\documentclass[10pt]{amsart}
\usepackage{amssymb,amsmath,amsfonts,amsthm,graphics,mathrsfs,amscd,hyperref}
\usepackage[hmargin=1in,vmargin=1in]{geometry}
\usepackage[all,cmtip]{xy}

\theoremstyle{plain}

\newtheorem{definition}[equation]{Definition}
\newtheorem{corollary}[equation]{Corollary}
\newtheorem{lemma}[equation]{Lemma}
\newtheorem{proposition}[equation]{Proposition}
\newtheorem{theorem}[equation]{Theorem}

\newtheorem{notation}{Notation}
\newtheorem{remark}[equation]{Remark}

\numberwithin{equation}{subsection}
\newtheorem{conjecture}{Conjecture}

\title{Deformations of holomorphic Poisson maps}
\author{Chunghoon Kim}

\email{ckim042@gmail.com}            

\begin{document}

\maketitle

\begin{abstract}
In this paper, we study deformations of holomorphic Poisson maps which extend Horikawa's series of papers \cite{Hor73},\cite{Hor74}, and \cite{Hor76} on deformations of holomorphic maps in the context of holomorphic Poisson deformations. In appendices, we present deformations of Poisson morphisms in the language of functors of Artin rings which is the algebraic version of deformations of holomorphic Poisson maps. We identity first-order deformations and obstructions.

\end{abstract}

\tableofcontents

\section{Introduction}

In this paper, we study deformations of holomorphic Poisson maps which extend Horikawa's series of papers \cite{Hor73},\cite{Hor74}, and \cite{Hor76} on deformations of holomorphic maps in the context of holomorphic Poisson deformations. We review deformation theory of holomorphic maps presented in \cite{Hor73},\cite{Hor74}, and \cite{Hor76} and explain how the theory can be extended in the context of holomorphic Poisson deformations.

 Let us review deformation theory of holomorphic maps presented in \cite{Hor73} and \cite{Hor74}. Horikawa defined a concept of a family of holomorphic maps of deformations of a holomorphic map $f:X\to Y$ of complex manifolds on the basis of Kodaira-Spencer's deformation theory. He studied two notions of a family of holomorphic maps as follows. Let $M$ and $S$ be two complex manifolds. 
 \begin{enumerate}
\item A family of holomorphic maps consists of a family $\{X_t|t\in M\}$ plus a collection $\{f_t:X_t\to Y|t\in M\}$ $(Y$ being fixed). \label{pp1}
\item A family of holomorphic maps consists of a family $\{X_t|t\in M\}$, a holomorphic map $s:M\to S$ and a collection $\{f_t:X_t\to Y_{s(t)}|t\in M\}$ ( $\mathcal{Y}=\{Y_{t'}\}_{t'\in S}\to S$ being fixed).\label{pp2}
\end{enumerate}
For each case, he studied infinitesimal deformations and proved theorem of completeness and existence. Let us consider the case $(\ref{pp1})$. For the precise statements, we recall the definitions of a family of holomorphic maps, and completeness of a family as follows.
\begin{definition}\label{pp10}
By a family of holomorphic maps into a compact complex manifold $Y$, we mean a collection $(\mathcal{X},\Phi, p,M)$ of complex manifold $\mathcal{X}$, a complex manifold $M$, and a holomorphic map $\Phi:\mathcal{X}\to \mathcal{Y}=Y\times M$, and $p:\mathcal{X}\to M$ with following properties:
\begin{enumerate}
\item $p$ is a surjective smooth proper holomorphic map.
\item $q\circ \Phi=p$, where $q:\mathcal{Y}\to M$ is the projection onto the second factor.
\end{enumerate}
Two families $(\mathcal{X}, \Phi,p, M)$ and $(\mathcal{X}',\Phi',p', M')$ of holomorphic maps into $Y$ are said to be equivalent if there exist holomorphic isomorphisms $\Psi:\mathcal{X}\to \mathcal{X}'$ and $\phi:M\to M'$ such the the following diagram is commutative
\begin{center}
$\begin{CD}
\mathcal{X}@>\Psi>> \mathcal{X}'\\
@V\Phi VV @VV\Phi'V\\
Y\times M@>id\times \phi>> Y\times M'
\end{CD}$
\end{center}
If $(\mathcal{X}, \Phi,p,M)$ is a family of holomorphic maps into $Y$, and if $h:N\to M$ is a holomorphic map, we can define the family $(\mathcal{X}',\Phi',p',N)$ induced by $h$ as follows:
\begin{enumerate}
\item $\mathcal{X}'=\mathcal{X}\times_M N$,
\item $\Phi'=\Phi\times id:\mathcal{X}'\to ( Y\times M)\times_M N=Y\times N,$
\item $p'=p_N:\mathcal{X}'\to N$.
\end{enumerate}
\end{definition}

\begin{definition}\label{pp11}
A family $(\mathcal{X},\Phi, p,M)$ of holomorphic maps into $Y$ is complete at $0\in M$ if, for any family $(\mathcal{X}',\Phi',p', N)$ such that $\Phi_{0'}':X_{0'}'\to Y$ is equivalent to $\Phi_0:X_0 \to Y$ for a point $0'\in N$, there exists a holomorphic map $h$ of a neighborhood $U$ of $0'$ in $N$ into $M$ with $h(0')=0$ such that the restriction of $(\mathcal{X}',\Phi', p', N)$ on $U$ is equivalent to the family induced by $h$ from $(\mathcal{X},\Phi, p, M)$.
\end{definition}

Let $(\mathcal{X},\Phi,p,M)$ be a family of holomorphic maps into $Y$, $0\in M$, $X=X_0$ and $f=\Phi_0:X\to Y$. Let $F:\Theta_X\to f^*\Theta_Y$ be the canonical homomorphism where $\Theta_X$ (resp. $\Theta_Y$) is the sheaf of germs of holomorphic vector fields on $X$ (resp. $Y$). Then we have an exact sequence of sheaves
\begin{align*}
0\to \Theta_{X/Y}\to \Theta_X\xrightarrow{F}f^*\Theta_Y\xrightarrow{P} \mathcal{N}_f\to 0
\end{align*}
where $\Theta_{X/Y}=ker(F:\Theta_X\to f^*\Theta_Y)$, and $\mathcal{N}_f=coker(F:\Theta_X\to f^*\Theta_Y)$.

Let $\mathcal{U}$ be a finite Stein covering of $X$. We set
\begin{align*}
D_{X/Y}=\frac{\{(\tau,\rho)\in C^0(\mathcal{U}, f^*\Theta_Y^\bullet)\oplus \mathcal{C}^1(\mathcal{U},\Theta_X)|\delta \tau=F\rho, \delta \rho=0\}}{\{(Fg, \delta g )|g\in C^0(\mathcal{U},\Theta_X)\}}
\end{align*}

Then infinitesimal deformations of $f:X\to Y$ in the family $(\mathcal{X},\Phi,p,M)$ are encoded in $D_{X/Y}$, and we can define the characteristic map
\begin{align*}
\tau:T_0(M)\to D_{X/Y}.
\end{align*}
We note that $D_{X,Y}\cong H^0(X, \mathcal{N}_f)$ when $f$ is non-degenerate (i.e $F$ is injective). In \cite{Hor73}, Horikawa proved theorem of completeness and existence for a non-degenerate holomorphic map, and in \cite{Hor74}, he proved theorem of completeness and existence for general case as follows.

\begin{theorem}[Theorem of completeness for deformations of holomorphic maps]\label{pp3}
Let $(\mathcal{X}, \Phi, p,M)$ be a family of holomorphic maps into $Y, 0\in M, X=X_0$, and $f=\Phi_0:X\to Y$. If the characteristic map $\tau:T_0(M)\to D_{X/Y}$ is surjective, then the family is complete at $0$.
\end{theorem}

\begin{theorem}[Theorem of existence for deformations of holomorphic maps]\label{pp4}
Let $f:X\to Y$ be a holomorphic map of a compact complex manifold $X$ into a complex manifold $Y$. Assume that the canonical homomorphism $F:\Theta_X\to f^*\Theta_Y$ satisfies the following conditions:
\begin{enumerate}
\item $F: H^1(X,\Theta_X)\to H^1(X, f^*\Theta_Y)$ is surjective,
\item $F:H^2(X,\Theta_X)\to H^2(X, f^*\Theta_Y)$ is injective.
\end{enumerate}
Then there exist a family $(\mathcal{X}, \Phi, p, M)$ of holomorphic maps into $Y$ and a point $0\in M$ such that 
\begin{enumerate}
\item $X=p^{-1}(0)$ and $\Phi_0:X\to Y$ coincides with $f:X\to Y$.
\item $\tau:T_0(M)\to D_{X/Y}$ is bijective.
\end{enumerate}
\end{theorem}

In section $\ref{section2}$, we extend the concept of a family of holomorphic maps, and prove an analogue of theorem of completeness (Theorem \ref{pp3}) and an analogue of theorem of existence (Theorem \ref{pp4}) in the context of holomorphic Poisson deformations. A holomorphic Poisson manifold $X$ is a complex manifold whose structure sheaf is a sheaf of Poisson algebras\footnote{We refer to \cite{Lau13} for general information on Poisson geometry}. A holomorphic Poisson structure is encoded in a holomorphic section (a holomorphic bivector field) $\Lambda\in H^0(X,\wedge^2 \Theta_X)$ with $[\Lambda,\Lambda]=0$, where the bracket $[-,-]$ is the Schouten bracket on $X$. In the sequel a holomorphic Poisson manifold will be denoted by $(X,\Lambda)$. Let $(X,\Lambda)$ and $(Y,\Pi)$ be two holomorphic Poisson manifolds. Then a holomorphic map $f:(X,\Lambda)\to (Y,\Pi)$ is called a holomorphic Poisson map if $f_* \Lambda_{x}=\Pi_{f(x)}$ for each $x\in X$, where $f_*:\wedge^2 \Theta_{X,x}\to \wedge^2 \Theta_{Y,f(x)}$ is the canonical map.

 It is natural to consider two notions of a family of holomorphic Poisson maps as analogues of case $(\ref{pp1})$ and $(\ref{pp2})$ as above in the context holomorphic Poisson deformations as follows. Let $M$ and $S$ be two complex manifolds.
\begin{enumerate}
\item A family of holomorphic Poisson maps consists of a family $\{(X_t,\Lambda_t)|t\in M\}$ plus a collection $\{f_t:(X_t,\Lambda_t)\to (Y,\Pi_0)|t\in M\}$ $((Y,\Pi_0)$ being fixed)
\item a family of holomorphic Poisson maps consists of a family $\{(X_t,\Lambda_t)|t\in M\}$, a holomorphic map $s:M\to S$ and a collection $\{f_t:(X_t,\Lambda_t)\to (Y_{s(t)},\Pi_{s(t)})|t\in M\}$ ($(\mathcal{Y},\Pi)=\{(Y_{t'},\Pi_{t'})\}_{t'\in S} \to S$ begin fixed.)\label{pp25}
\end{enumerate}
The case $(1)$ is the main topic of section $\ref{section2}$. We extend Definition $\ref{pp10}$ and Definition $\ref{pp11}$ to define concepts of a family of holomorphic Poisson maps, and completeness of a family (see Definition \ref{ll35}).

\begin{definition}
By a Poisson analytic family of holomorphic Poisson maps into a compact holomorphic Poisson manifold $(Y,\Pi_0)$, we mean a collection $(\mathcal{X},\Lambda,\Phi, p,M)$ of holomorphic Poisson manifold $(\mathcal{X},\Lambda)$, a complex manifold $M$, and a holomorphic Poisson map $\Phi:(\mathcal{X},\Lambda)\to (\mathcal{Y},\Pi_0)=(Y\times M,\Pi_0)$, and $p:(\mathcal{X},\Lambda)\to M$ with following properties:
\begin{enumerate}
\item $p$ is a surjective smooth proper holomorphic map, which makes $(\mathcal{X},\Lambda,p,M)$ a Poisson analytic family in the sense of \cite{Kim15}.
\item $q\circ \Phi=p$, where $q:(\mathcal{Y},\Pi_0)\to M$ is the projection onto the second factor which is the trivial Poisson analytic family over $M$.
\end{enumerate}
Two families $(\mathcal{X},\Lambda, \Phi, p, M)$ and $(\mathcal{X}',\Lambda', \Phi', p', M')$ of holomorphic Poisson maps into $(Y,\Pi_0)$ are said to be equivalent if there exist a holomorphic Poisson isomorphism $\Psi:(\mathcal{X},\Lambda)\to (\mathcal{X}',\Lambda')$ and a holomorphic isomorphism $\phi:M\to M'$ such the the following diagram is commutative
\begin{center}
$\begin{CD}
(\mathcal{X},\Lambda)@>\Psi>> (\mathcal{X}',\Lambda')\\
@V\Phi VV @VV\Phi'V\\
(Y\times M,\Pi_0)@>id\times \phi>> (Y\times M',\Pi_0)
\end{CD}$
\end{center}
If $(\mathcal{X},\Lambda, \Phi,p,M)$ is a family of holomorphic Poisson maps into $(Y,\Pi_0)$, and if $h:N\to M$ is a holomorphic map, we can define the family $(\mathcal{X}', \Lambda',\Phi',p',N)$ induced by $h$ as follows:
\begin{enumerate}
\item $(\mathcal{X}',\Lambda')=(\mathcal{X},\Lambda)\times_M N$ which is the induced Poisson analytic family by $h$ $($see \cite{Kim15}$)$.
\item $\Phi'=\Phi\times id:(\mathcal{X}',\Lambda')\to ( (Y\times M)\times_M N, \Pi_0)=(Y\times N, \Pi_0)$.
\item $p'=p_N:(\mathcal{X}',\Lambda')\to N$.
\end{enumerate}
\end{definition}

\begin{definition}
A family $(\mathcal{X},\Lambda,\Phi, p,M)$ of holomorphic Poisson maps into $(Y,\Pi_0)$ is complete at $0\in M$ if, for any family $(\mathcal{X}',\Lambda',\Phi',p', N)$ such that $\Phi_{0'}':(X_{0'}',\Lambda_{0'}')\to (Y,\Pi_0)$ is equivalent to $\Phi_0:(X_0,\Lambda_0)\to (Y,\Pi_0)$ for a point $0'\in N$, there exists a holomorphic map $h$ of a neighborhood $U$ of $0'$ in $N$ into $M$ with $h(0')=0$ such that the restriction of $(\mathcal{X}',\Lambda',\Phi', p', N)$ on $U$ is equivalent to the family induced by $h$ from $(\mathcal{X},\Lambda, \Phi, p, M)$.
\end{definition}

Let $(\mathcal{X},\Lambda, \Phi, p,M)$ be a family of holomorphic Poisson maps into $(Y,\Pi_0),0\in M,(X,\Lambda_0)=(X_0,\Lambda_0)$ and $f=\Phi_0:(X,\Lambda_0)\to (Y,\Pi_0)$. Then there is a complex of sheaves associated with $f$ on $X$ (see Appendix \ref{appendix1}):
\begin{align*}
f^*\Theta_Y^\bullet: f^* \Theta_Y\xrightarrow{\pi} \wedge^2 f^* \Theta_Y \xrightarrow{\pi}\wedge^3 f^* \Theta_Y\xrightarrow{\pi}\cdots
\end{align*}
such that the following diagram commutes
\begin{center}
$\begin{CD}
\Theta_X^\bullet :@. \Theta_X@>[\Lambda_0,-]>> \wedge^2 \Theta_X@>[\Lambda_0,-]>> \wedge^3 \Theta_X@>[\Lambda_0,-]>>\cdots\\
@.@VFVV @VFVV @VFVV\\
f^*\Theta_Y^\bullet:@.f^*\Theta_Y@> \pi >>  \wedge^2 f^*  \Theta_Y @> \pi >> \wedge^3 f^* \Theta_Y@> \pi >> \cdots
\end{CD}$
\end{center}
where $F:\Theta_X^\bullet\to f^*\Theta_Y^\bullet$ is the canonical homomorphism. Then we have an exact sequence of complex of sheaves
\begin{align*}
0\to \Theta_{X/Y}^\bullet \to \Theta_X^\bullet \xrightarrow{F} f^*\Theta_Y^\bullet \xrightarrow{P} \mathcal{N}_f^\bullet \to 0
\end{align*}
\begin{center}
$\begin{CD}
@.\cdots @. \cdots @.\cdots @. \cdots @.\\
@. @AAA @A[\Lambda_0,-]AA @A\pi AA @AAA\\
0@>>> \Theta_{X/Y}^3@>>> \wedge^3 \Theta_X@>F>> \wedge^3 f^*\Theta_Y@>P>> \mathcal{N}_f^3@>>> 0\\
@. @AAA @A[\Lambda_0,-]AA @A\pi AA @AAA\\
0@>>> \Theta_{X/Y}^2@>>> \wedge^2 \Theta_X@>F>> \wedge^2 f^*\Theta_Y@>P>> \mathcal{N}_f^2@>>> 0\\
@. @AAA @A[\Lambda_0,-]AA @A\pi AA @AAA\\
0 @>>> \Theta_{X/Y}@>>> \Theta_X@>F>> f^*\Theta_Y@>P>> \mathcal{N}_f@>>> 0
\end{CD}$
\end{center}
where $\Theta_{X/Y}^\bullet:=ker(\Theta_X^\bullet \xrightarrow{F} f^*\Theta_Y^\bullet)$ and $\mathcal{N}_f^\bullet:=coker(\Theta_X^\bullet\xrightarrow{F} f^*\Theta_Y)$.

Let $\mathcal{U}$ be a Stein covering of $X$. We set
\begin{align*}
PD_{(X,\Lambda_0)/(Y,\Pi_0)}=\frac{\{(\tau,\rho,\lambda)\in C^0(\mathcal{U}, f^*\Theta_Y^\bullet)\oplus \mathcal{C}^1(\mathcal{U},\Theta_X)\oplus C^0(\mathcal{U},\wedge^2 \Theta_X)|\frac{-\delta \tau=F\rho,\pi(\tau)=F\lambda,}{\delta\rho=0,\delta \lambda+[\Lambda_0,\rho]=0, [\Lambda_0,\lambda]=0}\}}{\{(Fg, -\delta g ,[\Lambda_0,g])|g\in C^0(\mathcal{U},\Theta_X)\}}
\end{align*}
Then infinitesimal deformations of $f:(X,\Lambda_0)\to (Y,\Pi_0)$ in the family $(\mathcal{X},\Lambda, \Phi,p,M)$ are encoded in $PD_{(X,\Lambda_0)/(Y,\Pi_0)}$, and we can define the characteristic map (see subsection \ref{jj1})
\begin{align*}
\tau:T_0(M)\to PD_{(X,\Lambda_0)/(Y,\Pi_0)}.
\end{align*}
We note that $PD_{(X,\Lambda_0)/(Y,\Pi_0)}\cong \mathbb{H}^0(X, \mathcal{N}_f^\bullet)$ when $f$ is non-degenerate.  We prove theorem of completeness (see Theorem \ref{ii1}) for deformations of holomorphic maps as follows.  

\begin{theorem}[Theorem of completeness for deformations of holomorphic Poisson maps]
Let $(\mathcal{X},\Lambda, \Phi, p,M)$ be a family of holomorphic Poisson maps into $(Y,\Pi_0), 0\in M, (X,\Lambda_0)=(X_0,\Lambda_0)$, and $f=\Phi_0:(X,\Lambda_0)\to (Y,\Pi_0)$. If the characteristic map $\tau:T_0(M)\to PD_{(X,\Lambda_0)/(Y,\Pi_0)}$ is surjective, then the family is complete at $0$.
\end{theorem}

We also prove theorem of existence for deformations of holomorphic Poisson maps under the assumption that the formal power series constructed in the proof converge.

\begin{theorem}[Theorem of existence for deformations of holomorphic Poisson maps]\label{pp43}
Let $f:(X,\Lambda_0)\to (Y,\Pi_0)$ be a holomorphic Poisson map of a compact holomorphic Poisson manifold $(X,\Lambda_0)$ into a holomorphic Poisson manifold $(Y,\Pi_0)$. Assume that the canonical homomorphism $F:\Theta_X^\bullet\to f^*\Theta_Y^\bullet$ satisfies the following conditions:
\begin{enumerate}
\item $F:\mathbb{H}^1(X,\Theta_X^\bullet)\to \mathbb{H}^1(X, f^*\Theta_Y^\bullet)$ is surjective,
\item $F:\mathbb{H}^2(X,\Theta_X^\bullet)\to \mathbb{H}^2(X, f^*\Theta_Y^\bullet)$ is injective.
\end{enumerate}
Then if the formal power series $(\ref{qq32})$ constructed in the proof converge, then  there exist a family $(\mathcal{X}, \Lambda, \Phi, p, M)$ of holomorphic Poisson maps into $(Y,\Pi_0)$ and a point $0\in M$ such that 
\begin{enumerate}
\item $(X,\Lambda_0)=p^{-1}(0)$ and $\Phi_0:(X,\Lambda_0)\to (Y, \Pi_0)$ coincides with $f:(X,\Lambda_0)\to (Y,\Pi_0)$.
\item $\tau:T_0(M)\to PD_{(X,\Lambda_0)/(Y,\Pi_0)}$ is bijective.
\end{enumerate}
\end{theorem}
 The author could not touch the convergence problem (see subsection \ref{pp40}). It seems that we can formally apply Horikawa's method in his proof of Theorem \ref{pp4} in the context of holomorphic Poisson deformations. But the author believes that we need a deep understanding of harmonic theories on the operators $L=\bar{\partial}+[\Lambda_0,-]$ on $\Theta_X^\bullet$ and $L_\pi=\bar{\partial}+\pi$ on $f^*\Theta_Y^\bullet$.

Next let us review the case $(\ref{pp2})$ in Horikawa's deformation theory presented in \cite{Hor74} and explain how the theory can be extended in the context of holomorphic Poisson deformations. Horikawa defined a notion of deformations of holomorphic maps into a complex analytic family and proved theorem of completeness and existence. For the precise statements, we recall the definition of a family of holomorphic maps into a family.
\begin{definition}  \label{pp30}
By a family of holomorphic maps into a complex analytic family $(\mathcal{Y},q,S)$, we mean a collection $(\mathcal{X},\Phi, p, M,s)$ where $(\mathcal{X}, p,M)$ is a complex analytic family, $\Phi:\mathcal{X}\to \mathcal{Y}$ is a holomorphic map, and $s:M\to S$ such that $s\circ p=q\circ \Phi$. Two families $(\mathcal{X}, \Phi,p,M,s)$ and $(\mathcal{X}',\Phi',p',M',s')$ are said to be equivalent if there exist holomorphic isomorphisms $g:\mathcal{X}\to \mathcal{X}'$, and $h:M\to M'$ such that the following diagram commutes
\[\xymatrix{
\mathcal{X} \ar[dd]^p \ar[drr]^g \ar[rrrr]^\Phi & & && \mathcal{Y} \ar[dd]^q\\
& & \mathcal{X}' \ar[dd]^{p'}   \ar[rru]^{\Phi'} \\
M  \ar[drr]^h \ar[rrrr]^s  & & & &S\\      
& &M' \ar[rru]^{s'} &       \\
}\]
\end{definition}
We can define induced families and the concept of completeness in a similar way as in Definition \ref{pp10}, and Definition \ref{pp11}.

Let $(\mathcal{X},\Phi, p,M,s)$ be a family of holomorphic maps into $(\mathcal{Y},q,S), 0\in M, X=X_0$ and let $\tilde{f}:X\to \mathcal{Y}$ be the restriction of $\Phi$ to $X$. Let $0^*=s(0), Y=Y_{0^*}$, and let $f:X\to Y$ be the holomorphic map induced by $\Phi$. Let $\tilde{F}:\Theta_X\to \tilde{f}^*\Theta_\mathcal{Y}$ be the canonical homomorphism induced by $\tilde{f}$.

Let $\mathcal{U}$ be a Stein covering of $X$, then we have
\begin{align*}
D_{X/\mathcal{Y}}=\frac{\{(\tilde{\tau},\rho) \in C^0(\mathcal{U},\tilde{f}^*\Theta_{\mathcal{Y}})\oplus \mathcal{C}^1(\mathcal{U},\Theta_X)|\delta \tilde{\tau}=\tilde{F}\rho, \delta\rho=0\}}{\{(\tilde{F}g,\delta g)|g\in C^0(\mathcal{U}, \Theta_X)\}}
\end{align*}
Then infinitesimal deformations of $f:X\to Y$ in the family $(\mathcal{X},\Phi, p,M,s)$ are encoded in $D_{X/\mathcal{Y}}$, and we have the characteristic map
\begin{align*}
\tau:T_0(M)\to D_{X/\mathcal{Y}}.
\end{align*}
In \cite{Hor74}, Horikawa proved theorem of completeness and existence for deformations of holomorphic maps into a family as follows.

\begin{theorem}[Theorem of completeness]
Let $(\mathcal{X},\Phi,p, M, s)$ be a family of holomorphic maps into a family $(\mathcal{Y},q,S),0\in M, X=X_0$, and let $\tilde{f}:X\to \mathcal{Y}$ be the restriction of $\Phi$ to $X$. If the characteristic map $\tau:T_0(M)\to D_{X/\mathcal{Y}}$ is surjective, then the family is complete at $0$.
\end{theorem}

\begin{theorem}[Theorem of existence for deformations of holomorphic maps into a family]\label{qq22}
Let $f:X\to Y$ be a holomorphic map of compact complex manifolds. Let $(\mathcal{Y}, q,S)$ be a complex analytic family, $0^*\in S, Y=Y_{0^*}$, and let $\rho':T_{0^*}(S)\to H^1(Y,\Theta_Y)$ be the Kodaira-Spencer map of $(\mathcal{Y},q,S)$. Assume that
\begin{enumerate}
\item $H^1(X,f^* \Theta_Y)$ is generated by the image of $F:H^1(X,\Theta_X  )  \to H^1(X, f^*\Theta_Y)$ and the image of $f^*\circ \rho': T_{0^*}(S)\to H^1(X, f^* \Theta_Y)$.
\item $F: H^2(X, \Theta_X)\to H^2(X, f^*\Theta_Y)$ is injective.
\end{enumerate}
Then there exists a family $(\mathcal{X},\Phi,p, M,s)$ of holomorphic maps into $(\mathcal{Y},q,S)$ and a point $0\in M$ such that
\begin{enumerate}
\item $s(0)=0^*,X=p^{-1}(0)$ and $\Phi_0$ coincides with $f$,
\item $\tau: T_0(M)\to D_{X/Y}$ is bijective.
\end{enumerate}

\end{theorem}

In section \ref{section3}, we extend the concept of a family of holomorphic maps into a family in the context of holomorphic Poisson deformations. In other words, the situation (\ref{pp25}) is the main topic of section \ref{section3}. We extend Definition \ref{pp30} to define a concept of a family of holomorphic Poisson maps into a Poisson analytic family (see Definition \ref{pp27}).

\begin{definition}
By a Poisson analytic family of holomorphic Poisson maps into a Poisson analytic family $(\mathcal{Y},\Pi,q,S)$, we mean a collection $(\mathcal{X},\Lambda,\Phi, p, M,s)$ where $(\mathcal{X},\Lambda, p,M)$ is a Poisson analytic family, $\Phi:(\mathcal{X},\Lambda)\to (\mathcal{Y},\Pi)$ is a holomorphic Poisson map, and $s:M\to S$ such that $s\circ p=q\circ \Phi$. Two families $(\mathcal{X},\Lambda, \Phi,p,M,s)$ and $(\mathcal{X}',\Lambda',M',s')$ are said to be equivalent if there exist a holomorphic Poisson isomorphism $g:(\mathcal{X},\Lambda)\to (\mathcal{X}',\Lambda')$ and an isomorphism $h:M\to M'$ such that the following diagram commutes
\[\xymatrix{
(\mathcal{X},\Lambda) \ar[dd]^p \ar[drr]^g \ar[rrrr]^\Phi & & && ( \mathcal{Y},\Pi) \ar[dd]^q\\
& & (\mathcal{X}',\Lambda') \ar[dd]^{p'}   \ar[rru]^{\Phi'} \\
M  \ar[drr]^h \ar[rrrr]^s  & & & &S\\      
& &M' \ar[rru]^{s'} &       \\
}\]
\end{definition}
We can define induced families and the concept of completeness (see Definition \ref{pp27} and Definition \ref{pp28}). 

Let $(\mathcal{X}, \Lambda, \Phi, p, M, s)$ be a Poisson analytic family of holomorphic Poisson maps into $(\mathcal{Y},\Pi, q, S)$, $0\in M$, $(X,\Lambda_0)=(X_0,\Lambda_0)$, and let $\tilde{f}:(X,\Lambda_0)\to (\mathcal{Y},\Pi)$ be the restriction of $\Phi$ to $(X,\Lambda_0)$. Let $\tilde{F}:\Theta_X^\bullet \to \tilde{f}^* \Theta_{\mathcal{Y}}^\bullet$ be the canonical homomorphism. Let $0^*=s(0),(Y,\Pi_0)=(Y_{0^*},\Pi_{0^*})$, and let $f:(X,\Lambda_0)\to (Y,\Pi_0)$ be the holomorphic Poisson map induced by $\Phi$. If $\mathcal{U}=\{U_i\}$ is a Stein covering of $X$, then we have
\begin{align*}
PD_{(X,\Lambda_0)/(\mathcal{Y},\Pi)}=\frac{\{(\tilde{\tau},\rho,\lambda) \in C^0(\mathcal{U},\tilde{f}^*\Theta_{\mathcal{Y}})\oplus\mathcal{C}^1(\mathcal{U},\Theta_X)\oplus C^0(\mathcal{U},\wedge^2 \Theta_X)|\frac{-\delta \tilde{\tau}=\tilde{F}\rho, \pi( \tilde{\tau} ) =\tilde{F}\lambda}{\delta\rho=0, \delta \lambda + [\Lambda_0, \rho]=0, [\Lambda_0,\lambda]=0 }\}}{\{(\tilde{F}g,-\delta g,[\Lambda_0, g])|g\in C^0(\mathcal{U}, \Theta_X)\}}
\end{align*}
Then infinitesimal deformations of $f:(X,\Lambda_0)\to (Y,\Pi_0)$ in the family $(\mathcal{X},\Lambda, \Phi, p, M,s)$ are encoded in $PD_{(X,\Lambda_0)/(\mathcal{Y},\Pi)}$, and we can define the characteristic map (see subsection \ref{qq20})
\begin{align*}
\tau:T_0(M)\to PD_{(X,\Lambda_0)/(\mathcal{Y},\Pi)}.
\end{align*}

We prove theorem of completeness (see Theorem \ref{ii5}) for deformations of holomorphic Poisson maps into a Poisson analytic family as follows.

\begin{theorem}[Theorem of completeness for deformations of holomorphic Poisson maps into a family]
Let $(\mathcal{X},\Lambda,\Phi,p, M, s)$ be a family of holomorphic Poisson maps into a family $(\mathcal{Y},\Pi, q,S),0\in M, (X,\Lambda_0)=(X_0,\Lambda_0)$, and let $\tilde{f}:(X,\Lambda_0)\to (\mathcal{Y},\Pi)$ be the restriction of $\Phi$ to $(X,\Lambda_0)$. If the characteristic map $\tau:T_0(M)\to PD_{(X,\Lambda_0)/(\mathcal{Y},\Pi)}$ is surjective, then the family is complete at $0$.
\end{theorem}

We prove theorem of existence (see Theorem \ref{ii80}) for deformations of holomorphic Poisson maps into a Poisson analytic family under the assumption that the formal power series constructed in the proof converge as in Theorem \ref{pp43}.

\begin{theorem}[Theorem of existence for deformations of holomorphic Poisson maps into a family]\label{qq21}
Let $f:(X,\Lambda_0)\to (Y, \Pi_0)$ be a holomorphic Poisson map of compact holomorphic Poisson manifolds. Let $(\mathcal{Y},\Pi, q,S)$ be a Poisson analytic family, $0^*\in S, (Y,\Pi_0)=(Y_{0^*},\Pi_{0^*})$, and let $\rho':T_{0^*}(S)\to \mathbb{H}^1(Y,\Theta_Y^\bullet)$ be the Poisson Kodaira-Spencer map of $(\mathcal{Y},\Pi,q,S)$ $($see \cite{Kim15}$)$. Assume that
\begin{enumerate}
\item $\mathbb{H}^1(X,f^* \Theta_Y^\bullet)$ is generated by the image of $F:\mathbb{H}^1(X,\Theta_X^\bullet  )  \to \mathbb{H}^1(X, f^*\Theta_Y^\bullet)$ and the image of $f^*\circ \rho': T_{0^*}(S)\to \mathbb{H}^1(X, f^* \Theta_Y^\bullet )$.
\item $F: \mathbb{H}^2(X, \Theta_X^\bullet)\to \mathbb{H}^2(X, f^*\Theta_Y^\bullet)$ is injective.
\end{enumerate}
If the formal power series $(\ref{qq50})$ constructed in the proof converge, then there exists a family $(\mathcal{X},\Lambda,\Phi,p, M,s)$ of holomorphic Poisson maps into $(\mathcal{Y},\Pi,q,S)$ and a point $0\in M$ such that
\begin{enumerate}
\item $s(0)=0^*,(X,\Lambda_0)=p^{-1}(0)$ and $\Phi_0$ coincides with $f$,
\item $\tau: T_0(M)\to PD_{(X,\Lambda_0)/(\mathcal{Y},\Pi)}$ is bijective.
\end{enumerate}

\end{theorem}

In \cite{Hor74}, Horikawa studied a stability problem of a holomorphic map over a complex manifold. Let $\{Y_t\}_{t\in M}$ be a family of deformations of complex manifolds with $Y=Y_0, 0\in M$, and let $f:X\to Y$ be a holomorphic map of a compact complex manifold $X$ into a complex manifold $Y$. The concept of stability lies in the question whether we can embed $X$ into a family $\{X_t\}_{t\in M}$ in such a manner that $f$ extends to a family of holomorphic maps $f_t:X_t\to Y_t,t\in M$. Horikawa proved the following theorem as a corollary of Theorem \ref{qq22}.
\begin{theorem}[Theorem of stability for holomorphic maps over complex manifolds] \label{ll10}
Let $f:X\to Y$ be a holomorphic map of a compact complex manifold $X$ into a complex manifold $Y$. Assume that
\begin{enumerate}
\item $F:H^1(X,\Theta_X)\to H^1(X, f^*\Theta_Y)$ is surjective,
\item $F:H^2(X,\Theta_X)\to H^2(X, f^*\Theta_Y)$ is injective.
\end{enumerate}
Then for any family $q:\mathcal{Y}\to M$ of complex manifolds such that $Y=q^{-1}(0)$ for some point $0\in M$, there exist
\begin{enumerate}
\item an open neighborhood $N$ of $0$,
\item a family $p:\mathcal{X}\to N$ of deformations of $X=p^{-1}(0)$,
\item a holomorphic map $\Phi:\mathcal{X}\to \mathcal{Y}|_N$ which induces $f$ over $0\in N$.
\end{enumerate}
\end{theorem}

Let us consider the above arguments in the context of holomorphic Poisson deformations and extend the concept of stability in terms of  a holomorphic Poisson map over a holomorphic Poisson manifold. Let $\{(Y_t, \Pi_t)\}_{t\in M}$ be a family of deformations of holomorphic Poisson manifolds with $(Y,\Pi_0)=(Y_i,\Pi_0),0\in M$, and let $f:(X,\Lambda_0)\to (Y,\Pi_0)$ be a holomorphic Poisson map of a compact holomorphic Poisson manifold $(X,\Lambda_0)$ into a holomorphic Poisson manifold $(Y,\Pi_0)$. It is natural to ask if we can embed $(X,\Lambda_0)$ into a family $\{(X_t,\Lambda_t)\}_{t\in M}$ in such a manner that $f$ extends to a family of holomorphic Poisson maps $f_t:(X_t,\Lambda_t)\to (Y_t,\Pi_t),t\in M$.

In section \ref{section4}, we prove an analogue of Theorem \ref{ll10} as a corollary of Theorem \ref{qq21} so that the author could not touch convergence problem. We prove (see Theorem \ref{stability})

\begin{theorem}[Theorem of stability of holomorphic Poisson maps]
Let $f:(X,\Lambda_0)\to (Y,\Pi_0)$ be a holomorphic Poisson map of compact holomorphic Poisson manifolds, where $X$ is compact. Assume that
\begin{enumerate}
\item $F:\mathbb{H}^1(X, \Theta_X^\bullet)\to \mathbb{H}^1(X, f^*\Theta_Y^\bullet)$ is surjective.
\item $F:\mathbb{H}^2(X, \Theta_X^\bullet)\to \mathbb{H}^2(X, f^*\Theta_Y^\bullet)$ is injective.
\end{enumerate}
Then for any family $q:(\mathcal{Y}, \Pi)\to M$ of holomorphic Poisson manifolds such that $(Y,\Pi_0)=q^{-1}(0)$ for some point $0\in M$, if the formal power series $(\ref{ll4})$ constructed in the proof converge, there exist
\begin{enumerate}
\item an open neighborhood $N$ of $0$,
\item a family $p:(\mathcal{X},\Lambda)\to N$ of deformations of $(X,\Lambda_0)=p^{-1}(0)$,
\item a holomorphic Poisson map $p:(\mathcal{X},\Lambda)\to (\mathcal{Y}|_N, \Pi|_N)$ over $N$ which induces $f$ over $0\in N$.
\end{enumerate}
\end{theorem}
In Appendix \ref{appendix3}, we also  prove theorem of stability of a Poisson morphism over a nonsingular Poisson variety by algebraic methods in the case of infinitesimal deformations over local aritinian $k$-algebras (see Theorem \ref{mm1}).

In \cite{Hor74}, Horikawa studied deformations of a composition of holomorphic maps and proved the following two theorems.

\begin{theorem}\label{ll15}
Let $f:X\to Y,g:Y\to Z$ and $h=g\circ f$ be holomorphic maps of complex manifolds. We assume that
\begin{enumerate}
\item $X$ and $Y$ are compact,
\item $g$ is non-degenerate and the canonical homomorphism $f^*G:f^*\Theta_Y\to h^*\Theta_Z$ is injective, where $G$ denote the homomorphism $\Theta_Y\to g^*\Theta_Z$,
\item there exist a family $(\mathcal{Y},\Psi, q, N)$ of holomorphic maps into $Z$ and a point $0'\in N$ such that $Y=q^{-1}(0')$ and such that $\Psi$ induces $g$ on $Y$,
\item the composition $f^*\circ \tau:T_{0'}(N)\to H^0(X, f^*\mathcal{N}_{g})$ is surjective, where $\tau:T_{0'}(N)\to H^0(X, f^* \mathcal{N}_{g})$ is the characteristic map of $(\mathcal{Y}, \Psi, q, N)$ at $0'$, and $f^*:H^0(Y, \mathcal{N}_{g})\to H^0(X, f^*\mathcal{N}_{g})$ is the pull-back homomorphism.
\end{enumerate}
Let $(\mathcal{X},\Upsilon,p, M)$ be a family of holomorphic maps into $Z$ and let $0$ be a point on $M$ such that $X=p^{-1}(0)$ and that $\Upsilon$ induces $h$ on $X$. Then there exist
\begin{enumerate}
\item an open neighborhood $M'$ of $0$,
\item a holomorphic map $s:M'\to N$,
\item a holomorphic map $\Phi:\mathcal{X}|_{M'}\to \mathcal{Y}$, 
such that $s(0)=0'$, $\Phi$ induces $f$ on $X$, and the diagram
\end{enumerate}
\[\xymatrix{
\mathcal{X}|_{M'} \ar[dd]^{p|_{M'}} \ar[drr]^{\Phi} \ar[rrrr]^{\Upsilon|_{M'}} & & && Z\times N \ar[dd]^{pr_2}\\
& & \mathcal{Y} \ar[dd]^{q}   \ar[rru]^{\Psi} \\
M'  \ar[drr]^s \ar[rrrr]^s  & & & &N\\      
& &N \ar[rru]^{id} &       \\
}\]
commutes.

\end{theorem}

\begin{theorem}\label{ll16}
Let $f:X,\to Y, g:Y\to Z$, and $h=g\circ f$ be holomorphic maps of complex manifolds. Let $p:\mathcal{X}\to M, q:\mathcal{Y}\to M$ and $\pi:\mathcal{Z} \to M$ be families of complex manifolds such that $X=p^{-1}(0), Y=q^{-1}(0)$ and $Z=\pi^{-1}(0)$ for some point $0\in M$. Let $\Phi:\mathcal{X} \to \mathcal{Y}$ and $\Upsilon: \mathcal{X} \to \mathcal{Z} $ be holomorphic maps over $M$ which induces $f$ and $h$ over $0\in M$, respectively. Assume that
\begin{enumerate}
\item $p$ and $q$ are proper.
\item $f^*:H^0(Y, g^*\Theta_Z)\to H^0(X,h^*\Theta_Z)$ is surjective.
\item $f^*: H^1(Y, g^*\Theta_Z)\to H^1(X, h^*\Theta_Z)$ is injective.
\end{enumerate}
Then there exists an open neighborhood $N$ of $0$, and a holomorphic map $\Psi:\mathcal{Y}|_N\to \mathcal{Z}$ over $N$ such that $\Upsilon|_N=\Psi\circ (\Phi|_N)$.
\end{theorem}

In section $\ref{section5}$, we prove an analogue of Theorem \ref{ll15}, and an analogue of Theorem \ref{ll16} (see Theorem \ref{ll17} and Theorem \ref{ll18})

\begin{theorem}
Let $f:(X,\Lambda_0)\to (Y,\Pi_0), g:(Y,\Pi_0)\to (Z, \Omega_0)$ and $h=g\circ f$. We assume that
\begin{enumerate}
\item $X$ and $Y$ are compact. 
\item $g$ is non-degenerate and the canonical homomorphism $f^* G:f^*\Theta_Y^\bullet \to h^* \Theta_Z^\bullet$ is injective, where $G$ denotes the injective homomorphism $\Theta_Y^\bullet \to g^*\Theta_Z^\bullet$ $($see Appendix \ref{appendix1}$)$.
\item there exist a family $(\mathcal{Y},\Pi, \Psi,q, N)$ of holomorphic Poisson maps into $(Z, \Omega_0)$ and a point $0'\in N$ such that $(Y,\Pi_0)=q^{-1}(0')$ and such that $\Psi$ induces $g$ on $(Y,\Pi_0)$.
\item the composition $f^*\circ \tau:T_{0'}(N)\to\mathbb{H}^0(X, f^*\mathcal{N}_g^\bullet)$ is surjective, where $\tau:T_{0'}(N)\to \mathbb{H}^0(Y, \mathcal{N}_g^\bullet)$ is the characteristic map of $(\mathcal{Y}, \Pi, \Psi, q, N)$ at $0'$, and $f^*:\mathbb{H}^0(Y,\mathcal{N}_g^\bullet)\to \mathbb{H}^0(X, f^*\mathcal{N}_g^\bullet)$ is the pullback homomorphism.
\end{enumerate}
Let $(\mathcal{X},\Lambda, \Upsilon, p, M)$ be a family of holomorphic Poisson maps into $(Z,\Omega_0)$ and let $0$ be a point in $M$ such that $(X,\Lambda_0)=p^{-1}(0)$ and that $\Upsilon$ induces $h$ on $X$. Then there exist 
\begin{enumerate}
\item an open neighborhood $M'$ of $0$,
\item a holomorphic map $s:M'\to N$,
\item a holomorphic Poisson map $\Phi:(\mathcal{X}|_{M'},\Lambda|_{M'})\to (\mathcal{Y},\Pi)$,
\end{enumerate}
such that $s(0)=0'$, $\Phi$ induces $f$ on $(X,\Lambda_0)$, and the diagram
\[\xymatrix{
(\mathcal{X}|_{M'},\Lambda|_{M'}) \ar[dd]^{p|_{M'}} \ar[drr]^{\Phi} \ar[rrrr]^{\Upsilon|_{M'}} & & && (Z\times N , \Omega_0) \ar[dd]^{pr_2}\\
& & ( \mathcal{Y} , \Pi) \ar[dd]^{q}   \ar[rru]^{\Psi} \\
M'  \ar[drr]^s \ar[rrrr]^s  & & & &N\\      
& &N \ar[rru]^{id} &       \\
}\]
commutes.
\end{theorem}

\begin{theorem}
Let $f:(X,\Lambda_0)\to (Y, \Pi_0), g:(Y,\Pi_0)\to (Z, \Omega_0)$, and $h=g\circ f$ be holomorphic Poisson maps of holomorphic Poisson manifolds. Let $p:(\mathcal{X},\Lambda)\to M, q:(\mathcal{Y},\Pi)\to M$ and $\pi:(\mathcal{Z},\Omega)\to M$ be families of holomorphic Poisson manifolds such that $(X,\Lambda_0)=p^{-1}(0), (Y,\Pi_0)=q^{-1}(0)$ and $(Z, \Omega_0)=\pi^{-1}(0)$ for some point $0\in M$. Let $\Phi:(\mathcal{X},\Lambda)\to (\mathcal{Y}, \Pi)$ and $\Upsilon:(\mathcal{X}, \Lambda)\to (\mathcal{Z}, \Omega)$ be holomorphic Poisson maps over $M$ which induces $f$ and $h$ over $0\in M$, respectively. Assume that
\begin{enumerate}
\item $p$ and $q$ are proper.
\item $f^*:\mathbb{H}^0(Y, g^*\Theta_Z^\bullet)\to \mathbb{H}^0(X,h^*\Theta_Z^\bullet)$ is surjective $($see Appendix \ref{appendix1}$)$.
\item $f^*:\mathbb{H}^1(Y, g^*\Theta_Z^\bullet)\to \mathbb{H}^1(X, h^*\Theta_Z^\bullet)$ is injective.
\end{enumerate}
Then there exists an open neighborhood $N$ of $0$, and a holomorphic Poisson map $\Psi:(\mathcal{Y}|_N, \Pi|_N)\to (\mathcal{Z}, \Omega)$ over $N$ such that $\Upsilon|_N=\Psi\circ (\Phi|_N)$.
\end{theorem}

In \cite{Hor76}, Horikawa studied a concept of costability. Let $\{X_t\}$ be a family of deformations of compact complex manifolds, $X=X_0,0\in M$, and let $f:X\to Y$ be a holomorphic map of compact complex manifolds. Horikawa asked if we can embed $Y$ into a complex analytic family in such a manner that $f$ extends to a family of holomorphic maps $f_t:X_t\to Y_t$, and proved the following theorem.
\begin{theorem}[Theorem of Costability for deformations of holomorphic maps]\label{costability22}
Let $X$ and $Y$ be compact complex manifolds, $f:X\to Y$ a holomorphic map, and let $p:\mathcal{X}\to M$ be a family of deformations of $X=p^{-1}(0),0\in M$. Assume that
\begin{enumerate}
\item $f^*:H^1(Y,\Theta_Y)\to \mathbb{H}^1(X, f^*\Theta_Y)$ is surjective.
\item $f^*:H^2(Y,\Theta_Y)\to \mathbb{H}^2(X, f^*\Theta_Y)$ is injective.
\end{enumerate}
 Then there exists an open neighborhood $N$ of $0$ in $M$, a family $q:\mathcal{Y}\to N$ of deformations of $Y=q{-1}(0)$, and a holomorphic map $\Phi:\mathcal{X}|_N\to \mathcal{Y}$ over $N$ which induces $f$ over $0\in N$.
\end{theorem}

Let us consider the above arguments in terms of holomorphic Poisson deformations. Let $\{(X_t,\Lambda_t)\}$ be a family of deformations of compact holomorphic Poisson manifolds, $(X,\Lambda_0)=(X_0,\Lambda_0)$, $ 0\in M$, and let $f:(X,\Lambda_0)\to (Y,\Pi_0)$ be a holomorphic Poisson map of $(X,\Lambda_0)$ into a compact holomorphic Poisson manifold $(Y,\Pi_0)$. It is natural to ask if we can embed $(Y,\Pi_0)$ into a family $(Y_t,\Pi_t)$ in such a manner that $f$ extends to a family of holomorphic Poisson maps $f_t:(X_t,\Lambda_t)\to (Y_t,\Pi_t),t\in M$. We conjecture an analogue of Theorem \ref{costability22} .

\begin{conjecture}[Theorem of Costatbiliy for deformations of holomorphic Poisson maps]
Let $(X,\Lambda_0)$ and $(Y,\Pi_0)$ be two compact holomorphic Poisson manifolds, $f:(X,\Lambda_0)\to (Y,\Pi_0)$ a holomorphic Poisson map, and let $p:(\mathcal{X},\Lambda)\to M$ be a Poisson analytic family of deformations of $(X,\Lambda_0)=p^{-1}(0),0\in M$.
Assume that
\begin{enumerate}
\item $f^*:\mathbb{H}^1(Y,\Theta_Y^\bullet)\to \mathbb{H}^1(X, f^*\Theta_Y^\bullet)$ is surjective $($see Appendix \ref{appendix1}$)$.
\item $f^*:\mathbb{H}^2(Y,\Theta_Y^\bullet)\to \mathbb{H}^2(X, f^*\Theta_Y^\bullet)$ is injective.
\end{enumerate}
 Then there exists an open neighborhood $N$ of $0$ in $M$, a Poisson analytic family $q:(\mathcal{Y},\Pi)\to N$ of deformations of $(Y,\Pi_0)=q^{-1}(0)$, and a holomorphic Poisson map $\Phi:(\mathcal{X},\Lambda)|_N\to (\mathcal{Y},\Pi)$ over $N$ which induces $f$ over $0\in N$.
\end{conjecture}

The author could not extend Horikawa's methods to prove theorem of costability for deformations of holomorphic Poisson maps since in order to prove Theorem \ref{costability22} he used the existence of Kuranishi space, and other spaces which have not been established in Poisson category. Though we can approach in an elementary meothd and so construct formal power series, we meet the difficulty in proving convergence as in \cite{Kod05} p.248-259. However, we can prove  theorem of costabiliy in the case of infinitesimal deformations of Poisson maps over local aritinan $k$-algebras. In Appendix \ref{appendix3}, we prove (see Theorem \ref{mm15})

\begin{theorem}[Theorem of Costability for infinitesimal deformations of Poisson maps]
Let $(X,\Lambda_0)$ and $(Y,\Pi_0)$ be two nonsingular projective Poisson varieties, $f:(X,\Lambda_0)\to (Y,\Pi_0)$ a Poisson map, and let $p:(\mathcal{X},\Lambda)\to Spec(A)$ be a flat Poisson deformation of $(X,\Lambda_0)$ over $Spec(A)$ for a local artinian $k$-algebra $A$ with the residue $k$. Assume that
Assume that
\begin{enumerate}
\item $f^*:\mathbb{H}^1(Y,T_Y^\bullet)\to \mathbb{H}^1(X, f^*T_Y^\bullet)$ is surjective\footnote{For algebraic case, we use $T_X$ instead of $\Theta_X$ for the tangent sheaf to keep notational consistency with \cite{Kim16}}.
\item $f^*:\mathbb{H}^2(Y,T_Y^\bullet)\to \mathbb{H}^2(X, f^*T_Y^\bullet)$ is injective.
\end{enumerate}
Then there exist a flat Poisson deformation  $q:(\mathcal{Y},\Pi)\to Spec( A)$ of $(Y,\Pi_0)$, and a Poisson map $\Phi:(\mathcal{X},\Lambda)\to (\mathcal{Y},\Pi)$ over $A$ which induces $f$.
\end{theorem}

In appendices, we present deformations of Poisson morphisms in the language of functors of Artin rings which is the algebraic version of deformations of holomorphic Poisson maps. We identity first-order deformations and obstructions.

\section{Deformations of holomorphic Poisson maps}\label{section2}

\begin{definition}\label{ll35}
By a Poisson analytic family of holomorphic Poisson maps into a compact holomorphic Poisson manifold $(Y,\Pi_0)$, we mean a collection $(\mathcal{X},\Lambda,\Phi, p,M)$ of holomorphic Poisson manifold $(\mathcal{X},\Lambda)$, a complex manifold $M$, and a holomorphic Poisson map $\Phi:(\mathcal{X},\Lambda)\to \mathcal{Y}=(Y\times M,\Pi_0)$, and $p:(\mathcal{X},\Lambda)\to M$ with following properties:
\begin{enumerate}
\item $p$ is a surjective smooth proper holomorphic map, which makes $(\mathcal{X},\Lambda,p,M)$ a Poisson analytic family in the sense of \cite{Kim15}.
\item $q\circ \Phi=p$, where $q:(\mathcal{Y},\Pi_0)\to M$ is the projection onto the second factor which is the trivial Poisson analytic family over $M$.
\end{enumerate}
Two families $(\mathcal{X},\Lambda, \Phi, p, M)$ and $(\mathcal{X}',\Lambda',, \Phi',p', M')$ of holomorphic Poisson maps into $(Y,\Pi_0)$ are said to be equivalent if there exist a holomorphic Poisson isomorphism $\Psi:(\mathcal{X},\Lambda)\to (\mathcal{X}',\Lambda')$ and a holomorphic isomorphism $\phi:M\to M'$ such the the following diagram is commutative
\begin{center}
$\begin{CD}
(\mathcal{X},\Lambda)@>\Psi>> (\mathcal{X}',\Lambda')\\
@V\Phi VV @VV\Phi'V\\
(Y\times M,\Pi_0)@>id\times \phi>> (Y\times M',\Pi_0)
\end{CD}$
\end{center}
If $(\mathcal{X},\Lambda, \Phi,p,M)$ is a family of holomorphic Poisson maps into $(Y,\Pi_0)$, and if $h:N\to M$ is a holomorphic map, we can define the family $(\mathcal{X}', \Lambda',\Phi',p',N)$ induced by $h$ as follows:
\begin{enumerate}
\item $(\mathcal{X}',\Lambda')=(\mathcal{X},\Lambda)\times_M N$ which is the induced Poisson analytic family by $h$ $($see \cite{Kim15}$)$.
\item $\Phi'=\Phi\times id:(\mathcal{X}',\Lambda')\to ( (Y\times M)\times_M N, \Pi_0)=(Y\times N, \Pi_0)$.
\item $p'=p_N:(\mathcal{X}',\Lambda')\to N$.
\end{enumerate}
In particular, if $N$ is a submanifold of $M$ and if $h$ is the natural injection, we call $(\mathcal{X}',\Lambda',\Phi', p', N)$ the restriction on $N$ and denote it by $(\mathcal{X}|_N, \Lambda|_N, \Phi |_N, p|_N, N)$.
\end{definition}

\begin{definition}
A family $(\mathcal{X},\Lambda,\Phi, p,M)$ of holomorphic Poisson maps into $(Y,\Pi_0)$ is complete at $0\in M$ if, for any family $(\mathcal{X}',\Lambda',\Phi',p', N)$ such that $\Phi_{0'}':(X_{0'}',\Lambda_{0'}')\to (Y,\Pi_0)$ is equivalent to $\Phi_0:(X_0,\Lambda_0)\to (Y,\Pi_0)$ for a point $0'\in N$, there exists a holomorphic map $h$ of a neighborhood $U$ of $0'$ in $N$ into $M$ with $h(0')=0$ such that the restriction of $(\mathcal{X}',\Lambda',\Phi', p', N)$ on $U$ is equivalent to the family induced by $h$ from $(\mathcal{X},\Lambda, \Phi, p, M)$.
\end{definition}

\subsection{Complex associated with a holomorphic Poisson map}\

Let $f:(X,\Lambda_0)\to (Y,\Pi_0)$ be a holomorphic Poisson map. Let $F:\Theta_X^\bullet\to f^*\Theta_Y^\bullet$ be the canonical homomorphism associated with $f$ (see Appendix \ref{appendix1}). Then we have an exact sequence of complex of sheaves on $X$
\begin{align*}
0\to \Theta_{X/Y}^\bullet \to \Theta_X^\bullet \xrightarrow{F} f^*\Theta_Y^\bullet \xrightarrow{P} \mathcal{N}_f^\bullet \to 0
\end{align*}
\begin{center}
$\begin{CD}
@.\cdots @. \cdots @.\cdots @. \cdots @.\\
@. @AAA @A[\Lambda_0,-]AA @A\pi AA @AAA\\
0@>>> \Theta_{X/Y}^3@>>> \wedge^3 \Theta_X@>F>> \wedge^3 f^*\Theta_Y@>P>> \mathcal{N}_f^3@>>> 0\\
@. @AAA @A[\Lambda_0,-]AA @A\pi AA @AAA\\
0@>>> \Theta_{X/Y}^2@>>> \wedge^2 \Theta_X@>F>> \wedge^2 f^*\Theta_Y@>P>> \mathcal{N}_f^2@>>> 0\\
@. @AAA @A[\Lambda_0,-]AA @A\pi AA @AAA\\
0 @>>> \Theta_{X/Y}@>>> \Theta_X@>F>> f^*\Theta_Y@>P>> \mathcal{N}_f@>>> 0
\end{CD}$
\end{center}
where $\Theta_{X/Y}^\bullet:=ker(\Theta_X^\bullet \xrightarrow{F} f^*\Theta_Y^\bullet)$ and $\mathcal{N}_f^\bullet:=coker(\Theta_X^\bullet\xrightarrow{F} f^*\Theta_Y)$.
\begin{definition}
We set
\begin{align*}
PD_{(X,\Lambda_0)/(Y,\Pi_0)}=\frac{\{(\tau,\rho,\lambda)\in C^0(\mathcal{U}, f^*\Theta_Y^\bullet)\oplus \mathcal{C}^1(\mathcal{U},\Theta_X)\oplus C^0(\mathcal{U},\wedge^2 \Theta_X)|\frac{-\delta \tau=F\rho,\pi(\tau)=F\lambda,}{\delta\rho=0,\delta \lambda+[\Lambda_0,\rho]=0, [\Lambda_0,\lambda]=0}\}}{\{(Fg, -\delta g ,[\Lambda_0,g])|g\in C^0(\mathcal{U},\Theta_X)\}}
\end{align*}
\end{definition}

\begin{lemma}
$PD_{(X,\Lambda_0)/(Y,\Pi_0)}$ does not depend on the choice of the Stein covering. $PD_{X/Y}$ is a finite dimensional vector space. We have two exact sequences
\begin{align}
\mathbb{H}^0(X,\Theta_X^\bullet)\xrightarrow{F}\mathbb{H}^0(X,f^*\Theta_Y^\bullet)\to PD_{(X,\Lambda_0)/(Y,\Pi_0)}\to \mathbb{H}^1 (X,\Theta_X^\bullet)\xrightarrow{F} \mathbb{H}^1(X, f^*\Theta_Y^\bullet)\\
0\to \mathbb{H}^1(X,\Theta_{X/Y}^\bullet)\to PD_{(X,\Lambda_0)/(Y,\Pi_0)}\to \mathbb{H}^0(X, \mathcal{N}_f^\bullet)\to \mathbb{H}^2(X, \Theta_{X/Y}^\bullet),\,\,\,\,\,\,\,\,\,\,\,\,\,\,\,\,\,
\end{align}
\end{lemma}

\begin{proof}
See Lemma \ref{ll31}.
\end{proof}

\begin{definition}
A holomorphic Poisson map $f:(X,\Lambda_0)\to (Y,\Pi_0)$ between two holomorphic Poisson manifolds is called non-degenerate if the rank of the Jacobian matrix at some point $x\in X$ is equal to dimension $X$. In this case we have the exact sequence $0\to \Theta_X^\bullet \xrightarrow{F} f^*\Theta_Y\xrightarrow{P} \mathcal{N}_f\to 0$.
\end{definition}

\begin{corollary}
If $f$ is non-degenerate, we have $PD_{(X,\Lambda_0)/(Y,\Pi_0)}\cong \mathbb{H}^0(X, \mathcal{N}_f^\bullet)$. 
\end{corollary}

\subsection{Infinitesimal deformations}\label{jj1}\

Let $(\mathcal{X},\Lambda, \Phi,p,M)$ be a family of holomorphic Poisson maps into a compact holomorphic Poisson manifold $(Y,\Pi_0)$. Let $0$ be a point on $M$, $X=X_0$, and let $f=\Phi_0:p^{-1}(0)=(X,\Lambda_0)\to (Y,\Pi_0)$ be the holomorphic Poisson map induced by $\Phi$. We will define a linear map $\tau:T_0(M)\to PD_{(X,\Lambda_0)/(Y,\Pi_0)}$ of the family at $0$. We may assume that
\begin{enumerate}
\item $M$ is an open set in $\mathbb{C}^r$ with coordinates $t=(t_1,...,t_r)$ and $0=(0,...,0)$.
\item $\mathcal{X}$ is covered by a finite number of Stein coordinate neighborhoods $\mathcal{U}_i$. Each $\mathcal{U}_i$ is covered by a system of coordinates $(z_i,t)=(z_i^1,...,z_i^n,t_1,...,t_r)$ such that $p(z_i,t)=t$.
\item $Y$ is covered by a system of a finite number of Stein coordinate neighborhoods $V_i$ with a system of coordinates $w_i=(w_i^1,...,w_i^m)$, and $\Phi(\mathcal{U}_i) \subset \mathcal{V}_i=V_i\times M$.
\item $\Phi$ is given by $w_i=\Phi_i(z_i,t)$, and let $f_i(z_i)=\Phi_i(z_i,0)$.
\item $(z_i,t)\in \mathcal{U}_i$ coincides with $(z_j,t)\in \mathcal{U}_j$ if and only if $z_i=\phi_{ij}(z_j,t)$
\item $w_i\in V_i$ coincides with $w_j\in V_j$ if and only if $w_i=\psi_{ij}(w_j)$.
\item $\Lambda$ is of the form $\Lambda=\sum_{\alpha,\beta=1}^n \Lambda_{\alpha\beta}^i(z_i,t)\frac{\partial}{\partial z_i^\alpha}\wedge \frac{\partial}{\partial z_i^\beta}$ with $\Lambda_{\alpha\beta}^i(z_i,t)=-\Lambda_{\beta\alpha}^i(z_i,t)$ on $\mathcal{U}_i$. Let $\Lambda_{\alpha\beta}^i(z_i)=\Lambda_{\alpha\beta}^i(z_i,0)$. Then $\Lambda_0$ is of the form $\Lambda_0=\sum_{\alpha,\beta=1}^n \Lambda_{\alpha\beta}^i(z_i)\frac{\partial}{\partial z_i^\alpha}\wedge \frac{\partial}{\partial z_i^\beta}$ on $X\cap \mathcal{U}_i$. 
\item $\Pi_0$ is of the form $\Pi_0=\sum   \Pi_{\alpha\beta}^i(w_i)\frac{\partial}{\partial w_i^\alpha}\wedge \frac{\partial}{\partial w_i^\beta}$ with $\Pi_{\alpha\beta}^i(w_i)=-\Pi_{\beta\alpha}^i(w_i)$ on $V_i$.
\end{enumerate}
Then we have
\begin{align}
\Phi_i(\phi_{ij}(z_j,t),t)&=\psi_{ij}(\Phi_j(z_j,t))\label{bb4}\\
\phi_{ij}(\phi_{jk}(z_k,t))&=\phi_{ik}(z_k,t)  \label{bb5}\\
[\Lambda,\Lambda]&=0 \label{bb6} \\
\sum_{\alpha\beta=1}^n \Lambda_{\alpha\beta}^j(z_j,t)\frac{\partial \phi_{ij}^p}{\partial z_j^\alpha}\frac{\partial \phi_{ij}^q}{\partial z_j^\beta}&= \Lambda_{pq}^i(\phi_{ij}(z_j,t),t) \label{bb7}\\
\sum_{\alpha,\beta=1}^n\Lambda_{\alpha\beta}^i(z_i,t)\frac{\partial \Phi_i^p}{\partial z_i^\alpha}\frac{\partial \Phi_i^q}{\partial z_i^\beta}&=\Pi_{pq}^i(\Phi_i(z_i,t)) \label{bb8}
\end{align}

We let $U_i=X\cap \mathcal{U}_i$ and denote by $\mathcal{U}$ the covering $\{U_i\}$ of $X$. For any element $\frac{\partial}{\partial t}\in T_0(M)$, let
\begin{align*}
-\tau_i&=\sum_{\beta=1}^m \frac{\partial \Phi_i^\beta}{\partial t}|_{t=0}\frac{\partial}{\partial w_i^\beta}\in \Gamma(U_i, f^*\Theta_Y)\\
\rho_{ij}&=\sum_{\alpha=1}^n \frac{\partial \phi_{ij}^\alpha}{\partial t}|_{t=0}\frac{\partial}{\partial z_i^\alpha}\in \Gamma(U_{ij},\Theta_X)\\
\Lambda_i'&=\sum_{\alpha,\beta=1}^n \frac{\partial \Lambda_{\alpha\beta}^i(z_i,t)}{\partial t}|_{t=0}\frac{\partial}{\partial z_i^\alpha}\wedge \frac{\partial}{\partial z_i^\beta}\in \Gamma(U_i, \wedge^2 \Theta_X)
\end{align*}

Then we claim that
\begin{align}
&\tau_i-\tau_j=F\rho_{ij}\label{bb1}\\
&\pi(\tau_i)=F \Lambda_i' \label{bb2}\\
\rho_{jk}-\rho_{ik}+\rho_{ij}=0,\,\,\,\,\,&\Lambda_j'-\Lambda_i'+[\Lambda_0, \rho_{ij}]=0,\,\,\,\,\,[\Lambda_0, \Lambda_i']=0. \label{bb3}
\end{align}

By taking the derivative of $(\ref{bb4})$ with respect to $t$ and setting $t=0$, we get $(\ref{bb1})$ (for the detail, see \cite{Hor73} p.375). By taking the derivative of $(\ref{bb5}),(\ref{bb6}),(\ref{bb7})$ with respect to $t$ and setting $t=0$, we get $(\ref{bb3})$ (for the detail, see \cite{Kim15}). It remains to show $(\ref{bb2})$. It is sufficient to show that $\pi(\tau_i)(w_i^p,w_i^q)=F\Lambda_i'(w_i^p,w_i^q)$ for any $p,q$. Indeed, we note that by taking the derivative of $(\ref{bb8})$ with respect to $t$ and setting $t=0$, we have
\begin{align}\label{bb10}
\sum_{\alpha,\beta=1}^n \frac{\partial \Lambda_{\alpha\beta}^i}{\partial t}|_{t=0}\frac{\partial f_i^p}{\partial z_i^\alpha}\frac{\partial f_i^q}{\partial z_i^\beta}+\sum_{\alpha,\beta=1}^n\Lambda_{\alpha\beta}^i(z_i)\frac{\partial}{\partial z_i^\alpha}\left(\frac{\partial \Phi_i^p}{\partial t}|_{t=0}\right)\frac{\partial f_i^q}{\partial z_i^\beta}+\sum_{\alpha,\beta=1}^n\Lambda_{\alpha\beta}^i(z_i)\frac{\partial f_i^p}{\partial z_i^\alpha}\frac{\partial}{\partial z_i^\beta}\left(\frac{\partial \Phi_i^q}{\partial t}|_{t=0}\right)=\sum_{\rho=1}^m \frac{\partial \Pi_{pq}^i}{\partial w_i^\rho}(f_i)\frac{\partial \Phi_i^\rho}{\partial t}|_{t=0}
\end{align}
Let us consider each side of $(\ref{bb2})$.
\begin{align}\label{bb11}
F(\sum_{\alpha,\beta=1}^n \frac{\partial \Lambda_{\alpha\beta}^i}{\partial t}|_{t=0}\frac{\partial}{\partial z_i^\alpha}\wedge \frac{\partial}{\partial z_i^\beta})(w_i^p,w_i^q)=2\sum_{\alpha,\beta=1}^n \frac{\partial \Lambda_{\alpha\beta}^i}{\partial t}|_{t=0}\frac{\partial f_i^p}{\partial z_i^\alpha} \frac{\partial f_i^q}{\partial z_i^\beta}
\end{align}
On the other hand,
\begin{align}\label{bb12}
\pi(\tau_i)(w_i^p,w_i^q)&=\Lambda_0(\tau_i(w_i^p), f_i^q)-\Lambda_0(\tau_i(w_i^q),f_i^p)-\tau_i(\Pi_0(w_i^p,w_i^q))=-\Lambda_0(\frac{\partial \Phi_i^p}{\partial t}|_{t=0},f_i^q)+\Lambda_0(\frac{\partial \Phi_i^q}{\partial t}|_{t=0},f_i^p)-2\tau_i(\Pi_{pq}^i)\\
&=-2\sum_{\alpha,\beta=1}^n\Lambda_{\alpha\beta}^i(z_i)\frac{\partial }{\partial z_i^\alpha}\left( \frac{\partial \Phi_i^p}{\partial t}|_{t=0}\right)\frac{\partial f_i^q}{\partial z_i^\beta}+2\sum_{\alpha,\beta=1}^n \Lambda_{\alpha\beta}^i(z_i)\frac{\partial}{\partial z_i^\alpha}\left( \frac{\partial \Phi_i^q}{\partial t}|_{t=0}\right)\frac{\partial f_i^p}{\partial z_i^\beta}+2\sum_{\rho=1}^m\frac{\partial \Phi_i^\rho}{\partial t}|_{t=0} \frac{\partial \Pi_{pq}^i}{\partial w_i^\rho}(f_i) \notag
\end{align}
Then $(\ref{bb2})$ follows from $(\ref{bb10}),(\ref{bb11})$, and $(\ref{bb12})$.

Hence from $(\ref{bb1}),(\ref{bb2})$, and $(\ref{bb3})$, the collection $(\{\tau_i\},\{\rho_{ij}\},\{\Lambda_i'\})$ defines an element in $PD_{(X,\Lambda_0)/(Y,\Pi_0)}$ and we have a linear map
\begin{align*}
\tau:T_0(M)\to PD_{(X,\Lambda_0)/(Y,\Pi_0)}
\end{align*}
We call $\tau$ the characteristic map of the family at $0$.

\begin{proposition}
The linear map $\tau:T_0(M)\to PD_{(X,\Lambda_0)/(Y,\Pi_0)}$ is independent of the choice of coverings and systems of coordinates.

\end{proposition}

\begin{proof}
It is sufficient to show that given another coordinate $(z_i',t)$ on $\mathcal{U}_i$ and $(w_i')$ on $V_i$, the associated element $(\{-\tau'_i \},\{\rho'_{ij}\},\{\lambda'_i\})$ defines the same class given by $(\{-\tau_i\},\{\rho_{ij}\},\{\lambda_i\})$ in $PD_{(X,\Lambda_0)/(Y,\Pi_0)}$ defined as above. Let $(z_i',t)$ coincide with $(z_i,t)$ if and only if $z_i'=h_i(z_i,t)$, and $(w_i')$ coincides with $(w_i)$ if and only if $w_i=g_i(w_i')$. Let $\Phi$ be given by $w_i'=\Phi_i'(z_i',t)$. Then we have $\Phi_i(z_i,t)=g_i(\Phi_i'(h_i(z_i,t),t))$. Let $\theta_i=\sum_{\sigma=1}^n \frac{\partial h_i^\sigma}{\partial t}|_{t=0}\frac{\partial}{\partial z_i^\sigma}$. Then $\theta_j-\theta_i=\rho_{ij}-\rho_{ij}'$, $\lambda_i-\lambda_i'=-[\Lambda_0, \theta_i]$ (see \cite{Kim15}), and $\tau_i-\tau_i'=F(\theta_i)$ (see \cite{Hor73} p.376). This completes the proof.
\end{proof}

\begin{remark}
In the case $(\mathcal{X},\Lambda)=(X\times M,\Lambda_0)$. Since $\rho_{ij}=0,\Lambda_i'=0$, $\{\tau_i\}$ defines an element in $\mathbb{H}^0(X, f^*\Theta_Y^\bullet)$ so that we have a linear map $\tau:T_0(M)\to \mathbb{H}^0(X,f^*\Theta_Y^\bullet)$.
\end{remark}

\begin{remark}
We say that a family $(\mathcal{X},\Lambda, \Phi, p, M)$ of holomorphic Poisson maps into $(Y,\Pi_0)$ is  non-degenerate  if $\Phi_t:(X_t,\Lambda_t)\to (Y,\Pi_0)$ is non-degenerate for any $t\in M$. Then $\{P\tau_i\}$ defines an element in $\mathbb{H}^0(X,\mathcal{N}_f^\bullet )$ and we have the characteristic map $\tau:T_0(M)\to \mathbb{H}^0(X, \mathcal{N}_f^\bullet)\cong PD_{(X,\Lambda_0)/(Y,\Pi_0)}$.
\end{remark}

\begin{remark}
Let $(\mathcal{X},\Lambda, \Phi, p, M)$ be a family of non-degenerate holomorphic Poisson maps into $(Y,\Pi_0), 0\in M$, and $(X,\Lambda_0)=(X_0,\Lambda_0)$, then the diagram 
\[\xymatrix{
T_0(M) \ar[rr]^\tau \ar[dr]_\sigma & & \mathbb{H}^0(X, N_f^\bullet) \ar[dl]^\delta  \\
& \mathbb{H}^1(X, \Theta_X^\bullet)
}\]
is commutative, where $\sigma$ is the Poisson Kodaira-Spencer map of the Poisson analytic family $(\mathcal{X},\Lambda, p, M)$ at $0$ $($see \cite{Kim15}$)$ and $\delta$ is the coboundary map induced from $0\to \Theta_X^\bullet \xrightarrow{F} f^* \Theta_Y^\bullet \xrightarrow{P} \mathcal{N}_f^\bullet \to 0$.
\end{remark}

\subsection{Theorem of completeness}
\begin{theorem}[Theorem of completeness]\label{ii1}
Let $(\mathcal{X},\Lambda, \Phi, p,M)$ be a family of holomorphic Poisson maps into $(Y,\Pi_0), 0\in M, (X,\Lambda_0)=(X_0,\Lambda_0)$, and $f=\Phi_0:(X,\Lambda_0)\to (Y,\Pi_0)$. If the characteristic map $\tau:T_0(M)\to PD_{(X,\Lambda_0)/(Y,\Pi_0)}$ is surjective, then the family is complete at $0$.
\end{theorem}

\begin{proof}
We extend the arguments in \cite{Hor73} p.377-382, and \cite{Hor74} p.650  in the context of holomorphic Poisson deformations. We tried to maintain notational consistency with \cite{Hor73} and \cite{Hor74}.

It suffices to prove that for any family $(\mathcal{X}',\Lambda',\Phi', p',M')$ of holomorphic Poisson maps into $(Y,\Pi_0)$, such that $\Phi_{0'}':(X_{0'}',\Lambda_{0'}')\to (Y,\Pi_0)$ is equivalent to $f:(X,\Lambda_0)\to (Y,\Pi_0)$ for some $0'\in M'$, there exist a holomorphic map $t:M'\to M$ with $t(0')=0$ and a holomorphic Poisson map $g:(\mathcal{X}',\Lambda')\to (\mathcal{X}, \Lambda)$ over $t$ which maps $(X_{t'}',\Lambda_{t'}')$ Poisson biholomorphically onto $(X_{t(t')},\Lambda_{t(t')})$ such that the following diagram commutes.
\begin{center}
$\begin{CD}
(\mathcal{X}',\Lambda')@>g>>(\mathcal{X},\Lambda)\\
@Vp'VV @VVpV\\
M'@>t>> M
\end{CD}$
\end{center}

We keep the notations in subsection \ref{jj1} for $(\mathcal{X},\Lambda,\Phi,p,M)$. By putting $``'"$ we represent something for $(\mathcal{X}',\Lambda',\Phi',p',M')$. 
\begin{remark}\label{qq30}
Since the problem is local, we may assume that $M=\{t\in \mathbb{C}^r||t|<1\},M'=\{t'\in \mathbb{C}^{r'}||t'|<1\},\mathcal{U}_i=\{(z_i,t)\in \mathbb{C}^n\times M||z_i|<1\},\mathcal{U}_i'=\{(z_i',t')\in \mathbb{C}^n\times M'||z_i'|<1\},U_i=\{z_i\in \mathbb{C}^n||z_i|<1\}$ and $X=\bigcup_i U_i^\delta$, where $U_i^\delta=\{z_i\in U_i||z_i|<1-\delta \} $. Since $\Phi_{0'}'$ is equivalent to $\Phi_0$, we may assume that $U_i'=U_i,z_i'=z_i$ on $U_i$, $\Phi_i'(z_i,0)=\Phi_i(z_i,0)=f_i(z_i),\phi_{ij}'(z_j,0)=\phi_{ij}(z_j,0)=b_{ij}(z_j)$, and $ \Lambda_0=\Lambda_{0'}'$.
\end{remark}

Let $M_\epsilon'=\{t'\in M'||t'|<\epsilon\}$ for a sufficiently small number $\epsilon >0$. We will construct holomorphic maps $t:M_\epsilon'\to M$ and $g_i:U_i\times M_\epsilon'\to\mathbb{C}^n$ which satisfy the following conditions:
\begin{align}
g_i(z_i,0)&=z_i,\label{jj2}\\
t(0)&=0, \label{jj3}\\
g_i(\phi_{ij}',t')&=\phi_{ij}(g_j,t(t')), \label{jj4}\\
\Phi_i(g_i,t(t'))&=\Phi_i' \label{jj5}\\
\sum_{\alpha,\beta=1}^n \Lambda_{\alpha\beta}'^{i}(z_i,t')\frac{\partial g_i^p}{\partial z_i^\alpha}\frac{\partial g_i^q}{\partial z_i^\beta}&=\Lambda_{pq}^i(g_i,t(t')) \label{jj6}
\end{align}

\subsubsection{Existence of formal solutions}\

\begin{notation}
Let $P(s)=P(s_1,...,s_r),Q(s)=Q(s_1,...,s_r)$ be power series in $s$ with coefficients in some module. We write $P(s)=P_0(s)+P_1(s)+\cdots+P_\mu(s)+\cdots$ where $P_\mu(s)$ is a homogenous polynomial in $s$ of degree $\mu$. We indicate $P^\mu(s)$ the polynomial $P^\mu(s)=P_0(s)+P_1(s)+\cdots+P_\mu(s)$. We write $P(s)\equiv_\mu 0$ if $P^\mu(s)=0$, and $P(s)\equiv_\mu Q(s)$ if $P(s)-Q(s)\equiv_\mu 0$.
\end{notation}

We note that $(\ref{jj4}),(\ref{jj5})$, and $(\ref{jj6})$ are equivalent to the following system of congruences:
\begin{align}
g_i^\mu(\phi_{ij}',t')&\equiv_\mu\phi_{ij}(g_j^\mu, t^\mu),\,\,\,\,\,\,\,\mu=0,1,2,\cdots \label{jj7}\\
\Phi_i(g_i^\mu, t^\mu)&\equiv_\mu \Phi_i',\,\,\,\,\,\,\,\,\,\,\,\,\,\,\,\,\,\,\,\,\,\,\,\,\,\,\,\mu=0,1,2,\cdots  \label{jj8}\\
\sum_{\alpha,\beta=1}^n \Lambda_{\alpha\beta}'^{i}(z_i,t')\frac{\partial g_i^{\mu p}}{\partial z_i^\alpha}\frac{\partial g_i^{\mu q}}{\partial z_i^\beta}&\equiv_\mu \Lambda_{pq}^i(g_i^\mu,t^\mu)\,\,\,\,\,\,\,\,\mu=0,1,2,\cdots \label{jj9}
\end{align}
We will construct $g_i^\mu(z_i,t')=g_{0|i}(z_i,t')+g_{i|1}(z_i,t')+\cdots +g_{i|\mu}(z_i,t')$ and $t^\mu(t')=t_0(t')+t_1(t')+\cdots +t_\mu(t')$ satisfying $(\ref{jj7})_\mu,(\ref{jj8})_\mu$ and $(\ref{jj9})_\mu$ by induction on $\mu$. By setting $g_{0|i}(z_i,t)=z_i$ and $t_0(t')=0$, it is clear that the induction holds for $\mu=0$.

Suppose that $t^{\mu-1}$ and $g_i^{\mu-1}$ satisfying $(\ref{jj7})_{\mu-1},(\ref{jj8})_{\mu-1}$, and $(\ref{jj9})_{\mu-1}$ are already determined. We define $\Gamma_{ij|\mu},\lambda_{i|\mu}$, and $\gamma_{i|\mu}$, where $\Gamma_{ij|\mu}$ is a homogenous polynomial of degree $\mu$ with coefficients in $\Gamma(U_{ij},\Theta_X)$, $\lambda_{i|\mu}$ is a homogeneous polynomial of degree $\mu$ with coefficients in $\Gamma(U_i, \wedge^2 \Theta_X)$, and $\gamma_{i|\mu}$ is a homogenous polynomial of degree $\mu$ with coefficients in $\Gamma(U_i, f^*\Theta_Y)$ by the following congruences
\begin{align}
\Gamma_{ij|\mu}&\equiv_{\mu} \sum_{\alpha=1}^n( g_i^{\mu-1}(\phi_{ij}',t')-\phi_{ij}(g_j^{\mu-1},t^{\mu-1})\frac{\partial}{\partial z_i^\alpha}\\
\lambda_{i|\mu}&\equiv_\mu \sum_{p,q=1}^n \lambda_{i|\mu}^{p,q} \frac{\partial}{\partial z_i^p}\wedge \frac{\partial}{\partial z_i^q}\\
&\equiv_\mu \sum_{p,q=1}^n      \left(-\Lambda_{pq}^i(g_i^{\mu-1},t^{\mu-1})+\sum_{\alpha,\beta=1}^n    \Lambda'^i_{\alpha\beta}(z_i,t')\frac{\partial g_i^{(\mu-1) p}}{\partial z_i^\alpha} \frac{\partial g_i^{(\mu-1)q}}{\partial z_i^\beta} \right)\frac{\partial}{\partial z_i^p}\wedge \frac{ \partial}{\partial z_i^q} \notag     \\
-\gamma_{i|\mu}&\equiv_\mu \sum_{\rho=1}^m (\Phi '^{\rho}_i -\Phi_i^\rho(g_i^{\mu-1},t^{\mu-1}))\frac{\partial}{\partial w_i^\rho}
\end{align} 

\begin{lemma}
We have
\begin{align}
\Gamma_{jk|\mu}-\Gamma_{ik|\mu}+\Gamma_{ij|\mu}&=0 \label{jj10}\\
[\Lambda_0, \lambda_{i|\mu}]&=0 \label{jj11}\\
\lambda_{j|\mu}-\lambda_{i|\mu}+[\Lambda_0, \Gamma_{ij|\mu}]&=0 \label{jj12}\\
F \Gamma_{ij|\mu}&=\gamma_{i|\mu}-\gamma_{j|\mu} \label{jj13}\\
F\lambda_{i|\mu}&=\pi(\gamma_{i|\mu}) \label{jj14}
\end{align}
\end{lemma}
\begin{proof}
For $(\ref{jj10}),(\ref{jj11})$, and $(\ref{jj12})$, see \cite{Kim15}. For $(\ref{jj13})$, see \cite{Hor73} p.378. We show $(\ref{jj14})$. It is sufficient to show that $F\lambda_{i|\mu}(w_i^r,w_i^s)=\pi(\gamma_{i|\mu})(w_i^r,w_i^s)$ for any $r,s$. We note that
\begin{align}
\sum_{\alpha,\beta=1}^n \Lambda_{\alpha\beta}^i(z_i,t) \frac{\partial \Phi_i^p}{\partial z_i^\alpha}\frac{\partial \Phi_i^q}{\partial z_i^\beta}&=\Pi_{pq}^i(\Phi_i(z_i,t))\\
\sum_{\alpha,\beta=1}^n \Lambda'^i_{\alpha\beta}(z_i,t') \frac{\partial \Phi_i'^p}{\partial z_i'^\alpha}\frac{\partial \Phi_i'^q}{\partial z_i'^\beta}&=\Pi_{pq}^i(\Phi_i'(z_i,t'))
\end{align}
Then we have
\begin{align*}
&F\lambda_{i|\mu}(w_i^r,w_i^s)\\
&=\lambda_{i|\mu}(f_i^r,f_i^s)=\sum_{p,q=1}^n2\left(-\Lambda_{pq}^i(g_i^{\mu-1},t^{\mu-1})+\sum_{\alpha,\beta=1}^n    \Lambda'^i_{\alpha\beta}(z_i,t')\frac{\partial g_i^{(\mu-1) p}}{\partial z_i^\alpha} \frac{\partial g_i^{(\mu-1)q}}{\partial z_i^\beta} \right)\frac{\partial f_i^r}{\partial z_i^p} \frac{ \partial f_i^s}{\partial z_i^q}\\
&\equiv_\mu \sum_{p,q=1}^n 2\left(-\Lambda_{pq}^i(g_i^{\mu-1},t^{\mu-1})+\sum_{\alpha,\beta=1}^n    \Lambda'^i_{\alpha\beta}(z_i,t')\frac{\partial g_i^{(\mu-1) p}}{\partial z_i^\alpha} \frac{\partial g_i^{(\mu-1)q}}{\partial z_i^\beta} \right)\frac{\partial \Phi_i^r}{\partial z_i^p}(g_i^{\mu-1},t^{\mu-1}) \frac{ \partial \Phi_i^s}{\partial z_i^q}(g_i^{\mu-1},t^{\mu-1})\\
&\equiv_\mu -2\Pi_{rs}^i(\Phi_i(g_i^{\mu-1},t^{\mu-1}))+2\sum_{\alpha,\beta=1}^n \Lambda_{\alpha\beta}'^i(z_i,t')\frac{\partial \Phi_i^r(g_i^{\mu-1},t^{\mu-1})}{\partial z_i^\alpha} \frac{\partial \Phi_i^s(g_i^{\mu-1},t^{\mu-1})}{\partial z_i^\beta}\\
&\equiv_\mu -2 \Pi_{rs}^i(\Phi_i'+\gamma_{i|\mu})+2\sum_{\alpha,\beta=1}^n \Lambda_{\alpha\beta}'^i(z_i,t')\frac{\partial(\Phi'^r_i+\gamma_{i|\mu}^r)}{\partial z_i^\alpha}\frac{\partial (\Phi'^s_i+\gamma_{i|\mu}^s)}{\partial z_i^\beta}\\
&\equiv_\mu -2\Pi_{rs}^i(\Phi_i')-2\sum_{\alpha=1}^m \gamma_{i|\mu}^\alpha\frac{\partial \Pi_{pq}^i}{\partial w_i^\alpha}(\Phi_i')+2\sum_{\alpha,\beta=1}^n \Lambda'^i_{\alpha\beta}(z_i,t')\frac{\partial \Phi'^r_i}{\partial z_i^\alpha}\frac{\partial \Phi'^s_i}{\partial z_i^\beta}+2\sum_{\alpha,\beta=1}^n \Lambda_{\alpha\beta}^i(z_i)\frac{\partial \gamma_{i|\mu}^r}{\partial z_i^\alpha}\frac{\partial \Phi'^s_i}{\partial z_i^\beta}+2\sum_{\alpha,\beta=1}^n\Lambda_{\alpha\beta}^i(z_i) \frac{\partial \Phi'^r_i}{\partial z_i^\alpha}\frac{\partial \gamma_{i|\mu}^s}{\partial z_i^\beta}\\
&\equiv_\mu -2\sum_{\alpha=1}^m \gamma_{i|\mu}^\alpha\frac{\partial \Pi_{pq}^i}{\partial w_i^\alpha}(f_i)+2\sum_{\alpha,\beta=1}^n \Lambda_{\alpha\beta}^i(z_i)\frac{\partial \gamma_{i|\mu}^r}{\partial z_i^\alpha}\frac{\partial f_i^s}{\partial z_i^\beta}+2\sum_{\alpha,\beta=1}^n\Lambda_{\alpha\beta}^i(z_i) \frac{\partial f_i^r}{\partial z_i^\alpha}\frac{\partial \gamma_{i|\mu}^s}{\partial z_i^\beta}\\
&=\Lambda_0(\gamma_{i|\mu}(w_i^r),f_i^s)-\Lambda_0(\gamma_{i|\mu}(w_i^s),f_i^r)-\gamma_{i|\mu}(\Pi_0(w_i^r,w_i^s))\\
&=\pi(\gamma_{i|\mu})(w_i^r,w_i^s).
\end{align*}
Hence we have $\pi(\gamma_{i|\mu})(w_i^r,w_i^s)=F\lambda_{i|\mu}(w_i^r,w_i^s)$.
\end{proof}

We will determine $t_\mu(t')$ and $g_{i| \mu}(t')$ such that $t^\mu(t')=t^{\mu-1}(t')+t_\mu(t')$, and $g_i^{\mu}(z_i,t')=g_i^{\mu-1}(z_i,t')+g_{i|\mu}(z_i,t')$
 satisfy $(\ref{jj7})_\mu, (\ref{jj8})_\mu$, and $(\ref{jj9})_\mu$.

\begin{lemma}
$(\ref{jj7})_\mu, (\ref{jj8})_\mu$, and $(\ref{jj9})_\mu$ are equivalent to
\begin{align}
\Gamma_{ij|\mu}&=g_{j|\mu}-g_{i|\mu}+\sum_{v=1}^r t_{v|\mu}\rho_{ijv} \label{jj30}\\
\gamma_{i|\mu}&=-Fg_{i|\mu}+\sum_{v=1}^r t_{v|\mu} \tau_{iv} \label{jj31}\\
\lambda_{i|\mu}&=-[\Lambda_0, g_{i|\mu}]+\sum_{v=1}^r t_{v|\mu}\Lambda_{iv}' \label{jj32}
\end{align}
where $g_{i|\mu}=\sum_{\alpha=1}^n g_{i|\mu}^\alpha \frac{\partial}{\partial z_i^\alpha},-\tau_{iv}=\sum_{\beta=1}^m \frac{\partial \Phi_i^\beta}{\partial t_v}|_{t=0}\frac{\partial}{\partial w_i^\beta}, \rho_{ijv}=\sum_{\alpha=1}^n \frac{\partial \phi_{ij}^\alpha}{\partial t_v}|_{t=0}\frac{\partial}{\partial z_i^\alpha}$, and $\Lambda_{iv}'=\sum_{\alpha,\beta=1}^n \frac{\partial \Lambda^i_{\alpha\beta}}{\partial t_v}|_{t=0}\frac{\partial}{\partial z_i^\alpha}\wedge \frac{\partial}{\partial z_i^\beta}$.
\end{lemma}
\begin{proof}
For $(\ref{jj30})$ and $(\ref{jj31})$, see \cite{Hor73} p.379. Let us show that $(\ref{jj9})_\mu$ is equivalent to $(\ref{jj32})$.
Indeed, $(\ref{jj9})_\mu$ is equivalent to
\begin{align*}
&\Lambda_{pq}^i(g_i^{\mu-1}+g_{i|\mu},t^{\mu-1}+t_\mu)\equiv_\mu \sum_{\alpha,\beta=1}^n \Lambda'^i_{\alpha\beta}(z_i,t')\frac{\partial (g_i^{(\mu-1)p}+g_{i|\mu}^p)}{\partial z_i^\alpha}\frac{\partial (g_i^{(\mu-1)q}+g_{i|\mu}^q)}{\partial z_i^\beta}\\
&\iff \sum_{\gamma=1}^n g_{i|\mu}^\alpha\frac{\partial \Lambda_{pq}^i}{\partial z_i^\gamma}(g_i^{\mu-1},t^{\mu-1})+\sum_{v=1}^r t_{v|\mu } \frac{\partial \Lambda_{pq}^i}{\partial t_v}(g_i^{\mu-1},t^{\mu-1})\\
&\equiv_\mu \lambda_{i|\mu}^{p,q}+\sum_{\alpha,\beta=1}^n \Lambda'^i_{\alpha\beta}(z_i,t')\frac{\partial g_i^{(\mu-1)p}}{\partial z_i^\alpha}\frac{\partial g_{i|\mu}^q}{\partial z_i^\beta}+\sum_{\alpha,\beta=1}^n \Lambda'^i_{\alpha\beta}(z_i,t')\frac{\partial g_{i|\mu}^p}{\partial z_i^\alpha}\frac{\partial g_i^{(\mu-1)q}}{\partial z_i^\beta}\\
&\iff  \sum_{\gamma=1}^n g_{i|\mu}^\gamma \frac{\partial \Lambda_{pq}^i}{\partial z_i^\gamma}(z_i,0)+\sum_{v=1}^r t_{v|\mu}\frac{\partial \Lambda_{pq}^i}{\partial t_v}(z_i,0)\equiv_\mu \lambda_{i|\mu}^{p,q}+\sum_{\alpha,\beta=1}^n \Lambda_{\alpha\beta}^i(z_i)\frac{\partial z_i^p}{\partial z_i^\alpha}\frac{\partial g_{i|\mu}^q}{\partial z_i^\beta}+\sum_{\alpha,\beta=1}^n \Lambda_{\alpha \beta}^i(z_i)\frac{\partial g_{i|\mu}^p}{\partial z_i^\alpha}\frac{\partial z_i^q}{\partial z_i^\beta}\\
&\iff \lambda_{i|\mu}(z_i^p,z_i^q)=-\Lambda_0(g_{i|\mu}^p,z_i^q)+\Lambda_0(g_{i|\mu}^q,z_i^p)+g_{i|\mu}(\Lambda_0(z_i^p,z_i^q))+\sum_{v=1}^r t_{v|\mu} \frac{\partial \Lambda_{pq}^i}{\partial t}(z_i,0)\\
&\iff \lambda_{i|\mu}=-[\Lambda_0, g_{i|\mu}]+\sum_{v=1}^r t_{v|\mu}\Lambda_{iv}'.
\end{align*}

\end{proof}

Let $ (t_1,...,t_r)$ be a system of coordinates on $M$ with center at $0$. Since $\tau:T_0(M)\to PD_{(X,\Lambda_0)/(Y,\Pi_0)}$ is surjective, $\tau(\frac{\partial}{\partial t_v}),v=1,...,r$ represented by 
\begin{align*}
(\{\tau_{iv}\},\{\rho_{ijv }\},\{\Lambda_{iv}'\})\in C^0(\mathcal{U}, f^*\Theta_Y)\oplus \mathcal{Z}^1(\mathcal{U},\Theta_X)\oplus C^0(\mathcal{U},\wedge^2 \Theta_X)
\end{align*}
 generate $PD_{(X,\Lambda_0)/(Y,\Pi_0)}$.

\begin{lemma}
We can find $t_\mu$ and $g_{i|\mu}$ satisfying $(\ref{jj30}),(\ref{jj31})$, and $(\ref{jj32})$.

\end{lemma}

\begin{proof}
From $(\ref{jj10})-(\ref{jj14})$, $(\{\gamma_{i|\mu}\},\{\Gamma_{ij|\mu}\},\{\lambda_{i|\mu}\})$ defines an element of $PD_{(X,\Lambda_0)/(Y,\Pi_0)}$. Then there exists $t_1,...,t_r$ such that $(\{\gamma_{i|\mu}\},\{\Gamma_{ij}\},\{\lambda_{i|\mu}\})$ is equivalent to $(\{\sum_{v=1}^n t_v\tau_{iv}\},\{\sum_{v=1}^r t_v\rho_{ijv}\},\{\sum_{v=1}^n t_v\Lambda_{iv}'\})$. Hence there exist $g=\{g_{i|\mu}\}\in C^0(\mathcal{U},\Theta_X)$ such that $\{\gamma_{i|\mu}-\sum_{v=1}^r t_v\tau_{iv}\}=F(\{-g_{i|\mu}\})$, $\Gamma_{ij|\mu}-\sum_{v=1}^r t_v\rho_{ijv}=-\delta(\{-g_{i|\mu}\})=\{g_{j|\mu}-g_{i|\mu}\}$, and $\{\lambda_{i|\mu}-\sum_{v=1}^r t_v \Lambda_{iv}'\}=[\Lambda_0, \{-g_{i|\mu}\}]$.
\end{proof}

Hence induction holds for $\mu$ so that we can construct formal power series $g_i$ and $t$ satisfying $(\ref{jj2})-(\ref{jj6})$.

\subsection{Proof of convergence}\

We prove that if we choose solutions $t_\mu(t')$ and $g_{i|\mu}(z_i,t')$ of the equations $(\ref{jj30}),(\ref{jj31})$, and $(\ref{jj32})$ properly in each inductive step, the power series $t(t')=t_1(t')+t_2(t')+\cdots +t_\mu (t')+\cdots$, and $g_i(z_i,t')=z_i+g_{i|1}(z_i,t')+\cdots +g_{i|\mu}(z_i,t)+\cdots$ converge absolutely and uniformly for $|t'|<\epsilon$ if $\epsilon>0$ is sufficiently small.

\begin{notation}
Let $g(s)=\sum_{v_1v_2...v_r} g_{v_1...v_r}s_1^{v_1}s_2^{v_2}\cdots s_r^{v_r}$ whose coefficients $g_{v_1...v_r}$ are vectors and a power sereis $a(s)=\sum a_{v_1...v_r}s_1^{v_1}\cdots s_r^{v_r}$ with $a_{v_1...v_r}\geq 0$. We indicate by writing $g(s)\ll a(s)$ that $|g_{v_1...v_r}|\leq a_{v_1...v_r}$. We let
\begin{align}\label{kk2}
A(s)=\frac{b}{16c} \sum_{\mu=1}^\infty \frac{1}{\mu^2}c^\mu(s_1+\cdots+s_r)^\mu
\end{align}
\end{notation}

It suffices to show the estimates
\begin{align}\label{kk1}
t^\mu(t') \ll A(t'),\,\,\,\,\,\,\,g_i^{\mu}(z_i,t')-z_i \ll A(t')
\end{align}
by induction on $\mu$ if the coefficients $b,c$ in $(\ref{kk2})$ are properly chosen. For $\mu=1$, the estimates $(\ref{kk1})_1$ holds if $b$ is sufficiently large. We recall Remark \ref{qq30}.
\begin{lemma}
Assume that the estimates $(\ref{kk1})_{\mu-1}$ holds for some $\mu$. Then we have
\begin{align}
\Gamma_{ij|\mu}(z_i,t')&\ll \left( \frac{K_1}{b}+\frac{K_2}{c}+\frac{K_3b}{c}       \right)  A(t'),\,\,\,\,\,z_i\in U_i\cap U_j   \label{aa1}   \\
\gamma_{i|\mu}(z_i,t')&\ll     \left( \frac{K_4}{b}+ \frac{K_5b}{c}  \right)A(t'),\,\,\,\,\,z_i\in U_i  \label{aa2}   \\
\lambda_{i|\mu}(z_i,t')&\ll \left(  \frac{K_6}{b}+\frac{K_7}{c}+\frac{K_8 b}{c}+\frac{K_9 b}{c^2}      \right)A(t'),\,\,\,\,\, z_i\in U_i^\delta \label{aa3}
\end{align}
where $K_1,...,K_9$ are constants independent of $\mu$.
\end{lemma}
\begin{proof}
See \cite{Kod05} p.302 for $(\ref{aa1})$, \cite{Hor73} p.381 for $(\ref{aa2})$ and \cite{Kim15} for $(\ref{aa3})$.
\end{proof}

For any $\sigma=( \gamma=\{\gamma_i\}, \Gamma=\{\Gamma_{ij}\},\lambda=\{\lambda_i\})\in C^0(\mathcal{U},f^*\Theta_Y)\oplus C^1(\mathcal{U}, \Theta_X)\oplus C^0(\mathcal{U}, \wedge^2 \Theta_X)$ such that
\begin{align} \label{aa7}
F\Gamma_{ij}=\gamma_i-\gamma_j,\,\,\,\,\, F\lambda_i=\pi(\gamma_i)
\end{align}

We define the norm $||\sigma||:=||\gamma||+||\Gamma||+||\lambda||$ where
\begin{align*}
||\gamma||=\max_i\sup_{z_i\in U_i}|\gamma_i(z_i)|,\,\,\,\,\,||\Gamma||=\max_{i,j}\sup_{z_i\in U_i\cap U_j}|\Gamma_{ij}(z_i)|,\,\,\,\,\,\,\,||\lambda||=\max_i\sup_{z_i\in U_i^\delta}|\lambda_i(z_i)|
\end{align*}

\begin{lemma}
For any $\sigma=(\gamma=\{\gamma_i\},\Gamma=\{\Gamma_{ij}\}, \lambda=\{\lambda_i\})$ satisfying $(\ref{aa7})$, we can find $g_i(z_i)$ and $t_v$ satisfying
\begin{align}\label{aa4}
\Gamma_{ij}=g_j-g_i+\sum_{v=1}^r t_v \rho_{ijv}, \,\,\,\,\,\,
\lambda_i=-[\Lambda_0, g_i] +\sum_{v=1}^r t_v \Lambda_{iv}',\,\,\,\,\,\,
\gamma_i=-Fg_i +\sum_{v=1}^r t_v \tau_{iv}
\end{align}
\begin{align*}
|g_i(z_i)|\leq K_{10}||\sigma||,\,\,\,\,\,\,|t|\leq K_{10}||\sigma||
\end{align*}
where $K_{10}$ is a constant independent of $\sigma$.
\end{lemma}

\begin{proof}
The proof is similar to \cite{Hor73} Lemma 2.3 p.381, \cite{Kod05} Lemma 6.2 p.295 and \cite{Kim15} to which we refer for the detail.
We define 
\begin{align*}
\iota(\sigma)=\inf \max\{\sup_{z_i\in U_i}|g_i(z_i)|,|t|\}
\end{align*}
 where $\inf$ is taken with respect to all solutions $g_i(z_i),t_v$ of the equations $(\ref{aa4})$. It suffices to show the existence of a constant $K_{10}$ such that $\iota(\sigma)\leq K_{10}||\sigma||$. Suppose that such a constant $K_{10}$ does not exist. Then we can find a sequence $\sigma^{(1)},\sigma^{(2)},...,\sigma^{(l)}=(\gamma^{(l)},\Gamma^{(l)},\lambda^{(l)}),...$ satisfying $(\ref{aa7})$ such that $\iota(\sigma^{(l)})=1,||\sigma^{(l)}||<\frac{1}{l}$ so that there exist $\{g_i^{(l)}\}$ and $t^{(l)}$ such that
\begin{align*}
\Gamma_{ij}^{(l)}=g_j^{(l)}-g_i^{(l)}+\sum_{v=1}^r t_v^{(l)}\rho_{ijv},\,\,\,\,\,\,
\lambda_i^{(l)}=-[\Lambda_0, g_i^{(l)}]+\sum_{v=1}^r t_v^{(l)}\Lambda_{iv}',\,\,\,\,\,\,
\gamma_i^{(l)}=-Fg_i^{(l)}+\sum_{v=1}^r t_v^{(l)}\tau_{iv},
\end{align*}
and $\{t_v^{(l)}\}$ converges, and $\{g_i^{l}(z_i)\}$ converges uniformly on $U_i$. Put $t_v=\lim t_v^{(l)}$ and $g_i(z_i)=\lim g_i^{(l)}(z_i)$. Let $g_i'(z_i)=g_i^{(l)}-g_i(z_i)$ and $t_v'=t_v^{(l)}-t_v$. Then we have
\begin{align*}
\Gamma_{ij}^{(l)}=g_j'-g_i'+\sum_{v=1}^r t_v'\rho_{ijv},\,\,\,\,\,\,
\lambda_i^{(l)}=-[\Lambda_0, g_i']+\sum_{v=1}^r t_p'\Lambda_{iv}',\,\,\,\,\,\,
\gamma_i^{(l)}=-Fg_i'+\sum_{v=1}^r t_v'\tau_{iv},
\end{align*}
This is a contradiction to $\iota(\sigma^{(l)})=1$.

\end{proof}
Hence we can choose solutions $g_{i|\mu}(z_i,t')$, and $t_\mu(t')$ of $(\ref{jj30}),(\ref{jj31})$, and $(\ref{jj32})$ such that
\begin{align*}
g_{i|\mu}(z_i,t')\ll K_{10}K^*A(t'),\,\,\,\,\,\,\,\,t_\mu(t')\ll K_{10}K^*A(t')
\end{align*}
where $K^*=\frac{K_1+K_4+K_6}{b} +\frac{K_2+K_7}{c}+\frac{(K_3+K_5+K_8)b}{c}+\frac{K_9b}{c^2}$. We can choose $b$ and $c$ in $(\ref{kk2})$ in a way that we have $K_{10}K^*<1$. Hence we have $g_{i|\mu}(z_i,t')\ll A(t')$ and $t_\mu(t')\ll A(t')$. This completes the proof of Theorem \ref{ii1}.

\end{proof}

\subsection{Theorem of existence}

\begin{theorem}[Theorem of existence]\label{hh1}
Let $f:(X,\Lambda_0)\to (Y,\Pi_0)$ be a holomorphic Poisson map of a compact holomorphic Poisson manifold $(X,\Lambda_0)$ into a holomorphic Poisson manifold $(Y,\Pi_0)$. Assume that the canonical homomorphism $F:\Theta_X^\bullet\to f^*\Theta_Y^\bullet$ satisfies the following conditions:
\begin{enumerate}
\item $F:\mathbb{H}^1(X,\Theta_X^\bullet)\to \mathbb{H}^1(X, f^*\Theta_Y^\bullet)$ is surjective,
\item $F:\mathbb{H}^2(X,\Theta_X^\bullet)\to \mathbb{H}^2(X, f^*\Theta_Y^\bullet)$ is injective.
\end{enumerate}
Then if the formal power series $(\ref{qq32})$ constructed in the proof converge, then  there exist a family $(\mathcal{X}, \Lambda, \Phi, p, M)$ of holomorphic Poisson maps into $(Y,\Pi_0)$ and a point $0\in M$ such that 
\begin{enumerate}
\item $(X,\Lambda_0)=p^{-1}(0)$ and $\Phi_0:(X,\Lambda_0)\to (Y, \Pi_0)$ coincides with $f:(X,\Lambda_0)\to (Y,\Pi_0)$.
\item $\tau:T_0(M)\to PD_{(X,\Lambda_0)/(Y,\Pi_0)}$ is bijective.
\end{enumerate}
\end{theorem}

\begin{remark}
If $f$ is non-degenerate, the above condition $(1)$ and $(2)$ are reduced to $\mathbb{H}^1(X, \mathcal{N}_f^\bullet)=0$.
\end{remark}

\begin{proof}
We extend the arguments in \cite{Hor73} and \cite{Hor74} in the context of holomorphic Poisson deformations (See \cite{Hor73} p.382-388 and \cite{Hor74} p.649-653). We tried to maintain notational consistency with \cite{Hor73} and \cite{Hor74}. 
We may assume that
\begin{enumerate}
\item $X$ is covered by a finite number of coordinate neighborhoods $U_i$ with a system of coordinates $(z_i^1,...,z_i^n)$ and $U_i=\{(z_i)\in \mathbb{C}^n||z_i|<1\}$.
\item $Y$ is covered by a finite number of coordinate neighborhoods $V_i$ with a system of coordinates $(w_i^1,..,w_i^m)$ and $V_i=\{(w_i)\in \mathbb{C}^m||w_i|<1\}$.
\item $f(U_i)\subset V_i$, and $f$ is given by $w_i=f_i(z_i)$.
\item $z_i\in U_i$ coincides with $z_j\in U_j$ if and only if $z_i=b_{ij}(z_j)$.
\item $w_i\in V_i$ coincides with $w_j\in V_j$ if and only if $w_i=g_{ij}(w_j)$.
\item $\Lambda_0$ is of the form $\Lambda_0=\sum_{\alpha,\beta=1}^n \Lambda_{\alpha\beta}^i(z_i)\frac{\partial}{\partial z_i^\alpha}\wedge \frac{\partial}{\partial z_i^\beta}$ with $\Lambda_{\alpha\beta}^i(z_i)=-\Lambda_{\beta\alpha}^i(z_i)$ on $U_i$.
\item $\Pi_0$ is of the form $\Pi_0=\sum_{\alpha,\beta=1}^m \Pi_{\alpha\beta}^i(w_i)\frac{\partial}{\partial w_i^\alpha}\wedge \frac{\partial}{\partial w_i^\beta}$ with $\Pi_{\alpha\beta}^i(w_i)=-\Pi_{\beta\alpha}^i(w_i)$ on $V_i$.
\end{enumerate}
Let $r=\dim PD_{(X,\Lambda_0)/(Y,\Pi_0)}$, and $M=\{t\in \mathbb{C}^r||t|<\epsilon\}$ for a sufficiently small number $\epsilon >0$.

We regard $X\times M$ as a differentiable manifold and prove the existence of a vector $(0,1)$-form 
\begin{align*}
\phi(t)=\sum_{v=1}^n \phi_i^v(z_i,t)\frac{\partial}{\partial z_i^v}=\sum_{v,\alpha=1}^n \phi_{i\alpha}^v(z_i,t)d\bar{z}_i^\alpha \frac{\partial}{\partial z_i^v}
\end{align*}
and the existence of a bivector of the form
\begin{align*}
\Lambda(t)=\sum_{\alpha,\beta=1}^n \Lambda_{\alpha\beta}^i(z_i,t)\frac{\partial}{\partial z_i^\alpha}\wedge \frac{\partial }{\partial z_i^\beta}
\end{align*}
depending holomorphically on $t$ and vector-valued differentiable functions $\Phi_i(z_i,t)$ on $U_i\times M$ depending holomorphically on $t$ which satisfy the following equalities:
\begin{align}
\phi(0)&=0 \label{ee1}\\
\Lambda(0)&=\Lambda_0\label{ee33}\\
\bar{\partial}\phi+\frac{1}{2}[\phi,\phi]&=0\label{ee3}\\
\frac{1}{2}[\Lambda,\Lambda]&=0\label{ee111}\\
\bar{\partial}\Lambda+[\phi,\Lambda]&=0\label{ee112}\\
\Phi_i(z_i,0)&=f_i(z_i) \label{ee34}\\
\bar{\partial} \Phi_i+[\phi, \Phi_i]&=0\label{ee113}\\
\Phi_i(b_{ij}(z_j),t)&=g_{ij}(\Phi_j(z_i,t))\label{ee114}\\
\sum_{\alpha,\beta=1}^n \Lambda_{\alpha\beta}^i(z_i,t) \frac{\partial \Phi_i^p}{\partial z_i^\alpha}\frac{\partial \Phi_i^q}{\partial z_i^\beta}&=\Pi_{pq}^i(\Phi_i(z_i,t))\label{ee2}
\end{align}

\subsubsection{Existence of formal solutions}\

We will prove the existence of formal solutions of $\phi(t),\Lambda(t)$, and $\Phi_i(z_i,t)$ satisfying $(\ref{ee1})-(\ref{ee2})$ as power series in $t$. Let $\phi(t)=\sum_{\mu=0}^\infty \phi_\mu(t),\Lambda(t)=\sum_{\mu=0}^\infty \Lambda_\mu(t)$, and $\Phi_i(z_i,t)=\sum_{\mu=0}^\infty \Phi_{i|\mu}(z_i,t)$, where $\phi_\mu(t),\Lambda_\mu(t)$, and $\Phi_{i|\mu}(t)$ are homogenous in $t$ of degree $\mu$, and let $\phi^{\mu}(t)=\phi_0(t)+\phi_1(t)+\cdots+\phi_\mu(t), \Lambda^\mu(t)=\Lambda_0(t)+\Lambda_1(t)+\cdots+\Lambda_\mu(t)$, and $\Phi_i^\mu(z_i,t)=\Phi_{i|0}(z_i,t)+\Phi_{i|1}(z_i,t)+\cdots+\Phi_{i|\mu}(z_i,t)$. We note that $(\ref{ee3})-(\ref{ee2})$ are equivalent to the following system of congruences:

\begin{align}
\bar{\partial} \phi^\mu+\frac{1}{2}[\phi^\mu,\phi^\mu]&\equiv_\mu 0 \label{gg30}\\
\frac{1}{2}[\Lambda^\mu,\Lambda^\mu]&\equiv_\mu 0 \label{gg31}\\
\bar{\partial}\Lambda^\mu+[\phi^\mu,\Lambda^\mu]&\equiv_\mu 0 \label{gg32}\\
\bar{\partial}\Phi_i^\mu +[\phi^\mu, \Phi_i^\mu]&\equiv_\mu 0  \label{gg33}\\
\Phi_i^\mu(b_{ij}(z_j),t)&\equiv_\mu g_{ij}(\Phi_j^\mu(z_j,t)) \label{gg34}\\
\sum_{\alpha,\beta=1}^n \Lambda_{\alpha\beta}^{\mu i}(z_i,t) \frac{\partial \Phi_i^{\mu p}}{\partial z_i^\alpha}\frac{\partial \Phi_i^{\mu q}}{\partial z_i^\beta}&\equiv_\mu \Pi_{pq}^{i}(\Phi_i^\mu(z_i,t))\label{gg35}
\end{align}
for $\mu=1,2,3,...$.
\begin{remark}\label{ii71}
By setting $\Lambda'^{\mu}:=\Lambda'^\mu-\Lambda_0$, $(\ref{gg30}),(\ref{gg31})$, and $(\ref{gg32})$ are equivalent to
\begin{align}\label{gg50}
L(\phi^\mu+\Lambda'^\mu)+\frac{1}{2}[\phi^\mu+\Lambda'^\mu,\phi^\mu+\Lambda'^\mu]\equiv_\mu 0
\end{align}
where $L=\bar{\partial}+[\Lambda_0,-]$.
\end{remark}

We shall construct solutions $\phi(t),\Lambda(t)$, and $\Phi_i(z_i,t)$ of $(\ref{ee3})-(\ref{ee2})$ by induction on $\mu$. From $(\ref{ee1}),(\ref{ee33})$, and $(\ref{ee34})$, we set $\phi_0(t)=0,\Lambda_0(t)=\Lambda_0$, and $\Phi_{i|0}(z_i,t)=f_i(z_i)$. Let us determine $\phi_1,\Lambda_1$ and $\Phi_{i|1}$. For this, we note the following lemma.

\begin{lemma}\label{kk10}
We have an isomorphism
\begin{align*}
PD_{(\Lambda_0,X)/(Y,\Pi_0)}\cong \frac{\{(\Phi,\phi,\psi)\in A^{0,0}(f^*\Theta_Y)\times A^{0,1}(\Theta_X)\times A^{0,0}(\wedge^2 \Theta_X)|\frac{\bar{\partial}\Phi=F\phi, \pi(\Phi)=F\psi, \bar{\partial}\phi=0, \bar{\partial}\psi+[\phi,\psi]=0,[\Lambda_0, \psi]=0}{\iff   L(\phi, \psi)=0, L_\pi(\Phi)=F(\phi, \psi)}\} }{\{(F \xi, \bar{\partial}\xi,[\Lambda_0, \xi]|\xi\in A^{0,0}(\Theta_X)\}}
\end{align*}
where $L=\bar{\partial}+[\Lambda_0,-]$, and $L_\pi=\bar{\partial}+\pi$.
\end{lemma}

\begin{proof}
Let $(\tau=\{\tau_i\},\rho=\{\rho_{ij}\},\lambda=\{\lambda_i\})\in C^0(\mathcal{U},f^*\Theta_Y)\oplus C^1(\mathcal{U},\Theta_X)\oplus C^0(\mathcal{U},\wedge^2 \Theta_X)$ be a representative of $PD_{(\Lambda_0,X)/(\Pi_0,Y)}$. Since $\delta \rho=0$, there exists $\eta_i\in \Gamma(U_i,\mathcal{A}^{0,0}(\Theta_X))$ such that $\rho_{ij}=\eta_i-\eta_j$. Then $\phi:=\bar{\partial}\eta_i \in A^{0,1}(\Theta_X)$ and $\psi:=[\Lambda_0, \eta_i]-\lambda_i\in A^{0,0}(\wedge^2 \Theta_X)$ satisfy $ \bar{\partial}\phi=0, \bar{\partial}\psi+[\phi,\psi]=0,[\Lambda_0, \psi]=0$ (for the detail, see \cite{Kim15}). On the other hand, $\tau_i-\tau_j=F\rho_{ij}=F\eta_i-F\eta_j$. We define $\Phi_i=F\eta_i-\tau_i$ on $U_i$. Then $\Phi=\{\Phi_i\}\in A^{0,0}(f^*\Theta_Y)$ with $\bar{\partial}\Phi=F\phi$ and $\pi(\Phi)=F[\Lambda_0, \eta_i]-F\tau_i=F\psi$. We define the correspondence by
\begin{align*}
(\tau=\{\tau_i\},\rho=\{\rho_{ij}\},\lambda=\{\lambda_i\}) \mapsto (\Phi,\phi,\psi)=(F\eta_i-\tau_i, \bar{\partial}\eta_i, [\Lambda_0,\eta_i]-\tau_i)
\end{align*}

Let $(\tau',\rho',\lambda')$ be another representative of the class defined by $(\tau,\rho,\lambda)$. Then there exists $g=\{g_i\}\in C^0(\mathcal{U},\Theta_X)$ such that $\tau_i-\tau_i'=Fg_i, \rho_{ij}-\rho_{ij}'=g_i-g_j$, and $\lambda_i-\lambda_i'=[\Lambda_0,g_i]$. Let $\rho_{ij}'=\eta_i'-\eta'_j$, and so $\phi'=\bar{\partial}\eta_i'$ and $\psi'=[\Lambda_0,\eta_i']-\lambda_i'$. Then $\Phi-\Phi'=F(\eta_i-\eta_i')-(\tau_i-\tau_i')=F(\eta_i-\eta_i'-g_i)$, $\phi-\phi'=\bar{\partial}(\eta_i-\eta_i'-g_i)$, and $\psi-\psi'=[\Lambda_0, \eta_i-\eta_i']-(\lambda_i- \lambda_i')=[\Lambda_0, \eta_i-\eta_i'-g_i]$. Hence the correspondence is well-defined.

Let us define the inverse map. Given $(\Phi,\phi,\psi)$ with $L(\phi,\psi)=0, L_\pi(\Phi)=F(\phi,\psi)$. Since $\bar{\partial}\phi=0$, there exists $\eta_i\in \Gamma(U_i,\mathcal{A}^{0,0}(\Theta_X))$ such that $\phi=\bar{\partial}\eta_i$. We define $\rho_{ij}:=\eta_i-\eta_j$ on $U_{ij}$, $\lambda_i:=[\Lambda_0,\eta_i]-\psi$ on $U_i$, and $\tau_i=F\eta_i-\Phi$ on $U_i$. 

\end{proof}

Choose a basis of $PD_{(X,\Lambda_0)/(Y,\Pi_0)}$ as
\begin{align*}
\{(\Phi_\lambda, -\phi_\lambda,- \Lambda_\lambda)\}_{\lambda=1,...,r}\in A^{0,0}(f^*\Theta_Y)\oplus A^{0,1}(\Theta_X)\oplus A^{0,0}(\wedge^2 \Theta_X)
\end{align*}
We write $\Phi_\lambda=\sum_{\alpha=1}^m \Phi_{\lambda i}^\alpha\frac{\partial}{\partial w_i^\alpha}$ on $U_i$, and $\Lambda_\lambda=\sum_{\alpha,\beta=1}^n \Lambda_{\alpha\beta \lambda}^i\frac{\partial}{\partial z_i^\alpha}\wedge \frac{\partial}{\partial z_i^\beta}$ on $U_i$. We define $\phi_1=\sum_{\lambda=1}^r \phi_\lambda t_\lambda, \Lambda_1=\sum_{\lambda=1}^r \Lambda_\lambda t_\lambda$, and $\Phi_{i|1}^\alpha=\sum_{\lambda=1}^r \Phi_{\lambda i}^\alpha t_\lambda$. We set $\Lambda^i_{\alpha\beta|1}(z_i,t)=\sum_{\lambda=1}^r\Lambda_{\alpha\beta\lambda}^it_\lambda$. Then $(\ref{gg30})_1-(\ref{gg34})_1$ holds (see \cite{Hor74} p.651). Let us check $(\ref{gg35})_1$. Indeed, $(\ref{gg35})_1$ is equivalent to
\begin{align*}
\sum_{\alpha,\beta=1}^n\Lambda_{\alpha\beta|1}^i(z_i,t)\frac{\partial f_i^p}{\partial z_i^\alpha}\frac{\partial f_i^q}{\partial z_i^\beta}+ \sum_{\alpha,\beta=1}^n \Lambda_{\alpha\beta}^i(z_i)\frac{\partial \Phi_{i|1}^p}{\partial z_i^\alpha}\frac{\partial f_i^q}{\partial z_i^\beta}+\sum_{\alpha,\beta=1}^n \Lambda_{\alpha\beta}^i(z_i)\frac{\partial f_i^p}{\partial z_i^\alpha}\frac{\partial \Phi_{i|1}^q}{\partial z_i^\beta}=\sum_{\beta=1}^m \frac{\partial \Pi_{pq}^i}{\partial w_i^\beta}(f_i)\Phi_{i|1}^\beta\\
\iff \pi(\Phi_{i|1})(w_i^p,w_i^q)=-F(\Lambda_1)(w_i^p,w_i^q)
\end{align*}

Hence induction holds for $\mu=1$.

Now we assume that $\phi^{\mu-1}, \Lambda^{\mu-1}$, and $\Phi_i^{\mu-1}$ satisfying $(\ref{gg30})_{\mu-1}-(\ref{gg35})_{\mu-1}$ are already determined. We define homogenous polynomials $\xi_\mu\in A^{0,2}(\Theta_X),\psi_\mu\in A^{0,1}(\wedge^2 \Theta_X), \eta_\mu\in A^{0,0}(\wedge^3 \Theta_X), E_{i|\mu}\in \Gamma(U_i, \mathcal{A}^{0,1}(f^* \Theta_Y)),\\\Gamma_{ij|\mu}\in \Gamma(U_{ij},\mathcal{A}^{0,0}(f^*\Theta_Y) ) $, and $\lambda_{i|\mu}\in \Gamma(U_i, \mathcal{A}^{0,0}(f^*\wedge^2 \Theta_Y ) )$ by the following congruences:

\begin{align}
\xi_\mu&\equiv_\mu \bar{\partial}\phi^{\mu-1}+\frac{1}{2}[\phi^{\mu-1},\phi^{\mu-1}] \label{gg1}\\
\psi_\mu &\equiv_\mu \bar{\partial}\Lambda^{\mu-1}+[\Lambda^{\mu-1},\phi^{\mu-1}] \label{gg2}\\
\eta_\mu &\equiv_\mu \frac{1}{2}[\Lambda^{\mu-1},\Lambda^{\mu-1}] \label{gg3}\\
E_{i|\mu}&\equiv_\mu \sum_{\alpha=1}^m(\bar{\partial} \Phi_i^{(\mu-1)\alpha}+[\phi^{\mu-1}, \Phi_i^{(\mu-1)\alpha}])\frac{\partial}{\partial w_i^\alpha}\label{gg4}\\
\Gamma_{ij|\mu}&\equiv_\mu \sum_{\alpha=1}^m (\Phi_i^{(\mu-1)\alpha}(b_{ij}(z_j),t)-g_{ij}^\alpha(\Phi_j^{\mu-1}(z_j,t)))\frac{\partial}{\partial w_i^\alpha} \label{gg5}\\
\lambda_{i|\mu}&\equiv_\mu\sum_{p,q=1}^m \lambda_{i|\mu}^{p,q}\frac{\partial}{\partial w_i^p}\wedge \frac{\partial}{\partial w_i^q} \label{gg6}\\
&= \sum_{p,q=1}^m \left( -\Pi_{pq}^{ i}(\Phi_i^{\mu-1}(z_i,t))+\sum_{\alpha,\beta=1}^n \Lambda_{\alpha\beta}^{(\mu-1) i} \frac{\partial \Phi_i^{(\mu-1) p}}{\partial z_i^\alpha}\frac{\partial \Phi_i^{(\mu-1) q}}{\partial z_i^\beta}   \right)\frac{\partial}{\partial w_i^p}\wedge \frac{\partial}{\partial w_i^q} \notag
\end{align}

\begin{remark}\label{ii70}
We note that $(\ref{gg1}),(\ref{gg2}),(\ref{gg3})$ are equivalent to
\begin{align*}
\xi_\mu+\psi_\mu+\eta_\mu=\bar{\partial}(\phi^{\mu-1}+\Lambda^{\mu-1})+\frac{1}{2}[\phi^{\mu-1}+\Lambda^{\mu-1},\phi^{\mu-1}+\Lambda^{\mu-1}]
\end{align*}
By letting $\Lambda'^{\mu-1}=\Lambda^{\mu-1}-\Lambda_0$, $(\ref{gg1},(\ref{gg2}),(\ref{gg3})$ are also equivalent to
\begin{align}\label{gg112}
\xi_\mu+\psi_\mu+\eta_\mu=L(\phi^{\mu-1}+\Lambda'^{\mu-1})+\frac{1}{2}[\phi^{\mu-1}+\Lambda'^{\mu-1},\phi^{\mu-1}+\Lambda'^{\mu-1}]
\end{align}
where $L=\bar{\partial}+[\Lambda_0,-]$.
\end{remark}

\begin{lemma} \label{ii76}
We have the following equalities:
\begin{align}
\bar{\partial} \xi_\mu&=0,\,\,\,\,\,\text{in $\Gamma(X, \mathcal{A}^{0,3}(\Theta_X))$}, \label{gg7}\\
[\Lambda_0,\xi_\mu]+\bar{\partial} \psi_\mu&=0\,\,\,\,\,\,\text{in $\Gamma(X, \mathcal{A}^{0,2}(\wedge^2\Theta_X))$}, \label{gg8}\\
[\Lambda_0,\psi_\mu]+\bar{\partial}\eta_\mu&=0\,\,\,\,\,\,\text{in $\Gamma(X, \mathcal{A}^{0,1}(\wedge^3\Theta_X))$}, \label{gg9}\\
[\Lambda_0, \eta_\mu]&=0\,\,\,\,\,\,\text{in $\Gamma(X, \mathcal{A}^{0,0}(\wedge^4\Theta_X))$}, \label{gg10}\\
\bar{\partial}E_{i|\mu}&=F \xi_\mu\,\,\,\,\,\,\text{in $\Gamma(U_i, \mathcal{A}^{0,2}(f^*\Theta_Y))$}, \label{gg11} \\
\pi(E_{i|\mu})+\bar{\partial} \lambda_{i|\mu}&=F \psi_\mu\,\,\,\,\,\,\text{in $\Gamma(U_i, \mathcal{A}^{0,1}(\wedge^2 f^*\Theta_Y))$}, \label{gg12}\\
\pi(\lambda_{i|\mu})&=F \eta_\mu\,\,\,\,\,\,\,\,\,\,\text{in $\Gamma(U_i, \mathcal{A}^{0,0}(\wedge^3 f^*\Theta_Y))$}, \label{gg13}\\
\lambda_{i|\mu}-\lambda_{j|\mu}&=\pi(\Gamma_{ij|\mu})\,\,\,\,\,\,\,\,\text{in $\Gamma(U_{ij},\mathcal{A}^{0,0}(\wedge^2 f^*\Theta_Y))$}, \label{gg14}\\
E_{i|\mu}-E_{j|\mu}&=\bar{\partial}\Gamma_{ij|\mu}\,\,\,\,\,\,\text{in $\Gamma(U_{ij},\mathcal{A}^{0,1}(f^*\Theta_Y))$}, \label{gg15}\\
\Gamma_{jk|\mu}-\Gamma_{ik|\mu}+\Gamma_{ij|\mu}&=0\,\,\,\,\,\,\text{in $\Gamma(U_{ijk},\mathcal{A}^{0,0}(f^*\Theta_Y))$}.   \label{gg16}
\end{align}
\end{lemma}
\begin{remark}\label{ii75}
We note that $(\ref{gg7})-(\ref{gg10})$ are equivalent to
\begin{align}\label{qq31}
L(\xi_\mu+\psi_\mu+\eta_\mu)=0.
\end{align}
\end{remark}
\begin{proof}

$(\ref{gg7}),(\ref{gg8}),(\ref{gg9})$, and $(\ref{gg10})$ follow from $(\ref{gg112}),(\ref{qq31})$, and $L\circ L=0$. $(\ref{gg11})$ follows from \cite{Hor73} p.385. $(\ref{gg15})$ and $(\ref{gg16})$ follow from \cite{Hor73} p.386.

We show $(\ref{gg12})$. It is sufficient to show that $\pi(E_{i|\mu})(w_i^p,w_i^q)+\bar{\partial} \lambda_{i|\mu}(w_i^p,w_i^q)=F\psi_\mu(w_i^p,w_i^q)$ for any $p,q$.
We note that
\begin{align}
F\psi_\mu(w_i^p,w_i^q) &\equiv_\mu F(\bar{\partial}\Lambda_i^{\mu-1}+[\Lambda_i^{\mu-1},\phi_i^{\mu-1}])(w_i^p,w_i^q) \notag\\
&\equiv_\mu \bar{\partial}\Lambda_i^{\mu-1}(\Phi_i^{(\mu-1)p},\Phi_i^{(\mu-1)q})+[\Lambda_i^{\mu-1},\phi_i^{\mu-1}](\Phi_i^{(\mu-1)p},\Phi_i^{(\mu-1)q}) \label{gg20}
\end{align}
Let us consider the first term of $(\ref{gg20})$.
\begin{align}\label{gg21}
\bar{\partial}\Lambda_i^{\mu-1}(\Phi_i^{(\mu-1)p},\Phi_i^{(\mu-1)q})=2\sum_{\alpha,\beta, r=1}^n\frac{\partial \Lambda_{\alpha\beta}^{(\mu-1)i}}{\partial \bar{z}_i^r } \frac{\partial \Phi_i^{(\mu-1)p}}{\partial z_i^\alpha}\frac{\partial \Phi_i^{(\mu-1)q}}{\partial z_i^\beta}d\bar{z}_i^r
\end{align}
Let us consider the second term of $(\ref{gg20})$.
\begin{align}\label{gg25}
&[\Lambda_i^{\mu-1},\phi_i^{\mu-1}](\Phi_i^{(\mu-1)p},\Phi_i^{(\mu-1)q})\\
&\equiv_\mu -\Lambda_i^{\mu-1}(\phi_i^{\mu-1}(\Phi_i^{(\mu-1)p}),\Phi_i^{(\mu-1)q})+\Lambda_i^{\mu-1}(\phi_i^{\mu-1}(\Phi_i^{(\mu-1)q}), \Phi_i^{(\mu-1)p})+\phi_i^{\mu-1}(\Lambda_i^{\mu-1}(\Phi_i^{(\mu-1)p},\Phi_i^{(\mu-1)q}))\notag \\
&\equiv_\mu -2\sum_{\alpha,\beta=1}^n \Lambda_{\alpha\beta}^{(\mu-1)i}\frac{\partial (\sum_{v=1}^n\phi_i^{(\mu-1)v}\frac{\partial \Phi_i^{(\mu-1)p}}{\partial z_i^v})}{\partial z_i^\alpha}\frac{\partial \Phi_i^{(\mu-1)q}}{\partial z_i^\beta}+2\sum_{\alpha,\beta=1}^n \Lambda_{\alpha\beta}^{(\mu-1)i}\frac{\partial (\sum_{v=1}^n\phi_i^{(\mu-1)v} \frac{\partial \Phi_i^{(\mu-1)q}}{\partial z_i^v})}{\partial z_i^\alpha}\frac{\partial \Phi_i^{(\mu-1)p}}{\partial z_i^\beta}\notag\\
&+\sum_{v=1}^n \phi_i^{(\mu-1)v}\frac{\partial}{\partial z_i^v}\left(\sum_{\alpha,\beta=1}^n 2\Lambda_{\alpha\beta}^{(\mu-1)i}\frac{\partial \Phi_i^{(\mu-1)p}}{\partial z_i^\alpha}\frac{\partial \Phi_i^{(\mu-1)q}}{\partial z_i^\beta}\right) \notag \\
&\equiv_\mu -2\sum_{\alpha,\beta,v=1}^n \Lambda_{\alpha\beta}^{(\mu-1)i}\frac{\partial}{\partial z_i^\alpha}\left(\phi_i^{(\mu-1)v}\frac{\partial \Phi_i^{(\mu-1)p}}{\partial z_i^v}\right)\frac{\partial \Phi_i^{(\mu-1)q}}{\partial z_i^\beta}+2\sum_{\alpha,\beta,v=1}^n \Lambda_{\alpha\beta}^{(\mu-1)i}\frac{\partial}{\partial z_i^\alpha}\left(\phi_i^{(\mu-1)v}\frac{\partial \Phi_i^{(\mu-1)q}}{\partial z_i^v}\right)\frac{\partial \Phi_i^{(\mu-1)p}}{\partial z_i^\beta} \notag\\
&+\sum_{v=1}^n \phi_i^{(\mu-1)v}\frac{\partial}{\partial z_i^v}( 2\lambda_{i| \mu}^{p,q}+2 \Pi_{pq}^i(\Phi_i^{\mu-1}))\notag\\
&\equiv_\mu - 2\sum_{\alpha,\beta,v=1}^n \Lambda_{\alpha\beta}^{(\mu-1)i}\frac{\partial}{\partial z_i^\alpha}\left(\phi_i^{(\mu-1)v}\frac{\partial \Phi_i^{(\mu-1)p}}{\partial z_i^v}\right)\frac{\partial \Phi_i^{(\mu-1)q}}{\partial z_i^\beta}+2\sum_{\alpha,\beta,v=1}^n \Lambda_{\alpha\beta}^{(\mu-1)i}\frac{\partial}{\partial z_i^\alpha}\left(\phi_i^{(\mu-1)v}\frac{\partial \Phi_i^{(\mu-1)q}}{\partial z_i^v}\right)\frac{\partial \Phi_i^{(\mu-1)p}}{\partial z_i^\beta} \notag\\
&+\sum_{v=1}^n2 \phi_i^{(\mu-1)v}\frac{\partial \Pi_{pq}^i(\Phi_i^{\mu-1})}{\partial z_i^v}\notag
\end{align}

On the other hand,
\begin{align}
\pi(E_{i|\mu})(w_i^p,w_i^q)=-\Lambda_0(E_{i|\mu}(w_i^p),f_i^q(z_i))+\Lambda_0(E_{i|\mu}(w_i^q),f_i^p(z_i))+E_{i|\mu}(\Pi_0(w_i^p,w_i^q)) \label{gg40}
\end{align}
We consider each term of $(\ref{gg40})$.
\begin{align}\label{gg27}
&-\Lambda_0(E_{i|\mu}(w_i^p), f_i^q(z_i))\equiv_\mu -\Lambda_i^{\mu-1}(E_{i|\mu}(w_i^p), \Phi_i^{(\mu-1)q})\\
&\equiv_\mu - \Lambda_i^{\mu-1}\left(\sum_{v=1}^n \frac{\partial \Phi_i^{(\mu-1)p}}{\partial \bar{z}_i^v} d\bar{z}_i^v+\sum_{v=1}^n \phi_i^{(\mu-1)v}\frac{\partial \Phi_i^{(\mu-1)p}}{\partial z_i^v}, \Phi_i^{(\mu-1)q}\right) \notag\\
&\equiv_\mu -\sum_{v=1}^n\Lambda_i^{\mu-1}\left(\frac{\partial \Phi_i^{(\mu-1)p}}{\partial \bar{z}_i^v},\Phi_i^{(\mu-1)q}\right)d\bar{z}_i^v -\sum_{v=1}^n\Lambda_i^{\mu-1}\left(\phi_i^{(\mu-1)v}\frac{\partial \Phi_i^{(\mu-1)p}}{\partial z_i^v},\Phi_i^{(\mu-1)q}\right) \notag \\
&\equiv_\mu -\sum_{\alpha,\beta,v=1}^n 2\Lambda_{\alpha\beta}^{(\mu-1)i}\frac{\partial}{\partial z_i^\alpha}\left( \frac{\partial \Phi_i^{(\mu-1)p}}{\partial \bar{z}_i^v}\right)\frac{\partial \Phi_i^{(\mu-1)q}}{\partial z_i^\beta} d\bar{z}_i^v-\sum_{\alpha,\beta,v=1}^n 2\Lambda_{\alpha\beta}^{(\mu-1)i}\frac{\partial}{\partial z_i^\alpha}\left(\phi_i^{(\mu-1)v} \frac{\partial \Phi_i^{(\mu-1)p}}{\partial z_i^v}\right)\frac{\partial \Phi_i^{(\mu-1)q}}{\partial z_i^\beta} \notag 
\end{align}

\begin{align}\label{gg28}
&\Lambda_0(E_{i|\mu}(w_i^q), f_i^p(z_i))\equiv_\mu \Lambda_i^{\mu-1}(E_{i|\mu}(w_i^q),\Phi_i^{(\mu-1)p})\\
&\equiv_\mu \sum_{\alpha,\beta,v=1}^n 2\Lambda_{\alpha\beta}^{(\mu-1)i}\frac{\partial}{\partial z_i^\alpha}\left( \frac{\partial \Phi_i^{(\mu-1)q}}{\partial \bar{z}_i^v}\right)\frac{\partial \Phi_i^{(\mu-1)p}}{\partial z_i^\beta} d\bar{z}_i^v+\sum_{\alpha,\beta,v=1}^n 2 \Lambda_{\alpha\beta}^{(\mu-1)i}\frac{\partial}{\partial z_i^\alpha}\left( \phi_i^{(\mu-1)v}\frac{\partial \Phi_i^{(\mu-1)q}}{\partial z_i^v}\right)\frac{\partial \Phi_i^{(\mu-1)p}}{\partial z_i^\beta} \notag
\end{align}
\begin{align}\label{gg29}
2E_{i|\mu}(\Pi_{pq}^i(w_i))&\equiv_\mu
2\sum_{\alpha=1}^m \bar{\partial}\Phi_i^{(\mu-1)\alpha}\frac{\partial \Pi_{pq}^i}{\partial w_i^\alpha}(\Phi_i^{(\mu-1)})+ 2\sum_{\alpha=1}^m\sum_{v=1}^n \phi_i^{\mu-1}\frac{\partial \Phi_i^{(\mu-1)\alpha}}{\partial z_i^v}\frac{\partial \Pi_{pq}^i}{\partial w_i^\alpha} (\Phi_i^{\mu-1})\\
&\equiv_\mu 2\sum_{v=1}^n \frac{\partial \Pi_{pq}^i(\Phi_i^{\mu-1})}{\partial \bar{z}_i^v}d\bar{z}_i^v +2\sum_{v=1}^n \phi_i^{(\mu-1)v}\frac{\partial \Pi_{pq}^i(\Phi_i^{\mu-1})}{\partial z_i^v} \notag
\end{align}
\begin{align}\label{gg291}
\bar{\partial}\lambda_{i|\mu}(w_i^p,w_i^q)&\equiv_\mu -2 \sum_{v=1}^n \frac{\partial \Pi_{pq}^i(\Phi_i^{\mu-1}(z_i,t))}{{\partial} \bar{z}_i^v}d\bar{z}_i^v+2\sum_{\alpha,\beta,v=1}^n\frac{\partial}{{\partial} \bar{z}_i^v}\left(\Lambda_{\alpha\beta}^{(\mu-1)i}\frac{\partial \Phi_i^{(\mu-1)p}}{\partial z_i^\alpha}\frac{\partial \Phi_i^{(\mu-1)q}}{\partial z_i^\beta} \right) d\bar{z}_v
\end{align}
From $(\ref{gg21}),(\ref{gg25}),(\ref{gg27}),(\ref{gg28}),(\ref{gg29})$, and $(\ref{gg291})$, we get $(\ref{gg12})$.

We show $(\ref{gg13})$. It is sufficient to show that $\pi(\lambda_{i|\mu})(w_i^a,w_i^b,w_i^c)=F \eta_\mu(w_i^a,w_i^b,w_i^c)$ for any $a,b,c$. Indeed,

\begin{align*}
&\pi(\lambda_{i|\mu})(w_i^a,w_i^b,w_i^c)\\
&=\Lambda_0(\lambda_{i|\mu}(w_i^a,w_i^b), f_i^c(z_i))-\Lambda_0(\lambda_{i|\mu}(w_i^a,w_i^c), f_i^b(z_i))+\Lambda_0(\lambda_{i|\mu}(w_i^b,w_i^c), f_i^a(z_i))\\
&+\lambda_{i|\mu}(\Pi_0(w_i^a,w_i^b),w_i^c)-\lambda_{i|\mu}(\Pi_0(w_i^a,w_i^c),w_i^b)+\lambda_{i|\mu}(\Pi_0(w_i^b,w_i^c),w_i^a)\\
&\equiv_\mu -2\Lambda_i^{\mu-1}(\Pi_{ab}^i(\Phi_i^{\mu-1}),\Phi_i^{(\mu-1)c})+\Lambda_i^{\mu-1}(\Lambda_i^{\mu-1}(\Phi_i^{(\mu-1)a},\Phi_i^{(\mu-1)b}),\Phi_i^{(\mu-1)c})\\ 
&+2\Lambda_i^{\mu-1}(\Pi_{ac}^i(\Phi_i^{\mu-1}),\Phi_i^{(\mu-1)b})+\Lambda_i^{\mu-1}(\Lambda_i^{\mu-1}(\Phi_i^{(\mu-1)a},\Phi_i^{(\mu-1)c})\Phi_i^{(\mu-1)b})\\
 &-2\Lambda_i^{\mu-1}(\Pi_{bc}^i(\Phi_i^{\mu-1}),\Phi_i^{(\mu-1)a})+\Lambda_i^{\mu-1}(\Lambda_i^{\mu-1}(\Phi_i^{(\mu-1)b},\Phi_i^{(\mu-1)c}),\Phi_i^{(\mu-1)a})\\
 &+ 2 \lambda_{i|\mu}(\Pi_{ab}^i(w_i),w_i^c)-2 \lambda_{i|\mu}(\Pi_{ac}^i(w_i),w_i^b)+2 \lambda_{i|\mu}(\Pi_{bc}^i(w_i),w_i^a)\\
 &\equiv_\mu -2\Lambda_i^{\mu-1}(\Pi_{ab}^i(\Phi_i^{\mu-1}),\Phi_i^{(\mu-1)c})+2\Lambda_i^{\mu-1}(\Pi_{ac}^i(\Phi_i^{\mu-1}),\Phi_i^{(\mu-1)b}) -2\Lambda_i^{\mu-1}(\Pi_{bc}^i(\Phi_i^{\mu-1}),\Phi_i^{(\mu-1)a})\\
 &+\frac{1}{2}[\Lambda_i^{\mu-1},\Lambda_i^{\mu-1}](\Phi_i^{(\mu-1)a},\Phi_i^{(\mu-1)b},\Phi_i^{(\mu-1)c})\\
 &+4\sum_{p=1}^m\left( -\Pi_{pc}^{ i}(\Phi_i^{\mu-1}(z_i,t))+\sum_{\alpha,\beta=1}^n \Lambda_{\alpha\beta}^{(\mu-1) i} \frac{\partial \Phi_i^{(\mu-1) p}}{\partial z_i^\alpha}\frac{\partial \Phi_i^{(\mu-1) c}}{\partial z_i^\beta}   \right)\frac{\partial \Pi_{ab}^i}{\partial w_i^p}(\Phi_i^{(\mu-1)})\\ 
 &-4\sum_{p=1}^m\left( -\Pi_{pb}^{ i}(\Phi_i^{\mu-1}(z_i,t))+\sum_{\alpha,\beta=1}^n \Lambda_{\alpha\beta}^{(\mu-1) i} \frac{\partial \Phi_i^{(\mu-1) p}}{\partial z_i^\alpha}\frac{\partial \Phi_i^{(\mu-1) b}}{\partial z_i^\beta}   \right)\frac{\partial \Pi_{ac}^i}{\partial w_i^p}(\Phi_i^{(\mu-1)})\\ 
& +4\sum_{p=1}^m\left( -\Pi_{pa}^{ i}(\Phi_i^{\mu-1}(z_i,t))+\sum_{\alpha,\beta=1}^n \Lambda_{\alpha\beta}^{(\mu-1) i} \frac{\partial \Phi_i^{(\mu-1) p}}{\partial z_i^\alpha}\frac{\partial \Phi_i^{(\mu-1) a}}{\partial z_i^\beta}   \right)\frac{\partial \Pi_{bc}^i}{\partial w_i^p}(\Phi_i^{(\mu-1)})\\
 &\equiv_\mu -2\Lambda_i^{\mu-1}(\Pi_{ab}^i(\Phi_i^{\mu-1}),\Phi_i^{(\mu-1)c})+2\Lambda_i^{\mu-1}(\Pi_{ac}^i(\Phi_i^{\mu-1}),\Phi_i^{(\mu-1)b}) -2\Lambda_i^{\mu-1}(\Pi_{bc}^i(\Phi_i^{\mu-1}),\Phi_i^{(\mu-1)a})\\
 &+\frac{1}{2}[\Lambda_i^{\mu-1},\Lambda_i^{\mu-1}](f_i^a,f_i^b,f_i^c)\\
 &-4\sum_{p=1}^n \left(\Pi_{pc}^{ i}(\Phi_i^{\mu-1}(z_i,t))\frac{\partial \Pi_{ab}^i}{\partial w_i^p } (\Phi_i^{(\mu-1)})- \Pi_{pb}^{ i}(\Phi_i^{\mu-1}(z_i,t))\frac{\partial \Pi_{ac}^i}{\partial w_i^p} (\Phi_i^{(\mu-1)})+\Pi_{pa}^{ i}(\Phi_i^{\mu-1}(z_i,t)) \frac{\partial \Pi_{bc}^i}{\partial w_i^p}(\Phi_i^{(\mu-1)}) \right)\\
 &+2\Lambda_i^{\mu-1} (\Pi_{ab}^i(\Phi_i^{(\mu-1)}),\Phi_i^{(\mu-1)c}) -2\Lambda_i^{\mu-1} (\Pi_{ac}^i(\Phi_i^{(\mu-1)}),\Phi_i^{(\mu-1)b}) +2\Lambda_i^{\mu-1} (\Pi_{bc}^i(\Phi_i^{(\mu-1)}),\Phi_i^{(\mu-1)a})\\
 &\equiv_\mu \frac{1}{2} F[\Lambda_i^{\mu-1},\Lambda_i^{\mu-1}](w_i^a,w_i^b,w_i^c)\equiv_\mu F\eta_\mu (w_i^a,w_i^b,w_i^c)
\end{align*}

Lastly, we show that $(\ref{gg14})$. It is sufficient to show that $\pi(\Gamma_{ij})(w_i^p,w_i^q)=(\lambda_{i|\mu}-\lambda_{j\mu})(w_i^p,w_i^q)$. Indeed,
\begin{align}\label{gg45}
&\pi(\Gamma_{ij})(w_i^p,w_i^q)=\Lambda_0(\Gamma_{ij}(w_i^p),f_i^q(z_i))-\Lambda_0(\Gamma_{ij}(w_i^q),f_i^p(z_i))-\Gamma_{ij}(\Pi_0(w_i^p,w_i^q))\\
&\equiv_\mu \Lambda_i^{\mu-1}(\Phi_i^{(\mu-1)p}-g_{ij}^p(\Phi_j^{\mu-1}),\Phi_i^{(\mu-1)q})-\Lambda_i^{\mu-1}(\Phi_i^{(\mu-1)q}-g_{ij}^q(\Phi_j^{\mu-1}),g_{ij}^p(\Phi_j^{\mu-1}))-\sum_{\alpha=1}^m 2(\Phi_i^{(\mu-1)\alpha}-g_{ij}^\alpha(\Phi_j^{\mu-1}))\frac{\partial \Pi_{pq}^i}{\partial w_i^\alpha}(\Phi_i^{\mu-1}) \notag\\
&\equiv_\mu \Lambda_i^{\mu-1}(\Phi_i^{(\mu-1)p},\Phi_i^{(\mu-1)q})-\Lambda_i^{\mu-1}(g_{ij}^p(\Phi_i^{\mu-1}),g_{ij}^q(\Phi_i^{(\mu-1)}))-\sum_{\alpha=1}^m 2(\Phi_i^{(\mu-1)\alpha}-g_{ij}^\alpha(\Phi_j^{\mu-1}))\frac{\partial \Pi_{pq}^i}{\partial w_i^\alpha}(\Phi_i^{\mu-1})\notag
\end{align}
\begin{align}\label{gg46}
&\lambda_{i|\mu}(w_i^p,w_i^q)\equiv_\mu-2\Pi_{pq}^i(\Phi_i^{\mu-1})+\Lambda_i^{\mu-1}(\Phi_i^{(\mu-1)p},\Phi_i^{(\mu-1)q})\\
&\lambda_{j|\mu}(w_i^p,w_i^q)\equiv_\mu 2\sum_{r,s=1}^m\left( -\Pi_{rs}^{ j}(\Phi_j^{\mu-1}(z_j,t))+\sum_{\alpha,\beta=1}^n \Lambda_{\alpha\beta}^{(\mu-1) j} \frac{\partial \Phi_j^{(\mu-1) r}}{\partial z_j^\alpha}\frac{\partial \Phi_j^{(\mu-1) s}}{\partial z_j^\beta}   \right)\frac{\partial g_{ij}^p}{\partial w_j^r}(\Phi_j^{\mu-1})\frac{\partial g_{ij}^q}{\partial w_j^s}(\Phi_j^{\mu-1}) \label{gg47}\\
&\equiv_\mu-2\Pi_{pq}^i(g_{ij}(\Phi_j^{\mu-1}))+2\sum_{\alpha,\beta=1}^m \Lambda_{\alpha\beta}^{(\mu-1) j} \frac{\partial g_{ij}^p( \Phi_j^{\mu-1})}{\partial z_j^\alpha}\frac{\partial g_{ij}^q(\Phi_j^{\mu-1})}{\partial z_j^\beta}=-2\Pi_{pq}^i(g_{ij}(\Phi_j^{\mu-1}))+2\sum_{\alpha,\beta=1}^m \Lambda_{\alpha\beta}^{(\mu-1) i} \frac{\partial g_{ij}^p( \Phi_j^{\mu-1})}{\partial z_i^\alpha}\frac{\partial g_{ij}^q(\Phi_j^{\mu-1})}{\partial z_i^\beta} \notag\\
&\equiv_\mu-2\Pi_{pq}^i(g_{ij}(\Phi_j^{\mu-1})-\Phi_i^{\mu-1}+\Phi_i^{\mu-1})+\Lambda_i^{\mu-1}(g_{ij}^p(\Phi_j^{\mu-1}),g_{ij}^q(\Phi_j^{\mu-1})) \notag \\
&\equiv_\mu -2\Pi_{pq}^i(\Phi_i^{\mu-1})+ \sum_{\alpha=1}^m -2(g_{ij}^\alpha(\Phi_j^{\mu-1})-\Phi_i^{(\mu-1)\alpha})\frac{\partial \Pi_{pq}^i}{\partial w_i^\alpha} (\Phi_i^{\mu-1}) + \Lambda_i^{\mu-1}(g_{ij}^p(\Phi_j^{\mu-1}),g_{ij}^q(\Phi_j^{\mu-1}))         \notag
\end{align}
From $(\ref{gg45}),(\ref{gg46})$, and $(\ref{gg47})$, we get $(\ref{gg14})$.
\end{proof}

We will determine $\phi_\mu, \Lambda_\mu$, and $\Phi_{i|\mu}$ such that $\phi^{\mu}:=\phi^{\mu-1}+\phi_\mu, \Lambda^{\mu}:=\Lambda^{\mu-1}+\Lambda_\mu$, and $\Phi_i^\mu:=\Phi_i^{\mu-1}+\Phi_{i|\mu}$ satisfy $(\ref{gg30})_\mu-(\ref{gg35})_\mu$. 

\begin{lemma}\label{ii86}
$(\ref{gg30})_\mu-(\ref{gg35})_\mu$ are equivalent to the following equalities:
\begin{align}
\bar{\partial} \phi_\mu&=-\xi_\mu \label{gg51}\\
\bar{\partial} \Lambda_\mu+[\Lambda_0, \phi_\mu]&=-\psi_\mu \label{gg52}\\
[\Lambda_0, \Lambda_\mu]&=-\eta_\mu \label{gg53}\\
-E_{i|\mu}&=\bar{\partial}\Phi_{i|\mu}+F\phi_\mu \label{gg54}\\
-\lambda_{i|\mu}&= \pi(\Phi_{i|\mu})+F\Lambda_\mu \label{gg55}\\
\Gamma_{ij|\mu}&=\Phi_{j|\mu}-\Phi_{i|\mu}\label{gg56}
\end{align}

\end{lemma}

\begin{remark}
$(\ref{gg54})$ and $(\ref{gg55})$ are equivalent to
\begin{align}\label{ii95}
(-E_{i|\mu},-\lambda_{i|\mu})=L_\pi(\Phi_{i|\mu})+F(\phi_\mu, \Lambda_\mu)
\end{align}
\end{remark}

\begin{proof}
$(\ref{gg51}),(\ref{gg52})$, and $(\ref{gg53})$ follow from $(\ref{gg50})$, and $(\ref{gg112})$. $(\ref{gg56})$ follows from \cite{Hor73} p.386. 

Let us show that $(\ref{gg33})_\mu$ is equivalent to $(\ref{gg54})$. Indeed,
\begin{align*}
0\equiv_\mu \bar{\partial} \Phi_i^\mu+[\phi^\mu,\Phi_i^\mu]\equiv_\mu \bar{\partial} \Phi_i^{\mu-1}+\bar{\partial} \Phi_{i|\mu}+[\phi^{\mu-1}+\phi_\mu, \Phi^{\mu-1}+\Phi_{i|\mu}]\equiv_\mu E_{i|\mu}+\bar{\partial}\Phi_{i|\mu}+[\phi_\mu, f]
\end{align*}

It remains to show that $(\ref{gg35})$ is equivalent to $(\ref{gg55})$. Indeed, $(\ref{gg35})_\mu$ is equivalent to the following.

\begin{align*}
\sum_{\alpha,\beta=1}^n (\Lambda_{\alpha\beta}^{(\mu-1)i}+\Lambda^i_{\alpha\beta|\mu})\frac{\partial (\Phi_i^{(\mu-1)p}+\Phi_{i|\mu}^p)}{\partial z_i^\alpha}\frac{\partial (\Phi_i^{(\mu-1)q}+\Phi_{i|\mu}^q)}{\partial z_i^\beta}\equiv_\mu \Pi_{pq}^i(\Phi_i^{\mu-1}+\Phi_{i|\mu})\equiv_\mu \Pi_{pq}^i(\Phi_i^{\mu-1})+\sum_{r=1}^m \frac{\partial \Pi_{pq}^i}{\partial w_i^r}(\Phi_i^{\mu-1})\Phi_{i|\mu}^r\\
\iff \lambda_{i|\mu}^{p,q}+\sum_{\alpha,\beta=1}^n \Lambda_{\alpha\beta }^i(z_i)\frac{\partial \Phi_{i|\mu}^p}{\partial z_i^\alpha}\frac{\partial f_i^q}{\partial z_i^\beta}+\sum_{\alpha,\beta=1}^n \Lambda_{\alpha\beta }^i(z_i)\frac{\partial f_i^p}{\partial z_i^\alpha}\frac{\partial \Phi_{i|\mu}^q}{\partial z_i^\beta}+\sum_{\alpha,\beta=1}^n \Lambda_{\alpha\beta|\mu}^i\frac{\partial f_i^p}{\partial z_i^\alpha}\frac{\partial f_i^q}{\partial z_i^\beta}-\sum_{r=1}^m \frac{\partial \Pi_{pq}^i}{\partial w_i^r}(f_i)\Phi_{i|\mu}^r\equiv_\mu 0\\
\iff \lambda_{i|\mu}(w_i^p,w_i^q)+\pi(\Phi_{i|\mu})(w_i^p,w_i^q)+ F\Lambda_\mu\equiv_\mu 0
\end{align*}
which comes from the following: for any $p,q$,
\begin{align*}
\pi(\Phi_{i|\mu})(w_i^p,w_i^q)&=\Lambda_0(\Phi_{i|\mu}(w_i^p),f_i^q)-\Lambda_0(\Phi_{i|\mu}(w_i^q), f_i^q)-\Phi_{i|\mu}(\Pi_0(w_i^p,w_i^q))\\
&=\Lambda_0(\Phi_{i|\mu}^p,f_i^q)-\Lambda_0(\Phi_{i|\mu}^q, f_i^p)-2\sum_{r=1}^n\Phi_{i|\mu}^r\frac{\partial \Pi_{pq}^i}{\partial w_i^r}(f_i)
\end{align*}
\end{proof}

\begin{lemma}
Under the hypothesis of Theorem $\ref{hh1}$, we can find
\begin{align*}
\phi_\mu\in A^{0,1}(\Theta_X),\,\,\,\,\,\Lambda_\mu\in A^{0,0}(\wedge^2 \Theta_X),\,\,\,\,\,\Phi_{i|\mu}\in \Gamma(U_i,\mathcal{A}^{0,0}(f^*\Theta_Y))
\end{align*}
satisfying $(\ref{gg51})-(\ref{gg56})$.

\end{lemma}

\begin{proof}
From $(\ref{gg16})$, we can find $\Gamma_{i|\mu}\in \Gamma(U_i,\mathcal{A}^{0,0}(f^*\Theta_Y))$ such that $\Gamma_{ij|\mu}=\Gamma_{i|\mu}-\Gamma_{j|\mu}$. From $(\ref{gg15})$, we see that $E_{i|\mu}':=E_{i|\mu}-\bar{\partial} \Gamma_{i|\mu}$ defines a global section $E_\mu'\in A^{0,1}(f^*\Theta_Y)$. On the other hand, from $(\ref{gg14})$, $\lambda_{i|\mu}':=\lambda_{i|\mu}-\pi(\Gamma_{i|\mu})$ defines a global section $\lambda_\mu' \in A^{0,0}(\wedge^2 \Theta_X)$. Then $E_{\mu}'$ and $\lambda_\mu'$ satisfy $\bar{\partial}E_{\mu}'=\bar{\partial} E_{i|\mu}=F\xi_\mu,\pi(E_\mu')+\bar{\partial} \lambda'_\mu=\pi(E_{i|\mu}')+\bar{\partial} \lambda_{i|\mu}=F \psi_\mu$, and $\pi(\lambda_\mu')=\pi(\lambda_\mu)=F \eta_\mu$. Since $F:\mathbb{H}^2(X,\Theta_X^\bullet)\to \mathbb{H}^2(X, f^*\Theta_Y^\bullet)$ is injective, there exist $\phi_\mu'$ and $\Lambda_\mu'$ such that $\bar{\partial}\phi_\mu'=-\xi_\mu$, $\bar{\partial}\Lambda_\mu'+[\Lambda_0, \phi_\mu']=-\psi_\mu$, and $[\Lambda_0,\Lambda_\mu']=-\eta_\mu$. Then we have $\bar{\partial} (E_\mu'+F\phi_\mu')=0, \pi(\lambda_\mu'+F\Lambda_\mu')=0$, and $\bar{\partial}(\lambda_\mu'+F\Lambda_\mu')+\pi(E_\mu'+F\phi_\mu')=0$. Since $F:\mathbb{H}^1(X,\Theta_X^\bullet)\to \mathbb{H}^1(X, f^*\Theta_Y^\bullet)$ is surjective, there exist $x_\mu\in A^{0,1}(\Theta_X)$ and $y_\mu\in A^{0,0}(\wedge^2 \Theta_X)$ with $\bar{\partial} x_\mu=0, \bar{\partial} y_\mu+[\Lambda_0, x_\mu]=0$, and $[\Lambda_0, y_\mu]=0$, and there exists $\Phi_\mu'$ such that $\bar{\partial}\Phi_\mu'=Fx_\mu-( E_\mu'+F\phi_\mu'),\,\,\,\,\, \pi(\Phi'_\mu)=Fy_\mu-(\lambda_\mu'+F\Lambda_\mu')$, equivalently, $L_\pi(\Phi'_\mu)=F(x_\mu,y_\mu)-((E'_\mu,\lambda'_\mu)+F(\phi_\mu',\Lambda_\mu'))$.

Let $\phi_\mu:=\phi_\mu'-x_\mu$, $\Lambda_\mu:=\Lambda_\mu'-y_\mu$, and $\Phi_{i|\mu}:=\Phi_\mu'-\Gamma_{i|\mu}$. Then $\bar{\partial} \phi_\mu=\bar{\partial}\phi_\mu'=-\xi_\mu$, $\bar{\partial}\Lambda_\mu+[\Lambda_0,\phi_\mu]=-\psi_\mu$, $[\Lambda_0, \Lambda_\mu]=-\eta_\mu$, and we have $L_\pi(\Phi_{i|\mu})+F(\phi_\mu,\Lambda_\mu)=-(E_{i|\mu},\lambda_\mu)$. Lastly, $\Phi_{j|\mu}-\Phi_{i|\mu}=-\Gamma_{j|\mu}+\Gamma_{i|\mu}=\Gamma_{ij|\mu}$.

\end{proof}

This completes the inductive constructions of formal power series
 \begin{align}\label{qq32}
\phi(t)=\sum_{\mu=1}^\infty \phi_\mu(t),\,\,\,\,\,\Lambda(t)=\Lambda_0+\sum_{\mu=1}^\infty \Lambda_\mu(t),\,\,\,\,\,\Phi_i(z_i,t)=f_i(z_i)+\sum_{\mu=1}^\infty \Phi_{i|\mu}(z_i,t)
\end{align}
 satisfying $(\ref{ee1})-(\ref{ee2})$.

\end{proof}

\subsubsection{Remark on Proof of convergence} \label{pp40}\

Though the author could not touch the proof of convergence, it seems that we can formally apply Horikawa's method presented in \cite{Hor73} p.389-393 in the context of holomorphic Poisson deformations. But the author believes that we need a deep understanding of harmonic theories on the operators $L=\bar{\partial}+[\Lambda_0,-]$ on $\Theta_X^\bullet$ and $L_\pi=\bar{\partial}+\pi$ on $f^*\Theta_Y^\bullet$. If it would turn out that we can formally apply Horikawa's method, the proof depends on the following lemma which is the modification of the key lemma in \cite{Hor73} p.390.

\begin{lemma}[conjecture]
Suppose that $\phi\in A^{0,1}(\Theta_X), \psi\in A^{0,0}(\wedge^2 \Theta_X)$, $E\in A^{0,1}(f^* \Theta_Y)$, and $a\in A^{0,0}(\wedge^2 f^* \Theta_Y)$ satisfying
$L_\pi((E,a)+F(\phi,\psi))=0$ are given. Then we can find $x\in A^{0,1}(\Theta_X), y\in A^{0,0}(\wedge^2 \Theta_X)$, and $\Phi\in A^{0,0}(f^* \Theta_Y)$ in such a way that
\begin{align*}
\Box(x,y)&=0\\
(E,a)+F(\phi, \psi)&=F(x,y)+L_\pi(\Phi)\\
|(x,y)|_{k+\alpha}\ll K_2(|(\phi,& \psi)|_{k+\alpha}+|(E,a)|_{k+1-\alpha})\\
|\Phi |_{k+\alpha}\ll K_2(|(\phi,& \psi)|_{k+\alpha}+|(E,a)|_{k+1-\alpha})
\end{align*}
where $K_2$ is a constant which is independent of $(\phi,\psi)$, and $(E,a)$. Here $\Box=L^*L+LL^*$, where $L^*$ is the adjoint operator of $L$, and $|(-,-)|_{k+\alpha}$ is defined in a similar way as in \cite{Kim15}.
\end{lemma}

\subsubsection{Construction of a family} \label{kk20}\

Now assume that the formal power series $\phi(t),\Lambda(t)$ and $\Phi_i(z_i,t)$ satisfying $(\ref{ee1})-(\ref{ee2})$ constructed in $(\ref{qq32})$ converge for $|t|<\epsilon$ for a sufficiently small number $\epsilon$. As in \cite{Kim15}, $(\ref{ee3}),(\ref{ee111})$, and $(\ref{ee112})$ give a Poisson analytic family $p:(\mathcal{X},\Lambda)\to M$ such that $p^{-1}(0)=(X,\Lambda_0)$, and each fiber $p^{-1}(t)$ is endowed with a holomorphic Poisson structure determined by $\phi(t)$ and $\Lambda(t)$. From $(\ref{ee34}),(\ref{ee113}),(\ref{ee114})$, and $(\ref{ee2})$, $\{\Phi_i(z_i,t)\}$ defines a holomorphic map Poisson map $\Phi:(\mathcal{X},\Lambda)\to (Y\times M,\Pi_0)$.

In more detail, on each $U_i$, there exists $n$ linearly independent $C^\infty$ function $\eta_\alpha^\sigma=\eta_\alpha^\sigma(z_\alpha,t)$ with $\bar{\partial}\eta_\alpha^\sigma+[\phi,\eta_\alpha^\sigma]=0$ and $\eta_\alpha^\sigma(z_\alpha,0)=z_\alpha^\sigma$ such that $ (\eta_\alpha^1(z_\alpha,t),...,\eta_j^n(z_\alpha,t),t_1,...,t_r)$ gives complex coordinates on $U_\alpha\times M\subset \mathcal{X}$. We have $\eta_\alpha^\sigma=\phi_{\alpha\beta}^\sigma(\eta_\beta,t)$ on $U_{\alpha\beta}$, where $\phi_{\alpha\beta}^\sigma$ are holomorphic functions of $(\eta_\beta,t)$ on $U_{\alpha\beta}$. $\Lambda$ is given by $\Lambda_\alpha=\sum_{p,q=1}^n g_{pq}^\alpha(\eta_\alpha,t)\frac{\partial}{\partial \eta_\alpha^p }\wedge \frac{\partial}{\partial \eta_\alpha^q}$ on $U_\alpha\times M$ such that $g_{pq}^\alpha(\eta_\alpha(z_\alpha,t),t)=\sum_{r,s=1}^n \Lambda_{rs}^\alpha(z_\alpha,t)\frac{\partial \eta_\alpha^p}{\partial z_\alpha^s}\frac{\partial \eta_\alpha^q}{\partial z_\alpha^s}$. On the other hand $\Phi$ is given by $w_\alpha^\lambda=\Psi_\alpha^\lambda(\eta_\alpha,t)$ on $U_\alpha$, where $\Psi_\alpha^\lambda$ are holomorphic functions of $(\eta_\alpha,t)$. Then $\Phi_{\alpha}^\lambda(z_\alpha,t)=\Psi_\alpha^\lambda(\eta_\alpha(z_\alpha,t),t)$ on $U_\alpha$ and $\Pi_{ab}^i(\Psi_\alpha^\lambda(\eta_\alpha(z_\alpha,t),t))=\sum_{r,s=1}^n \Lambda_{rs}^\alpha(z_i,t)\frac{\partial \Psi_\alpha^a(\eta_\alpha,t)}{\partial z_\alpha^r}\frac{\partial \Psi_\alpha^b(\eta_\alpha,t)}{\partial z_\alpha^s}=\sum_{p,q=1}^n g_{pq}^\alpha(\eta_\alpha(z_\alpha,t),t)\frac{\partial \Psi_\alpha^a}{\partial \eta_\alpha^p}\frac{\partial \Psi_\alpha^b}{\partial \eta_\alpha^q}$ so that $\Psi_\alpha^\lambda$ is a holomorphic Poisson map.

Let $\left(\frac{\partial}{\partial t}\right)\in T_0(M)$ and let $``\,\,\,\dot{}\,\,\,"$ denote the operation $\frac{\partial}{\partial t}|_{t=0}$. Then we have (see \cite{Kim15})
\begin{align}
&\bar{\partial}(\sum_{\sigma=1}^n \dot{\eta}_\alpha^\sigma \frac{\partial}{\partial z_\alpha^\sigma})=-\dot{\phi} \,\,\,\,\,\,\,\,\,\,\,\,\,\,\,\,\,\,\,\,\,\,\,\,\,\,\text{on $U_\alpha$}, \label{kk30}\\
&\sum_{\sigma=1}^n \dot{\eta}_\alpha^\sigma\frac{\partial}{\partial z_\alpha^\sigma}=\sum_{\sigma=1}^n \dot{\eta}_\beta^\sigma \frac{\partial}{\partial z_\beta^\sigma}+\sum_{\sigma=1}^n \dot{\phi}_{\alpha\beta}^\sigma\frac{\partial}{\partial z_\alpha^\sigma},\,\,\,\,\,\,\,\,\text{ on $U_{\alpha\beta}$}, \label{kk31}\\
&\sum_{p,q=1}^n \dot{\Lambda}_{pq}^\alpha\frac{\partial}{\partial z_\alpha^p}\wedge \frac{\partial}{\partial z_\alpha^q} -\sum_{p,q=1}^n \dot{g}_{pq}^\alpha\frac{\partial}{\partial z_\alpha^p}\wedge \frac{\partial}{\partial z_\alpha^q}   + [\Lambda_0, \sum_{\sigma=1}^n \dot{\eta}_\alpha^\sigma\frac{\partial}{\partial z_\alpha^\sigma}  ]=0,\,\,\,\,\,\,\,\,\,\text{on $U_\alpha$},  \label{kk32}   \\
&\sum_{\rho=1}^m \dot{\Phi}_\alpha^\rho\frac{\partial}{\partial w_\alpha^\rho}=\sum_{\rho=1}^m \dot{\Psi}_\alpha^\rho\frac{\partial}{\partial w_\alpha^\rho}+F(\sum_{\sigma=1}^n \dot{\eta}_\alpha^\sigma\frac{\partial}{\partial z_\alpha^\sigma}), \,\,\,\,\,\,\,\,\text{on $U_\alpha$}. \label{kk33}
\end{align}
We note that $\tau\left( \frac{\partial}{\partial t}  \right)$ is represented by $(-\sum_{\rho=1}^m \dot{\Psi}_\alpha^\rho\frac{\partial}{\partial w_\alpha^\rho}, \sum_{\sigma=1}^n \dot{\phi}_{\alpha\beta}^\sigma\frac{\partial}{\partial z_\alpha^\sigma}, \sum_{p,q=1}^n \dot{g}_{pq}^\alpha\frac{\partial}{\partial z_\alpha^p}\wedge \frac{\partial}{\partial z_\alpha^q})\in C^0(\mathcal{U}, f^*\Theta_Y)\oplus C^1(\mathcal{U}, \Theta_X)\oplus C^0(\mathcal{U},\wedge^2 \Theta_X)$. Hence by the isomorphism in the proof of Lemma \ref{kk10}, $\tau:T_0(M)\to PD_{(X,\Lambda_0)/(Y,\Pi_0)}$ is bijective. This completes the proof of Theorem \ref{hh1}.

\section{Deformations of holomorphic Poisson maps into a Poisson analytic family}\label{section3}

\begin{definition}\label{pp27}
By a Poisson analytic family of holomorphic Poisson maps into a Poisson analytic family $(\mathcal{Y},\Pi,q,S)$, we mean a collection $(\mathcal{X},\Lambda,\Phi, p, M,s)$ where $(\mathcal{X},\Lambda, p,M)$ is a Poisson analytic family, $\Phi:(\mathcal{X},\Lambda)\to (\mathcal{Y},\Pi)$ is a holomorphic Poisson map, and $s:M\to S$ such that $s\circ p=q\circ \Phi$. Two families $(\mathcal{X},\Lambda, \Phi,p,M,s)$ and $(\mathcal{X}',\Lambda',M',s')$ are said to be equivalent if there exist a holomorphic Poisson isomorphism $g:(\mathcal{X},\Lambda)\to (\mathcal{X}',\Lambda')$ and an isomorphism $h:M\to M'$ such that the following diagram commutes
\[\xymatrix{
(\mathcal{X},\Lambda) \ar[dd]^p \ar[drr]^g \ar[rrrr]^\Phi & & && (\mathcal{Y},\Pi) \ar[dd]^q\\
& & (\mathcal{X}',\Lambda') \ar[dd]^{p'}   \ar[rru]^{\Phi'} \\
M  \ar[drr]^h \ar[rrrr]^s  & & & &S\\      
& &M' \ar[rru]^{s'} &       \\
}\]
If $(\mathcal{X},\Lambda, \Phi,p,M,s)$ is a family of holomorphic Poisson maps into $(\mathcal{Y},\Pi,q,S)$, and $h:N\to M$ is a holomorphic map, we can define the family $(\mathcal{X}',\Lambda',\Phi',p',N,s')$ induced by $h$ as follows.
\begin{enumerate}
\item $(\mathcal{X}',\Lambda')=(\mathcal{X},\Lambda)\times_M N$.
\item $\Phi': (\mathcal{X}',\Lambda')\to (\mathcal{Y},\Pi)$ in the following way:
\begin{center}
$\begin{CD}
(x,t)\in(\mathcal{X},\Lambda)\times_M N:=(\mathcal{X}',\Lambda')@>\Phi'>> \Phi(x)\in (\mathcal{Y},\Pi)\\
@Vp'VV @VVqV\\
t\in N@>s':=s\circ h>>  s(h(t))\in S
\end{CD}$
\end{center}
\item $p'=p_N:(\mathcal{X}',\Lambda')\to N$.
\end{enumerate}
\end{definition}

\begin{definition}\label{pp28}
A family $(\mathcal{X},\Lambda,\Phi, p,M,s)$ of holomorphic Poisson maps into $(\mathcal{Y},\Pi,q,S)$ is complete at $0\in M$ if, for any family $(\mathcal{X}',\Lambda',\Phi',p', N,s')$ of holomorphic Poisson maps into $(\mathcal{Y},\Pi,q,S)$ such that $\Phi_{0'}':(X_{0'}',\Lambda_{0'}')\to (\mathcal{Y},\Pi)$ is equivalent to $\Phi_0:(X_0,\Lambda_0)\to (\mathcal{Y},\Pi)$ for a point $0'\in N$, there exists a holomorphic map $h$ of a neighborhood $U$ of $0'$ in $N$ into $M$ with $h(0')=0$ such that the restriction of $(\mathcal{X}',\Lambda',\Phi', p', N,s')$ on $U$ is equivalent to the family induced by $h$ from $(\mathcal{X},\Lambda, \Phi, p, M,s)$.
\end{definition}

\begin{remark}
Let $(\mathcal{X}, \Lambda, \Phi, p, M, s)$ be a Poisson analytic family of holomorphic Poisson maps into $(\mathcal{Y},\Pi, q, S)$, $0\in M$, $(X,\Lambda_0)=(X_0,\Lambda_0)$, and let $\tilde{f}:(X,\Lambda_0)\to (\mathcal{Y},\Pi)$ be the restriction of $\Phi$ to $(X,\Lambda_0)$. Let $\tilde{F}:\Theta_X^\bullet \to \tilde{f}^* \Theta_{\mathcal{Y}}^\bullet$  be the canonical homomorphism. If $\mathcal{U}=\{U_i\}$ is a Stein covering of $X$, then we have
\begin{align*}
PD_{(X,\Lambda_0)/(\mathcal{Y},\Pi)}=\frac{\{(\tilde{\tau},\rho,\lambda) \in C^0(\mathcal{U},\tilde{f}^*\Theta_{\mathcal{Y}})\oplus\mathcal{C}^1(\mathcal{U},\Theta_X)\oplus C^0(\mathcal{U},\wedge^2 \Theta_X)|\frac{-\delta \tilde{\tau}=\tilde{F}\rho, \tilde{\pi}( \tilde{\tau} ) =\tilde{F}\lambda}{\delta\rho=0,\delta \lambda+[\Lambda_0, \rho]=0, [\Lambda_0,\lambda]=0}\}}{\{(\tilde{F}g,-\delta g,[\Lambda_0, g])|g\in C^0(\mathcal{U}, \Theta_X)\}}
\end{align*}
\end{remark}

\subsection{Infinitesimal deformations} \label{qq20}\

Let $(\mathcal{X}, \Lambda, \Phi, p, M, s)$ be a Poisson analytic family of holomorphic Poisson maps into $(\mathcal{Y},\Pi, q, S)$, $0\in M$, $(X,\Lambda_0)=(X_0,\Lambda_0)$, and let $\tilde{f}:(X,\Lambda_0)\to (\mathcal{Y},\Pi)$ be the restriction of $\Phi$ to $(X,\Lambda_0)$. Let $0^*=s(0), (Y,\Pi_0)=(Y_{0^*},\Pi_{0^*})$, and let $f:(X,\Lambda_0)\to (Y,\Pi_0)$ be the holomorphic Poisson map induced by  $ \Phi$. Let $F:\Theta_{X}^\bullet \to f^*\Theta_{Y}^\bullet$ be the canonical homomorphism. We shall define a linear map $\tau:T_0(M)\to PD_{(X,\Lambda_0)/(\mathcal{Y},\Pi)}$. We may assume the following:
\begin{enumerate}
\item $M$ is an open set in $\mathbb{C}^r$ with coordinates $t=(t_1,...,t_r)$ and $0=(0,...,0)$.
\item $\mathcal{X}$ is covered by a finite number of Stein coordinate neighborhoods $\mathcal{U}_i$ Each $\mathcal{U}_i$ is covered by a system of coordinates $(z_i,t)=(z_i^1,...,z_i^n,t_1,...,t_r)$ such that $p(z_i,t)=t$.
\item $S$ is an open set in $\mathbb{C}^{r'}$ with a system of coordinate $s=(s^1,...,s^{r'})$ and $0^*=(0,...,0)$.
\item $\mathcal{Y}$ is covered by a finite number of Stein coordinate neighborhoods $\mathcal{V}_i$ with a system of coordinates $(w_i,s)=(w_i^1,...,w_i^m,s^1,...,s^{r'})$ with $q(w_i,s)=s$, such that $\Phi_i(\mathcal{U}_i)\subset \mathcal{V}_i$. In terms of these coordinates $\Phi$ is given by $\Phi(z_i,t)=(\Phi_i(z_i,t),s(t))$ with $w_i=\Phi_i(z_i,t)$, and let $f_i(z_i)=\Phi_i(z_i,0)$.
\item $(z_i,t)\in \mathcal{U}_i$ coincides with $(z_j,t)\in \mathcal{U}_j$ if and only if $z_i=\phi_{ij}(z_j,t)$. We set $b_{ij}(z_j)=\phi_{ij}(z_j,0)$.
\item $(w_i,s)\in \mathcal{V}_i$ coincides with $(w_j,s)\in \mathcal{V}_j$ if and only if $w_i=\psi_{ij}(w_j,s)$. We set $g_{ij}(w_j)=\psi_{ij}(w_j,0)$. 
\item On $\mathcal{U}_i$, the holomorphic Poisson structure $\Lambda$ is of the form $ \Lambda=\sum_{\alpha,\beta=1}^n \Lambda_{\alpha\beta}^i (z_i,t)\frac{\partial}{\partial z_i^\alpha}\wedge \frac{\partial}{\partial z_i^\beta}$. We set $\Lambda_{\alpha\beta}^i(z_i)=\Lambda_{\alpha\beta}^i(z_i,0)$ with $\Lambda_{\alpha\beta}^i(z_i,t)=-\Lambda_{\beta\alpha}^i(z_i,t)$. Then  the holomorphic Poisson structure $\Lambda_0$ is of the form $\Lambda_0=\sum_{\alpha,\beta=1}^n \Lambda_{\alpha\beta}^i(z_i)\frac{\partial}{\partial z_i^\alpha}\wedge \frac{\partial}{\partial z_i^\beta}$ on $X\cap \mathcal{U}_i$.
\item On $\mathcal{V}_i$, the holomorphic Poisson structure $\Pi$ is of the form $\Pi=\sum_{\alpha,\beta=1}^m \Pi_{\alpha\beta}^i(w_i,s)\frac{\partial}{\partial w_i^\alpha}\wedge \frac{\partial}{\partial w_i^\beta}$ with $\Pi_{\alpha\beta}^i(w_i,s)=-\Pi_{\beta\alpha}^i(w_i,s)$.
\end{enumerate}

Then we have
\begin{align}
\Phi_i(\phi_{ij}(z_j,t),t)&=\psi_{ij}(\Phi_j(z_j,t),s(t)) \label{dd1}\\
\phi_{ij}(\phi_{jk}(z_k,t))&=\phi_{ik}(z_k,t) \label{dd2}\\
[\Lambda,\Lambda]&=0 \label{dd3}\\
\sum_{\alpha,\beta=1}^n \Lambda_{\alpha\beta}^j(z_j,t)\frac{\partial \phi_{ij}^p}{\partial z_j^\alpha}\frac{\partial \phi_{ij}^q}{\partial z_j^\beta}&=\Lambda_{pq}^i(\phi_{ij}(z_j,t),t) \label{dd4}\\
\sum_{\alpha,\beta=1}^n \Lambda_{\alpha\beta}^i(z_i,t)\frac{\partial \Phi_i^p}{\partial z_i^\alpha}\frac{\partial \Phi_i^q}{\partial z_i^\beta}&=\Pi_{pq}^i(\Phi_i(z_i,t),s(t)) \label{dd5}
\end{align}

We let $U_i=X\cap \mathcal{U}_i$ and denote by $\mathcal{U}$ the covering $\{U_i\}$ of $X$. For any element $\frac{\partial}{\partial t}\in T_0(M)$, let
\begin{align}
-\tau_i&=\sum_{\beta=1}^m \frac{\partial \Phi_i^{\beta}}{\partial t}|_{t=0}\frac{\partial}{\partial w_i^\beta}+\sum_{v=1}^{r'}\frac{\partial s_v}{\partial t}|_{t=0}\frac{\partial}{s_v}\in \Gamma(U_i, \tilde{f}^* \Theta_{\mathcal{Y}}) \label{dd6}\\
\rho_{ij}&=\sum_{\alpha=1}^n \frac{\partial \phi_{ij}^\alpha}{\partial t}|_{t=0}\frac{\partial}{\partial z_i^\alpha}\in \Gamma(U_{ij}, \Theta_X) \label{dd7}\\
\Lambda_i'&=\sum_{\alpha,\beta=1}^n \frac{\partial \Lambda_{\alpha\beta}^i(z_i,t)}{\partial t}|_{t=0}\frac{\partial}{\partial z_i^\alpha}\wedge \frac{\partial}{\partial z_i^\beta}\in \Gamma(U_i, \wedge^2 \Theta_X)\label{dd8}
\end{align}

Then we claim that
\begin{align}
\tau_i-\tau_j&=\tilde{F}\rho_{ij} \label{dd8}\\
\pi(\tau_i)&=\tilde{F}\Lambda_i' \label{dd9}\\
\rho_{jk}-\rho_{ik}+\rho_{ij},\,\,\,\,\,\Lambda_j'-\Lambda_i'&+[\Lambda_0, \rho_{ij}]=0,\,\,\,\,\,[\Lambda_0,\Lambda_i']=0 \label{dd10}
\end{align}

$(\ref{dd8})$ and $(\ref{dd10})$ follow from $(\ref{dd1}),(\ref{dd2}),(\ref{dd3})$ and $(\ref{dd4})$. It remains to show $(\ref{dd9})$. It is sufficient to show that $\pi(\tau_i)(w_i^p,w_i^q)=\tilde{F}(\Lambda_i')(w_i^p,w_i^q)$ for any $p,q$. We note that by taking the derivative of $(\ref{dd5})$ with respect to $t$ and setting $t=0$, we get
\begin{align}
\sum_{\alpha,\beta=1}^n \frac{\partial \Lambda_{\alpha\beta}^i}{\partial t}|_{t=0}\frac{\partial f_i^p}{\partial z_i^\alpha}\frac{\partial f_i^q}{\partial z_i^\beta}+\sum_{\alpha,\beta=1}^n\Lambda_{\alpha\beta}^i(z_i)\frac{\partial}{\partial z_i^\alpha}\left(\frac{\partial \Phi_i^p}{\partial t}|_{t=0}\right)\frac{\partial f_i^q}{\partial z_i^\beta}+\sum_{\alpha,\beta=1}^n\Lambda_{\alpha\beta}^i(z_i)\frac{\partial f_i^p}{\partial z_i^\alpha}\frac{\partial}{\partial z_i^\beta}\left(\frac{\partial \Phi_i^q}{\partial t}|_{t=0}\right) \label{dd18}\\
=\sum_{\rho=1}^m \frac{\partial \Pi_{pq}^i}{\partial w_i^\rho}(f_i,0)\frac{\partial \Phi_i^\rho}{\partial t}|_{t=0}+\sum_{v=1}^{r'}\frac{\partial \Pi_{pq}^i}{\partial s_v}(f_i,0)\frac{\partial s_v}{\partial t}|_{t=0} \notag
\end{align}

Let us consider each side of $(\ref{dd9})$,
\begin{align}\label{dd19}
\tilde{F}\left(    \sum_{\alpha,\beta=1}^n \frac{\partial \Lambda_{\alpha\beta}^i(z_i,t)}{\partial t}|_{t=0}\frac{\partial}{\partial z_i^\alpha}\wedge \frac{\partial}{\partial z_i^\beta}    \right)(w_i^p,w_i^q)=2\sum_{\alpha,\beta=1}^n \frac{\partial \Lambda_{\alpha\beta}^i}{\partial t}|_{t=0}\frac{\partial f_i^p}{\partial z_i^\alpha}\frac{\partial f_i^q}{\partial z_i^\beta}
\end{align}

On the other hand,
\begin{align}
\pi(\tau_i)(w_i^p,w_i^q)&=\Lambda_0(\tau_i(w_i^p),f_i^q)-\Lambda_0(\tau_i(w_i^q),f_i^p)-\tau_i(\Pi(w_i^p,w_i^q)) \label{dd20}\\
&=-2\sum_{\alpha,\beta=1}^n\Lambda_{\alpha\beta}^i(z_i)\frac{\partial }{\partial z_i^\alpha}\left( \frac{\partial \Phi_i^p}{\partial t}|_{t=0}\right)\frac{\partial f_i^q}{\partial z_i^\beta}+2\sum_{\alpha,\beta=1}^n \Lambda_{\alpha\beta}^i(z_i)\frac{\partial}{\partial z_i^\alpha}\left( \frac{\partial \Phi_i^q}{\partial t}|_{t=0}\right)\frac{\partial f_i^p}{\partial z_i^\beta} \notag\\
&+2\sum_{\rho=1}^m\frac{\partial \Phi_i^\rho}{\partial t}|_{t=0} \frac{\partial \Pi_{pq}^i}{\partial w_i^\rho}(f_i,0) +2\sum_{v=1}^{r'}\frac{\partial s_v}{\partial t}|_{t=0}\frac{\partial \Pi_{pq}^i}{\partial  s_v}(f_i,0) \notag
\end{align}
Then $(\ref{dd9})$ follows from $(\ref{dd18}),(\ref{dd19})$, and $(\ref{dd20})$.

Hence from $(\ref{dd8}),(\ref{dd9})$, and $(\ref{dd10})$, the collection $(\{\tau_i\},\{\rho_{ij}\},\{\Lambda_i'\})$ defines an element in $PD_{(X,\Lambda_0)/(\mathcal{Y},\Pi)}$ and we have a linear map
\begin{align*}
\tau:T_0(M)\to PD_{(X,\Lambda_0)/(\mathcal{Y},\Pi)}
\end{align*}
We call $\tau$ the characteristic map of the family at $0$.

\begin{remark}
 We have a canonical homomorphism $\pi: PD_{(X,\Lambda_0)/(\mathcal{Y},\Pi)}\to \mathbb{H}^1(X,\Theta_X^\bullet)$ such that $\pi\circ \tau:T_0(M)\to \mathbb{H}^1(X,\Theta_X^\bullet)$ gives the Poisson Kodaira-Spencer map of the Poisson analytic family $(\mathcal{X},\Lambda, p, M)$.
\end{remark}

\subsection{Two representations of $PD_{(X,\Lambda_0)/(\mathcal{Y},\Pi)}$ in terms of $f^* \Theta_Y^\bullet$}\

Let $\rho':T_{0^*}(S)\to \mathbb{H}^1(Y,\Theta_Y^\bullet)$ be the Poisson Kodaira-Spencer map of the Poisson analytic family $(\mathcal{Y},\Pi,q,S)$ at $0^*$. Then each $f^*\rho'(\frac{\partial}{\partial s^v})\in \mathbb{H}^1(X, f^*\Theta_Y^\bullet)$ is represented by a $1$-cocycle $({\rho}_v',\gamma_v)=(\{{\rho}_{vij}'\},\{ \gamma_{vi}\})$  with
\begin{align*}
{\rho}_{vij}'=\sum_{\lambda=1}^m \frac{\partial \psi_{ij}^\lambda}{\partial s_j^v}(f_j(z_j),0)\frac{\partial}{\partial w_i^\lambda},\,\,\,\,\, \gamma_{vi}=\sum_{\alpha,\beta=1}^m\frac{\partial \Pi_{\alpha\beta}^i}{\partial s_i^v}(f_i(z_i),0)\frac{\partial}{\partial w_i^\alpha}\wedge \frac{\partial}{\partial w_i^\beta}
\end{align*}

\begin{lemma}\label{kk22}
We have an isomorphism
\begin{align*}
&PD_{(X,\Lambda_0)/(\mathcal{Y},\Pi)}\\
&=\frac{\{(\tau,\rho, \lambda, (\theta^v))\in C^0(\mathcal{U}, f^*\Theta_Y)\oplus C^1(\mathcal{U}, \Theta_X)\oplus C^0(\mathcal{U},\wedge^2 \Theta_X)\oplus \mathbb{C}^{r'}| \frac{-\delta \tau=F\rho+\sum_{v=1}^{r'} \theta^v {\rho}_v', \pi(\tau)=F \lambda+\sum_{v=1}^{r'}\theta^v \gamma_v}{\delta \rho=0, \delta \lambda+[\Lambda,\rho]=0,[\Lambda,\lambda]=0 }\}}{ \{ (F \xi, -\delta \xi, [\Lambda_0, \xi],0)| \xi\in C^0(\mathcal{U},\Theta_X)\}}
\end{align*}
\end{lemma}

\begin{proof}
Let $(\tilde{\tau}=\{\tilde{\tau}_i\},\rho=\{\rho_{ij}\},\lambda=\{\lambda_i\})\in C^0(\mathcal{U},\tilde{f}^*\Theta_{\mathcal{Y}})\oplus C^1(\mathcal{U},\Theta_X)\oplus C^0(\mathcal{U},\wedge^2 \Theta_X)$ be a representative of an element of $PD_{(X,\Lambda_0)/(\mathcal{Y},\Pi)}$. We write each $\tilde{\tau}_i$ by
\begin{align*}
\tilde{\tau}_i=\sum_{\alpha=1}^m \tau_i^\lambda\frac{\partial}{\partial w_i^\alpha}+\sum_{v=1}^{r'} \theta_i^v \frac{\partial}{\partial s_i^v},\,\,\,\,\,\,\,\text{on $U_i$}.
\end{align*}

Let $\tau_i=\sum_{\alpha=1}^m \tau_i^\lambda\frac{\partial}{\partial w_i^\alpha}$. From $-\delta \tilde{\tau}=\tilde{F}\rho$, we get $\theta^v=\theta_i^v=\theta_j^v \in \mathbb{C}$, and $-\delta({\tau})=F\rho+\sum_{v=1}^{r'} \theta^v \rho_{v}'$ (for the detail, see \cite{Hor74} p.655). On the other hand, from $\pi(\tilde{\tau})=Fl$, and
\begin{align*}
\tilde{\pi}(\tilde{\tau}_i)=\pi(\tau_i)+[\sum_{\alpha,\beta=1}^m \Pi_{\alpha\beta}^i(w_i,s_i)\frac{\partial}{\partial w_i^\alpha}\wedge \frac{\partial}{\partial w_i^\beta},\sum_{v=1}^{r'}\theta^v \frac{\partial}{\partial s_i^v}]|_{w_i=f_i(z_i)}=\pi(\tau_i)-\sum_{v=1}^{r'}\theta^v \gamma_{vi}
\end{align*}
 we have $\pi(\tau)=Fl+\sum_{v=1}^{r'}\theta^v \gamma_{vi}.$ Then $(\tilde{\tau},\rho, \lambda  )\to (\tau,\rho, \lambda ,(\theta^v))$ gives an isomorphism.

\end{proof}

\begin{remark}

 An element $(\rho,\lambda)\in \mathbb{H}^1(X,\Theta_X^\bullet)$ is in the image of $\pi$ if and only if $(F\rho,F\lambda)\in \mathbb{H}^1(X, f^*\Theta_Y^\bullet)$ is contained in the image of $f^*\circ \rho':T_{0^*}(S)\to \mathbb{H}^1(X,f^*\Theta_Y^\bullet)$.
\end{remark}

Since $f^*\rho'(\frac{\partial}{\partial s^v})=(\rho_{vij}',\gamma_{vi})$ is a $1$-cocycle, $\delta \{\rho_{vij}'\}=0$ so that there exist $\xi_{vi}\in \Gamma(U_i,\mathcal{A}^{0,0}(f^*\Theta_Y))$ such that $\rho_{vij}'=\xi_{vi}-\xi_{vj}$. Let $\tilde{\varphi}_v:=\bar{\partial}\xi_{vi}$ and let $\tilde{\chi}_v:=\pi(\xi_{vi})-\gamma_{vi}$.

\begin{lemma}\label{kk23}

We have an isomorphism
\begin{align*}
&PD_{(X,\Lambda_0)/(\mathcal{Y},\Pi)}\\
&=\frac{(\Phi, \phi, \psi ,(\theta^v))\in A^{0,0}(f^* \Theta_Y)\oplus A^{0,1}(\Theta_X)\oplus A^{0,0}(\wedge^2 \Theta_X) \oplus \mathbb{C}^{r'}| \frac{ \bar{\partial} \Phi=F \phi+\sum_{v=1}^{r'} \theta^v \tilde{\varphi}_v, \pi(\Phi)=F\psi+\sum_{v=1}^{r'} \theta^v \tilde{\chi}_v,}{ \bar{\partial} \phi=0,\bar{\partial}\psi +[\Lambda_0,\phi]=0, [\Lambda_0,\psi]=0 } \}}{\{(F \xi,\bar{\partial}\xi, [\Lambda_0, \xi],0 |\xi\in A^{0,0}(\Theta_X)\}}
\end{align*}

\end{lemma}

\begin{proof}

Let $(\tau=\{\tau_i\},\rho=\{\rho_{ij}\}, \lambda=\{\lambda_i\},(\theta^v))$ be a representative of an element of $PD_{(X,\Lambda_0)/(\mathcal{Y},\Pi)}$ as in Lemma \ref{kk22}. Since $\delta \rho=0$, there exists $\eta_i\in \Gamma(U_i, \mathcal{A}^{0,0}(\Theta_X))$ such that $\rho_{ij}=\eta_i-\eta_j$. Then $\phi:=\bar{\partial}\eta_i\in A^{0,1}(\Theta_X)$ and $\psi:=[\Lambda_0,\eta_i]-\lambda_i\in A^{0,0}(\wedge^2 \Theta_X)$. On the other hand, $\tau_i-\tau_j=F\eta_i-F\eta_j+\sum_v \theta^v(\xi_{vi}-\xi_{vj})$. We define $\Phi_i=F\eta_i-\tau_i+\sum_{v=1}^{r'} \theta^v \xi_{vi}$ on $U_i$. Then $\Phi:=\{\Phi_i\}\in A^{0,0}(f^*\Theta_Y)$ with $\bar{\partial}\Phi=F\phi+\sum_{v=1}^{r'} \theta^v \tilde{\varphi}_{v}$, and $\pi(\Phi)=F[\Lambda_0,\eta_i]-\pi(\tau_i)+\sum_{v=1}^{r'}\theta^v \pi(\xi_{vi})=F[\Lambda_0,\eta_i]-F\lambda_i-\sum_{v=1}^{r'} \theta^v \gamma_{vi}+\sum_{v=1}^{r'} \theta^v \pi(\xi_{vi})=F\psi+\sum_{v=1}^{r'} \theta^v \tilde{\chi}_v$.
Then $(\tau,\rho,\lambda,(\theta^v))\to (\Phi, \phi, \psi, ( \theta^v  ) )$ gives an isomorphism.
\end{proof}

\subsection{Theorem of completeness}

\begin{theorem}[Theorem of completeness]\label{ii5}
Let $(\mathcal{X},\Lambda,\Phi,p, M, s)$ be a family of holomorphic Poisson maps into a family $(\mathcal{Y},\Pi, q,S),0\in M, (X,\Lambda_0)=(X_0,\Lambda_0)$, and let $\tilde{f}:(X,\Lambda_0)\to (\mathcal{Y},\Pi)$ be the restriction of $\Phi$ to $(X,\Lambda_0)$. If the characteristic map $\tau:T_0(M)\to PD_{(X,\Lambda_0)/(\mathcal{Y},\Pi)}$ is surjective, then the family is complete at $0$.
\end{theorem}

\begin{proof}
We extend the arguments in \cite{Hor74} p.655-656 in the context of holomorphic Poisson deformations. We tried to maintain notational consistency with \cite{Hor74}. Let $(\mathcal{X}',\Lambda',\Phi',p',M',s')$ be another family of holomorphic Poisson maps into $(\mathcal{Y},\Pi, q, S)$. Assume that there exists a point $0'\in M'$ such that the restriction $\tilde{f}':(X_{0'}',\Lambda_{0'}')\to (\mathcal{Y},\Pi)$  to $(X_{0'}',\Lambda_{0'}')=p'^{-1}(0')$ is equivalent to $\tilde{f}$. By defining $\Psi=(\Phi,p):(\mathcal{X},\Lambda)\to (\mathcal{Y}\times M, \Pi)$, we have a deformation of holomorphic Poisson maps of $\tilde{f}$ into $(\mathcal{Y},\Pi)$ over $M$ in the sense of Definition \ref{ll35} such that the following commutative diagram commutes.
\begin{center}
$\begin{CD}
(\mathcal{X},\Lambda)@>\Psi=(\Phi, p)>> (\mathcal{Y}\times M, \Pi) @>pr_1>> (\mathcal{Y}, \Pi)\\
@VpVV @Vpr_2VV @VqVV\\
M @>>> M @>s>> S
\end{CD}$
\end{center}
Similarly by defining $\Psi'=(\Phi',p'):(\mathcal{X}', \Lambda')\to (\mathcal{Y}\times M',\Pi)$ over $M'$, we have a deformation of holomorphic Poisson maps of $\tilde{f}'$ into $(\mathcal{Y},\Pi)$ over $M'$ in the sense of Definition \ref{ll35} such that
\begin{center}
$\begin{CD}
(\mathcal{X},\Lambda)@>\Psi'=(\Phi', p')>> (\mathcal{Y}\times M', \Pi) @>pr_1>> (\mathcal{Y}, \Pi)\\
@VpVV @Vpr_2VV @VqVV\\
M' @>>> M' @>s'>> S
\end{CD}$
\end{center}

Then since $\tau:T_0(M)\to PD_{(X,\Lambda_0)/(\mathcal{Y},\Pi)}$ is surjective, by Theorem \ref{ii1}, we have a holomorphic Poisson maps $g:(\mathcal{X}',\Lambda')\to (\mathcal{X},\Lambda)$ and a holomorphic map $h:M'\to M$ (by replacing $M'$ by an open neighborhood of $0'$ if necessary) such that $(\mathcal{X}',\Lambda',\Psi',p',M')$ is induced by $h$ from $(\mathcal{X},\Lambda, \Psi, p,M)$, and the following diagram commutes.
\[\xymatrix{
(\mathcal{X}',\Lambda') \ar[dd]^{p'} \ar[drr]^g \ar[rrrr]^{\Psi'=(\Phi',p')}& & & & (\mathcal{Y}\times M',\Pi) \ar[dd] \ar[drr] \ar[rrrr]^{pr_1} & & && (\mathcal{Y},\Pi) \ar[dd]^q\\
& & (\mathcal{X},\Lambda) \ar[dd]^p \ar[rrrr]^{\Psi=(\Phi,p)} & & & & (\mathcal{Y}\times M,\Pi) \ar[dd]   \ar[rru]^{pr_1} \\
M'  \ar[drr]^h \ar[rrrr]^{id}& & & & M'  \ar[drr]^h \ar[rrrr]^{s'}  & & & &S\\      
& & M \ar[rrrr]^{id} & & & & M \ar[rru]^{s} &       \\
}\]

which proves Theorem \ref{ii5}.

\end{proof}

\subsection{Theorem of existence}\

\begin{theorem}[Theorem of existence]\label{ii80}
Let $f:(X,\Lambda_0)\to (Y, \Pi_0)$ be a holomorphic Poisson map of compact holomorphic Poisson manifolds. Let $(\mathcal{Y},\Pi, q,S)$ be a Poisson analytic family, $0^*\in S, (Y,\Pi_0)=(Y_{0^*},\Pi_{0^*})$, and let $\rho':T_{0^*}(S)\to \mathbb{H}^1(Y,\Theta_Y^\bullet)$ be the Poisson Kodaira-Spencer map of $(\mathcal{Y},\Pi,q,S)$. Assume that
\begin{enumerate}
\item $\mathbb{H}^1(X,f^* \Theta_Y^\bullet)$ is generated by the image of $F:\mathbb{H}^1(X,\Theta_X^\bullet  )  \to \mathbb{H}^1(X, f^*\Theta_Y^\bullet)$ and the image of $f^*\circ \rho': T_{0^*}(S)\to \mathbb{H}^1(X, f^* \Theta_Y^\bullet )$.
\item $F: \mathbb{H}^2(X, \Theta_X^\bullet)\to \mathbb{H}^2(X, f^*\Theta_Y^\bullet)$ is injective.
\end{enumerate}
If the formal power series $(\ref{qq50})$ constructed in the proof converge, then there exists a family $(\mathcal{X},\Lambda,\Phi,p, M,s)$ of holomorphic Poisson maps into $(\mathcal{Y},\Pi,q,S)$ and a point $0\in M$ such that
\begin{enumerate}
\item $s(0)=0^*,(X,\Lambda_0)=p^{-1}(0)$ and $\Phi_0$ coincides with $f$,
\item $\tau: T_0(M)\to PD_{(X,\Lambda_0)/(\mathcal{Y},\Pi)}$ is bijective.
\end{enumerate}

\end{theorem}

\begin{remark}
If $f$ is non-degenerate, then the above condition $(1)$ and $(2)$ are reduced to the following: $P\circ f^*\circ \rho':T_{0^*}(S)\to \mathbb{H}^1(X, \mathcal{N}_f^\bullet)$ is surjective.
\end{remark}

\begin{proof}
We extend the arguments in \cite{Hor74} p.657-660 in the context of holomorphic Poisson deformations. We tried to keep notational consistency with \cite{Hor74}.

 We take systems of coordinates on $X$ and on $\mathcal{Y}$ as in . Moreover, we assume that $U_i=\{z_i\in \mathbb{C}^n||z_i|<1\}$.
  
 Let $r=\dim PD_{(X,\Lambda_0)/(\mathcal{Y},\Pi)}$ and $M=\{t\in \mathbb{C}^r||t|<\epsilon\}$ with a sufficiently small number $\epsilon>0$. We regard $X\times M$ as  a differentiable manifold and prove the existence of a vector $(0,1)$-form
 \begin{align*}
 \phi(t)=\sum_{v=1}^n \phi_i^v(z_i,t)\frac{\partial}{\partial z_i^v}=\sum_{v,\alpha=1}^n \phi_{i\alpha}^v(z_i,t)d\bar{z}_i^\alpha\frac{\partial}{\partial z_i^v}
 \end{align*}
 and the existence of a bivector field of the form
 \begin{align*}
 \Lambda(t)=\sum_{\alpha,\beta=1}^n \Lambda_{\alpha\beta}^i(z_i,t)\frac{\partial}{\partial z_i^\alpha}\wedge\frac{\partial}{\partial z_i^\beta}
 \end{align*}
 depending holomorphically on $t$, vector-valued differentiable functions $\Phi_i:U_i\times M\to \mathbb{C}^m$ depending holomorphically on $t$, and a holomorphic map $h:M\to \mathbb{C}^{r'}$,
 such that
 \begin{align}
 h(0)&=0, \label{ii30}\\
 \phi(0)&=0, \label{ii31}\\
 \Lambda(0)&=\Lambda_0, \label{ii32}\\
 \bar{\partial} \phi+\frac{1}{2}[\phi,\phi]&=0,  \label{ii33}\\
 \bar{\partial}\Lambda+[\phi,\Lambda]&=0, \label{ii34}\\
 [\Lambda,\Lambda]&=0, \label{ii35}\\
 \Phi_i(z_i,0)&=f_i(z_i),   \label{ii36}\\
 \bar{\partial}\Phi_i+[\phi, \Phi_i]&=0,  \label{ii37}\\
 \Phi_i(b_{ij}(z_j),t)&=\psi_{ij}(\Phi_j(z_j,t),h(t)), \label{ii38}\\
 \Pi_{pq}^i(\Phi_i(z_i,t),h(t))&=\sum_{\alpha,\beta=1}^n \Lambda_{\alpha\beta}^{ i}(z_i,t)\frac{\partial \Phi_i^p}{\partial z_i^\alpha}\frac{\partial \Phi_i^q}{\partial z_i^\beta}. \label{ii39}
 \end{align}

\subsubsection{Existence of formal solutions}\

We will prove the existence of formal solutions of $\phi(t),\Lambda(t),\Phi_i(z_i,t)$, and $h(t)$ satisfying $(\ref{ii30})- (\ref{ii39})$ as power series in $t$. Let $\phi(t)=\sum_{\mu=1}^\infty \phi_\mu(t),\Lambda(t)=\sum_{\mu=0}^\infty, \Phi_i(z_i,t)=\sum_{\mu=0}^\infty \Phi_{i|\mu}(z_i,t)$, and $h(t)=\sum_{\mu=0}^\infty h_\mu(t)$, where $\phi_\mu(t),\Lambda_\mu(t),\Phi_{i|\mu}(t)$, and $h_\mu(t)$ are homogenous in $t$ of degree $\mu$, and let $\phi^\mu(t)=\phi_0(t)+\phi_1(t)+\cdots+\phi_\mu(t),\Lambda^\mu(t)=\Lambda_0(t)+\Lambda_1(t)+\cdots+\Lambda_\mu(t), \Phi_i^\mu(z_i,t)=\Phi_{i|0}(z_i,t)+\Phi_{i|1}(z_i,t)+\cdots+\Phi_{i|\mu}(z_i,t)$, and $h^\mu(t)=h_0(t)+h_1(t)+\cdots + h_\mu(t)$. We note that $(\ref{ii30})- (\ref{ii39})$ are equivalent to the following system of congruences:

\begin{align}
 \bar{\partial} \phi^\mu+\frac{1}{2}[\phi^\mu,\phi^\mu]&\equiv_\mu 0, \label{ii40}\\
 \bar{\partial}\Lambda^\mu+[\phi^\mu,\Lambda^\mu]&\equiv_\mu 0, \label{ii41}\\
 [\Lambda^\mu,\Lambda^\mu]&\equiv_\mu 0, \label{ii42}\\
 \bar{\partial}\Phi_i^\mu+[\phi^\mu, \Phi_i^\mu]&\equiv_\mu 0, \label{ii43}\\
 \Phi_i^\mu(b_{ij}(z_j),t)&\equiv_\mu \psi_{ij}(\Phi_j^\mu(z_j,t),h^\mu(t)), \label{ii44}\\
 \Pi_{pq}^i(\Phi_i^\mu(z_i,t),h^\mu(t))&\equiv_\mu \sum_{\alpha,\beta=1}^n \Lambda_{\alpha\beta}^{\mu i}(z_i,t)\frac{\partial \Phi_i^{\mu p}}{\partial z_i^\alpha}\frac{\partial \Phi_i^{\mu q}}{\partial z_i^\beta}, \label{ii45}
 \end{align}
for $\mu=1,2,3,...$. We also recall Remark \ref{ii71}.

We shall construct solutions $\phi(t),\Lambda(t)$, and $\Phi_i(z_i,t)$ of $(\ref{ii30})-(\ref{ii39})$ by induction on $\mu$. From $(\ref{ii30}),(\ref{ii31})$, and $(\ref{ii32})$, we set $\phi_0=0, \Phi_{i|0}=f_i(z_i), \Lambda_{i|0}=\Lambda_0$, and $h_0=0$. For $\mu=1$, we determine $\phi_1,\Lambda_1, \Phi_{i|1}$ and $h_1$ as follows: take $(\Phi_{1\lambda}',-\phi_{1\lambda},-\Lambda_{1\lambda},(-\theta_\lambda^v))\in A^{0,0}(f^*\Theta_Y)\oplus A^{0,1}(\Theta_X)\oplus A^{0,0}(\wedge^2 \Theta_X)\oplus \mathbb{C}^{r'}$ for $\lambda=1,...,r$ which represent a basis of $PD_{(X,\Lambda_0)/(\mathcal{Y},\Pi)}$ via the isomorphism in Lemma \ref{kk23}. Then we have $\bar{\partial}\Phi_{1\lambda}'=-F\phi_{1\lambda}-\sum_{v=1}^{r'}\theta_\lambda^v \tilde{\varphi}_v$, and $\pi(\Phi_{1\lambda}')=-F\Lambda_{1\lambda}-\sum_{v=1}^{r'}\theta_\lambda^v\tilde{\chi}_v$. We set $\phi_1=\sum_{\lambda=1}^r \phi_{1\lambda}t_\lambda, \Lambda_{i|1}=\sum_{\lambda=1}^r \Lambda_{1\lambda}t_\lambda, \Phi_{i|1}=\sum_{\lambda=1}^r(\Phi_{1\lambda}'+\sum_{v=1}^{r'} \theta^v_\lambda\xi_{vi})t_\lambda, h_1^v=\sum_{\lambda=1}^r\theta_\lambda^v t_\lambda$. Then $(\ref{ii40})_1-(\ref{ii45})_1$ hold. We only check $(\ref{ii45})_1$. Indeed, $(\ref{ii45})_1$ is equivalent to
\begin{align*}
&\Pi_{pq}^i(f_i+\Phi_{i|1},h_1)\equiv_1 \sum_{\alpha,\beta=1}^n (\Lambda_{\alpha\beta}^i(z_i)+\Lambda_{\alpha\beta|1}^i)\frac{\partial (f_i^p+\Phi_{i|1}^p)}{\partial z_i^\alpha}\frac{\partial(f_i^q+\Phi_{i|1}^q)}{\partial z_i^\beta} \iff \\
&\sum_{\alpha=1}^m \Phi_{i|1}^\alpha\frac{\partial \Pi_{pq}^i}{\partial w_i^\alpha}(f_i,0)+\sum_{v=1}^{r'}h_1^v\frac{\partial \Pi_{pq}^i}{\partial s_i^v}(f_i,0)=\sum_{\alpha,\beta=1}^n\Lambda_{\alpha\beta}^i(z_i)\frac{\partial \Phi_{i|1}^p}{\partial z_i^\alpha}\frac{\partial f_i^q}{\partial z_i^\beta}+\sum_{\alpha,\beta=1}^n\Lambda_{\alpha\beta}^i(z_i)\frac{\partial f_i^p}{\partial z_i^\alpha}\frac{\partial \Phi_{i|1}^q}{\partial z_i^\beta}+\sum_{\alpha,\beta=1}^n \Lambda_{\alpha\beta|1}^i\frac{\partial f_i^p}{\partial z_i^\alpha}\frac{\partial f_i^q}{\partial z_i^\beta}\\
&\iff \pi(\Phi_{i|1})+F\Lambda_{i|1}-\sum_{v=1}^{r'} h_i^v \gamma_{vi}=0
\end{align*}
On the other hand, we have 
\begin{align*}
\sum_{\lambda=1}^r\left(\pi(\Phi_{1\lambda}')+\sum_{v=1}^{r'}\theta_\lambda^v\pi(\xi_{vi})+F \Lambda_{1\lambda} -\sum_{v=1}^{r'}\theta_\lambda^vh_i^v \gamma_{vi} \right)t_\lambda =\sum_{\lambda=1}^r\left(\pi(\Phi_{1\lambda}')+\sum_{v=1}^{r'}\theta_\lambda^v\tilde{\chi}_v+F \Lambda_{1\lambda} \right)t_\lambda=0.
\end{align*}

 Hence induction holds for $\mu=1$.

Now assume that $\phi^{\mu-1},\Lambda^{\mu-1},\Phi_i^{\mu-1}$, and $h^{\mu-1}$ satisfying $(\ref{ii40})_{\mu-1}-(\ref{ii45})_{\mu-1}$ are already determined. We define homogeneous polynomials $\xi_\mu\in A^{0,2}(\Theta_X), \psi_\mu\in A^{0,1}(\wedge^2 \Theta_X)$, and $\eta_\mu \in A^{0,0}(\wedge^2 \Theta_X)$, $E_{i|\mu}\in \Gamma(U_i, \mathcal{A}^{0,1}(f^*\Theta_Y))$, $\Gamma_{ij|\mu}\in \Gamma(U_{ij},\mathcal{A}^{0,0}(f^*\Theta_Y))$ and $\lambda_{i|\mu}\in \Gamma(U_i, \mathcal{A}^{0,0}(\wedge^2 f^*\Theta_Y))$ by the following congruences:
 \begin{align}
 \xi_\mu&\equiv_\mu \phi^{\mu-1}+\frac{1}{2}[\phi^{\mu-1},\phi^{\mu-1}], \label{ii87}\\
 \psi_\mu&\equiv_\mu \bar{\partial}\Lambda^{\mu-1}+[\phi^{\mu-1},\Lambda^{\mu-1}] \label{ii88}\\
 \eta_\mu &\equiv_\mu \frac{1}{2}[\Lambda^\mu,\Lambda^\mu] \label{ii89}\\
 E_{i|\mu}&\equiv_\mu \sum_{\alpha=1}^m (\Phi_i^{\mu-1}+[\phi^{\mu-1},\Phi_i^{\mu-1}])\frac{\partial}{\partial w_i^\alpha} \label{ii90}\\
 \Gamma_{ij|\mu}&\equiv_\mu \sum_{\alpha=1}^m (\Phi_i^{(\mu-1)\alpha}-\phi_{ij}^\alpha(\Phi_j^{\mu-1}, h^{\mu-1})\frac{\partial}{\partial w_i^\alpha} \label{ii91}\\
 \lambda_{i|\mu}&\equiv_\mu \sum_{p,q=1}^m \lambda_{i|\mu}^{p,q}\frac{\partial}{\partial w_i^p}\wedge \frac{\partial}{\partial w_i^q} \label{ii92}\\
 &\sum_{p,q=1}^m  \left(-\Pi_{pq}^i(\Phi_i^\mu(z_i,t),h^\mu(t))+\sum_{\alpha,\beta=1}^n \Lambda_{\alpha\beta}^i(z_i,t)\frac{\partial \Phi_i^{\mu p}}{\partial z_i^\alpha}\frac{\partial \Phi_i^{\mu q}}{\partial z_i^\beta}\right)\frac{\partial}{\partial w_i^p}\wedge \frac{\partial}{\partial w_i^q}. \notag
 \end{align}
We also recall Remark \ref{ii70}.
 
\begin{lemma} \label{ll1}
 We have the following equalities:
\begin{align}
\bar{\partial}\xi_\mu&=0,\,\,\,\,\,\,\,\,\,\,\,\,\,\,\,\,\,\,\text{in $\Gamma(X,\mathcal{A}^{0,3}(\Theta_X))$}, \label{hh2}\\
\bar{\partial}\psi_\mu+[\Lambda_0, \xi_\mu]&=0,\,\,\,\,\,\,\,\,\,\,\,\,\,\,\,\,\,\text{in $\Gamma(X, \mathcal{A}^{0,2}(\wedge^2 \Theta_X))$}, \label{hh3}\\
\bar{\partial}\eta_\mu+[\Lambda_0, \psi_\mu]&=0,\,\,\,\,\,\,\,\,\,\,\,\,\,\,\,\,\,\,\text{in $\Gamma(X, \mathcal{A}^{0,1}(\wedge^3 \Theta_X))$}, \label{hh4}\\
[\Lambda_0,\eta_\mu]&=0, \,\,\,\,\,\,\,\,\,\,\,\,\,\,\,\,\,\,\,\text{in $\Gamma(X, \mathcal{A}^{0,0}(\wedge^4 \Theta_X))$} \label{hh5}\\
\bar{\partial} E_{i|\mu}&=F\xi_\mu,\,\,\,\,\,\,\,\,\,\,\,\,\text{in $\Gamma(U_i,\mathcal{A}^{0,2}(f^*\Theta_Y))$}, \label{hh6}\\
\pi(E_{i|\mu})+\bar{\partial} \lambda_{i|\mu}&=F\psi_\mu,\,\,\,\,\,\,\,\,\,\,\,\,\text{in $\Gamma(U_i,\mathcal{A}^{0,1}(\wedge^2 f^*\Theta_Y))$}, \label{hh6}\\
\pi(\lambda_{i|\mu})&=F\eta_\mu,\,\,\,\,\,\,\,\,\,\,\,\,\,\text{in $\Gamma(U_i, \mathcal{A}^{0,0}(\wedge^3 f^* \Theta_Y))$}, \label{hh7}\\
E_{i|\mu}-E_{j|\mu}&=\bar{\partial} \Gamma_{ij|\mu}, \,\,\,\,\,\,\,\,\text{in $\Gamma(U_{ij}, \mathcal{A}^{0,1}(f^*\Theta_Y))$}, \label{hh8}\\
\lambda_{i|\mu}-\lambda_{j|\mu}&=\pi(\Gamma_{ij|\mu}),\,\,\,\,\text{in $\Gamma(U_{ij}, \mathcal{A}^{0,0}(\wedge^2 f^* \Theta_Y))$}, \label{hh9}\\
\Gamma_{jk|\mu}-\Gamma_{ik|\mu}+\Gamma_{ij|\mu}&=0,\,\,\,\,\,\,\,\,\,\,\,\,\,\,\,\,\,\,\,\,\,\text{in $\Gamma(U_{ijk},\mathcal{A}^{0,0}(f^*\Theta_Y))$}. \label{hh10}
\end{align}
We aslo recall Remark $\ref{ii75}$.
\end{lemma}
\begin{proof}
The proofs are similar to Lemma $\ref{ii76}$.
\end{proof}
 
 We will determine $\phi_\mu, \Lambda_\mu,\Phi_{i|\mu}$, and $h_\mu$ such that $\phi^\mu:=\phi^{\mu-1}+\phi_\mu, \Lambda^\mu:=\Lambda^{\mu-1}+\Lambda_\mu, \Phi_i^\mu:=\Phi_i^{\mu-1}+\Phi_{i|\mu}$, and $h^\mu:=h^{\mu-1}+h_\mu$ satisfy $(\ref{ii40})_\mu-(\ref{ii45})_\mu$.
 
 
 \begin{lemma}\label{ll2}
$(\ref{ii87})_\mu-(\ref{ii92})_\mu$ are equivalent to
 \begin{align}
-\xi_\mu&= \bar{\partial} \phi_\mu\label{ii98}\\
- \psi_\mu&=\bar{\partial}\Lambda_\mu+[\Lambda_0, \phi_\mu]\\
- \eta_\mu&=[\Lambda_0, \Lambda_\mu]\\
 -E_{i|\mu}&=\bar{\partial} \Phi_{i|\mu}+F\phi_\mu \label{ii96}\\
 -\lambda_{i|\mu}&=\pi(\Phi_{i|\mu})+F\Lambda_\mu-\sum_{v=1}^{r'} h_\mu^v \tilde{\Pi}_{vi}', \label{ii85}\\
 \Gamma_{ij|\mu}&=\Phi_{j|\mu}-\Phi_{i|\mu}+\sum_{v=1}^{r'} h_\mu^v \tilde{\rho}'_{vij} \label{ii99}
 \end{align}
 where $\Phi_i=\sum_{\alpha=1}^m \Phi_i^\alpha\frac{\partial}{\partial w_i^\alpha}\in \Gamma(U_i, f^*\Theta_Y)$, $\tilde{\rho}_{vij}'=\sum_{\alpha=1}^m\frac{\psi_{ij}^\alpha}{\partial s_j^v}(f_j(z_j),0)\frac{\partial}{\partial w_i^\alpha}$, and $\tilde{\Pi}_{vi}'=\sum_{\alpha,\beta=1}^m \frac{\partial \Pi_{\alpha\beta}^i}{\partial s_i^v}(f_j(z_j),0)\frac{\partial}{\partial w_i^\alpha}\wedge \frac{\partial}{\partial w_i^\beta}$. 
 \end{lemma}
 \begin{remark}
 $(\ref{ii96})$ and $(\ref{ii85})$ are equivalent to
 \begin{align}
 (-E_{i|\mu},-\lambda_{i|\mu}+\sum_{v=1}^{r'}h_\mu^v \tilde{\Pi}_{vi}')=L_\pi(\Phi_{i|\mu})+F (\phi_\mu, \Lambda_\mu)
 \end{align}
 \end{remark}
 \begin{proof}
 The proofs are similar to Lemma \ref{ii86}. Here we only prove that $(\ref{ii39})_\mu$ is equivalent to $(\ref{ii85})$. Indeed, $(\ref{ii39})$ is equivalent to the following.
 \begin{align*}
&\sum_{\alpha,\beta=1}^n (\Lambda_{\alpha\beta}^{(\mu-1)i}+\Lambda^i_{\alpha\beta|\mu})\frac{\partial (\Phi_i^{(\mu-1)p}+\Phi_{i|\mu}^p)}{\partial z_i^\alpha}\frac{\partial (\Phi_i^{(\mu-1)q}+\Phi_{i|\mu}^q)}{\partial z_i^\beta}\\
&\equiv_\mu \Pi_{pq}^i(\Phi_i^{\mu-1}+\Phi_{i|\mu},h^{\mu-1}+h_\mu)\equiv_\mu \Pi_{pq}^i(\Phi_i^{\mu-1},h^{\mu-1})+\sum_{r=1}^m \frac{\partial \Pi_{pq}^i}{\partial w_i^r}(\Phi_i^{\mu-1}, h^{\mu-1})\Phi_{i|\mu}^r +\sum_{v=1}^{r'} \frac{\partial  \Pi_{pq}^i}{\partial s^v}(\Phi_i^{\mu-1}, h^{\mu-1} ) h_\mu^v\\
&\iff \lambda_{i|\mu}^{p,q}+\sum_{\alpha,\beta=1}^n \Lambda_{\alpha\beta }^i(z_i)\frac{\partial \Phi_{i|\mu}^p}{\partial z_i^\alpha}\frac{\partial f_i^q}{\partial z_i^\beta}+\sum_{\alpha,\beta=1}^n \Lambda_{\alpha\beta }^i(z_i)\frac{\partial f_i^p}{\partial z_i^\alpha}\frac{\partial \Phi_{i|\mu}^q}{\partial z_i^\beta}+\sum_{\alpha,\beta=1}^n \Lambda_{\alpha\beta|\mu}^i\frac{\partial f_i^p}{\partial z_i^\alpha}\frac{\partial f_i^q}{\partial z_i^\beta}\\
&-\sum_{r=1}^m \frac{\partial \Pi_{pq}^i}{\partial w_i^r}(f_i,0)\Phi_{i|\mu}^r-\sum_{v=1}^{r'} \frac{\partial  \Pi_{pq}^i}{\partial s^v}(f_i,0 ) h_\mu^v\equiv_\mu 0\\
&\iff \lambda_{i|\mu}(w_i^p,w_i^q)+\pi(\Phi_{i|\mu})(w_i^p,w_i^q)+F \Lambda_\mu(w_i^p,w_i^q) -\sum_{v=1}^{r'}h_\mu^v \tilde{\Pi}'_{vi}(w_i^p,w_i^q)\equiv_\mu 0
\end{align*}
which follows from the following: for any $p,q$,
\begin{align*}
\pi(\Phi_{i|\mu})(w_i^p,w_i^q)&=\Lambda_0(\Phi_{i|\mu}(w_i^p),f_i^q)-\Lambda_0(\Phi_{i|\mu}(w_i^q), f_i^q)-\Phi_{i|\mu}(\Pi_0(w_i^p,w_i^q))\\
&=\Lambda_0(\Phi_{i|\mu}^p,f_i^q)-\Lambda_0(\Phi_{i|\mu}^q, f_i^p)-2\sum_{r=1}^n\Phi_{i|\mu}^r\frac{\partial \Pi_{pq}^i}{\partial w_i^r}(f_i,0)
\end{align*}

 \end{proof}

 \begin{lemma}\label{ll3}
Under the hypothesis of Theorem $\ref{ii80}$, we can find 
\begin{align*}
\phi_\mu\in A^{0,1}(\Theta_X),\,\,\,\,\, \Lambda_\mu\in A^{0,0}(\wedge^2 \Theta_X),\,\,\,\,\,\Phi_{i|\mu} \in \Gamma(U_i, \mathcal{A}^{0,0}(f^*\Theta_Y)),\,\,\,\,\,h_\mu=(h_\mu^v)\in \mathbb{C}^{r'}
\end{align*}
satisfying $(\ref{ii98})-(\ref{ii99})$.
 
 \end{lemma}
 
 \begin{proof}
From $(\ref{hh10})$, we can find $\Gamma_{i|\mu}\in \Gamma(U_i, \mathcal{A}^{0,0}(f^*\Theta_Y))$ such that $\Gamma_{ij|\mu}=\Gamma_{i|\mu}-\Gamma_{j|\mu}$. From $(\ref{hh8})$ $E_{i|\mu}':=E_{i|\mu}-\bar{\partial}\Gamma_{i|\mu}$ defines a global section $E_\mu'\in A^{0,1}(f^*\Theta_Y)$. On the other hand, from $(\ref{hh9})$, $\lambda_{i|\mu}':=\lambda_{i|\mu}-\pi(\Gamma_{i|\mu})$ defines a global section $\lambda_\mu'\in A^{0,0}(\wedge^2 f^* \Theta_Y)$. Then $E_\mu'$ and $\lambda_\mu'$ satisfy $\bar{\partial} E_\mu'=\bar{\partial} E_{i|\mu}=F\xi_\mu, \pi(E_\mu')+\bar{\partial}\lambda_\mu'=\pi(E_{i|\mu}')+\bar{\partial} \lambda_{i|\mu}=F\psi_\mu$, and $\pi(\lambda_\mu')=\pi(\lambda_\mu)=F\eta_\mu$. Since $F:\mathbb{H}^2(X,\Theta_X^\bullet)\to \mathbb{H}^2(X, f^*\Theta_Y^\bullet)$ is injective, there exist $\phi_\mu'$ and $\Lambda_\mu'$ such that $\bar{\partial} \phi_\mu'=-\xi_\mu,\bar{\partial} \Lambda_\mu'+[\Lambda_0, \phi_\mu']=-\psi_\mu$, and $[\Lambda_0, \Lambda_\mu']=-\eta_\mu$. Then we have $\bar{\partial}(E_\mu'+F\phi_\mu')=0, \pi(\lambda_\mu'+F\Lambda_\mu')=0$, and $\bar{\partial}(\lambda_\mu'+F\Lambda_\mu')+\pi(E_\mu'+F\phi_\mu')=0$. Since $\mathbb{H}^1(X, f^*\Theta_Y^\bullet)$ is generated by the image of $F:\mathbb{H}^1(X,\Theta_X^\bullet)\to \mathbb{H}^1(X, f^*\Theta_Y^\bullet)$ and the image of $f^*\circ \rho':T_{0^*}(S)\to \mathbb{H}^1(X, f^*\Theta_Y^\bullet)$, there exist $x_\mu\in A^{0,1}(\Theta_X), y_\mu\in A^{0,0}(\wedge^2 \Theta_X)$ with $\bar{\partial}x_\mu=0, \bar{\partial}y_\mu+[\Lambda_0, x_\mu]=0$, and $[\Lambda_0, y_\mu]=0$, and there exist $\Phi_\mu'\in A^{0,0}(f^*\Theta_Y)$, and $h_\mu=(h_\mu^v)\in \mathbb{C}^{r'}$ such that $\bar{\partial} \Phi_\mu'=Fx_\mu-\sum_{v=1}^{r'} h_\mu^v  \tilde{a}_v-(E_\mu'+F\phi_\mu')$, $\pi(\Phi_\mu')=Fy_\mu-\sum_{v=1}^{r'} h_\mu^v\tilde{b}_v-(\lambda_\mu'+F\phi_\mu')$, equivalently, $L_\pi(\Phi_\mu')=F(x_\mu,y_\mu)-\sum_{v=1}^{r'}h_\mu^v(\tilde{a}_v,\tilde{b}_v)-((E_\mu',\lambda_\mu')+F(\phi_\mu',\Lambda_\mu'))$ where each $(\tilde{a}_v,\tilde{b}_v)\in A^{0,1}(f^*\Theta_Y)\oplus A^{0,0}(\wedge^2 f^*\Theta_Y)$ represents the cohomology class $f^*\rho'\left(\frac{\partial}{\partial s^v} \right)$ which is determined in the following way: we take $\xi_{vi}\in \Gamma(U_i, \mathcal{A}^{0,0}(f^* \Theta_Y)$ such that $\xi_{vi}-\xi_{vj}=\tilde{\rho}_{vij}'$. Then $f^*\rho'\left(\frac{\partial}{\partial s^v}\right)=(\{\tilde{\rho}_{vij}\},\{\tilde{\Pi}_{vi})\in C^1(\mathcal{U},f^*\Theta_Y)\oplus C^0(\mathcal{U},\wedge^2 \Theta_Y)$ corresponds to $(\bar{\partial}\xi_{vi}, \pi(\xi_{vi})-\tilde{\Pi}_{vi})=(\tilde{a}_v,\tilde{b}_v)\in A^{0,1}(f^*\Theta_Y)\oplus A^{0,0}(\wedge^2 f^*\Theta_Y)$.

Let $\phi_\mu:=\phi_\mu'-x_\mu, \Lambda_\mu:=\Lambda_\mu'-y_\mu$, and $\Phi_{i|\mu}:=\Phi_\mu'-\Gamma_{i|\mu}+\sum_{v=1}^{r'} h_\mu^v \xi_{vi}$. Then $\bar{\partial}\phi_\mu=\bar{\partial}\phi_{\mu}'=-\xi_\mu, \bar{\partial}\Lambda_\mu+[\Lambda_0,\phi_\mu]=-\psi_\mu, [\Lambda_0, \Lambda_\mu]=-\eta_\mu$, and we have $L_\pi(\Phi_{\mu}')+F(\phi_\mu,\Lambda_\mu)=-\sum_{v=1}^{r'}h_\mu^v(\tilde{a}_v,\tilde{b}_v)-(E_{i|\mu},\lambda_\mu)+(\bar{\partial}\Gamma_{i|\mu},\pi(\Gamma_{i|\mu}))=-\sum_{v=1}^{r'}h_\mu^v(\xi_{vi},\pi(\xi_{iv}))+\sum_{v=1}^{r'}(0, h_\mu^v\tilde{\Pi}_{iv})-(E_{i|\mu},\lambda_\mu)+(\bar{\partial}\Gamma_{i|\mu},\pi(\Gamma_{i|\mu}))$ so that we get $L_\pi(\Phi_{i|\mu})+F(\phi_\mu, \Lambda_\mu)= (-E_{i|\mu},-\lambda_\mu+\sum_{v=1}^{r'} h_\mu^v \tilde{\Pi}_{iv})$. Lastly, $\Phi_{j|\mu}-\Phi_{i|\mu}=-\Gamma_{j|\mu}+\Gamma_{i|\mu}+\sum_{v=1}^{r'}h_\mu^v(\xi_{vj}-\xi_{vi})=\Gamma_{ij|\mu}-\sum_{v=1}^{r'}h_\mu^v \tilde{\rho}_{vij}'$.
 \end{proof}
 This completes the inductive constructions of formal power series 
 \begin{align}\label{qq50}
\phi(t)=\sum_{\mu=1}^\infty \phi_\mu(t),\,\,\,\,\,\Lambda(t)=\Lambda_0+\sum_{\mu=1}^\infty \Lambda_\mu(t),\,\,\,\,\,\Phi_i(z_i,t)=f_i(z_i)+\sum_{\mu=1}^\infty \Phi_{i|\mu}(z_i,t),\,\,\,\,\,\,h(t)=\sum_{\mu=1}^\infty h_\mu(t)
\end{align}
satisfying $(\ref{ii30})- (\ref{ii39})$.

\subsubsection{Proof of convergence}\

As in Theorem \ref{hh1}, the author could not prove the convergence of the formal power series.

\subsubsection{Construction of a family} \label{ll5}\
 
Now we assume that  our construction $\phi(t),\Lambda(t),\Phi_i(t)$ and $h(t)$ converge for $|t|<\epsilon$ for a sufficiently small number $\epsilon>0$. By the same arguments in subsection \ref{kk20}, we have a Poisson analytic family $p:(\mathcal{X},\Lambda)\to M$ with $p^{-1}(0)=(X,\Lambda_0)$, and a holomorphic Poisson map $\Phi:(\mathcal{X},\Lambda)\to (\mathcal{Y},\Pi)$, and a holomorphic map $h:M\to S$ such that $q\circ \Phi=h\circ p$ such that $s(0)=0^*$ and $\Phi$ induces $f$ on $(X,\Lambda_0)$. It remains to show $\tau$ is bijective. We keep the notations in subsection \ref{kk20}. 
Then we have $(\ref{kk30}),(\ref{kk31}),(\ref{kk32})$, and $(\ref{kk33})$.

We note that $\tau\left(\frac{\partial}{\partial t}\right)$ is represented by $(-\sum_{\rho=1}^m \dot{\Psi}_\alpha^\rho \frac{\partial}{\partial w_\alpha^\rho}-\sum_{v=1}^{r'}\dot{h}^v\frac{\partial}{\partial s_\alpha^v}, \sum_{\sigma=1}^n \dot{\phi}_{\alpha\beta}^\sigma\frac{\partial}{\partial z_\alpha^\sigma},\sum_{p,q=1}^n \dot{g}_{pq}^\alpha\frac{\partial}{\partial z_\alpha^p}\wedge \frac{\partial}{\partial z_\alpha^q})\in C^0(\mathcal{U},\tilde{f}^*\Theta_{\mathcal{Y}})\oplus C^1(\mathcal{U},\Theta_X)\oplus C^0(\mathcal{U},\wedge^2 \Theta_X)$. 
By the isomorphisms in the proofs of Lemma \ref{kk22}, and Lemma \ref{kk23},  $\tau$ is bijective. This completes the proof of Theorem \ref{ii80}.

\end{proof}

\section{Stability of holomorphic Poisson manifolds over $(Y,\Pi_0)$}\label{section4}

\begin{theorem}[compare Theorem \ref{mm1}]\label{stability}
Let $f:(X,\Lambda_0)\to (Y,\Pi_0)$ be a holomorphic Poisson map of compact holomorphic Poisson manifolds, where $X$ is compact. Assume that
\begin{enumerate}
\item $F:\mathbb{H}^1(X, \Theta_X^\bullet)\to \mathbb{H}^1(X, f^*\Theta_Y^\bullet)$ is surjective.
\item $F:\mathbb{H}^2(X, \Theta_X^\bullet)\to \mathbb{H}^2(X, f^*\Theta_Y^\bullet)$ is injective.
\end{enumerate}
Then for any family $q:(\mathcal{Y}, \Pi)\to M$ of holomorphic Poisson manifolds such that $(Y,\Pi_0)=q^{-1}(0)$ for some point $0\in M$, if the formal power series $(\ref{ll4})$ constructed in the proof converge, then there exist
\begin{enumerate}
\item an open neighborhood $N$ of $0$,
\item a family $p:(\mathcal{X},\Lambda)\to N$ of deformations of $(X,\Lambda_0)=p^{-1}(0)$,
\item a holomorphic Poisson map $p:(\mathcal{X},\Lambda)\to (\mathcal{Y}|_N, \Pi|_N)$ over $N$ which induces $f$ over $0\in N$.
\end{enumerate}
\end{theorem}

\begin{remark}
If $f$ is non-degenerate, two conditions $(1)$ and $(2)$ are equivalent to $\mathbb{H}^1(X, \mathcal{N}_f^\bullet)$.
\end{remark}

\begin{remark}
Let $X$ be a compact holomorphic Poisson submanifold of $(Y,\Lambda_0)$ so that we have a Poisson embedding $i:X\hookrightarrow (Y,\Lambda_0)$. By Lemma $\ref{aa10}$, if $\mathbb{H}^1(X,\mathcal{N}_{X/Y}^\bullet)=0$, then we have $\mathbb{H}^1(X,\mathcal{N}_i^\bullet)=0$ so that $(X,\Lambda_0)$ is a stable holomorphic Poisson submanifold of $(Y,\Lambda_0)$ provided  that the formal power series in Theorem $\ref{stability}$ converge. This was already proved in \cite{Kim17}.
\end{remark}

\begin{proof}[Proof of Theorem $\ref{stability}$ ]
We extend the arguments in \cite{Hor74} p.660-661 in the context of holomorphic Poisson deformations. We tried to maintain notational consistency with \cite{Hor74}. Let $\dim PD_{(X,\Lambda_0)/(Y,\Pi_0)}=r$. We may assume that $M=\{t=(t_\lambda)\in \mathbb{C}^r||t|<\epsilon\}$. We will construct formal power series satisfying $(\ref{ii30})-(\ref{ii39})$ with $h(t)=t$ by induction on $\mu$  in the exact same way to the proof of Theorem $\ref{ii80}$.  We keep the notations in the proof of Theorem \ref{ii80} with $s=t$ and $h(t)=t$. For each $\lambda=1,...,r$, let $(\tilde{\varphi}_\lambda,\tilde{\chi}_\lambda)\in A^{0,1}(f^*\Theta_Y)\oplus A^{0,0}(\wedge^2 f^*\Theta_Y)$ with $L_\pi(\tilde{\varphi}_\lambda,\tilde{\chi}_\lambda)=0$ be a representative of the cohomology class of $(\{\rho_{\lambda ij}' \}, \{ \gamma_{\lambda i}\})\in C^1(\mathcal{U}, f^*\Theta_Y)\oplus C^0 (\mathcal{U}, \wedge^2 f^*\Theta_Y)$ where $\rho_{\lambda ij}'=\sum_{\alpha=1}^m \frac{\partial \psi_{ij}^\alpha}{\partial t_j^\lambda}(f_j(z_j),0)\frac{\partial}{\partial w_i^\alpha}$, and $\gamma_{\lambda i}=\sum_{\alpha,\beta=1}^m \frac{\partial \Pi_{\alpha\beta}^i}{\partial t_i^\lambda}(f_j(z_j),0)\frac{\partial}{\partial w_i^\alpha}\wedge \frac{\partial}{\partial w_i^\beta}.$ Then we can find $\xi_{\lambda ij}\in \Gamma(U_i, \mathcal{A}^{0,0}(f^*\Theta_Y))$ such that $\rho_{\lambda ij}'=\xi_{\lambda i}-\xi_{\lambda j}$ such that $\tilde{\varphi}_\lambda=\bar{\partial} \xi_{\lambda i}$, and $\tilde{\chi}_v=\pi(\xi_{\lambda i})-\gamma_{\lambda i}$.

Since $F: \mathbb{H}^1(X,\Theta_X^\bullet)\to \mathbb{H}^1(X,f^*\Theta_Y^\bullet)$ is surjective, there exist $\Phi_{i\lambda}'\in A^{0,0}(f^*\Theta_Y),\phi_{1\lambda}\in A^{0,1}(\Theta_X)$, and $\Lambda_{1\lambda}\in A^{0,0}(\wedge^2 \Theta_X)$ such that $\bar{\partial} \Phi_{1\lambda}'=-F\phi_{1\lambda}-\tilde{\varphi}_\lambda$, and $\pi(\Phi_{i\lambda}')=-F\Lambda_{i \lambda}-\tilde{\chi}_\lambda$ with $L(\phi_{1\lambda}+\Lambda_{1\lambda})=0$. We set $\phi_1=\sum_{\lambda=1}^r \phi_{1\lambda}t_\lambda, \Lambda_1=\sum_{\lambda=1}^r \Lambda_{1\lambda}, \Phi_{i|1}=\sum_{\lambda=1}^r (\Phi_1'+\xi_{\lambda i})t_\lambda$. Then $\phi_1,\Lambda_1,\Phi_{i|1}$ and $h_1^\lambda=t_\lambda$ satisfy $(\ref{ii40})_1-(\ref{ii45})_1$ so that induction holds for $\mu=1$. Let us assume that induction holds for $\mu-1$. Then we have $(\ref{ii87})-(\ref{ii92})$, Lemma \ref{ll1} and Lemma \ref{ll2} with $h^\mu(t)=0, h_\mu^v=0$, and $r'=r$. Since $F:\mathbb{H}^1(X,\Theta_X^\bullet)\to \mathbb{H}^1(X, f^*\Theta_Y^\bullet)$ is surjective and $F:\mathbb{H}^2(X,\Theta_X^\bullet)\to \mathbb{H}^2(X, f^*\Theta_Y^\bullet)$ is injective, we have Lemma $(\ref{ll3})$ with $h_\mu=0$ so that we have formal power series 
\begin{align}\label{ll4}
\phi(t)=\sum_{\mu=1}^\infty \phi_\mu(t),\,\,\,\,\,\Lambda(t)=\Lambda_0+\sum_{\mu=1}^\infty \Lambda_\mu(t),\,\,\,\,\,\Phi_i(z_i,t)=f_i(z_i)+\sum_{\mu=1}^\infty \Phi_{i|\mu}(z_i,t).
\end{align}
satisfying $(\ref{ii30})-(\ref{ii39})$. Now assume that $(\ref{ll4})$ converge for $|t|<\epsilon'$ for a sufficiently small number $\epsilon'>0$. Then as in subsection \ref{ll5}, we can construct a Poisson analytic family $p:(\mathcal{X},\Lambda)\to N$ of deformations of $(X,\Lambda_0)=p^{-1}$ and a holomorphic Poisson map $p:(\mathcal{X},\Lambda)\to(\mathcal{Y}|_N,\Pi|_N)$ which induces $f$ over $0\in N$. 
\end{proof}

\section{Deformations of compositions of holomorphic Poisson maps}\label{section5}

\begin{theorem}\label{ll17}
Let $f:(X,\Lambda_0)\to (Y,\Pi_0), g:(Y,\Pi_0)\to (Z, \Omega_0)$ and $h=g\circ f$. We assume that
\begin{enumerate}
\item $X$ and $Y$ are compact. 
\item $g$ is non-degenerate and the canonical homomorphism $f^* G:f^*\Theta_Y^\bullet \to h^* \Theta_Z^\bullet$ is injective, where $G$ denotes the injective homomorphism $\Theta_Y^\bullet \to g^*\Theta_Z^\bullet$ $($see Appendix \ref{appendix1}$)$.
\item there exist a family $(\mathcal{Y},\Pi, \Psi,q, N)$ of holomorphic Poisson maps into $(Z, \Omega_0)$ and a point $0'\in N$ such that $(Y,\Pi_0)=q^{-1}(0')$ and such that $\Psi$ induces $g$ on $(Y,\Pi_0)$.
\item the composition $f^*\circ \tau:T_{0'}(N)\to\mathbb{H}^0(X, f^*\mathcal{N}_g^\bullet)$ is surjective, where $\tau:T_{0'}(N)\to \mathbb{H}^0(Y, \mathcal{N}_g^\bullet)$ is the characteristic map of $(\mathcal{Y}, \Pi, \Psi, q, N)$ at $0'$, and $f^*:\mathbb{H}^0(Y,\mathcal{N}_g^\bullet)\to \mathbb{H}^0(X, f^*\mathcal{N}_g^\bullet)$ is the pullback homomorphism.
\end{enumerate}
Let $(\mathcal{X},\Lambda, \Upsilon, p, M)$ be a family of holomorphic Poisson maps into $(Z,\Omega_0)$ and let $0$ be a point in $M$ such that $(X,\Lambda_0)=p^{-1}(0)$ and that $\Upsilon$ induces $h$ on $X$. Then there exist 
\begin{enumerate}
\item an open neighborhood $M'$ of $0$,
\item a holomorphic map $s:M'\to N$,
\item a holomorphic Poisson map $\Phi:(\mathcal{X}|_{M'},\Lambda|_{M'})\to (\mathcal{Y},\Pi)$,
\end{enumerate}
such that $s(0)=0'$, $\Phi$ induces $f$ on $(X,\Lambda_0)$, and the diagram
\[\xymatrix{
(z_i,t)\in (\mathcal{X}|_{M'} ,\Lambda|_{M'}) \ar[dd]^{p|_{M'}} \ar[drr]^{\Phi} \ar[rrrr]^{\Upsilon|_{M'}} & & && (y_i,s(t))\in (Z\times N, \Omega_0) \ar[dd]^{pr_2}\\
& & (w_i,s)\in (\mathcal{Y},\Pi) \ar[dd]^{q}   \ar[rru]^{\Psi} \\
t\in M'  \ar[drr]^s \ar[rrrr]^s  & & & & s\in N\\      
& & s\in N \ar[rru]^{id} &       \\
}\]
commutes.
\end{theorem}

\begin{proof}
We extend the arguments in \cite{Hor74} p.662-664 in the context of holomorphic Poisson deformations. We tried to keep notational consistency with \cite{Hor74}. We may assume the following:
\begin{enumerate}
\item $M=\{t\in \mathbb{C}^r||t|<\epsilon \},0=(0,...,0),N=\{s\in \mathbb{C}^{r'}||s|<1\}$, and $0'=(0,...,0)$, where $\epsilon>0$ is a sufficiently small number.
\item $\mathcal{X}$ (resp. $\mathcal{Y}$) is covered by a finite number of coordinate neighborhoods $\mathcal{U}_i$ (resp. $\mathcal{V}_i)$ and each $\mathcal{U}_i$ (resp. $\mathcal{V}_i$) is covered by a system of coordinates $(z_i,t)$ (resp. $(w_i,s)$) such that $p(z_i,t)=t$ (resp. $q(w_i,s)=s)$. Moreover, each $\mathcal{U}_i$ is a polydisc $\{(z_i,t)||z_i|<1,|t|<\epsilon\}$. We set $U_i=\mathcal{U}_i\cap X$ and $V_i=\mathcal{V}_i\cap Y$.
\item $I\subset J$, and $f(U_i)$ is contained in $V_i$ for each $i\in I$. $f$ is given by $w_i=f_i(z_i)$ on $U_i$.
\item for each $i\in J$, there exists a coordinate neighborhood $W_i$ on $Z$ such that $\Upsilon (\mathcal{U}_i)\in W_i\times M$ for $i\in I$ and $\Psi(\mathcal{V}_i)\subset W_i\times N$ for $i\in J$.
\item Each $W_i$ is covered by a system of coordinates $y_i$.
\item $\Upsilon $ and $\Psi$, respectively, given by $y_i=\Upsilon_i(z_i,t)$, and $y_i=\Psi_i(w_i,s)$.
\item $(z_i,t)\in \mathcal{U}_i$ and $(w_i,s)\in \mathcal{V}_i$ coincides with $(z_j,t)\in \mathcal{U}_j$ and $(w_j,s)\in \mathcal{V}_j$, respectively, if and only if $z_i=\phi_{ij}(z_k,t)$ and $w_i=\psi_{ij}(w_j,s)$.
\item On $\mathcal{U}_i$, $\Lambda$ is given by $\sum_{\alpha,\beta=1}^n \Lambda_{\alpha\beta}^i(z_i,t)\frac{\partial}{\partial z_i^\alpha}\wedge \frac{\partial}{\partial z_i^\beta}$ with $\Lambda_{\alpha\beta}^i(z_i,t)=-\Lambda_{\beta\alpha}^i(z_i,t)$.
\item On $\mathcal{V}_i$, $\Pi$ is given by $\sum_{\alpha,\beta=1}^m \Pi_{\alpha\beta}^i(w_i,s)\frac{\partial}{\partial w_i^\alpha}\wedge \frac{\partial}{\partial w_i^\beta}$ with $\Pi_{\alpha\beta}^i(w_i,s)=-\Pi_{\alpha\beta}^i(w_i,s)$.
\item $y_i\in W_i$ coincides with $y_j\in W_j$ if and only if $y_i=e_{ij}(y_j)$.
\item On $W_i$, $\Omega$ is given by $\sum_{\alpha,\beta=1}^l \Omega_{\alpha\beta}^i (y_i)\frac{\partial}{\partial y_i^\alpha}\wedge \frac{\partial}{\partial y_i^\beta}$ with $\Omega_{\alpha\beta}^i(y_i)=-\Omega_{\beta\alpha}^i(y_i)$.
\end{enumerate} 
We set $g_i(w_i)=\Psi_i(w_i,0)$, $h(z_i)=\Upsilon_i(z_i,0)$, $b_{ij}(z_j)=\phi_{ij}(z_j,0)$, and $c_{ij}(w_j)=\psi_{ij}(w_j,0)$.

We shall construct holomorphic functions
\begin{align*}
s=(s^v):M\to \mathbb{C}^{r'},\,\,\,\,\,\Phi_i:\mathcal{U}_i\to \mathbb{C}^m,\,\,\,\,\,(i\in I)
\end{align*}
such that
\begin{align}
s(0)=0,\,\,\,\,\,&\Phi_i(z_i,0)=f_i(z_i),\,\,\,\,\,\text{on $U_i\,\,\,(i\in I)$}, \label{mm20}\\
\Phi_i(\phi_{ij}(z_j,t),t)&=\psi_{ij}(\Phi_j(z_j,t),s(t)),\,\,\,\,\,\text{on $\mathcal{U}_{ij}\,\,\,(i,j\in I)$}, \label{mm21}\\
\Upsilon_i (z_i,t)&=\Psi_i(\Phi_i(z_i,t),s(t)),\,\,\,\,\,\text{on $\mathcal{U}_i\,\,\,(i\in I)$}, \label{mm22}\\
\Pi_{pq}^i(\Phi_i(z_i,t),s(t))&=\sum_{\alpha,\beta=1}^n \Lambda_{\alpha\beta}^i(z_i, t )\frac{\partial \Phi_i^{p}(z_i,t)}{\partial z_i^\alpha}\frac{\partial \Phi_i^{q}(z_i,t)}{\partial z_i^\beta},\,\,\,\,\,\text{on $\mathcal{U}_i\,\,\,(i\in I)$}. \label{mm23}
\end{align}

\subsubsection{ Existence of formal solutions}\

We will prove the existence of formal  solutions of $s(t)$ and $\Phi_i(z_i,t)$ satisfying $(\ref{mm20})-(\ref{mm23})$ as power series in $t$. Let $s(t)=\sum_{\mu=0}^\infty s_\mu(t)$, and $\Phi_i(z_i,t)=\sum_{\mu=0}^\infty \Phi_{i|\mu}(z_i,t)$, where $s_\mu(t)$ and $\Phi_{i|\mu}(z_i,t)$ are homogenous in $t$ of degree $\mu$, and let $s^\mu(t)=s_0(t)+s_1(t)+\cdots+s_\mu(t)$, and $\Phi_i^{\mu}(z_i,t)=\Phi_{i|0}(z_i,t)+\Phi_{i|1}(z_i,t)+\cdots+\Phi_{i|\mu}(z_i,t)$. We note that $(\ref{mm20})-(\ref{mm23})$ are equivalent to the following congruences:
\begin{align}
\Phi_i^\mu(\phi_{ij}(z_j,t),t)&\equiv_\mu\psi_{ij}(\Phi_j^\mu(z_j,t),s^\mu(t)),\,\,\,\,\,\text{on $\mathcal{U}_{ij}(i,j\in I)$},\label{cc10}\\
\Upsilon (z_i,t)&\equiv_\mu \Psi_i(\Phi_i^\mu(z_i,t),s^\mu(t)),\,\,\,\,\,\text{on $\mathcal{U}_i(i\in I)$},\label{cc11}\\
\Pi_{pq}^i(\Phi_i^{\mu-1}(z_i,t),s)&\equiv_\mu\sum_{\alpha,\beta=1}^n \Lambda_{\alpha\beta}^i(z_i,t)\frac{\partial \Phi_i^{\mu p}(z_i,t)}{\partial z_i^\alpha}\frac{\partial \Phi_i^{\mu q}(z_i,t)}{\partial z_i^\beta},\,\,\,\,\,\text{on $\mathcal{U}_i(i\in I)$}.\label{cc12}
\end{align}
for $\mu=0,1,2,\cdots$.

We shall construct formal solutions of $s(t)$ and $\Phi_i(z_i,t)$ satisfying $(\ref{mm20})-(\ref{mm23})$ by induction on $\mu$. We set $s_0=0$, and $\Phi_{i|0}=f_i(z_i)$. Then the induction holds for $\mu=0$.

Suppose that we have already determined $s^{\mu-1}$ and $\Phi_i^\mu$ satisfying $(\ref{cc10})_{\mu-1}-(\ref{cc12})_{\mu-1}$. We define homogenous polynomials $\Gamma_{ij|\mu}\in \Gamma(U_{ij},f^*\Theta_Y),\gamma_{i|\mu}\in \Gamma(U_i, h^*\Theta_Z)$ and $\lambda_{i|\mu}\in \Gamma(U_i, \wedge^2 f^*\Theta_Y)$ of degree $\mu$ by the following congruences:
\begin{align}
\Gamma_{ij|\mu}&\equiv_\mu \sum_{\alpha=1}^m (\Phi_i^{(\mu-1)\alpha}(\phi_{ij},t)-\psi_{ij}^\alpha(\Phi_j^{\mu-1},s^{\mu-1})\frac{\partial}{\partial w_i^\alpha}\\
-\gamma_{i|\mu}&\equiv_\mu \sum_{\beta=1}^l(\Upsilon_i^\beta-\Psi_i^\beta(\Phi_i^{\mu-1},s^{\mu-1}))\frac{\partial}{\partial y_i^\beta}\\
\lambda_{i|\mu}&\equiv_\mu \sum_{p,q=1}^m \lambda_{i|\mu}^{p,q}\frac{\partial}{\partial w_i^p}\wedge \frac{\partial}{\partial w_i^q}\\
&\equiv_\mu \sum_{p,q=1}^m \left( -\Pi_{pq}^i(\Phi_i^{\mu-1}(z_i,t),s^{\mu-1})+ \sum_{\alpha,\beta=1}^n \Lambda_{\alpha\beta}^i(z_i,t)\frac{\partial \Phi_i^{(\mu-1) p}(z_i,t)}{\partial z_i^\alpha}\frac{\partial \Phi_i^{(\mu-1) q}(z_i,t)}{\partial z_i^\beta}   \right)\frac{\partial}{\partial w_i^p}\wedge \frac{\partial}{\partial w_i^q}. \notag
\end{align}
\begin{lemma}\label{qq60}
\begin{align}
\Gamma_{ij|\mu}-\Gamma_{ik|\mu}+\Gamma_{ij|\mu}&=0, \label{cc4}\\
\pi_f(\lambda_{i|\mu})&=0\label{cc5}\\
\lambda_{j|\mu}-\lambda_{i|\mu}+\pi_f(\Gamma_{ij|\mu})&=0 \label{cc6}\\
\gamma_{i|\mu}-\gamma_{j|\mu}&=(f^*G) \Gamma_{ij|\mu} \label{cc7}\\
\pi_h(\gamma_{i|\mu})&=(f^* G)\lambda_{i|\mu}\label{cc8}
\end{align}
\end{lemma}
\begin{proof}
For $(\ref{cc4})$ and $(\ref{cc7})$, see \cite{Hor74} p.663. We show $(\ref{cc5})$. It is sufficient to show  that $\pi_f(\lambda_{i|\mu})(w_i^a,w_i^b,w_i^c)=0$ for any $a,b,c$. Indeed, since $[\Lambda_i,\Lambda_i]=0$, and $[\Pi_i,\Pi_i]=0$, we have

\begin{align*}
&\pi_f(\lambda_{i|\mu})(w_i^a,w_i^b,w_i^c)\\
&=\Lambda_0(\lambda_{i|\mu}(w_i^a,w_i^b), f_i^c(z_i))-\Lambda_0(\lambda_{i|\mu}(w_i^a,w_i^c), f_i^b(z))+\Lambda_0(\lambda_{i|\mu}(w_i^b,w_i^c), f_i^a(z))\\
&+\lambda_{i|\mu}(\Pi_0(w_i^a,w_i^b),w_i^c)-\lambda_{i|\mu}(\Pi_0(w_i^a,w_i^c),w_i^b)+\lambda_{i|\mu}(\Lambda_0(w_i^b,w_i^c),w_i^a)\\
&\equiv_\mu -2\Lambda_i(\Pi_{ab}^i(\Phi_i^{\mu-1},s^{\mu-1}),\Phi_i^{(\mu-1)c})+\Lambda_i(\Lambda_i(\Phi_i^{(\mu-1)a},\Phi_i^{(\mu-1)b}),\Phi_i^{(\mu-1)c})\\
&+2\Lambda_i(\Pi_{ac}^i(\Phi_i^{\mu-1},s^{\mu-1}),\Phi_i^{(\mu-1)b})-\Lambda_i(\Lambda_i(\Phi_i^{(\mu-1)a},\Phi_i^{(\mu-1)c}),\Phi_i^{(\mu-1)b})\\
&-2\Lambda_i(\Pi_{bc}^i(\Phi_i^{\mu-1},s^{\mu-1}),\Phi_i^{(\mu-1)a})+\Lambda_i(\Lambda_i(\Phi_i^{(\mu-1)b},\Phi_i^{(\mu-1)c}),\Phi_i^{(\mu-1)a})\\
&+2\lambda_{i|\mu}(\Pi_{ab}^i(w_i,s),w_i^c)-2\lambda_{i|\mu}(\Pi_{ac}^i(w_i,s),w_i^b)+2\lambda_{i|\mu}(\Pi_{bc}^i(w_i,s),w_i^a)\\
&\equiv_\mu -2\Lambda_i(\Pi_{ab}^i(\Phi_i^{\mu-1},s^{\mu-1}),\Phi_i^{(\mu-1)c})
+2\Lambda_i(\Pi_{ac}^i(\Phi_i^{\mu-1},s^{\mu-1}),\Phi_i^{(\mu-1)b})
-2\Lambda_i(\Pi_{bc}^i(\Phi_i^{\mu-1},s^{\mu-1}),\Phi_i^{(\mu-1)a})\\
&+4\sum_{p=1}^n \left( -\Pi_{pc}^i(\Phi_i^{\mu-1},s^{\mu-1})+\sum_{\alpha,\beta=1}^n \Lambda_{\alpha\beta}^i(z_i,t)\frac{\partial \Phi_i^{(\mu-1)p}}{\partial z_i^\alpha}\frac{\partial \Phi_i^{(\mu-1)c}}{\partial z_i^\beta}         \right)\frac{\partial \Pi_{ab}^i}{\partial w_i^p}(\Phi_i^{\mu-1},s^{\mu-1})\\
&-4\sum_{p=1}^n \left( -\Pi_{pb}^i(\Phi_i^{\mu-1}, s^{\mu-1})+\sum_{\alpha,\beta=1}^n \Lambda_{\alpha\beta}^i(z_i,t)\frac{\partial \Phi_i^{(\mu-1)p}}{\partial z_i^\alpha}\frac{\partial \Phi_i^{(\mu-1)b}}{\partial z_i^\beta}         \right)\frac{\partial \Pi_{ac}^i}{\partial w_i^p}(\Phi_i^{\mu-1},s^{\mu-1})\\
&+4\sum_{p=1}^n \left( -\Pi_{pa}^i(\Phi_i^{\mu-1},s^{\mu-1})+\sum_{\alpha,\beta=1}^n \Lambda_{\alpha\beta}^i(z_i,t)\frac{\partial \Phi_i^{(\mu-1)p}}{\partial z_i^\alpha}\frac{\partial \Phi_i^{(\mu-1)a}}{\partial z_i^\beta}         \right)\frac{\partial \Pi_{bc}^i}{\partial w_i^p}(\Phi_i^{\mu-1}, s^{\mu-1})\\
&\equiv_\mu -2\Lambda_i(\Pi_{ab}^i(\Phi_i^{\mu-1},s^{\mu-1}),\Phi_i^{(\mu-1)c})
+2\Lambda_i(\Pi_{ac}^i(\Phi_i^{\mu-1},s^{\mu-1}),\Phi_i^{(\mu-1)b})
-2\Lambda_i(\Pi_{bc}^i(\Phi_i^{\mu-1},s^{\mu-1}),\Phi_i^{(\mu-1)a})\\
&-4\sum_{p=1}^n \left(   \Pi_{pc}^i(\Phi_i^{\mu-1},s^{\mu-1})  \frac{\partial \Pi_{ab}^i}{\partial w_i^p}(\Phi_i^{\mu-1}, s^{\mu-1})  - \Pi_{pb}^i(\Phi_i^{\mu-1},s^{\mu-1})  \frac{\partial \Pi_{ac}^i}{\partial w_i^p}(\Phi_i^{\mu-1},s^{\mu-1})  + \Pi_{pa}^i(\Phi_i^{\mu-1},s^{\mu-1})  \frac{\partial \Pi_{bc}^i}{\partial w_i^p}(\Phi_i^{\mu-1},s^{\mu-1})            \right)\\
&+4\sum_{\alpha,\beta=1}^n\Lambda_{\alpha\beta}^i(z_i,t)\frac{\partial \Pi_{ab}^i(\Phi_i^{\mu-1},s^{\mu-1})}{\partial z_i^\alpha}\frac{\partial \Phi_i^{(\mu-1)c}}{\partial z_i^\beta}-4\sum_{\alpha,\beta=1}^n\Lambda_{\alpha\beta}^i(z_i,t)\frac{\partial \Pi_{ac}^i(\Phi_i^{\mu-1}, s^{\mu-1})}{\partial z_i^\alpha}\frac{\partial \Phi_i^{(\mu-1)b}}{\partial z_i^\beta}\\
&+4\sum_{\alpha,\beta=1}^n\Lambda_{\alpha\beta}^i(z_i,t)\frac{\partial \Pi_{bc}^i(\Phi_i^{\mu-1}, s^{\mu-1})}{\partial z_i^\alpha}\frac{\partial \Phi_i^{(\mu-1)a}}{\partial z_i^\beta}\\
&=0
\end{align*}

Next we show $(\ref{cc6})$. It is sufficient to show that $\lambda_{j|\mu}(w_i^p,w_i^q)-\lambda_{i|\mu}(w_i^p,w_i^q)+\pi_f(\Gamma_{ij|\mu})(w_i^p,w_i^q)=0$ for any $p,q$. We note that $\Pi_{pq}^i(\psi_{ij}(w_j,s),s)=\sum_{r,s=1}^m \Pi_{rs}^j(w_j,s) \frac{\partial \psi_{ij}^p}{\partial w_j^r}\frac{\partial \psi_{ij}^q}{\partial w_j^s} $. Then

\begin{align}\label{mm31}
&\lambda_{j|\mu}(w_i^p,w_i^q)\equiv_\mu \lambda_{j|\mu}(\psi_{ij}^p,\psi_{ij}^q)\\
&\equiv_\mu2\sum_{r,s=1}^m \left( -\Pi_{rs}^j(\Phi_j^{\mu-1}(z_j,t),s^{\mu-1})+ \sum_{\alpha,\beta=1}^n \Lambda_{\alpha\beta}^j(z_j,t)\frac{\partial \Phi_j^{(\mu-1) r}(z_j,t)}{\partial z_j^\alpha}\frac{\partial \Phi_j^{(\mu-1) s}(z_j,t)}{\partial z_j^\beta}   \right)\frac{\partial \psi_{ij}^p}{\partial w_j^r}(\Phi_j^{\mu-1},s^{\mu-1} )\frac{\partial \psi_{ij}^q}{\partial w_j^s}(\Phi_j^{\mu-1},s^{\mu-1})\notag \\
&\equiv_\mu -2\Pi_{pq}^i(\psi_{ij}(\Phi_j^{\mu-1}(z_j,t),s^{\mu-1}),s^{\mu-1})+2\sum_{\alpha,\beta=1}^n \Lambda_{\alpha\beta}^i(\phi_{ij}(z_j,t),t)\frac{\partial \psi_{ij}^p(\Phi_j^{(\mu-1)},s^{\mu-1})}{\partial z_i^\alpha}\frac{\partial \psi_{ij}^q(\Phi_j^{(\mu-1)},s^{\mu-1})}{\partial z_i^\beta}\notag\\
&\equiv_\mu -2\Pi_{pq}^i(\Phi_i^{\mu-1}(\phi_{ij}(z_j,t))-\Gamma_{ij|\mu},s^{\mu-1}) \notag\\
&+2\sum_{\alpha,\beta=1}^n \Lambda_{\alpha\beta}^i(\phi_{ij}(z_j,t),t)\frac{\partial (\Phi_i^{(\mu-1)p}(\phi_{ij}(z_j,t),t)-\Gamma_{ij|\mu}^p)}{\partial z_i^\alpha}\frac{\partial (\Phi_i^{(\mu-1)q} (\phi_{ij}(z_j,t),t)-\Gamma_{ij|\mu}^q)}{\partial z_i^\beta}\notag\\
&\equiv_\mu -2\Pi_{pq}^i(\Phi_i^{\mu-1}(\phi_{ij}(z_j,t),s^{\mu-1})+2\sum_{\gamma=1}^m \Gamma_{ij}^\gamma\frac{\partial \Pi_{pq}^i}{\partial w_i^\gamma}(\Phi_i^{\mu-1},s^{\mu-1})\notag\\
&+2 \sum_{\alpha,\beta=1}^n \Lambda_{\alpha\beta}^i(\phi_{ij}(z_j,t),t)\frac{\partial \Phi_i^{(\mu-1) p}(\phi_{ij}(z_j,t),t)}{\partial z_i^\alpha}\frac{\partial \Phi_i^{(\mu-1) q}(\phi_{ij}(z_j,t),t)}{\partial z_i^\beta}-2\sum_{\alpha,\beta=1}^n\Lambda_{\alpha\beta}^i(z_i,t)\frac{\partial \Gamma_{ij|\mu}^p}{\partial z_i^\alpha}\frac{\partial \Phi_i^{(\mu-1)q}}{\partial z_i^\beta}\notag\\
&-2\sum_{\alpha,\beta=1}^n\Lambda_{\alpha\beta}^i(z_i,t)\frac{\partial \Phi_i^{(\mu-1)p}}{\partial z_i^\alpha}\frac{\partial \Gamma_{ij|\mu}^q}{\partial z_i^\beta}\notag
\end{align}

\begin{align}\label{mm32}
-\lambda_{i|\mu}(w_i^p,w_i^q)\equiv_\mu 2\Pi_{pq}^i(\Phi_i^{\mu-1}(\phi_{ij}(z_j,t),t),s^{\mu-1})-2\sum_{\alpha,\beta=1}^n \Lambda_{\alpha\beta}^i(\phi_{ij}(z_j,t),t)\frac{\partial \Phi_i^{(\mu-1) p}(\phi_{ij}(z_j,t),t)}{\partial z_i^\alpha}\frac{\partial \Phi_i^{(\mu-1) q}(\phi_{ij}(z_j,t),t)}{\partial z_i^\beta}
\end{align}

\begin{align}\label{mm33}
&\pi_f(\Gamma_{ij|\mu})(w_i^p,w_i^q)=\Lambda_0(\Gamma_{ij|\mu}^p,f_i^q)-\Lambda_0(\Gamma_{ij|\mu}^q,f_i^p)-\Gamma_{ij|\mu}(\Pi_0(w_i^p,w_i^q))\\
&\equiv_\mu2\sum_{\alpha,\beta=1}^n\Lambda_{\alpha\beta}^i(z_i,t)\frac{\partial \Gamma_{ij|\mu}^p}{\partial z_i^\alpha}\frac{\partial \Phi_i^{(\mu-1)q}}{\partial z_i^\beta}-2\sum_{\alpha,\beta=1}^n\Lambda_{\alpha\beta}^i(z_i,t)\frac{\partial \Gamma_{ij|\mu}^q}{\partial z_i^\alpha}\frac{\partial \Phi_i^{(\mu-1)p}}{\partial z_i^\beta}-2\sum_{\gamma=1}^m\Gamma_{ij|\mu}^\gamma \frac{\partial \Pi_{pq}^i}{\partial w_i^\gamma}(\Phi_i^{\mu-1},s^{\mu-1}) \notag
\end{align}

Then from $(\ref{mm31}),(\ref{mm32})$, and $(\ref{mm33})$, we get $(\ref{cc6})$.

Lastly we show $(\ref{cc8})$. It is sufficient to show that $\pi_g(\gamma_{i|\mu})(y_i^p,y_i^q)=(f^*G)\lambda_{i|\mu}(y_i^p,y_i^q)$ for any $p,q$. We note that $\Omega_{pq}^i(\Psi_i(w_i,s))=\sum_{\alpha\beta=1}^n \Pi_{\alpha\beta}^i(w_i,s)\frac{\partial \Psi_i^p}{\partial w_i^\alpha}\frac{\partial \Psi_i^q}{\partial w_i^\beta}$, and $\Omega_{pq}^i(\Upsilon_i(z_i,t))=\sum_{\alpha,\beta=1}^n \Lambda_{\alpha\beta}^i(z_i,t)\frac{\partial \Upsilon_i^p}{\partial z_i^\alpha}\frac{\partial \Upsilon_i^q}{\partial z_i^\beta}$.
\begin{align*}
&\pi_h(\gamma_{i|\mu})(y_i^p,y_i^q)=\Lambda_0(\gamma_{i|\mu}(y_i^p),(g_i\circ f_i)^q)-\Lambda_0(\gamma_{i|\mu}(y_i^q),(g_i\circ f_i)^p)-\gamma_{i|\mu}(\Omega_0(y_i^p,y_i^q))\\
&\equiv_\mu -\Lambda_i( \Upsilon_i^p-\Psi_i^p(\Phi_i^{\mu-1},s^{\mu-1}),\Psi_i^q(\Phi_i^{\mu-1},s^{\mu-1}))+\Lambda_i(\Upsilon_i^q-\Psi_i^q(\Phi_i^{\mu-1},s^{\mu-1}),\Psi^p(\Phi_i^{\mu-1},s^{\mu-1}))\\
&-2\sum_{\beta=1}^l (\Upsilon_i^\beta-\Psi_i^\beta(\Phi_i^{\mu-1},s^{\mu-1}))\frac{\partial \Omega_{pq}^i }{\partial y_i^\beta }(\Psi_i(\Phi_i^{\mu-1},s^{\mu-1}))\\
&\equiv_\mu -\Lambda_i( \Upsilon_i^p-\Psi_i^p(\Phi_i^{\mu-1},s^{\mu-1}),\Psi_i^q(\Phi_i^{\mu-1},s^{\mu-1})-\Upsilon_i^q+\Upsilon_i^q)+\Lambda_i(\Upsilon_i^q-\Psi_i^q(\Phi_i^{\mu-1},s^{\mu-1}),\Psi^p(\Phi_i^{\mu-1},s^{\mu-1}))\\
&+2\sum_{\beta=1}^l (\Upsilon_i^\beta-\Psi_i^\beta(\Phi_i^{\mu-1},s^{\mu-1}))\frac{\partial \Omega_{pq}^i }{\partial y_i^\beta }(\Psi_i(\Phi_i^{\mu-1},s^{\mu-1}))\\
&\equiv_\mu -\Lambda_i(\Upsilon_i^p,\Upsilon_i^q)-\Lambda_i(\Psi_i^q(\Phi_i^{\mu-1},s^{\mu-1}),\Psi_i^p(\Phi_i^{\mu-1},s^{\mu-1}))+2\sum_{\beta=1}^l (\Upsilon_i^\beta-\Psi_i^\beta(\Phi_i^{\mu-1},s^{\mu-1}))\frac{\partial \Omega_{pq}^i }{\partial y_i^\beta }(\Psi_i(\Phi_i^{\mu-1},s^{\mu-1}))
\end{align*}
On the other hand, we have
\begin{align*}
&(f^* G (\lambda_{i|\mu}))(y_i^a,y_i^b)=\lambda_{i|\mu}(g_i^p,g_i^q)\equiv_\mu \lambda_{i|\mu}( \Psi_i^p,\Psi_i^q)\\
&\equiv_\mu 2\sum_{r,s=1}^l-\Pi_{rs}^i(\Phi_i^{\mu-1},s^{\mu-1})\frac{\partial \Psi_i^p}{\partial w_i^r}(\Phi_i^{\mu-1},s^{\mu-1})\frac{\partial \Psi_i^q}{\partial w_i^s} (\Phi_i^{\mu-1},s^{\mu-1})+2\sum_{r,s=1}^l\sum_{\alpha,\beta=1}^n \Lambda_{\alpha\beta}^i(z_i,t)\frac{\partial \Phi_i^{(\mu-1)r}}{\partial z_i^\alpha}\frac{\partial \Phi_i^{(\mu-1)s}}{\partial z_i^\beta}\frac{\partial \Psi_i^p}{\partial w_i^r}\frac{\partial \Psi_i^q}{\partial w_i^s}\\
&\equiv_\mu -2\Omega_{pq}^i(\Psi_i(\Phi_i^{\mu-1},s^{\mu-1}))+\Lambda_i(\Psi_i^p(\Phi_i^{\mu-1},s^{\mu-1}),\Psi_i^q(\Phi_i^{\mu-1},s^{\mu-1}))\\
&\equiv_\mu -2\Omega_{pq}^i(\Psi_i(\Phi_i^{\mu-1},s^{\mu-1})-\Upsilon_i+\Upsilon_i)+\Lambda_i(\Psi_i^p(\Phi_i^{\mu-1},s^{\mu-1}),\Psi_i^q(\Phi_i^{\mu-1},s^{\mu-1}))\\
&\equiv_\mu -2\Omega_{pq}^i(\Gamma_i)-2\sum_{\beta=1}^l (\Psi_i^\beta(\Phi_i^{\mu-1},s^{\mu-1})-\Upsilon_i^\beta)\frac{\partial \Omega_{pq}^i}{\partial y_i^\beta}(\Psi_i(\Phi_i^{\mu-1},s^{\mu-1}))+\Lambda_i(\Psi_i^p(\Phi_i^{\mu-1},s^{\mu-1}),\Psi_i^q(\Phi_i^{\mu-1},s^{\mu-1}))\\
&\equiv_\mu -2\sum_{\alpha,\beta=1}^n \Lambda_{\alpha\beta}^i(z_i,t)\frac{\partial \Upsilon_i^p}{\partial z_i^\alpha}\frac{\partial \Upsilon_i^q}{\partial z_i^\beta} - 2\sum_{\beta=1}^l (\Psi_i^\beta(\Phi_i^{\mu-1},s^{\mu-1})-\Upsilon_i^\beta)\frac{\partial \Omega_{pq}^i}{\partial y_i^\beta}(\Psi_i(\Phi_i^{\mu-1},s^{\mu-1}))+\Lambda_i(\Psi_i^p(\Phi_i^{\mu-1},s^{\mu-1}),\Psi_i^q(\Phi_i^{\mu-1},s^{\mu-1})
\end{align*}
Hence we get the claim.
\end{proof}

We set $\rho_{vij}'=\sum_\lambda \frac{\partial \psi_{ij}^\lambda}{\partial s^v}|_{s=0} \frac{\partial}{\partial w_i^\lambda}\in \Gamma(V_{ij},\Theta_Y)$, and $\Pi_{vi}'=\sum \frac{\partial \Pi_{\alpha\beta}^i(w_i,s)}{\partial s^v}\frac{\partial}{\partial w_i^\alpha}\wedge \frac{\partial}{\partial w_i^\beta}$, and $-\tau_{vi}'=\sum \frac{\partial \Psi_i^\alpha}{\partial s_i^v}|_{s=0} \frac{\partial}{\partial y_i^\alpha}\in \Gamma(V_i, g^*\Theta_Z)$. Then we have $\tau_{vi}'-\tau_{vj}'=G \rho_{vij}'$ on $V_{ij}$, and $\pi_g(\tau_{vi}')=G \Pi_{vi}'$. Then we show that

\begin{lemma}
$(\ref{cc10})_\mu,(\ref{cc11})_\mu$ and $(\ref{cc12})_\mu$ are equivalent to the following:
\begin{align}
\Gamma_{ij|\mu}&=\Phi_{j|\mu}-\Phi_{i|\mu}+\sum_{v=1}^{r'} s_\mu^v f^*\rho_{vij}' \label{cc1}\\
-\gamma_{i|\mu}&=(f^* G)\Phi_{i|\mu}-\sum_{v=1}^{r'} s_\mu^v f^*\tau_{vi}' \label{cc2}\\
-\lambda_{i|\mu}&=\pi_f(\Phi_{i|\mu})-\sum_{v=1}^{r'} s_\mu^v  f^*\Pi_{vi}'\label{cc3}
\end{align}
where $\Phi_{i|\mu}=\sum_{\alpha=1}^m \Phi_{i|\mu}^\alpha \frac{\partial}{\partial w_i^\alpha}$.
\end{lemma}

\begin{proof}
For $(\ref{cc1})$, see \cite{Hor74} p.663. Let us show $(\ref{cc2})$. Rewrite $(\ref{cc11})_\mu$ in the following: for any $\beta$,
\begin{align*}
\Upsilon_i^\beta\equiv_\mu \Psi_i^\beta(\Phi_i^{\mu-1}+\Phi_{i|\mu}, s^{\mu-1}+s_\mu)
\end{align*}
Then we have
\begin{align*}
\Upsilon_i^\beta\equiv_\mu \Psi_i^\beta(\Phi_i^{\mu-1},s^{\mu-1})+\sum_{\alpha=1}^m \Phi_{i|\mu}^\alpha\frac{\partial \Psi_i^\beta}{\partial w_i^\alpha}(\Phi_i^{\mu-1},s^{\mu-1})+\sum_{v=1}^{r'}s_\mu^v\frac{\partial \Psi_i^\beta }{\partial s_i^v }(\Phi_i^{\mu-1},s^{\mu-1})\\
\iff -\gamma_{i|\mu}^\beta \equiv_\mu \sum_{\alpha=1}^m \Phi_{i|\mu}^\alpha\frac{\partial g_i^\beta}{\partial w_i^\alpha}(f_i)+\sum_{v=1}^{r'}s_\mu^v\frac{\partial \Psi_i^\beta}{\partial s_i^v}(f_i,0)
\end{align*}
which are equivalent to $(\ref{cc2})$.

Next let us show $(\ref{cc3})$. Rewrite $(\ref{cc12})_\mu$ in the following.
\begin{align*}
\Pi_{pq}^i(\Phi_i^{\mu-1}+\Phi_{i|\mu},s^{\mu-1}+s_\mu)\equiv_\mu \sum_{\alpha,\beta=1}^n \Lambda_{\alpha\beta}^i(z_i,t)\frac{\partial ( \Phi_i^{(\mu-1)p}+\Phi_{i|\mu}^p ) }{\partial z_i^\alpha}\frac{\partial (\Phi_i^{(\mu-1)q}+\Phi_{i|\mu}^q )}{\partial z_i^\beta}
\end{align*}
Then we have
\begin{align*}
& \sum_{\alpha=1}^m \Phi_{i|\mu}^\alpha \frac{\partial \Pi_{pq}^i}{\partial w_i^\alpha}(\Phi_i^{\mu-1},s^{\mu-1})+\sum_{v=1}^{r'} s_\mu^v\frac{\partial \Pi_{pq}^i}{\partial s^v}(\Phi_i^{\mu-1},s^{\mu-1})\\
&\equiv_\mu \lambda_{i|\mu}^{p,q}+\sum_{\alpha,\beta=1}^n\Lambda_{\alpha\beta}^i(z_i,t)\frac{\partial \Phi_i^{(\mu-1)p}}{\partial z_i^\alpha}\frac{\partial \Phi_{i|\mu}^q}{\partial z_i^\beta}+\sum_{\alpha,\beta=1}^n \Lambda_{\alpha\beta}^i(z_i,t)\frac{\partial \Phi_{i|\mu}^p}{\partial z_i^\alpha}\frac{\partial \Phi_i^{(\mu-1)q}}{\partial z_i^\beta}
\end{align*}
which are equivalent to
\begin{align}\label{mm40}
 \sum_{\alpha=1}^m \Phi_{i|\mu}^\alpha \frac{\partial \Pi_{pq}^i}{\partial w_i^\alpha}(f_i,0)+\sum_{v=1}^{r'} s_\mu^v\frac{\partial \Pi_{pq}^i}{\partial s^v}(f_i,0)
\equiv_\mu \lambda_{i|\mu}^{p,q}+\sum_{\alpha,\beta=1}^n\Lambda_{\alpha\beta}^i(z_i)\frac{\partial f_i^p}{\partial z_i^\alpha}\frac{\partial \Phi_{i|\mu}^q}{\partial z_i^\beta}+\sum_{\alpha,\beta=1}^n \Lambda_{\alpha\beta}^i(z_i)\frac{\partial \Phi_{i|\mu}^p}{\partial z_i^\alpha}\frac{\partial f_i^q}{\partial z_i^\beta}
\end{align}
On the other hand, we have
\begin{align}
\pi_f(\Phi_{i|\mu})(w_i^p,w_i^q)=\Lambda_0(\Phi_{i|\mu}^p, f_i^q)-\Lambda_0(\Phi_{i|\mu}^q,f_i^p)-2\sum_{\alpha=1}^m\Phi_{i|\mu}^\alpha\frac{\partial \Pi_{pq}^i}{\partial w_i^\alpha}(f_i)
\end{align}\label{mm41}
From $(\ref{mm40})$ and $(\ref{mm41})$, we get $(\ref{cc3})$.
\end{proof}

\begin{lemma}
Under the hypothesis of Theorem $\ref{ll17}$, we can find $\Phi_{i|\mu}\in \Gamma(U_i,f^*\Theta_Y)$ and $s_\mu=(s_\mu^v)\in \mathbb{C}^{r'}$ which satisfy $(\ref{cc1}),(\ref{cc2})$, and $(\ref{cc3})$.
\end{lemma}

\begin{proof}
Since we have an exact sequence $0\to f^*\Theta_Y^\bullet\xrightarrow{f^* G} h^* \Theta_Z\to f^*\mathcal{N}_g^\bullet\to 0$, from $(\ref{cc7})$ and $(\ref{cc8})$, the collection $\{\gamma_{i|\mu}\}$ represents an element of $\mathbb{H}^0(X, f^*\mathcal{N}_g^\bullet)$. Since $f^*\circ\tau:T_{0'}(N)\to \mathbb{H}^0(X, f^*\mathcal{N}_g^\bullet)$ is surjective, we can find $s_\mu=(s_\mu^v)$ such that $\{\gamma_{i|\mu}\}$ and $\{\sum_{v=1}^{r'}s_\mu^vf^* \tau_{v_i}'\}$ represent the same element in $\mathbb{H}^0(X, f^*\mathcal{N}_g^\bullet)$. Hence there exist $\Phi_{i|\mu}\in \Gamma(U_i, f^* \Theta_Y)$ such that $-\gamma_{i|\mu}=(f^* G)\Phi_{i|\mu}-\sum_v s_\mu^v f^*\tau_{vi}'$, which proves $(\ref{cc2})$. on the other hand, from $(\ref{cc8})$, we have
\begin{align*}
&(f^* G)(-\lambda_{i|\mu})=\pi_h(-\gamma_{i|\mu})=\pi_h(f^*G(\Phi_{i|\mu}))-\sum_v s_\mu^v\pi_h  f^*\tau_{vi'}=f^*G(\pi_f(\Phi_{i|\mu}))-\sum_v s_\mu^vf^* \pi_g\tau_{vi'}\\
&=f^*G(\pi_f(\Phi_{i|\mu}))-\sum_v s_\mu^vf^* (G\Pi_{iv}')=f^*G(\pi_f(\Phi_{i|\mu}))-\sum_{v=1}^{r'} s_\mu^v(f^*G)f^* \Pi_{iv}'=f^*G\left(\pi_f(\Phi_{i|\mu})-\sum_{v=1}^{r'} s_\mu^v f^*\Pi_{iv}' \right).
\end{align*}
Hence since $f^*G: f^*\Theta_Y^\bullet \to h^*\Theta_Z^\bullet$ is injective, we get $(\ref{cc3})$. For $(\ref{cc1})$, we note that 
\begin{align*}
(f^*G)\Gamma_{ij|\mu}&=\gamma_{i|\mu}-\gamma_{j|\mu}=-(f^*G)\Phi_{i|\mu}+\sum_{v=1}^{r'} s_\mu^v f^*\tau_{vi}'+(f^*G)\Phi_{j|\mu}-\sum_{v=1}^{r'}s_\mu^v f^*\tau_{vj}'\\
&=f^*G(\Phi_{j|\mu}-\Phi_{i|\mu})+\sum_{v=1}^{r'} f^* (G\rho_{vij}')=f^*G(\Phi_{j|\mu}-\Phi_{i|\mu}+\sum_{v=1}^{r'} f^*\rho_{vij}')
\end{align*}
Hence we get $(\ref{cc1})$.
\end{proof}

Hence induction holds for $\mu$ so that we have formal power series $s(t)$ and $\Phi_i(z_i,t)$ satisfying $(\ref{mm20})-(\ref{mm23})$.

\subsubsection{Proof of convergence}\

By a similar argument as in the proof of Theorem \ref{ii1}, we can choose solutions $\Phi_{i|\mu}$ and $s_\mu$ in each inductive step so that $\Phi_i$ and $s$ converge absolutely and uniformly for $|t|<\epsilon$ for a sufficiently small number $\epsilon>0$. The completes the proof of Theorem \ref{ll17}.

\end{proof}

\begin{theorem}\label{ll18}
Let $f:(X,\Lambda_0)\to (Y, \Pi_0), g:(Y,\Pi_0)\to (Z, \Omega_0)$, and $h=g\circ f$ be holomorphic Poisson maps of holomorphic Poisson manifolds. Let $p:(\mathcal{X},\Lambda)\to M, q:(\mathcal{Y},\Pi)\to M$ and $\pi:(\mathcal{Z},\Omega)\to M$ be families of holomorphic Poisson manifolds such that $(X,\Lambda_0)=p^{-1}(0), (Y,\Pi_0)=q^{-1}(0)$ and $(Z, \Omega_0)=\pi^{-1}(0)$ for some point $0\in M$. Let $\Phi:(\mathcal{X},\Lambda)\to (\mathcal{Y}, \Pi)$ and $\Upsilon:(\mathcal{X}, \Lambda)\to (\mathcal{Z}, \Omega)$ be holomorphic Poisson maps over $M$ which induces $f$ and $h$ over $0\in M$, respectively. Assume that
\begin{enumerate}
\item $p$ and $q$ are proper.
\item $f^*:\mathbb{H}^0(Y, g^*\Theta_Z^\bullet)\to \mathbb{H}^0(X,h^*\Theta_Z^\bullet)$ is surjective.
\item $f^*:\mathbb{H}^1(Y, g^*\Theta_Z^\bullet)\to \mathbb{H}^1(X, h^*\Theta_Z^\bullet)$ is injective.
\end{enumerate}
Then there exists an open neighborhood $N$ of $0$, and a holomorphic Poisson map $\Psi:(\mathcal{Y}|_N, \Pi|_N)\to (\mathcal{Z}, \Omega)$ over $N$ such that $\Upsilon|_N=\Psi\circ (\Phi|_N)$.
\end{theorem}

\begin{proof}
We extend arguments in \cite{Hor74} p.664-p.665 in the context of holomorphic Poisson deformations. We tried to maintain notational consistency with \cite{Hor74}.
We may assume the following:
\begin{enumerate}
\item $M=\{t\in \mathbb{C}^r||t|<\epsilon\}$ for a sufficiently small number $\epsilon >0$, and $0=(0,...,0)$.
\item $\mathcal{X}$ (resp. $\mathcal{Y}$) is covered by a finite number of coordinate neighborhoods $\mathcal{U}_i(i\in I)$ (resp. $\mathcal{V}_i (i\in J)$) and each $\mathcal{U}_i$ (resp. $\mathcal{V}_i$) is covered by a system of coordinates $(z_i,t)$ (resp. $(w_i,t))$ such that $p(z_i,t)=t$ (resp. $q(w_i,t)=t)$. Moreover, each $\mathcal{V}_i$ is a polydisc $\{(w_i,t)||w_i|<1,|t|<\epsilon\}$. We set $U_i=\mathcal{U}_i\cap X$ and $V_i=\mathcal{V}_i\cap Y$.
\item $I\subset J$, and $\Phi(\mathcal{U}_i)$ is contained in $\mathcal{V}_i$ for each $i\in I$. $\Phi$ is given by $w_i=\Phi_i(z_i,t)$ and we set $f_i(z_i)=\Phi_i(z_i,0)$.
\item For each $i\in J$, there exists a coordinate neighborhood $\mathcal{W}_i$ on $\mathcal{Z}$ such that $\Phi(\mathcal{U}_i)\subset \mathcal{W}_i$ for $i\in I$, and $g(V_i)\subset \mathcal{W}_i\cap Z$ for $i\in J$.
\item Each $\mathcal{W}_i$ is covered by a system of coordinate $(y_i,t)$ such that $\pi(y_i,t)=t$.
\item $\Upsilon$ and $g$ are given, respectively, by $y_i=\Upsilon_i(z_i,t)$ and $y_i=g_i(w_i)$. We set $h_i(z_i)=\Upsilon_i(z_i,0)$.
\item $(z_i,t)\in \mathcal{U}_i, (w_i,t)\in \mathcal{V}_i$, and $(y_i,t)\in \mathcal{W}_i$ coincide with $(z_j,t)\in \mathcal{U}_j, (w_j,t)\in \mathcal{V}_j$, and $(y_j,t)\in \mathcal{W}_j$, respectively, if and only if $z_i=\phi_{ij}(z_j,t), w_i=\psi_{ij}(w_j,t)$, and $y_i=\theta_{ij}(y_j,t)$. We set $e_{ij}(y_j)=\theta_{ij}(y_j,0)$.\\
\item On $\mathcal{U}_i$, $\Lambda$ is given by $\sum_{\alpha,\beta=1}^n \Lambda_{\alpha\beta}^i(z_i,t)\frac{\partial}{\partial z_i^\alpha}\wedge \frac{\partial}{\partial z_i^\beta}$ with $\Lambda_{\alpha\beta}^i(z_i,t)=-\Lambda_{\beta\alpha}^i(z_i,t)$.
\item On $\mathcal{V}_i$, $\Pi$ is given by $\sum_{\alpha,\beta=1}^m \Pi_{\alpha\beta}^i(w_i,t)\frac{\partial}{\partial w_i^\alpha}\wedge \frac{\partial}{\partial w_i^\beta}$ with $\Pi_{\alpha\beta}^i(w_i,t)=-\Pi_{\alpha\beta}^i(w_i,t)$.
\item On $\mathcal{W}_i$, $\Omega$ is given by $\sum_{\alpha,\beta=1}^l \Omega_{\alpha\beta}^i (y_i,t)\frac{\partial}{\partial y_i^\alpha}\wedge \frac{\partial}{\partial y_i^\beta}$ with $\Omega_{\alpha\beta}^i(y_i,t)=-\Omega_{\beta\alpha}^i(y_i,t)$.
\end{enumerate}
We shall construct holomorphic functions $\Psi_i:\mathcal{V}_i\to \mathbb{C}^l$ such that
\begin{align}
\Psi_i(w_i,0)&=g_i(z_i),\,\,\,\,\,\text{on $V_i$, for $i\in I$}, \label{mm50}\\
\Psi_i(\psi_{ij},t)&=\theta_{ij}(\Psi_j,t),\,\,\,\,\,\text{on $\mathcal{V}_{ij}$, for $i,j\in J$}, \label{mm51}\\
\Upsilon_i(z_i,t)&=\Psi_i(\Phi_i,t),\,\,\,\,\,\text{on $\mathcal{U}_i$, for $i\in I$}, \label{mm52}\\
\Omega_{pq}^i(\Psi_i(w_i,t),t)&=\sum_{\alpha,\beta=1}^m \Pi_i(w_i,t)\frac{\partial \Psi_i^p}{\partial w_i^\alpha}\frac{\partial \Psi_i^q}{\partial w_i^\beta}. \label{mm75}
\end{align}

\subsubsection{Existence of formal solutions}\

We will prove the existence of formal solutions of $\Psi_i(z_i,t)$ satisfying $(\ref{mm50})-(\ref{mm52})$ as power series in $t$. Let $\Psi_i(w_i,t)=\sum_{\mu=0}^\infty\Psi_{i|\mu}(w_i,t)$, where $\Psi_{i|\mu}(w_i,t)$ is a homogenous polynomial in $t$ of degree  $\mu$, and let $\Psi_i^\mu(w_i,t)=\Psi_{i|0}(w_i,t)+\Psi_{i|1}(w_i,t)+\cdots+\Psi_{i|\mu}(w_i,t)$. We note that $(\ref{mm50})-(\ref{mm52})$ are equivalent to the following system of congruences:
\begin{align}
\Psi_i^\mu(\psi_{ij},t)&\equiv_\mu\theta_{ij}(\Psi_j^\mu,t),\,\,\,\,\,\text{on $\mathcal{V}_{ij}$, for $i,j\in J$}, \label{mm53}\\
\Upsilon_i(z_i,t)&\equiv_\mu\Psi_i^\mu(\Phi_i,t),\,\,\,\,\,\text{on $\mathcal{U}_i$, for $i\in I$}, \label{mm54}\\
\Omega_{pq}^i(\Psi_i^\mu(w_i,t),t)&\equiv_\mu\sum_{\alpha,\beta=1}^m \Pi_{\alpha\beta}^i(w_i,t)\frac{\partial \Psi_i^{\mu p}}{\partial w_i^\alpha}\frac{\partial \Psi_i^{\mu q}}{\partial w_i^\beta}, \label{mm55}
\end{align}
for $\mu=0,1,2,\cdots$.

We shall construct formal solutions of $\Psi_i(w_i,t)$ satisfying $(\ref{mm53})-(\ref{mm55})$ by induction on $\mu$. We set $\Psi_{i|0}=g_i(w_i)$. Then  the induction holds for $\mu=0$.

Suppose that we have already determined $\Psi^{\mu-1}$ satisfying $(\ref{mm53})_{\mu-1}-(\ref{mm55})_{\mu-1}$. We define homogeneous polynomials $\Gamma_{ij|\mu}\in \Gamma(V_{ij},g^*\Theta_Z)(i,j\in J)$, $\gamma_{i|\mu}\in \Gamma(U_i, h^*\Theta_Z)(i\in I )$ and $\lambda_{i|\mu}\in \Gamma(V_i, \wedge^2 g^*\Theta_Z)(i\in J)$ of degree $\mu$ by the following congruences:
\begin{align}
\Gamma_{ij|\mu}&\equiv_\mu \sum_{\alpha=1}^l (\Psi_i^{(\mu-1)\alpha}(\psi_{ij},t)-\theta_{ij}^\alpha(\Psi_j^{\mu-1},t))\frac{\partial}{\partial y_i^\alpha} \label{mm60}\\
-\gamma_{i|\mu}&\equiv_\mu \sum_{\beta=1}^l (\Upsilon_i^\beta-\Psi_i^{(\mu-1)\beta}(\Phi_i,t))\frac{\partial}{\partial y_i^\beta} \label{mm61}\\
\lambda_{i|\mu}&\equiv_\mu \sum_{p,q=1}^l\left( - \Omega_{pq}^i(\Psi_i^\mu(w_i,t),t)+ \sum_{\alpha,\beta=1}^m \Pi_{\alpha\beta}^i(w_i,t)\frac{\partial \Psi_i^{\mu p}}{\partial w_i^\alpha}\frac{\partial \Psi_i^{\mu q}}{\partial w_i^\beta}    \right)\frac{\partial}{\partial y_i^p}\wedge \frac{\partial}{\partial y_i^q} \label{mm62}
\end{align}

\begin{lemma}
\begin{align}
\Gamma_{jk|\mu}-\Gamma_{ik|\mu}+\Gamma_{ij|\mu}&=0,\,\,\,\,\,\,\,\,\,\,\,\,\,\,\,\,\,\,\,\,\,\,\,\,\,\,\,\,\,\,\,\,\text{on $V_{ijk}$, for $i,j,k\in J$}, \label{mm63}\\
\lambda_{j|\mu}-\lambda_{i|\mu}+\pi_g( \Gamma_{ij|\mu})&=0,\,\,\,\,\,\,\,\,\,\,\,\,\,\,\,\,\,\,\,\,\,\,\,\,\,\,\,\,\,\,\,\,\text{on $V_{ij}$, for $i,j\in J$} \label{mm64}\\
\pi_g( \lambda_{i|\mu})&=0,\,\,\,\,\,\,\,\,\,\,\,\,\,\,\,\,\,\,\,\,\,\,\,\,\,\,\,\,\,\,\,\,\text{on $V_{i}$, for $i\in J$} \label{mm65}\\
\gamma_{i|\mu}-\gamma_{j|\mu}&=f^* \Gamma_{ij|\mu},\,\,\,\,\,\,\,\,\,\,\,\,\,\,\,\,\,\text{on $U_{ij}$, for $i,j\in I$}, \label{mm66}\\
\pi_h(\gamma_{i|\mu})&=f^* \lambda_{i|\mu},\,\,\,\,\,\,\,\,\,\,\,\,\,\,\,\,\,\,\,\text{on $U_i$, for  $i\in I$}\label{mm67}
\end{align}
\end{lemma}

\begin{proof}
For $(\ref{mm63}),(\ref{mm64})$, and $(\ref{mm65})$, see Lemma \ref{qq60}. For $(\ref{mm66})$, see \cite{Hor74} p.665. Let us show $(\ref{mm67})$. It is sufficient to show that $\pi_h(\lambda_{i|\mu})(y_i^p,y_i^q)=F \lambda_{i|\mu}(y_i^p,y_i^q)$ for any $p,q$. We note that $\Omega_{pq}^i(\Upsilon_i(z_i,t),t)=\sum_{\alpha,\beta=1}^n \Lambda_{\alpha\beta}^i(z_i,t)\frac{\partial \Upsilon_i^p}{\partial z_i^\alpha}\frac{\partial \Upsilon_i^q}{\partial z_i^\beta}$ and $\Pi_{pq}^i(\Phi_i(z_i,t),t)=\sum_{\alpha,\beta=1}^n \Lambda_{\alpha\beta}^i(z_i,t)\frac{\partial \Phi_i^p}{\partial z_i^\alpha}\frac{\partial \Phi_i^q}{\partial z_i^\beta}$. Then we have
\begin{align*}
&\pi_h(\gamma_{i|\mu})(y_i^p,y_i^q)=\Lambda_0(\gamma_{i|\mu}(y_i^p), h_i^q)-\Lambda_0(\gamma_{i|\mu}(y_i^q),h_i^p)-\gamma_{i|\mu}(\Omega_0(y_i^p,y_i^q))\\
&\equiv_\mu\Lambda_i(\Psi_i^{(\mu-1)p}(\Phi_i,t)-\Upsilon_i^p,\Psi_i^{(\mu-1)q }(\Phi_i,t))-\Lambda_i(\Psi_i^{(\mu-1)q}-\Upsilon_i^q,\Psi_i^{(\mu-1)p}(\Phi_i,t)-\Upsilon_i^p+\Upsilon_i^p)\\
&-2\sum_{\beta=1}^l(\Psi_i^{(\mu-1)\beta}(\Phi_i,t)-\Upsilon_i^\beta)\frac{\partial \Omega_{pq}^i}{\partial y_i^\beta}(\Psi_i^{\mu-1}(\Phi_i,t),t)\\
&\equiv_\mu \Lambda_i(\Psi_i^{(\mu-1)p}(\Phi_i,t),\Psi_i^{(\mu-1)q}(\Phi_i,t))-\Lambda_i(\Upsilon_i^p,\Upsilon_i^q)-2\sum_{\beta=1}^l(\Psi_i^{(\mu-1)\beta}(\Phi_i,t)-\Upsilon_i^\beta)\frac{\partial \Omega_{pq}^i}{\partial y_i^\beta}(\Psi_i^{\mu-1}(\Phi_i,t),t)
\end{align*}
On the other hand, we have
\begin{align*}
&f^*\lambda_{i|\mu}(y_i^p,y_i^q)\equiv_\mu -2\Omega_{pq}^i(\Psi_i^{\mu-1}(f_i(z_i),t),t)+2 \sum_{\alpha,\beta=1}^m \Pi_{\alpha\beta}^i(f_i(z_i),t)\frac{\partial \Psi_i^{(\mu-1) p}}{\partial w_i^\alpha} (f_i(z_i),t) \frac{\partial \Psi_i^{(\mu-1) q}}{\partial w_i^\beta}(f_i(z_i),t)\\
&\equiv_\mu -2\Omega_{pq}^i(\Psi_i^{\mu-1}(\Phi_i,t),t)+2 \sum_{\alpha,\beta=1}^m \Pi_{\alpha\beta}^i(\Phi_i,t)\frac{\partial \Psi_i^{(\mu-1) p}}{\partial w_i^\alpha} (\Phi_i,t) \frac{\partial \Psi_i^{(\mu-1) q}}{\partial w_i^\beta}(\Phi_i,t)\\
&\equiv_\mu-2\Omega_{pq}^i(\Psi_i^{\mu-1}(\Phi_i,t)-\Upsilon_i+\Upsilon_i,t)+2\sum_{\alpha,\beta=1}^n \Lambda_{\alpha\beta}^i(z_i,t)\frac{\partial \Psi_i^{(\mu-1)p}(\Phi_i,t)}{\partial z_i^\alpha}\frac{\partial \Psi_i^{(\mu-1)q}(\Phi_i,t)}{\partial z_i^\beta}\\
&\equiv_\mu-2\Omega_{pq}^i(\Upsilon_i,t)-2\sum_{\beta=1}^l (\Psi_i^{(\mu-1)\beta}(\Phi_i,t)-\Upsilon_i^\beta)\frac{\partial \Omega_{pq}^i}{\partial y_i^\beta}(\Psi_i^{\mu-1}(\Phi_i,t),t)+\Lambda_i(\Psi_i^{(\mu-1)p}(\Phi_i,t),\Psi_i^{(\mu-1)q}(\Phi_i,t))\\
&\equiv_\mu -\Lambda_i(\Upsilon_i^p,\Upsilon_i^q)-2\sum_{\beta=1}^l (\Psi_i^{(\mu-1)\beta}(\Phi_i,t)-\Upsilon_i^\beta)\frac{\partial \Omega_{pq}^i}{\partial y_i^\beta}(\Psi_i^{\mu-1}(\Phi_i,t),t)+\Lambda_i(\Psi_i^{(\mu-1)p}(\Phi_i,t),\Psi_i^{(\mu-1)q}(\Phi_i,t))
\end{align*}
Hence we get $(\ref{mm67})$.
\end{proof}

\begin{lemma}
$(\ref{mm53})_{\mu},(\ref{mm54})_{\mu}$, and $(\ref{mm55})_\mu$ are equivalent to the following:
\begin{align}
\Gamma_{ij|\mu}&=\Psi_{j|\mu}-\Psi_{i|\mu},\,\,\,\,\,\text{on $V_{ij}$, for $i,j\in J$}, \label{mm68}\\
-\gamma_{i|\mu}&=f^* \Psi_{i|\mu},\,\,\,\,\,\,\,\,\,\,\,\,\,\,\,\,\,\,\text{on $U_i$, for $i\in I$}, \label{mm69}\\
-\lambda_{i|\mu}&=\pi_g(\Psi_{i|\mu}),\,\,\,\,\,\,\,\,\,\,\,\,\,\text{on $V_i$, for $i\in J$}, \label{mm70}
\end{align}
where $\Psi_{i|\mu}=\sum_{\alpha=1}^l \Psi_{i|\mu}^\alpha\frac{\partial}{\partial y_i^\alpha}\in \Gamma(U_i, g^*\Theta_Z)$.
\end{lemma}

\begin{proof}
For $(\ref{mm68})$ and $(\ref{mm69})$, see \cite{Hor74} p.665. Let us show that $(\ref{mm55})_\mu$ is equivalent to $(\ref{mm70})$. $(\ref{mm55})_\mu$ is equivalent to
\begin{align*}
&\Omega_{pq}^i(\Psi_i^{\mu-1}(w_i,t)+\Psi_{i|\mu},t)\equiv_\mu\sum_{\alpha,\beta=1}^m \Pi_{\alpha\beta}^i(w_i,t)\frac{\partial (\Psi_i^{(\mu-1)p}+\Psi_{i|\mu}^p)}{\partial w_i^\alpha}\frac{\partial (\Psi_i^{(\mu-1) q}+\Psi_{i|\mu}^q ) }{\partial w_i^\beta}\\
&\iff \sum_{\alpha=1}^l \Psi_{i|\mu}^\alpha\frac{\partial \Omega_{pq}^i}{\partial y_i^\alpha}(g_i,0)\equiv_\mu\lambda_{i|\mu}^{p,q}+\sum_{\alpha,\beta=1}^m \Pi_{\alpha\beta}^i(w_i)\frac{\partial \Psi_{i|\mu}^p}{\partial w_i^\alpha}\frac{\partial g_i^q}{\partial w_i^\beta}+\sum_{\alpha,\beta=1}^m  \Pi_{\alpha\beta}^i(w_i)\frac{\partial g_i^p}{\partial w_i^\alpha}\frac{\partial \Psi_{i|\mu}^q}{\partial w_i^\beta}
\end{align*}
On the other hand,
\begin{align*}
\pi_g(\Psi_{i|\mu})(w_i^p,w_i^q)=\Pi_0(\Psi_{i|\mu}^p,g_i^q)-\Pi_0(\Psi_{i|\mu}^q,g_i^p)-2\Psi_{i|\mu}(\Omega_{pq}^i(y_i))
\end{align*}
Hence we get $(\ref{mm70})$.
\end{proof}

\begin{lemma}
Under the hypothesis of Theorem $\ref{ll18}$, we can find $\Psi_{i|\mu}\in \Gamma(V_i,g^*\Theta_Z)$ which satisfy $(\ref{mm68}),(\ref{mm69})$, and $(\ref{mm70})$.
\end{lemma}

\begin{proof}
Since $f^*:\mathbb{H}^1(Y,g^*\Theta_Z^\bullet)\to\mathbb{H}^1(X, h^*\Theta_Z)$ is injective and $(\{f^*\Gamma_{ij|\mu}\},\{f^*\lambda_{i|\mu}\})$ is cohomologous to zero by $(\ref{mm66})$ and $(\ref{mm67})$, there exist $\Psi_{i|\mu}'\in \Gamma(V_i,g^*\Theta_Z)$ such that
\begin{align*}
\Gamma_{ij|\mu}=\Psi_{j|\mu}'-\Psi_{i|\mu}',\,\,\,\,\,\,\,-\lambda_{i|\mu}=\pi_g( \Psi_{i|\mu}').
\end{align*}
Then $\{-\gamma_{i|\mu}-f^*\Psi_{i|\mu}'\}$ defines a global section in $H^0(X, h^*\Theta_Z)$ since $\gamma_{i|\mu}+f^* \Psi_{i|\mu}'-\gamma_{j|\mu}-f^* \Psi_{j|\mu}' =f^*\Gamma_{ij}-f^* \Gamma_{ij}=0$. On the other hand, $\pi_h(-\gamma_{i|\mu}-f^*\Psi_{i|\mu}')=-f^*\lambda_{i|\mu}-\pi_h f^*\Psi_{i|\mu}'=f^*\pi_g(\Psi_{i|\mu}')-f^*\pi_g(\Psi_{i|\mu}')=0$ so that $\{\gamma_{i|\mu}-f^*\Psi_{i|\mu}'\}$ represents a homogenous polynomial with coefficients in $\mathbb{H}^0(X, h^*\Theta_Z)$. Since $f^*:\mathbb{H}^0(Y,  g^*\Theta_Z)\to \mathbb{H}^0(X, h^*\Theta_Z)$ is surjective, there exists $\chi_\mu\in \mathbb{H}^0(Y, g^*\Theta_Z)$ with $\pi_g(\chi_\mu)=0$ such that $-\gamma_{i|\mu}-f^*\Psi_{i|\mu}'=f^*\chi_\mu$ on $U_i$ for $i\in I$. We set $\Psi_{i|\mu}:=\Psi_{i|\mu}'+\chi_\mu$. Then $\Gamma_{ij}=\Psi_{j|\mu}-\Psi_{i|\mu}$, $f^*\Psi_{i|\mu}=f^*\Psi_{i|\mu}'+f^*\chi_\mu=-\gamma_{i|\mu}$, and $\pi_g(\Psi_{i|\mu})=\pi_g(\Psi_{i|\mu}')+\pi_g(\chi_\mu)=-\lambda_{i|\mu}$.
\end{proof}

Hence induction holds for $\mu$ so that we have a formal power series $\Psi_i(w_i,t)$ satisfying $(\ref{mm50})-(\ref{mm75})$.

\subsubsection{Proof of convergence}\

By a similar argument as in the proof of Theorem \ref{ii1}, we can choose solutions $\Psi_{i|\mu}$ in each inductive step so that $\Psi_i$ converges absolutely and uniformly for $|t|<\epsilon$ for a sufficiently small number $\epsilon>0$. This completes the proof of Theorem \ref{ll18}.

\end{proof}

\appendix

\section{Complex associated with a Poisson map}\label{appendix1}
Let $k$ be an algebraically closed field with characteristic $0$.
\begin{definition}
Let $(B,\Lambda_0)$ and $(C,\Pi_0)$ be a Poisson $k$-algebra with $\Lambda_0\in Hom_B(\wedge^2 \Omega_{B/k},B)$ and $\Pi_0\in Hom_C(\wedge^2 \Omega_{C/k}^1,C)$. Let $f:(C,\Pi_0)\to (B, \Lambda_0)$ be a Poisson $k$-homomorphism. We define a complex
\begin{align*}
Hom_C(\Omega_{C/k}^1,B)\xrightarrow{\pi=\pi_f} Hom_C(\wedge^2 \Omega_{C/k}^1,B)\xrightarrow{\pi} Hom_C(\wedge^3, \Omega_{C/k}^1,B)\xrightarrow{\pi} \cdots
\end{align*}
in the following way$:$ for $Q\in Hom_C(\wedge^q\Omega_{C/k},B)$ and $a_1,\cdots, a_{q+1}\in C$,
\begin{align*}
\pi(Q)(a_1, \cdots, a_{q+1})=&\sum_{\sigma\in S_{q,1}} sgn(\sigma) \Lambda_0( Q(a_{\sigma(1)}, \cdots ,a_{\sigma(q)}),  f(a_{\sigma(q+1)}))\\
&-(-1)^{q-1}\sum_{\sigma\in S_{2,q-1}}sgn(\sigma) Q(\Pi_0(a_{\sigma(1)},a_{\sigma(2)}),a_{\sigma(3)}, \cdots ,a_{\sigma(q+1)})
\end{align*}
\end{definition}

\begin{remark}
Let $f:(C,\Pi_0)\to (B,\Lambda_0)$ be a Poisson $k$-homomorphism. We have the following commutative diagram.
\begin{center}
$\begin{CD}
Hom_B(\Omega_{B/k}^1,B)@>[\Lambda_0,-]>> Hom_B(\wedge^2 \Omega_{B/k}^1,B)@>[\Lambda_0,-]>> Hom_B(\wedge^3 \Lambda_{B/k}^1,B)@>[\Lambda_0,-]>>\cdots\\
@VF VV @VF VV @VF VV\\
Hom_C(\Omega_{C/k}^1,B)@>\pi>> Hom_C(\wedge^2 \Omega_{C/k}^1,B)@>\pi>> Hom_C(\wedge^3, \Omega_{C/k}^1,B)@>\pi>> \cdots
\end{CD}$
\end{center}
where $F$ is defined in the following way$:$ for $P\in Hom_B(\wedge^q\Omega_{B/k}^1, B)$, and $a_1,...,a_q\in C$,
\begin{align*}
F(P)(a_1,...,a_q):=P(f(a_1),...,f(a_q))
\end{align*}
\end{remark}

\begin{remark}
Let $f:(C,\Pi_0)\to (B,\Lambda_0)$ be a Poisson $k$-homomorphism. Then we have the following commutative diagram
\begin{center}
$\begin{CD}
Hom_C(\Omega_{C/k}^1,C)@>[\Pi_0,-]>> Hom_C(\wedge^2 \Omega_{C/k}^1,C)@>[\Pi_0,-]>> Hom_C(\wedge^3 \Omega_{C/k}^1,C)@>[\Pi_0,-]>>\cdots\\
@Vf^* VV @Vf^* VV @Vf^* VV\\
Hom_C(\Omega_{C/k}^1,B)@>\pi>> Hom_C(\wedge^2 \Omega_{C/k}^1,B)@>\pi>> Hom_C(\wedge^3, \Omega_{C/k}^1,B)@>\pi>> \cdots
\end{CD}$
\end{center}
where $f^*$ is defined in the following way$:$ for $Q\in Hom_C(\wedge^q \Omega_{C/k}^1,C)$, and $a_1,...,a_q\in C$,
\begin{align*}
f^*(Q)(a_1,...,a_q):=f(Q(a_1,...,a_q))
\end{align*}
\end{remark}

\begin{remark}
Let $f:(C,\Pi_0)\to (B,\Lambda_0)$ and $g:(D,\Omega_0)\to (C, \Pi_0)$ be two Poisson $k$-homomorphisms. Let $h=f\circ g$. Then
\begin{center}
$\begin{CD}
Hom_C(\Omega_{C/k}^1,C)@>[\Pi_0,-]>> Hom_C(\wedge^2 \Omega_{C/k}^1,C)@>[\Pi_0,-]>> Hom_C(\wedge^3 \Omega_{C/k}^1,C)@>[\Pi_0,-]>>\cdots\\
@VG VV @VG VV @VG VV\\
Hom_D(\Omega_{D/k}^1,C)@>\pi_g>> Hom_D(\wedge^2 \Omega_{D/k}^1,C)@>\pi_g>> Hom_D(\wedge^3, \Omega_{D/k}^1,C)@>\pi_g>> \cdots
\end{CD}$
\end{center}
induces
\begin{center}
$\begin{CD}
Hom_C(\Omega_{C/k}^1,B)@>\pi_f>> Hom_C(\wedge^2 \Omega_{C/k}^1,B)@>\pi_f>> Hom_C(\wedge^3 \Lambda_{C/k}^1,B)@>\pi_f>>\cdots\\
@Vf^* G VV @Vf^* G VV @Vf^* G VV\\
Hom_D(\Omega_{D/k}^1,B)@>\pi_h>> Hom_D(\wedge^2 \Omega_{D/k}^1,B)@>\pi_h>> Hom_D(\wedge^3, \Omega_{D/k}^1,B)@>\pi_h>> \cdots
\end{CD}$
\end{center}
where $f^* G$ is defined in the following way$:$ for $P\in Hom_C(\wedge^q \Omega_{C/k}^1,B)$, and $b_1,...,b_q\in D$,
\begin{align*}
f^*G(P)(b_1,...,b_q)=P(g(b_1),...,g(b_q))
\end{align*}
\end{remark}

\begin{remark}
Let $f:(C,\Pi_0)\to (B,\Lambda_0)$ and $g:(D,\Omega_0)\to (C, \Pi_0)$ be two Poisson $k$-homomorphisms. Let $h=f\circ g$. We have the following commutative diagram.
\begin{center}
$\begin{CD}
Hom_D(\Omega_{D/k}^1,C)@>\pi_g>> Hom_D(\wedge^2 \Omega_{D/k}^1,C)@>\pi_g>> Hom_C(\wedge^3 \Omega_{D/k}^1,C)@>\pi_g>>\cdots\\
@Vf^* VV @Vf^* VV @Vf^* VV\\
Hom_D(\Omega_{D/k}^1,B)@>\pi_h>> Hom_D(\wedge^2 \Omega_{D/k}^1,B)@>\pi_h>> Hom_D(\wedge^3, \Omega_{D/k}^1,B)@>\pi_h>> \cdots
\end{CD}$
\end{center}
where $f^*$ is defined in the following way$:$ for $P\in Hom_D(\wedge^q\Omega_{D/k}^1,C)$, and $b_1,...,b_q\in D$,
\begin{align*}
f^*(P)(b_1,...,b_q)=f(P(b_1,...,b_q)).
\end{align*}
\end{remark}

The above arguments are translated in the context of a holomorphic Poisson map between two holomorphic Poisson maps and a Poisson map between two nonsingular Poisson varieties. For a nonsingular Poisson variety or a holomorphic Poisson manifold $X$, we denote the tangent sheaf of $X$ by $T_X$. We remark that we use the notation $\Theta_X$ for a holomorphic Poisson manifold $X$ instead of $T_X$ in the main body of the paper to keep notational consistency with \cite{Hor73} and \cite{Hor74}.

\begin{definition}
Let $f:(X,\Lambda_0)\to (Y, \Pi_0)$ be a holomorphic Poisson map between holomorphic Poisson manifolds or a Poisson map between nonsingular Poisson varieties. Then we have a complex of sheaves
\begin{align*}
f^* T_Y^\bullet :f^* T_Y\xrightarrow{\pi} \wedge^2 f^* T_Y \xrightarrow{\pi}\wedge^3 f^* T_Y\xrightarrow{\pi}\cdots
\end{align*}
We will denote the $i$-th hypercohomology group by $\mathbb{H}^i(X, f^* T_Y^\bullet)$.
\end{definition}

\begin{remark}
Let $f:(X,\Lambda_0)\to (Y,\Pi_0)$ be a holomorphic Poisson map between holomorphic Poisson manifolds or a Poisson map between nonsingular Poisson varieties. Then we have the following commutative diagram.
\begin{center}
$\begin{CD}
T_X^\bullet  :@. T_X@>[\Lambda_0,-]>> \wedge^2 T_X@>[\Lambda_0,-]>> \wedge^3 T_X@>[\Lambda_0,-]>> \cdots \\
@.@VFVV @VFVV @VFVV\\
f^*T_Y^\bullet : @. f^* T_Y@>\pi >> \wedge^2 f^* T_Y@>\pi>> \wedge^3 f^* T_Y@>\pi>>\cdots
\end{CD}$ 
\end{center}
This induces 
\begin{align*}
F:\mathbb{H}^i(X, T_X^\bullet)\to \mathbb{H}^i(X, f^* T_Y^\bullet).
\end{align*}
\end{remark}

\begin{remark}
Let $f:(X,\Lambda_0)\to (Y,\Pi_0)$ be a holomorphic Poisson map between holomorphic Poisson manifolds or a Poisson map between nonsingular Poisson varieties. Then we have the homomorphism
\begin{align*}
f^*:\mathbb{H}^i(Y, T_Y^\bullet)\to \mathbb{H}^i(X, f^* T_Y^\bullet).
\end{align*}

\end{remark}

\begin{remark}
Let $f:(X,\Lambda_0)\to (Y,\Pi_0)$ and $g:(Y,\Pi_0)\to (Z, \Omega_0)$ be two holomorphic Poisson maps between holomorphic Poisson manifolds or two Poisson maps between nonsingular Poisson varieties. Let $h=g\circ f$. Then the canonical homomorphism $G: T_Y^\bullet \to g^* T_Z^\bullet$ induces $f^*G:f^* T_Y^\bullet \to h^* T_Z^\bullet$ so that we have a homomorphism
\begin{align*}
f^* G:\mathbb{H}^i(X, f^*T_Y^\bullet)\to \mathbb{H}^i(X, h^* T_Z^\bullet).
\end{align*}
\end{remark}

\begin{remark}
Let $f:(X,\Lambda_0)\to (Y,\Pi_0)$ and $g:(Y,\Pi_0)\to (Z, \Omega_0)$ be two holomorphic Poisson maps between holomorphic Poisson manifolds or two Poisson maps between nonsingular Poisson varieties. Let $h=g\circ f$. Then we have the natural homomorphism 
\begin{align*}
f^*:\mathbb{H}^i(Y, g^* T_Z^\bullet)\to \mathbb{H}^i(X, h^* T_Z^\bullet).
\end{align*}
\end{remark}

\section{Deformations of Poisson morphisms}\label{appendix2}

We denote by $\bold{Art}$ the category of local artinian $k$-algebras with residue field $k$, where $k$ is a algebraically closed filed with characteristic $0$.

\begin{definition}[general definition]\label{ll30}
Let $f:(X,\Lambda_0)\to (Y,\Pi_0)$ be a Poisson morphism of nonsingular Poisson varieties. Let $A\in \bold{Art}$. An infinitesimal deformation of $f$ over $Spec(A)$ is a cartesian diagram
\begin{center}
$\xi:\begin{CD}
(X,\Lambda_0)@>>> (\mathcal{X},\Lambda)\\
@VfVV @VV\Psi V\\
(Y,\Pi_0)@>>>(\mathcal{Y},\Pi)\\
@VVV @VVq V\\
Spec(k)@>>> Spec(A)
\end{CD}$
\end{center}
where $(\mathcal{X},\Lambda)$, $(\mathcal{Y},\Pi)$ are Poisson schemes over $Spec(A)$ and $\Psi$ is a morphism of Poisson schemes over $Spec(A)$ via $q$, $\Psi$ and $ q\circ \Psi$ are flat such that $(\mathcal{X},\Lambda)\times_{Spec(A)} Spec(k)\cong (X,\Lambda_0)$ via $q\circ \Psi$ and $(\mathcal{Y},\Pi)\times_{Spec(A)} Spec(k)=(Y,\Pi_0)$ via $q$ so that $(\mathcal{X},\Lambda)$ is a flat Poisson deformation of $(X,\Lambda_0)$ over $Spec(A)$, and $(\mathcal{Y},\Pi)$ is a flat Poisson deformation of $(Y,\Pi_0)$ over $Spec(A)$.

Let $\xi'$ be another infinitesimal deformation of $f$ which is given by $\Psi':(\mathcal{X}',\Lambda)\to (\mathcal{Y}',\Lambda)$. Then we say that $\xi$ is isomorphic to $\xi'$ if there exist Poisson isomorphisms $r:(\mathcal{X},\Lambda)\to (\mathcal{X}',\Lambda')$ and $s:(\mathcal{Y},\Pi)\to (\mathcal{Y}',\Pi')$ such that the following diagram commutes.
\begin{center}
$\begin{CD}
(\mathcal{X},\Lambda)@>r>> (\mathcal{X}',\Lambda')\\
@V\Psi VV @VV\Psi' V\\
(\mathcal{Y},\Pi)@>s>> (\mathcal{Y}',\Pi')
\end{CD}$
\end{center}
Then we can define a functor of Artin rings of deformations of $f$:
\begin{align*}
Def_f:\bold{Art}&\to (Sets)\\
A &\mapsto \{\text{infinitesimal deformations of $f$ over $A$}\}/\text{isomorphisms}
\end{align*}
\end{definition}

\begin{definition}
Given an infinitesimal deformation $\xi$ of $f:(X,\Lambda_0)\to (Y,\Pi_0)$ over $Spec(A)$ as in Definition $\ref{ll30}$ and a small extension $e:0\to (t)\to \tilde{A}\to A\to 0$ in $\bold{Art}$, a lifting of $\xi$ to $\tilde{A}$ is an infinitesimal deformation $\tilde{\Psi}:(\tilde{\mathcal{X}},\tilde{\Lambda})\to (\tilde{\mathcal{Y}},\tilde{\Pi})$ of $f$ over $Spec(\tilde{A})$ which induces $\Psi$ so that we have the following cartesian diagram
\begin{center}
$\begin{CD}
(X,\Lambda_0)@>>> (\mathcal{X},\Lambda)@>>> (\tilde{\mathcal{X}},\tilde{\Lambda})\\
@VVV @V\Psi VV @V\tilde{\Psi}VV\\
(Y,\Pi_0)@>>> (\mathcal{Y},\Pi)@>>> (\tilde{\mathcal{Y}},\tilde{\Pi})\\
@VVV @VqVV @V\tilde{q}VV\\
Spec(k) @>>> Spec(A)@>>> Spec(\tilde{A})
\end{CD}$
\end{center}
\end{definition}

\subsection{Deformations of a Poisson morphism leaving domain and target fixed}\

Let $f:(X,\Lambda_0)\to (Y,\Lambda_0)$ be a Poisson morphism of nonsingular Poisson varieties.  An infinitesimal deformation $\xi$ of $f$ over $A\in \bold{Art}$ leaving domain and target fixed is a cartesian diagram of the form
\begin{center}
$\xi:\begin{CD}
(X,\Lambda_0)@>>> (\mathcal{X}=X\times Spec(A),\Lambda_0)\\
@VfVV @VV\Psi V\\
(Y,\Pi_0)@>>>(\mathcal{Y}=Y\times Spec(A),\Pi_0)\\
@VVV @VVpr_2V\\
Spec(k)@>>> Spec(A)
\end{CD}$
\end{center}
where $\Psi$ is a Poisson morphism inducing $f$ and $pr_2$ is the projection. We note that $(\mathcal{X},\Lambda_0)$ and $(\mathcal{Y},\Pi_0)$ is a trivial Poisson deformation of $(X,\Lambda_0)$ and $(Y,\Pi_0)$, respectively in Definition \ref{ll30}. Then we can define a functor of Artin rings which is a subfunctor of $Def_f$ 
\begin{align*}
Def_{(X,\Lambda_0)/f/(Y,\Pi_0)}:\bold{Art}&\to (Sets)\\
A &\mapsto\{\text{infinitesimal deformations of $f$ over $A$ with fixed domain and target}\}
\end{align*}

\begin{remark}[Notation]\label{notation12}
Let $f:C\to B$ be a $k$-homomorphism. Let $P\in Hom_C(\Omega_{C/k}^1,C)$. We define $f P\in Hom_C(\wedge^2, \Omega_{C/k}^1,B)$ by the rule $P(a)=f(P(a))$ for $a\in C$.

\end{remark}

\begin{remark}[Notation]\label{notation11}
Let $f:C\to B$ be a $k$-homomorphism. Let $\Pi_0\in Hom_C(\wedge^2 \Omega_{C/k}^1,C)$. We define $f^* \Pi_0\in Hom_C(\wedge^2 \Omega_{C/k}^1,B)$ by the rule $f^*\Pi_0(a,b):=f(\Pi_0(a,b))$ for $a,b\in C$. On the other hand, let $\Lambda_0\in Hom_B(\wedge^2 \Omega_{B/k}^1, B)$. We define $f \Lambda_0\in Hom_C(\wedge^2 \Omega_{C/k}^1,B)$ by the rule $f\Lambda_0(a,b):=\Lambda_0(f(a),f(b))$. We note that when $f^*\Pi_0=f\Lambda_0$, $f:(C,\Pi_0)\to (B, \Lambda_0)$ is biderivation-preserving, in other words, $f(\Pi_0(a,b))=\Lambda_0(f(a),f(b))$.
\end{remark}

\begin{proposition}[compare \cite{Ser06} Proposition 3.4.2 p.158]
$f:(X,\Lambda_0)\to (Y,\Pi_0)$ be a Poisson morphism of nonsingular Poisson varieties. Then
\begin{enumerate}
\item There is a natural identification 
\begin{align*}
Def_{(X,\Lambda_0)/f/(Y,\Pi_0)}(k[\epsilon])\cong \mathbb{H}^0(X,f^*T_Y^\bullet)
\end{align*}
\item Given an infinitesimal deformation of $f$ over $A\in \bold{Art} $ leaving domain and target fixed, and a small extension $e:0\to (t)\to \tilde{A}\to A\to 0$, we can associate an element $o_\xi(e)\in \mathbb{H}^1(X, f^* T_Y^\bullet)$, which is zero if and only if there is a lifting of $\xi$ to $\tilde{A}$.
\end{enumerate}
\end{proposition}

\begin{proof}
Consider a first-order deformation $\Psi:(X\times Spec(k[\epsilon]),\Lambda_0)\to (Y\times Spec(k[\epsilon]),\Pi_0)$ of $f$. Choose an affine open covering $\mathcal{U}=\{U_i=Spec(B_i)\}_{i\in I}$ of $X$ such that for each $i$, $f(U_i)$ is contained in an affine open set $V_i=Spec(C_i)$ of $Y$ and $\Psi$ is locally given as
\begin{align*}
\Psi_i:(U_i\times Spec(k[\epsilon]), \Lambda_0)\to (V_i\times Spec(k[\epsilon]),\Pi_0)
\end{align*}
which corresponds
\begin{align*}
\Phi_i:(C_i\otimes_k k[\epsilon],\Pi_0)  \to (B_i\otimes_k k[\epsilon], \Lambda_0)
\end{align*}
which induces $f|_{U_i}$ by modulo $\epsilon$ so that we have an element $v_i\in Der_k(C_i,B_i)=\Gamma(U_i, f^*T_Y)$ such that $\Phi_i=f+\epsilon v_i$. Since $v_i=v_j$ on $U_{ij}$, $v=\{v_i\}$ defines an element in $H^0(X, f^* T_Y)$.  On the other hand, we claim that $\pi(v_i)=0$. Indeed, since $\Phi_i=f+\epsilon v_i$, we have, for any $x,y\in C_i$,
\begin{align*}
&\Phi_i(\Pi_0(x,y))=\Lambda_0(\Phi_ix, \Phi_iy)\\
&f(\Pi_0(x,y))+tv_i(\Pi_0(x,y))=\Lambda_0(fx+tv_ix,fy+tv_iy))\\
&f(\Pi_0(x,y))+tv_i(\Pi_0(x,y))=\Lambda_0(fx,fy)+t\Lambda_0(v_ix,fy)+t\Lambda_0(fx,v_iy)
\end{align*}
so that we have $v_i(\Pi_0(x,y))-\Lambda_0(v_ix,fy)-\Lambda_0(fx,v_iy)=0$ which is equivalent to $\pi(v_i)=0$. Hence $v=\{v_i\}$ defines an element in $\mathbb{H}^0(X, f^* T_Y^\bullet)$ in the following \u{C}ech resolution of $f^* T_Y^\bullet$.
\begin{center}
$\begin{CD}
C^0(\mathcal{U}, \wedge^3 f^* T_Y)\\
@A\pi AA \\
C^0(\mathcal{U}, \wedge^2 f^* T_Y)@>\delta>> C^1(\mathcal{U}, \wedge^2 f^* T_Y)\\
@A\pi AA @A\pi AA\\
C^0(\mathcal{U},f^* T_Y) @>-\delta >> C^1(\mathcal{U}, f^* T_Y)@>\delta >> C^2(\mathcal{U}, f^* T_Y)
\end{CD}$
\end{center}

Next we identify obstructions. Consider a small extension $e:0\to (t)\to \tilde{A}\to A\to 0$, and an infinitesimal deformation $\xi=(\Psi:(X\times Spec(A),\Lambda_0)\to (Y\times Spec(A),\Pi_0))$ of $f$ over $Spec(A)$ with the domain and the target fixed. Choose an affine open cover $\mathcal{U}=\{U_i=Spec(B_i)\}_{i\in I}$ of $X$ such that for each $i$, $f(U_i)$ is contained in an affine open set $V_i=Spec(C_i)$ of $Y$, $\Psi$ is locally represented as
\begin{align*}
\Psi_i:(U_i\times Spec(A),\Lambda_0) \to (V_i\times Spec(A),\Pi_0)
\end{align*}
which corresponds to
\begin{align*}
\Phi_i:(C_i\otimes_k A,\Pi_0)\to (B_i\otimes_k A,\Lambda_0)
\end{align*}
Let $\tilde{\Phi}_i$ be a lifting of $\Phi_i$ to $\tilde{A}$ inducing $\Phi_i$
\begin{align*}
\tilde{\Phi}_i:(C_i\otimes_k \tilde{A},\Pi_0)\to (B_i\otimes_k \tilde{A},\Lambda_0)
\end{align*}
Then since $\Phi_i=\Phi_j$ on $U_{ij}$, we have $\tilde{\Phi}_j- \tilde{\Phi}_i=t\xi_{ij}$ for some $\xi_{ij}\in \Gamma(U_{ij}, f^* T_Y)$. Then 
\begin{align*}
t(\xi_{ij}-\xi_{ik}+\xi_{jk})=\tilde{\Phi}_j-\tilde{\Phi}_i-\tilde{\Phi}_k+\tilde{\Phi}_i+\tilde{\Phi}_k-\tilde{\Phi}_j=0
\end{align*}
Since $\Phi_i^*\Pi_0-\Phi_i \Lambda_0=0$, we have $\tilde{\Phi}_{i}^*\Pi_0-\tilde{\Phi}_i \Lambda_0=tP_i$ for some $P_i\in Hom_{C_i} (\wedge^2 \Omega_{C_i}^1,B_i)$. We recall Notation \ref{notation11}. Then for any $a,b,c\in C_i$, since $[\Lambda_0,\Lambda_0]=0$, and $[\Pi_0,\Pi_0]=0$, we have
\begin{align*}
&t\pi(P_i)(a,b,c)\\
&=t\Lambda_0(P_i(a,b), f(c))-t\Lambda_0(P_i(a,c),f(b))+t\Lambda_0(P_i(b,c),f(a))
+tP_i(\Pi_0(a, b),c)-tP_i(\Pi_0(a, c),b)+tP_i(\Pi_0(b, c), a)\\
&=\Lambda_0((\tilde{\Phi}_{i}^*\Pi_0-\tilde{\Phi}_i \Lambda_0)(a,b),\tilde{\Phi}_i(c))-\Lambda_0((\tilde{\Phi}_{i}^*\Pi_0-\tilde{\Phi}_i \Lambda_0)(a,c),\tilde{\Phi}(b))+\Lambda_0((\tilde{\Phi}_i^*\Pi_0-\tilde{\Phi}_i \Lambda_0)(b,c),\tilde{\Phi}_i(a))\\
&+(\tilde{\Phi}_{i}^*\Pi_0-\tilde{\Phi}_i \Lambda_0)(\Pi_0(a,b),c)-(\tilde{\Phi}_{i}^*\Pi_0-\tilde{\Phi}_i \Lambda_0)(\Pi_0(a,c),b)+(\tilde{\Phi}_{i}^*\Pi_0-\tilde{\Phi}_i \Lambda_0)(\Pi_0(b,c),a)\\
&=\Lambda_0(\tilde{\Phi}_i(\Pi_0(a,b)),\tilde{\Phi}_i(c))-\Lambda_0(\Lambda_0(\tilde{\Phi}_i(a),\tilde{\Phi}_i(b)),\tilde{\Phi}_i(c))-\Lambda_0(\tilde{\Phi}_i(\Pi_0(a,c)),\tilde{\Phi}_i(b))+\Lambda_0(\Lambda_0(\tilde{\Phi}_i(a),\tilde{\Phi}_i(c)),\tilde{\Phi}_i(b))\\
&+\Lambda_0(\tilde{\Phi}_i(\Pi_0(b,c)),\tilde{\Phi}_i(a))-\Lambda_0(\Lambda_0(\tilde{\Phi}_i(b),\tilde{\Phi}_i(c)),\tilde{\Phi}_i(a))\\
&+\tilde{\Phi}_i(\Pi_0(\Pi_0(a,b),c))-\Lambda_0(\tilde{\Phi}_i(\Pi_0(a,b)),\tilde{\Phi}_i(c))-\tilde{\Phi}_i(\Pi_0(\Pi_0(a,c),b))+\Lambda_0(\tilde{\Phi}_i(\Pi_0(a,c)),\tilde{\Phi}_i(b))\\
&+\tilde{\Phi}_i(\Pi_0(\Pi_0(b,c),a))-\Lambda_0(\tilde{\Phi}_i(\Pi_0(b,c)),\tilde{\Phi}_i(a))=0
\end{align*}
Hence we have 
\begin{align}
\pi(P_i)=0
\end{align}
On the other hand, for $a,b\in \Gamma(V_{ij}, \mathcal{O}_Y)$,
\begin{align*}
&t(P_i-P_j)(a,b)=\tilde{\Phi}_{i}^*\Pi_0(a,b)-\tilde{\Phi}_i\Lambda_0(a,b)-\tilde{\Phi}_j^*\Pi_0(a,b)+\tilde{\Phi}_j\Lambda_0(a,b)\\
&=\tilde{\Phi}_i(\Pi_0(a,b))-\Lambda_0(\tilde{\Phi}_i(a),\tilde{\Phi}_i(b))-\tilde{\Phi}_j(\Pi_0(a,b))+\Lambda_0(\tilde{\Phi}_j(a),\tilde{\Phi}_j(b))\\
&=\tilde{\Phi}_i(\Pi_0(a,b))-\Lambda_0(\tilde{\Phi}_i(a),\tilde{\Phi}_i(b))-\tilde{\Phi}_j(\Pi_0(a,b))+\Lambda_0(\tilde{\Phi}_i(a)+t\xi_{ij}(a),\tilde{\Phi}_i(b)+t\xi_{ij}(b))\\
&=\tilde{\Phi}_i(\Pi_0(a,b))-\tilde{\Phi}_j(\Pi_0(a,b))+\Lambda_0(t\xi_{ij}(a),\tilde{\Phi}_i(b))+\Lambda_0(\tilde{\Phi}_i(a),t\xi_{ij}(b))\\
&=-t\xi_{ij}(\Pi_0(a,b))+t\Lambda_0(\xi_{ij}(a),f(b))+t\Lambda_0(f(a),\xi_{ij}(b))
\end{align*}
Hence we have
\begin{align}
P_j-P_i+\pi(\xi_{ij})=0
\end{align}
From, $\alpha:=(\{P_i\},\{\xi_{ij}\})\in C^0(\mathcal{U}, \wedge^2 f^* T_Y)\oplus C^1(\mathcal{U}, f^* T_Y)$ defines an element in $\mathbb{H}^1(X, f^* T_Y^\bullet)$.

Let $\tilde{\Phi}_i'$ be an another arbitrary lifting of $\Phi_i$. Let $\beta:=(\{P_i'\},\{\xi_{ij}'\})$ be the associated $1$-cocycle. Then $\tilde{\Phi}_i'-\tilde{\Phi}_i=t Q_i$ for some $Q_i\in Hom_{C_i}(\wedge^2 \Omega_{C_i}^1, B_i)$. Then
\begin{align*}
(tP_i'-tP_i)(a,b)&= \tilde{\Phi}'^*_{i}\Pi_0(a,b)-\tilde{\Phi}_i'\Lambda_0(a,b)-\Phi_{i}^*\Pi(a,b)+\tilde{\Phi}_i'\Lambda_(a,b)\\
&=\tilde{\Phi}'_i(\Pi_0(a,b))-\Lambda_0(\tilde{\Phi}_i'(a),\tilde{\Phi}_i'(b))-\tilde{\Phi}_i(\Pi_0(a,b))+\Lambda_0(\tilde{\Phi}_i(a),\tilde{\Phi}_i(b))\\
&=tQ_i(\Pi_0(a,b))-\Lambda_0(f(a), tQ_i(b))-\Lambda_0(tQ_i(a),f(b))=-\pi(Q_i)(a,b)
\end{align*}
Hence we have
\begin{align}
P_i'-P_i=\pi(-Q_i)
\end{align}
On the other hand, $t\xi_{ij}'-t\xi_{ij}=\tilde{\Phi}_j'-\tilde{\Phi}_i'-\tilde{\Phi}_j+\tilde{\Phi}_i=tQ_j-tQ_i$. Hence we have
\begin{align}
\xi_{ij}'-\xi_{ij}=Q_j-Q_i
\end{align}
Then $\alpha-\beta=(\pi(\{Q_i\},-\delta(\{Q_i\}))$ so that $\alpha$ is cohomologous to $\beta$. Hence given a small extension $e:0\to (t)\to \tilde{A}\to A\to 0$, we can associate an element $o_\xi(e):=$ the cohomology class of $\alpha\in \mathbb{H}^0(X, f^* T_Y^\bullet)$. We note that $o_\xi(e)=0$ if and only if there exists collections $P_i=0$ and $\xi_{ij}=0$ if and only if there is a Poisson morphism $\tilde{\Psi}:(X\times Spec(\tilde{A}),\Lambda_0)\to (Y\times Spec(\tilde{A}),\Pi_0)$ inducing $\Psi$ if and only if there is a lifting of $\xi$ to $\tilde{A}$.
\end{proof}

\subsection{Deformations of a Poisson morphism leaving the target fixed}\

Let $f:(X,\Lambda_0)\to (Y,\Pi_0)$ be a Poisson morphism of nonsingular Poisson varieties. An infinitesimal deformation of $f$ over $A\in \bold{Art}$ leaving the target fixed is a cartesian diagram of the form
\begin{center}
$\xi:\begin{CD}
(X,\Lambda_0)@>>> (\mathcal{X},\Lambda)\\
@VfVV @VV\Psi V\\
(Y,\Pi_0)@>>>(\mathcal{Y}=Y\times Spec(A),\Pi_0)\\
@VVV @VVs V\\
Spec(k)@>>> Spec(A)
\end{CD}$
\end{center}
We note that $(\mathcal{Y},\Pi_0)$ is a trivial Poisson deformation of $(Y,\Lambda_0)$ in Definition \ref{ll30}. Then we can define a functor of Artin rings which is a subfunctor of $Def_f$
\begin{align*}
Def_{f/(Y,\Pi_0)}:\bold{Art}&\to (Sets)\\
A &\mapsto \{\text{infinitesimal deformations of $f$ over $A$ with leaving the target fixed}\}/\text{isomorphisms}
\end{align*}

Let $f:(X,\Lambda_0)\to (Y,\Pi_0)$ be a Poisson morphism of nonsingular Poisson varieties. Let us consider an exact sequence of a complex of coherent sheaves on $X$
\begin{center}
$\begin{CD}
@. \cdots @.\cdots @. \cdots @. \cdots\\
@. @AAA @A[\Lambda_0,-]AA @A\pi AA @AAA@.\\
0@>>> T_{X/Y}^3@>J>> \wedge^3 T_X@>F>> \wedge^3 f^* T_Y@>P>> N_f^2@>>> 0\\
@.@AAA @A[\Lambda_0,-]AA @A\pi AA @AAA @.\\
0@>>> T_{X/Y}^2@>J>>\wedge^2 T_X @>F>>\wedge^2 f^* T_Y @>P>> N_f^1@>>>0\\
@. @AAA @A[\Lambda_0,-]AA @A\pi AA @AAA @.\\
0@>>> T_{X/Y}^1@>J>> T_X @>F>> f^*T_Y@>P>> N_f@>>> 0
\end{CD}$
\end{center}
where $T_{X/Y}^\bullet:=Ker(T_X^\bullet \xrightarrow{F} f^*T_Y^\bullet)$ and $N_f^\bullet:=Coker(T_X^\bullet \xrightarrow{F} f^* T_Y^\bullet)$. $f$ is called non-degenerate when $F$ is injective so that $T_{X/Y}^\bullet=0$. When $f$ is smooth, $F$ is surjective so that $\mathcal{N}_f^\bullet=0$.

\begin{definition}
Let $f:(X,\Lambda_0)\to (Y,\Pi_0)$ be a Poisson morphism between two nonsingular projective Poisson varieties. Let $\mathcal{U}=\{U_i\}_{i\in I}$ be an affine open cover of $X$ and define
\begin{align*}
PD_{(X,\Lambda_0)/(Y,\Pi_0)}=\frac{\{(v,t,\lambda)\in C^0(\mathcal{U}, f^*T_Y)\oplus C^1(\mathcal{U},T_X)\oplus C^0(\mathcal{U},\wedge^2 T_X)|\frac{-\delta v=F(t), \pi(v)=F(\lambda),}{  \delta t=0, \delta \lambda+[\Lambda_0,t]=0,[\Lambda_0,\lambda]=0  } \}}{\{ (F(w),-\delta(w),[\Lambda_0,w])|w\in C^0(\mathcal{U},T_X)\}}
\end{align*}
\small{\begin{align*}
&PD_{(X,\Lambda_0)/(Y,\Pi_0)}^1=\\
&\frac{\{(\xi, \eta, s, r,w)\in C^1(\mathcal{U},f^*T_Y)\oplus C^0(\mathcal{U},\wedge^2 f^*T_Y)\oplus C^2(\mathcal{U},T_X)\oplus C^1(\mathcal{U},\wedge^2 T_X)\oplus C^0(\mathcal{U},\wedge^3 T_X)| \frac{\delta \xi =F(s),\delta \eta+\pi(\xi)=F( r),\pi(\eta)=F(w),}{ [\Lambda_0,w]=0, -\delta(w)+[\Lambda_0,r]=0, -\delta(r)+[\Lambda_0,s]=0,\delta s=0}\}}{\{(F(u)-\delta(a),F(c)+\pi(a),\delta(u),\delta(c)+[\Lambda_0,u] ,[\Lambda_0,c] ) |a\in C^0(\mathcal{U},f^* T_Y),u\in C^1(\mathcal{U},T_X), c\in C^0(\mathcal{U},\wedge^2 T_X)\}}
\end{align*}}
\end{definition}

\begin{lemma}\label{ll31}\

\begin{enumerate}
\item $PD_{(X,\Lambda_0)/(Y,\Pi_0)}$ and $PD_{(X,\Lambda_0)/(Y,\Pi_0)}^1$ do not depend on the choice of the affine cover $\mathcal{U}$ of $X$.
\item We have the following exact sequences
\begin{enumerate}
\item $\mathbb{H}^0(X,T_X^\bullet)\to \mathbb{H}^0(X, f^* T_Y^\bullet)\to PD_{(X,\Lambda_0)/(Y,\Pi_0)}\to \mathbb{H}^1(X,T_X^\bullet)\to \mathbb{H}^1(X, f^* T_Y^\bullet)$.
\item $0\to \mathbb{H}^1(X, T_{X/Y}^\bullet)\to PD_{(X,\Lambda_0)/(Y,\Pi_0)}\to \mathbb{H}^0(X,\mathcal{N}_f^\bullet)\to \mathbb{H}^2(X, T_{X/Y}^\bullet)$.
\item $\mathbb{H}^1(X,T_X^\bullet)\to \mathbb{H}^1(X, f^*T_Y^\bullet)\to PD_{(X,\Lambda_0)/(Y,\Pi_0)}^1\to \mathbb{H}^2(X, T_X^\bullet)\to \mathbb{H}^2(X,f^*T_Y^\bullet)$.
\item $0\to \mathbb{H}^2(X, T_{X/Y}^\bullet)\to PD_{(X,\Lambda_0)/(Y,\Pi_0)}^1\to \mathbb{H}^1(X, \mathcal{N}_f^\bullet)\to \mathbb{H}^3(X, T_{X/Y}^\bullet)$.
\end{enumerate}
\item If $f$ is non-degenerate, then $PD_{(X,\Lambda_0)/(Y,\Pi_0)}\cong \mathbb{H}^0(X, \mathcal{N}_f^\bullet)$ and $PD_{X/Y}^1\cong \mathbb{H}^1(X, \mathcal{N}_f^\bullet)$.
\item if $f$ is smooth, $PD_{(X,\Lambda_0)/(Y,\Pi_0)}\cong \mathbb{H}^1(X, T_{X/Y}^\bullet)$ and $PD_{(X,\Lambda_0)/(Y,\Pi_0)}^1\cong \mathbb{H}^2(X, T_{X/Y}^\bullet)$.
\end{enumerate}
\end{lemma}

\begin{proof}

By five lemma, $(2)$ implies $(1)$. $(2)$ implies $(3)$ since $T_{X/Y}^\bullet=0$. $(2)$ implies $(4)$ since $N_f^\bullet=0$. Hence it is sufficient to show $(2)$. 

Let us show that $(a)$ is exact. We define $\mathbb{H}^0(X, f^* T_Y^\bullet)\to PD_{(X,\Lambda_0)/(Y,\Pi_0)}$ by $u\mapsto (u,0,0)$ and define $PD_{(X,\Lambda_0)/(Y,\Pi_0)}\to \mathbb{H}^1(X, T_X^\bullet)$ by $(v,t,\lambda)\mapsto (t,\lambda)$. Then we can check that $(a)$ is exact.

Let us show that $(b)$ is exact. We define $\mathbb{H}^1(X, T_{X/Y}^\bullet)\to PD_{(X,\Lambda_0)/(Y,\Pi_0)}$ by
\begin{align*}
(\psi,\phi)\in C^1(\mathcal{U}, T_{X/Y}^1)\oplus C^0(\mathcal{U}, T_{X/Y}^2)\mapsto (0,J\psi, J\phi)\in PD_{(X,\Lambda_0)/(Y,\Pi_0)}
\end{align*}
We define $PD_{(X,\Lambda_0)/(Y,\Pi_0)}\to \mathbb{H}^0(X, \mathcal{N}_f^\bullet)$ by $(v,t,\lambda)\to Pv$. Let us define $\mathbb{H}^0(X,\mathcal{N}_f^\bullet)\to\mathbb{H}^2(X, T_{X/Y}^\bullet)$. Every element of $\mathbb{H}^0(X, \mathcal{N}_f^\bullet)$ is represented by some $v\in C^0(\mathcal{U}, f^*T_Y)$ such that $-\delta v=F(t)$ for some $t\in C^1(\mathcal{U}, T_X)$ and $\pi(v)=F(r)$ for some $r\in C^0(\mathcal{U},\wedge^2 T_X)$. Then $\delta t\in C^1(\mathcal{U}, T_{X/Y}^2)$ since $F(\delta t)=\delta F(t)=-\delta\delta v=0$, $[\Lambda_0,t]+\delta r\in \mathcal{C}^1(\mathcal{U}, T_{X/Y}^2)$ since $F([\Lambda,t]_0+\delta r)=\pi(F(t))+\delta F(r)=-\pi(\delta v)+\delta(\pi(v))=0$, and $[\Lambda_0, r]\in C^0(\mathcal{U}, T_{X/Y}^1)$ since $F([\Lambda_0,r])=\pi(F(r))=\pi(\pi(r))=0$. Then we define $\mathbb{H}^0(X, \mathcal{N}_f^\bullet)\to \mathbb{H}^2(X, T_{X/Y}^\bullet)$ by $v\mapsto (\delta t,[\Lambda_0,t]+\delta r, [\Lambda_0,r])$ which define a $2$-cocycle of $T_{X/Y}^\bullet$. Then we can check that $(b)$ is exact.

Let us show that $(c)$ is exact. We define $\mathbb{H}^1(X, f^* T_Y^\bullet)\to PD^1_{(X,\Lambda_0)/(Y,\Pi_0)}$ by $(\xi, \eta)\mapsto(\xi,\eta,0,0)$ and define $PD^1_{(X,\Lambda_0)/(Y,\Pi_0)}\to \mathbb{H}^2(X, T_X^\bullet)$ by $(\xi,\eta,s,r,w)\mapsto (s,r,w)$. Then we can check that $(c)$ is exact.

Let us show that $(d)$ is exact. We define $\mathbb{H}^2(X, T_{X/Y}^\bullet)\to PD^1_{(X,\Lambda_0)/(Y,\Pi_0)}$ by 
\begin{align*}
(\psi, \phi, \theta)\in C^2(\mathcal{U}, T_{X/Y}^1)\oplus C^1(\mathcal{U}, T_{X/Y}^2)\oplus C^0(\mathcal{U}, T_{X/Y}^3)\mapsto (0,0, J\psi,J\phi,J\theta)\in PD^1_{(X,\Lambda_0)/(Y,\Pi_0)}.
\end{align*}
We define $PD^1_{(X,\Lambda_0)/(Y,\Pi_0)}\to \mathbb{H}^1(X, \mathcal{N}_f^\bullet)$ by $(\xi,\eta, s, r, w)\mapsto (P\xi, P \eta)$. Let us define $\mathbb{H}^1(X, \mathcal{N}_f^\bullet)\to\mathbb{H}^3(X, T_{X/Y}^\bullet)$. Every element of $\mathbb{H}^1(X, \mathcal{N}_f^\bullet)$ is represented by some $(\xi, \eta)\in C^1(\mathcal{U}, f^*T_Y)\oplus C^0(\mathcal{U}, \wedge^2 f^*T_Y)$ such that $\delta \xi=F(a), \pi(\xi)+\delta(\eta)=F(b)$, and $\pi(\eta)=F(c)$ for some $a\in C^2(\mathcal{U}, T_X), b\in C^1(\mathcal{U}, \wedge^2 T_X)$, and $c\in C^0(\mathcal{U}, \wedge^3 T_X)$. Then $-\delta a\in C^3(\mathcal{U}, T_{X/Y}^1)$ since $F(-\delta a)=-\delta \delta \xi=0$, $[\Lambda_0,a]-\delta(b)\in C^2(\mathcal{U}, T_{X/Y}^2)$ since $F([\Lambda_0,a]-\delta(b))=\pi(F(a))-\delta(F(b))=\pi(\delta(\xi))-\delta(\pi(\xi))-\delta(\delta (\eta))=0$, $[\Lambda_0,b]-\delta(c)\in C^1(\mathcal{U}, T_{X/Y}^3)$ since $F([\Lambda_0,b]-\delta(c))=\pi(F(b))-\delta(F(c))=\pi(\delta(\eta))-\delta(\pi(\eta))=0$, and $[\Lambda_0,c]\in C^0(\mathcal{U}, T_{X/Y}^4)$ since $F([\Lambda_0,c])=\pi(F(c))=\pi(\pi(\eta))=0$. Then we define $\mathbb{H}^1(X, \mathcal{N}_f^\bullet)\to\mathbb{H}^3(X, T_{X/Y}^\bullet)$ by $(\eta,\xi)\mapsto(-\delta a,[\Lambda_0,a]-\delta(b),[\Lambda_0,b]-\delta(c),[\Lambda_0,c])$ which is $3$-cocycle of $T_{X/Y}^\bullet$. Then we can check that $(d)$ is exact.
\end{proof}

\begin{theorem}[compare \cite{Ser06} Theorem 3.4.8 p.164]
Let $f:(X,\Lambda_0)\to (Y, \Pi_0)$ be a Poisson morphism of nonsingular Poisson varieties. Then
\begin{enumerate}
\item There is a natural identification
\begin{align*}
Def_{f/(Y,\Pi_0)}(k[\epsilon])\cong PD_{(X,\Lambda_0)/(Y,\Pi_0)}
\end{align*}
\item Given an infinitesimal deformation $\xi$ of $f$ over $A\in \bold{Art}$ leaving the target fixed and a small extension $0\to (t)\to \tilde{A}\to A\to 0$, we can associate an element $o_\xi(e)\in PD^1_{(X,\Lambda_0)/(Y,\Pi_0)}$, which is zero if and only if there is a lifting of $\xi$ to $\tilde{A}$.
\end{enumerate}
\end{theorem}

\begin{proof}
Consider a first-order deformation and induced first-order deformation of $(X,\Lambda_0)$,
\begin{center}
$\begin{CD}
(X,\Lambda_0)@>>> (\mathcal{X}, \Lambda)\\
@VfVV @VV\Psi V\\
(Y,\Pi_0)@>>> (Y\times Spec(k[\epsilon]),\Pi_0)\\
@VVV @VVs V\\
Spec(k)@>>> Spec(k[\epsilon])
\end{CD}
\,\,\,\,\,\,\,\,\,\,
\begin{CD}
(X,\Lambda_0)@>>> (\mathcal{X},\Lambda)\\
@VVV @VVs \Psi V\\
Spec(k)@>>> Spec(k[\epsilon])
\end{CD}$
\end{center}
Choose an affine open cover $\mathcal{U}=\{U_i=Spec(B_i)\}_{i\in I}$ of $X$ such that for each $i$, $f(U_i)$ is contained in an affine open set $V_i=Spec(C_i)$ of  $Y$, $\Psi$ is locally represented as
\begin{align*}
\Psi:(\mathcal{X}|_{U_i},\Lambda|_{U_i})\to (V_i\times Spec(k[\epsilon]),\Pi_0)
\end{align*}
and we have a Poisson isomorphism $\theta_i:(U_i\times Spec(k[\epsilon]),\Lambda_0+\epsilon \Lambda_i)\to (\mathcal{X}|_{U_i},\Lambda|_{U_i})$ for some $\Lambda_i\in \Gamma(U_i, \wedge^2 T_X)$. Then for each $i,j\in I$, $\theta_{ij}:=\theta_j^{-1}\theta_i:(U_{ij}\times Spec(k),\Lambda_0+\epsilon \Lambda_j)\to (U_{ij}\times Spec(k),\Lambda_0+\epsilon \Lambda_i)$ is a Poisson isomorphism inducing the identity on $(U_{ij},\Lambda_0)$ modulo by $\epsilon$. Then $\theta_{ij}$ corresponds to $Id+\epsilon p_{ij}:(\Gamma(U_{ij},\mathcal{O}_X)\otimes k[\epsilon], \Lambda_0+\epsilon \Lambda_j)\to (\Gamma(U_{ij},\mathcal{O}_X)\otimes k[\epsilon],\Lambda_0+\epsilon \Lambda_i)$ where $p_{ij}\in \Gamma(U_i, T_X)$ and $p_{ij}+p_{jk}-p_{ik}=0, \Lambda_j-\Lambda_i-[\Lambda_0, p_{ij}]=0,[\Lambda_0, \Lambda_i]=0$ (see \cite{Kim15}). Hence we have $\delta(\{p_{ij}\})=0, \delta(\{-\Lambda_i\})+[\Lambda_0, \{p_{ij}\}]=0, [\Lambda_0,\{-\Lambda_i\}]=0$. We set
\begin{align*}
\Psi_i:=\Psi\theta_i:(U_i\times Spec(k[\epsilon]),\Lambda_0+\epsilon \Lambda_i)\to (V_i\times Spec(k[\epsilon]),\Pi_0)\subset (Y\times Spec(k[\epsilon]),\Pi_0)
\end{align*}
which corresponds to a Poisson $k[\epsilon]$-homomorphism
\begin{align*}
\Phi_i:(C_i\otimes_k k[\epsilon],\Pi_0)\to (B_i\otimes_k k[\epsilon], \Lambda_0+\epsilon \Lambda_i)
\end{align*}
which induces $f|_{U_i}$ by modulo $\epsilon$ so that we have an element $v_i\in \Gamma(U_i, f^* T_Y)$ such that $\Phi_i=f+\epsilon v_i$ and for any $x,y\in C_i$,
\begin{align*}
&\Phi_i(\Pi_0(x,y))=(\Lambda_0+\epsilon \Lambda_i)(\Phi_i x, \Phi_i y)\\
&f(\Pi_0(x,y))+\epsilon v_i(\Pi_0(x,y))=(\Lambda_0+\epsilon \Lambda_i)(fx+\epsilon v_ix,fy+\epsilon v_iy)\\
&f(\Pi_0(x,y))+\epsilon v_i(\Pi_0(x,y))=\Lambda_0(fx,fy)+\epsilon\Lambda_0(v_ix,fy)+\epsilon \Lambda_0(fx,v_iy)+\epsilon \Lambda_i(fx,fy)\\
&\iff \pi(v_i)+F\Lambda_i=0 \iff \pi(v_i)=F(-\Lambda_i).
\end{align*}

On the other hand, restricting to $U_{ij}$, we have $\Psi_j|_{U_{ij}} \theta_j^{-1}\theta_i=\Psi_i|_{U_{ij}}$which corresponds to $(1+\epsilon p_{ij})\Phi_j=\Phi_i$, equivalently, $(1+\epsilon p_{ij})(f+\epsilon v_j)=1+\epsilon v_i$ so that we have $p_{ij}\circ f+v_j=v_i$, and so $v_i-v_j=F(p_{ij}) (-\delta \{v_i\}=F(\{p_{ij}\}))$. Hence $(\{v_{ij}\},\{p_{ij}\},\{-\Lambda_i\})\in C^0(\mathcal{U},f^* T_Y)\oplus C^1(\mathcal{U},T_X)\oplus C^0(\mathcal{U}, \wedge^2 T_X)$ defines an element in $PD_{(X,\Lambda_0)/(Y,\Pi_0)}$.

Next we identify obstructions. Consider a small extension $e:0\to(t)\to\tilde{A}\to A\to 0$, and an infinitesimal deformation $\xi=(\Psi: (\mathcal{X},\Lambda)\to (Y\times_k Spec(A),\Pi_0))$ of $f$ over $Spec(A)$ leaving the target fixed.
Choose an affine open cover $\mathcal{U}=\{U_i=Spec(B_i)\}_{i\in I}$ of $X$ such that for each $i$, $f(U_i)$ is contained in an affine open set $V_i=Spec(C_i)$ of $Y$, $\Psi$ is locally represented as
\begin{align*}
\Psi:(\mathcal{X}|_{U_i},\Lambda |_{U_i})\to (V_i\times Spec(A),\Pi_0)
\end{align*}
and we have a Poisson isomorphism $\theta_i:(U_i\times Spec(A),\Lambda_i)\to (\mathcal{X}|_{U_i},\Lambda |_{U_i})$, where $\Lambda_i\in \Gamma(U_i, \wedge^2 T_X)\otimes_k A$. Then for each $i,j\in I$, $\theta_{ij}:=\theta_j^{-1}\theta_i:(U_{ij}\times Spec(A),\Lambda_j)\to (U_{ij}\times Spec(A), \Lambda_i)$ is a Poisson isomorphism inducing the identity on $(U_{ij},\Lambda_0)$, which corresponds to a Poisson $A$-isomorphism $r_{ij}:(\Gamma(U_{ij},\mathcal{O}_X)\otimes_k A,\Lambda_j)\to (\Gamma(U_{ij},\mathcal{O}_X)\otimes_k A, \Lambda_i)$. We set
\begin{align*}
\Psi_i:=\Psi\theta_i:(U_i\times Spec(A),\Lambda_i)\to (V_i\times Spec(A),\Pi_0)
\end{align*}
which corresponds to a Poisson $A$-homomorphism
\begin{align*}
\Phi_i:(C_i\otimes_k A,\Pi_0)\to (B_i\otimes_k A,\Lambda_i)
\end{align*}

Now given a small extension $e:0\to (t)\to \tilde{A}\to A\to 0$, we associated an element $o_\eta(e)\in PD_{(X,\Lambda_0)/(Y,\Pi_0)}^1$.

Let $\tilde{\Phi}_i:C_i\otimes_k \tilde{A}\to B_i\otimes_k \tilde{A}$ be an arbitrary lifting of $\Phi_i$ to $\tilde{A}$ inducing $\Phi_i$, $\tilde{\Lambda}_i\in \Gamma(U_i, \wedge^2 T_X)\otimes_k \tilde{A}$ be an arbitrary lifting of $\Lambda_i$ to $\tilde{A}$ inducing $\Lambda_i$, and $\tilde{r}_{ij}:\Gamma(U_{ij},\mathcal{O}_X)\otimes_k \tilde{A} \to \Gamma(U_{ij},\mathcal{O}_X)\otimes_k \tilde{A}$ be an arbitrary lifting of $r_{ij}$ to $\tilde{A}$ inducing $r_{ij}$. Then $\tilde{r}_{ij}\tilde{r}_{jk}\tilde{r}_{ki}=Id+t\tilde{d}_{ijk}$ for some $\tilde{d}_{ijk}\in \Gamma(U_{ijk}, \mathcal{O}_X)$, $[\tilde{\Lambda}_i,\tilde{\Lambda}_i]=tq_i$ for some $q_i\in \Gamma(U_i,\wedge^3 T_X)$, and there exist $\Lambda_{ij}'\in \Gamma(U_{ij},\wedge^2 T_X)$ such that $\tilde{r}_{ij}(\tilde{\Lambda}_j(c,d))=(\tilde{\Lambda}_i+t\Lambda_{ij}')(\tilde{r}_{ij}(c),\tilde{r}_{ij}(d))$ for $c,d\in \Gamma(U_{ij},\mathcal{O}_X)\otimes_k \tilde{A}$. Then $[\Lambda_0, \{\frac{1}{2}q_i\}]=0, -\delta(\{\frac{1}{2}q_i\})+[\Lambda_0,\{ \Lambda_{ij}' \}]=0, -\delta(\{\Lambda_{ij}'\})+[\Lambda_0, \{-\tilde{d}_{ijk}\}]=0, - \delta(\{-\tilde{d}_{ijk}\})=0$ (see \cite{Kim16}).

Since $r_{ij}\Phi_j=\Phi_i$, we have $\tilde{r}_{ij}\tilde{\Phi}_j-\tilde{\Phi}_i=t\xi_{ij}$ for some $\xi_{ij}\in \Gamma(U_{ij},f^* T_Y)$. Then since $\tilde{r}_{ij}\tilde{r}_{jk}=\tilde{r}_{ik}+t\tilde{d}_{ijk}\tilde{r}_{ik}$, we have $t(\xi_{ij}-\xi_{ik}+\xi_{jk})=t(\xi_{ij}-\xi_{ik}+r_{ij}\xi_{jk})=\tilde{r}_{ij}\tilde{\Phi}_j-\tilde{\Phi}_i-\tilde{r}_{ik}\tilde{\Phi}_k+\tilde{\Phi}_i+\tilde{r}_{ij}\tilde{r}_{jk}\tilde{\Phi}_k-\tilde{r}_{ij}\tilde{\Phi}_j=-\tilde{r}_{ik}\tilde{\Phi}_k+\tilde{r}_{ij}\tilde{r}_{jk}\tilde{\Phi}_k=t\tilde{d}_{ijk}f$. Hence we have
\begin{align}
\xi_{ij}-\xi_{ik}+\xi_{jk}=F\tilde{d}_{ijk}
\end{align}

Since $\Phi_i^* \Pi_0-\Phi_i \Lambda_i=0$ (recall Notation \ref{notation11}), we have $\tilde{\Phi}_{i}^*\Pi_0-\tilde{\Phi}_i\tilde{\Lambda}_i=tP_i$ for some $P_i\in Hom_{C_i}(\wedge^2 \Omega_{C_i}^1,B_i)$. Then for $a,b,c\in C_i$, since $[\Pi_0,\Pi_0]=0$, we have
\begin{align*}
&t\pi(P_i)(a,b,c)\\
&=t\Lambda_0(P_i(a, b), f(c))-t\Lambda_0(P_i(a, c)\wedge f(b))+t\Lambda_0(P_i(b, c),f(a))+tP_i(\Pi_0(a,b),c)-tP_i(\Pi_0(a, c),b)+tP_i(\Pi_0(b,c),a)\\
&=\tilde{\Lambda}_i((\tilde{\Phi}_i^*\Pi_0-\tilde{\Phi}_i\tilde{\Lambda})(a,b),\tilde{\Phi}_i(c))-\tilde{\Lambda}_i((\tilde{\Phi}_i^*\Pi_0-\tilde{\Phi}_i\tilde{\Lambda}_i)(a,c),\tilde{\Phi}_i(b))+\tilde{\Lambda}_i((\tilde{\Phi}_{i}^*\Pi_0-\tilde{\Phi}_i \tilde{\Lambda}_i)(b,c),\tilde{\Phi}_i(a))\\
&+(\tilde{\Phi}_{i}^*\Pi_0-\tilde{\Phi}_i \tilde{\Lambda}_i)(\Pi_0(a,b),c)-(\tilde{\Phi}_i^*\Pi_0-\tilde{\Phi}_i\tilde{\Lambda}_i)(\Pi_0(a,c),b)+(\tilde{\Phi}_{i}^*\Pi_0-\tilde{\Phi}_i \tilde{\Lambda}_i)(\Pi_0(b,c),a)\\
&=\tilde{\Lambda}_i(\tilde{\Phi}_i(\Pi_0(a,b)),\tilde{\Phi}_i(c))-\tilde{\Lambda}_i(\tilde{\Lambda}_i(\tilde{\Phi}_i(a),\tilde{\Phi}_i(b)),\tilde{\Phi}_i(c))-\tilde{\Lambda}_i(\tilde{\Phi}_i(\Pi_0(a,c)),\tilde{\Phi}_i(b))+\tilde{\Lambda}_i(\tilde{\Lambda}_i(\tilde{\Phi}_i(a),\tilde{\Phi}_i(c)),\tilde{\Phi}_i(b))\\
&+\tilde{\Lambda}_i(\tilde{\Phi}_i(\Pi_0(b,c)),\tilde{\Phi}_i(a))-\tilde{\Lambda}_i(\tilde{\Lambda}_i(\tilde{\Phi}_i(b),\tilde{\Phi}_i(c)),\tilde{\Phi}_i(a))\\
&+\tilde{\Phi}_i(\Pi_0(\Pi_0(a,b),c))-\tilde{\Lambda}_i(\tilde{\Phi}_i(\Pi_0(a,b)),\tilde{\Phi}_i(c))-\tilde{\Phi}_i(\Pi_0(\Pi_0(a,c),b))+\tilde{\Lambda}_i(\tilde{\Phi}_i(\Pi_0(a,c)),\tilde{\Phi}_i(b))\\
&+\tilde{\Phi}_i(\Pi_0(\Pi_0(b,c),a))-\tilde{\Lambda}_i(\tilde{\Phi}_i(\Pi_0(b,c)),\tilde{\Phi}_i(a))\\
&=-\tilde{\Lambda}_i(\tilde{\Lambda}_i(\tilde{\Phi}_i(a),\tilde{\Phi}_i(b)),\tilde{\Phi}_i(c))+\tilde{\Lambda}_i(\tilde{\Lambda}_i(\tilde{\Phi}_i(a),\tilde{\Phi}_i(c)),\tilde{\Phi}_i(b))-\tilde{\Lambda}_i(\tilde{\Lambda}_i(\tilde{\Phi}_i(b),\tilde{\Phi}_i(c)),\tilde{\Phi}_i(a))
\end{align*}

On the other hand, since $[\tilde{\Lambda}_i,\tilde{\Lambda}_i]=t q_i$
\begin{align*}
tF(q_i)(a,b,c)&=tq_i(f(a),f(b),f(c))=[\tilde{\Lambda}_i,\tilde{\Lambda}_i](\tilde{\Phi}_i(a),\tilde{\Phi}_i(b),\tilde{\Phi}_i(c))\\
&=2\tilde{\Lambda}_i(\tilde{\Lambda}_i(\tilde{\Phi}_i(a),\tilde{\Phi}_i(b)),\tilde{\Phi}_i(c))-2\tilde{\Lambda}_i(\tilde{\Lambda}_i(\tilde{\Phi}_i(a),\tilde{\Phi}_i(c)),\tilde{\Phi}_i(b))+2\tilde{\Lambda}_i(\tilde{\Lambda}_i(\tilde{\Phi}_i(b),\tilde{\Phi}_i(c)),\tilde{\Phi}_i(a))
\end{align*}
Hence we have
\begin{align}
\pi(P_i)=-\frac{1}{2}Fq_i
\end{align}

On the other hand, let $\tilde{r}_{ij}(\tilde{\Lambda}_j(c,d))=(\tilde{\Lambda}_i+t\Lambda_{ij}')(\tilde{r}_{ij}(c),\tilde{r}_{ij}(d))$ for $c,d\in \Gamma(U_{ij},\mathcal{O}_X)$. We recall Notation \ref{notation12}. Then for $a,b\in \Gamma(V_{ij}, \mathcal{O}_Y)$, 
\begin{align*}
t(P_i-P_j)&=t(P_i-r_{ij}^*P_j)(a,b)=\tilde{\Phi}_{i}^*\Pi_0(a,b)-\tilde{\Phi}_i\tilde{\Lambda}_i(a,b)- \tilde{r}_{ij}(\tilde{\Phi}_{j}^*\Pi_0(a,b))+\tilde{r}_{ij}(\tilde{\Phi}_j \tilde{\Lambda}_j(a,b))\\
&=\tilde{\Phi}_i(\Pi_0(a,b))-\tilde{\Lambda}_i(\tilde{\Phi}_i(a),\tilde{\Phi}_i(b))-\tilde{r}_{ij}(\tilde{\Phi}_j(\Pi_0(a,b)))+\tilde{r}_{ij}(\tilde{\Lambda}_j(\tilde{\Phi}_j(a),\tilde{\Phi}_j(b)))\\
&=\tilde{\Phi}_i(\Pi_0(a,b))-\tilde{\Lambda}_i(\tilde{\Phi}_i(a),\tilde{\Phi}_i(b))-\tilde{r}_{ij}(\tilde{\Phi}_j(\Pi_0(a,b)))+(\tilde{\Lambda}_i+t\Lambda_{ij}')(\tilde{r}_{ij}\tilde{\Phi}_j(a),\tilde{r}_{ij}\tilde{\Phi}_j(b)))\\
&=\tilde{\Phi}_i(\Pi_0(a,b))-\tilde{\Lambda}_i(\tilde{\Phi}_i(a),\tilde{\Phi}_i(b))-\tilde{r}_{ij}(\tilde{\Phi}_j(\Pi_0(a,b)))+(\tilde{\Lambda}_i+t\Lambda_{ij}')(\tilde{r}_{ij}\tilde{\Phi}_j(a),\tilde{r}_{ij}\tilde{\Phi}_j(b)))\\
&=\tilde{\Phi}_i(\Pi_0(a,b))-\tilde{\Lambda}_i(\tilde{\Phi}_i(a),\tilde{\Phi}_i(b))-\tilde{r}_{ij}(\tilde{\Phi}_j(\Pi_0(a,b)))+(\tilde{\Lambda}_i+t\Lambda_{ij}')(\tilde{\Phi}_i(a)+t\xi_{ij}(a),\tilde{\Phi}_i(b)+t\xi_{ij}(b))\\
&=\tilde{\Phi}_i(\Pi_0(a,b))-\tilde{r}_{ij}(\tilde{\Phi}_j(\Pi_0(a,b)))+\tilde{\Lambda}_i(t\xi_{ij}(a),\tilde{\Phi}_i(b))+\tilde{\Lambda}_i(\tilde{\Phi}_i(a),t\xi_{ij}(b))+t\Lambda_{ij}'(\tilde{\Phi}_i(a),\tilde{\Phi}_i(b))\\
&=-t\xi_{ij}(\Pi_0(a,b))+t\Lambda_0(\xi_{ij}(a),f(b))+t\Lambda_0(f(a),\xi_{ij}(b))+t\Lambda_{ij}'(f(a),f(b))
\end{align*}

Hence we have
\begin{align}
\pi(\xi_{ij})=P_i-P_j-F\Lambda_{ij}'
\end{align}

Then from , $\alpha:=(\{-\xi_{ij}\},\{-P_i\},\{-\tilde{d}_{ijk}\},\{\Lambda_{ij}'\},\{\frac{1}{2}q_i\})\in \mathcal{C}^1(\mathcal{U}, f^* T_Y)\oplus C^0(\mathcal{U}, \wedge^2 f^* T_Y)\oplus C^2(\mathcal{U},T_X)\oplus C^1(\mathcal{U},\wedge^2 T_X)\oplus C^0(\mathcal{U},\wedge^2 T_X)$ defines an element $PD^1_{(X,\Lambda_0)/(Y,\Pi_0)}$.

Let $\tilde{\Phi}_i'$ be an another arbitrary lifting of $\Phi_i$, $\tilde{r}_{ij}'$ be an another arbitrary lifting of $r_{ij}$, $\tilde{\Lambda}_i'$ be an another arbitrary lifting of $\Lambda_i$. Let $\beta=(\{-\xi_{ij}'\},\{-P_i'\},\{-\tilde{d}_{ijk}'\},\{\Lambda_{ij}''\},\{\frac{1}{2}q_i'\})\in \mathcal{C}^1(\mathcal{U}, f^* T_Y)\oplus C^0(\mathcal{U}, \wedge^2 f^* T_Y)\oplus C^2(\mathcal{U},T_X)\oplus C^1(\mathcal{U},\wedge^2 T_X)\oplus C^0(\mathcal{U},\wedge^2 T_X)$ be the associated element in  $PD^1_{(X,\Lambda_0)/(Y,\Pi_0)}$. Then $\tilde{\Phi}_i'=\tilde{\Phi}_i+t Q_i$ for some $Q_i\in \Gamma(U_i, f^* T_Y)$, $\tilde{r}_{ij}'=\tilde{r}_{ij}+tp_{ij}$ for some $p_{ij}\in \Gamma(U_{ij}, T_X)$ and $\tilde{\Lambda}_i'=\tilde{\Lambda}_i+t\Lambda_i'$ for some $\Lambda_i'\in \Gamma(U_i,\wedge^2 T_X)$. $\{-\tilde{d}_{ijk}+\tilde{d}_{ijk}'\}=\delta(\{p_{ij}\}),\Lambda_{ij}'-\Lambda_{ij}''=\delta(-\Lambda_i')+[\Lambda_0, p_{ij}]=0, \frac{1}{2}q_i-\frac{1}{2}q_i'=[\Lambda_0, -\Lambda_i']$ .

On the other hand, for $a,b\in C_i$,
\begin{align*}
(tP_i'-tP_i)(a,b)&= \tilde{\Phi}'^*_{i}\Pi_0(a,b)-\tilde{\Phi}_i'\tilde{\Lambda}_i'(a,b)-\tilde{\Phi}^*_{i}\Pi_0(a,b)+\tilde{\Phi}_i\tilde{\Lambda}_i(a,b)\\
&=\tilde{\Phi}'_i((\Pi_0(a,b))-\tilde{\Lambda}_i'(\tilde{\Phi}_i'(a),\tilde{\Phi}_i'(b))-\tilde{\Phi}_i(\Pi_0(a,b))+\tilde{\Lambda}_i(\tilde{\Phi}_i(a),\tilde{\Phi}_i(b))\\
&=(\tilde{\Phi}_i+tQ_i)((\Pi_0(a,b))-(\tilde{\Lambda}_i+t\Lambda_i')(\tilde{\Phi}_i(a)+tQ_i(a),\tilde{\Phi}_i(b)+tQ_i(b))-\tilde{\Phi}_i(\Pi_0(a,b))+\tilde{\Lambda}_i(\tilde{\Phi}_i(a),\tilde{\Phi}_i(b))\\
&=tQ_i(\Pi_0(a,b))-\Lambda_0(f(a), tQ_i(b))-\Lambda_0(tQ_i(a),f(b))-t\Lambda_i'(f(a),f(b))\\
&=-\pi(Q_i)(a,b)-\Lambda_i'(f(a),f(b)).
\end{align*}

Hence we have
\begin{align}
P_i'-P_i=\pi(-Q_i)+F(- \Lambda_i')
\end{align}

On the other hand, 
\begin{align*}
t\xi_{ij}'(a)-t\xi_{ij}(a)&=\tilde{r}_{ij}'\tilde{\Phi}_j'(a)-\tilde{\Phi}_i'(a)-\tilde{r}_{ij}\tilde{\Phi}_j(a)+\tilde{\Phi}_i(a)\\
&=(\tilde{r}_{ij}+tp_{ij})(\tilde{\Phi}_j+tQ_j)(a)-(\tilde{\Phi}_i+tQ_i)(a)-\tilde{r}_{ij}\tilde{\Phi}_j(a)+\tilde{\Phi}_i(a)\\
&=tp_{ij}f(a)+tQ_j(a)-tQ_i(a)
\end{align*}

Hence we have
\begin{align}
\xi_{ij}'-\xi_{ij}=F p_{ij}+Q_j-Q_i
\end{align}
Therefore $\alpha-\beta=(F(\{p_{ij}\})-\delta(\{-Q_i\}),\pi(\{-Q_i\})+F(\{-\Lambda_i' \}),\delta(\{p_{ij}\}),\delta(\{p_{ij}\}),\delta(\{-\Lambda_i'\})+[\Lambda_0, \{p_{ij}\}, [\Lambda_0, \{-\Lambda_i'\}])$ so that $\alpha,\beta$ define the same class in $PD_{(X,\Lambda_0)/(Y,\Pi_0)}^1$. Hence given a small extension $e:0\to (t)\to \tilde{A}\to A\to 0$, we can associate an element $o_\xi(e):=$ the cohomology class of $\alpha \in PD_{(X,\Lambda_0)/(Y,\Pi_0)}$. We note that $o_\xi(e)=0$ if and only if there exist collections $\xi_{ij}=0,P_i=0,\tilde{d}_{ijk}=0,\Lambda_{ij}'=0$, and $q_i=0$:
\begin{enumerate}
\item if $\tilde{d}_{ijk}=0,\Lambda_{ij}'=0$, and $q_i=0$, then we have a flat Poisson deformation $(\tilde{\mathcal{X}},\tilde{\Lambda})$ of $(X,\Lambda_0)$ over $Spec(\tilde{A})$ inducing $(\mathcal{X},\Lambda)$.
\item if $\xi_{ij}=0$, and $P_i=0$, we have a Poisson morphism $\tilde{\Psi}:(\tilde{\mathcal{X}},\tilde{\Lambda})\to (Y\times Spec(\tilde{A}),\Pi_0)$ inducing $\Psi$.
\end{enumerate}
Hence $o_\xi(e)=0$ if and only if there is a lifting of $\xi$ to $\tilde{A}$.
\end{proof}

\section{Stability and Costability}\label{appendix3}

Let $f:(X,\Lambda_0)\to (Y,\Pi_0)$ be a Poisson morphism of nonsingular Poisson varieties with $X$ projective. Let us consider a deformation of $f$ over $A\in \bold{Art}$ in the sense of Definition \ref{ll30} so that we have a Poisson morphism $\Phi:(\mathcal{X},\Lambda)\to (\mathcal{Y},\Pi)$ over $A$ which induces $f$. Then we have morphisms of functors of Artin rings which are defined by forgetful morphisms
\begin{align}\label{ll70}
f_X:Def_f\to Def_{(X,\Lambda_0)},\,\,\,\,\,\,\,\,\,\, f_Y:Def_f\to Def_{(Y,\Pi_0)}
\end{align}
where $Def_{(X,\Lambda_0)}$ is the functor of Artin rings of flat Poisson deformations of $(X,\Lambda_0)$, and $Def_{(Y,\Pi_0)}$ is the functor of Artin rings of flat Poisson deformations of $(Y,\Pi_0)$ (see \cite{Kim16}).

\begin{theorem}[stability]\label{mm1}
Let $f:(X,\Lambda_0)\to (Y,\Pi_0)$ be a Poisson morphism of nonsingular Poisson varieties with $X$ projective. Assume that
\begin{enumerate}
\item $F:\mathbb{H}^1(X, T_X^\bullet)\to \mathbb{H}^1(X, f^*T_Y^\bullet)$ is surjective.
\item $F:\mathbb{H}^2(X, T_X^\bullet)\to \mathbb{H}^2(X, f^* T_Y^\bullet)$ is injective.
\end{enumerate}
Then the functor $f_Y$ in $(\ref{ll70})$ is smooth. In other words, for any Poisson flat deformation  $q:(\mathcal{Y}, \Pi)\to Spec(A)$ of $(Y,\Pi_0)$ over $Spec(A)$ for a local artinian $k$-algebra $A$ with residue $k$,  there exist
\begin{enumerate}
\item a falt Poisson deformation $p:(\mathcal{X},\Lambda)\to Spec(A)$ of $(X,\Lambda_0)$ over $Spec(A)$,
\item a Poisson morphism $\Phi:(\mathcal{X},\Lambda)\to (\mathcal{Y}, \Pi)$ over $Spec(A)$ which induces $f$.
\end{enumerate}

\end{theorem}

\begin{proof}
We will prove by the induction on the dimension of $A$. When $\dim_k A=1$, there is nothing to prove. Assume that the theorem holds for $\dim_k A\leq n-1$. Let $A$ be a local artinian $k$-algebra with $\dim_k A=n$. Assume that the maximal ideal $\mathfrak{m}$ of $A$ satisfies $\mathfrak{m}^{p-1}\ne 0$ and $\mathfrak{m}^p= 0$. Choose $t\in \mathfrak{m}^{p-1}$. Then $A/(t)\in \bold{Art}$ with $\dim_k A/(t)\leq n-1$ and $0\to (t)\to A\to A/(t)\to 0$ is a small extension.

Let $q:(\mathcal{Y},\Pi)\to Spec(A)$ be a Poisson deformation of $(Y,\Pi_0)$ over $Spec(A)$. Let $\bar{q}:(\bar{Y},\bar{\Pi}):=(\mathcal{Y},\Pi)\times_{Spec(A)} Spec(A/(t))\to Spec(A/(t))$ be the induced deformation over $Spec(A/(t))$. Then by the induction hypothesis, there exists a deformation $\bar{p}:(\bar{\mathcal{X}},\bar{\Lambda})\to Spec(A/(t))$ of $(X,\Lambda_0)$ and $\bar{\Psi}: (\bar{\mathcal{X}},\bar{\Lambda})\to (\bar{\mathcal{Y}},\bar{\Pi})$ which induces $f$.

Let $\mathcal{U}=\{U_i=Spec(B_i)\}_{i\in I}$ be an affine open cover of $X$ and $\mathcal{V}=\{V_i=Spec(C_i)\}_{i \in I}$ be an affine open cover of $Y$ such that $f(U_i)\subset V_i$. We have a Poisson isomorphism  $\bar{\theta}_i:(U_i\times Spec(A/(t)),\bar{\Lambda}_i)\to (\bar{\mathcal{X}}|_{U_i},\bar{\Lambda}|_{U_i})$, where $\bar{\Lambda}_i\in \Gamma(U_i,\wedge^2 T_X)\otimes_k A/(t)$. Then for each $i,j\in I$, $\bar{\theta}_{ij}:=\bar{\theta}_j^{-1}\bar{\theta}_i:(U_{ij}\times Spec(A/(t)),\bar{\Lambda}_j)\to (U_{ij}\times Spec(A/(t)), \bar{\Lambda}_i)$ is a Poisson isomorphism inducing the identity on $(U_{ij},\Lambda_0)$, which corresponds to a Poisson $A/(t)$-isomorphism $\bar{r}_{ij}:(\Gamma(U_{ij},\mathcal{O}_X)\otimes_k A/(t), \bar{\Lambda}_j)\to(\Gamma(U_{ij},\mathcal{O}_X)\otimes_k A/(t),\bar{\Lambda}_i)$. On the other hand, similarly we have a Poisson isomorphism $\bar{e}_i:(V_i\times Spec(A/(t)),\bar{\Pi}_i)\to (\bar{\mathcal{Y}}|_{V_i},\bar{\Pi}|_{V_i})$, where $\bar{\Pi}_i \in \Gamma(V_i, \wedge^2 T_Y)\otimes_k A/(t)$. Then for each $i,j\in I$, $\bar{e}_{ij}:=\bar{e}_j^{-1}\bar{e}_i:(V_{ij}\times Spec(A/(t)),\bar{\Pi}_i)\to (V_{ij}\times Spec(A/(t)),\bar{\Pi}_i)$ is a Poisson isomorphism inducing the identity on $(V_{ij},\Pi_0)$, which corresponds to a Poisson $A/(t)$-isomorphism $\bar{g}_{ij}:(\Gamma(V_{ij},\mathcal{O}_Y)\otimes_k A/(t),\bar{\Pi}_j)\to (\Gamma(V_{ij},\mathcal{O}_Y)\otimes_k A/(t),\bar{\Pi}_i)$. We set $\bar{\Psi}_i=e_i^{-1} \Psi\theta_i:(U_i\times Spec(A/(t)),\bar{\Lambda}_i)\to (V_i\times Spec(A/(t)) ,\bar{\Pi}_i)$ which corresponds to $\bar{\Phi}_i:(C_i\otimes_k A/(t),\bar{\Pi}_i)\to (B_i\otimes_k A/(t), \bar{\Lambda}_i)$.

Let $\bar{\Phi}_i:C_i\otimes A\to B_i\otimes A$ be a lifting of $\bar{\Phi}_i$, $r_{ij}:\Gamma(U_{ij},\mathcal{O}_X)\otimes A\to \Gamma(U_{ij},\mathcal{O}_X)\otimes A$ be a lifting of $\bar{r}_{ij}$, and $\Lambda_i\in \Gamma(U_i,\wedge^2 T_X)\otimes_k A$ be a lifting of $\bar{\Lambda}_i$. On the other hand, Let $g_{ij}:(\Gamma(V_{ij},\mathcal{O}_Y)\otimes_k A, \Pi_j)\to ( \Gamma(V_{ij},\mathcal{O}_Y)\otimes_k A, \Pi_i)$ be a lifting of $\bar{g}_{ij}$ and $\bar{\Pi}_i$ defining $(\mathcal{Y},\Pi)$. Then $r_{ij}r_{jk}r_{ki}=Id+t\eta_{ijk}$ for some $\eta_{ijk}\in \Gamma(U_{ijk},T_X)$, $[\Lambda_i,\Lambda_i]=tQ_i$ for some $Q_i\in \Gamma(U_i,\wedge^2 T_X)$, and there exist $\Lambda_{ij}'\in \Gamma(U_{ij}, \wedge^2 T_X)$ such that
$r_{ij}(\Lambda_j(a,b))=(\Lambda_i+t\Lambda_{ij}')(r_{ij}(a),r_{ij}(b))$ for $a,b\in \Gamma(U_{ij},\mathcal{O}_X)\otimes A$ (for the detail, see \cite{Kim16}). Conisider the following diagram (which is not necessarily commutative nor Poisson except for $g_{ij}$)

\begin{center}
$\begin{CD}
(\Gamma(V_{ij},\mathcal{O}_Y)\otimes_k A, \Pi_j)@> g_{ij}>> (\Gamma(V_{ij} , \mathcal{O}_Y)\otimes_k A,\Pi_i)\\
@V\Phi_j VV @VV\Phi_i V\\
(\Gamma(U_{ij},\mathcal{O}_X)\otimes_k A, \Lambda_j) @>r_{ij}>> (\Gamma(U_{ij},\mathcal{O}_X)\otimes_k A,\Lambda_i)
\end{CD}$
\end{center}
Then since $\bar{r}_{ij} \bar{\Phi}_j-\bar{\Phi}_i \bar{g}_{ij}=0$, we have $r_{ij} \Phi_j-\Phi_i g_{ij}=t E_{ij}$ for some $E_{ij}\in \Gamma(U_{ij},f^* T_Y)$. Then
\begin{align*}
t(E_{ij}-E_{ik}+E_{jk})&=t(E_{ij}g_{jk}-E_{ik}+r_{ij} E_{jk})=r_{ij} \Phi_jg_{jk}-\Phi_ig_{ij}g_{jk}-r_{ik} \Phi_k+\Phi_i g_{ik}+r_{ij}r_{jk}\Phi_k-r_{ij} \Phi_j g_{jk}\\
&= -r_{ik} \Phi_k+r_{ij}r_{jk}\Phi_k=-r_{ik} \Phi_k+(r_{ik}+t\eta_{ijk}r_{ik} )\Phi_k=t\eta_{ijk}f
\end{align*}
Hence  we have
\begin{align}\label{ll80}
F(\{\eta_{ijk}\})=\delta(\{E_{ij}\}).
\end{align}

On the other hand, we recall Notation \ref{notation11}. Then since $\bar{\Phi}_i^* \bar{\Pi}_i-\bar{\Phi}_i\bar{\Lambda}_i=0$, we have $\Phi_{i}^*\Pi_i-\Phi_i^*\Lambda_i=tP_i$ for some $P_i\in Hom_{C_i}(\wedge^2 \Omega_{C_i},B_i)=\Gamma(U_i, \wedge^2 f^*T_Y)$. Then for $a,b,c\in C_i$, since $[\Pi_i,\Pi_i]=0$, we have
\begin{align*}
&t\pi(P_i)(a,b,c)\\
&=t\Lambda_0(P_i(a,b),f(c))-t\Lambda_0(P_i(a,c),f(b))+t\Lambda_0(P_i(b,c),f(a))+tP_i(\Pi_0(a,b),c)-tP_i(\Pi_0(a,c),b)+tP_i(\Pi_0(b,c),a)\\
&=\Lambda_i((\Phi_{i}^*\Pi_i-\Phi_i\Lambda_i)(a,b),\Phi_i(c))-\Lambda_i((\Phi_{i}^*\Pi_i-\Phi_i\Lambda_i)(a,c),\Phi_i(b))+\Lambda_i((\Phi_{i}^*\Pi_i-\Phi_i\Lambda_i)(b,c),\Phi_i(a))\\
&+(\Phi_{i}^*\Pi_i-\Phi_i\Lambda_i)(\Pi_i(a,b),c)-(\Phi_{i}^*\Pi_i-\Phi_i\Lambda_i)(\Pi_i(a,c),b)+(\Phi_{i}^*\Pi_i-\Phi_i\Lambda_i)(\Pi_i(b,c),a)\\
&=\Lambda_i(\Phi_i(\Pi_i(a,b)),\Phi_i(c))-\Lambda_i((\Lambda_i(\Phi_i(a),\Phi_i(b)),\Phi_i(c))-\Lambda_i(\Phi_i(\Pi_i(a,c)),\Phi_i(b))+\Lambda_i((\Lambda_i(\Phi_i(a),\Phi_i(c)),\Phi_i(b))\\
&+\Lambda_i(\Phi_i(\Pi_i(b,c)),\Phi_i(a))-\Lambda_i((\Lambda_i(\Phi_i(b),\Phi_i(c)),\Phi_i(a))\\
&+\Phi_i(\Pi_i(\Pi_i(a,b)),c)-\Lambda_i(\Phi_i(\Pi_i(a,b)),\Phi_i(c))-\Phi_i(\Pi_i(\Pi_i(a,c)),b)+\Lambda_i(\Phi_i(\Pi_i(a,c)),\Phi_i(b))\\
&+\Phi_i(\Pi_i(\Pi_i(b,c)),a)-\Lambda_i(\Phi_i(\Pi_i(b,c)),\Phi_i(a))\\
&=-(\frac{1}{2}[\Lambda_i,\Lambda_i](\Phi_i(a),\Phi_i(b),\Phi_i(c)))=tF(-\frac{1}{2}Q_i)(a,b,c)
\end{align*}
Hence we have 
\begin{align}\label{ll81}
\pi(\{P_i\})=F(-\frac{1}{2} \{Q_i\}).
\end{align}
Lastly, we have, for $a,b\in \Gamma(V_{ij},\mathcal{O}_Y)$,
\begin{align*}
&t(P_i-P_j)(a,b)=t(g_{ij}P_i-r_{ij}^*P_j)(a,b)=(\Phi_{i}^*\Pi_i-\Phi_i\Lambda_i)(g_{ij}(a),g_{ij}(b))-r_{ij}(\Phi_{j}^*\Pi_j-\Phi_j\Lambda_j)(a,b)\\
&=\Phi_i(\Pi_i(g_{ij}(a),g_{ij}(b)))-\Lambda_i(\Phi_i( g_{ij}(a)),\Phi_i ( g_{ij}(b)))-r_{ij}(\Phi_j(\Pi_j(a,b)))+r_{ij}(\Lambda_j(\Phi_j (a),\Phi_j(b)))\\
&=\Phi_i(g_{ij}(\Pi_j(a,b)))-\Lambda_i((r_{ij}\Phi_j-tE_{ij})(a),(r_{ij}\Phi_j-tE_{ij})(b))\\
&-\Phi_i (g_{ij}(\Pi_j(a,b)))- t E_{ij}(\Pi_j(a,b))+\Lambda_i(r_{ij}\Phi_j(a),r_{ij}\Phi_j(b))+t\Lambda_{ij}'(r_{ij}\Phi_j(a),r_{ij}\Phi_j(b))\\
&=t \Lambda_{ij}'(f(a),f(b))+t\Lambda_0(f(a),E_{ij}(b))+t\Lambda_0(E_{ij}(a), f(b))-t E_{ij}(\Pi_0(a,b))\\
&=tF(\Lambda_{ij}')(a,b)+t\pi(E_{ij})(a,b).
\end{align*}
Hence we have 
\begin{align}\label{ll82}
\delta(\{P_i\})+\pi(E_{ij})=F(-\Lambda_{ij}')
\end{align}

We note that $(\{-\frac{1}{2}Q_i\},\{ -\Lambda_{ij}'\},\{ \eta_{ijk}\})$ is a $2$-cocycle of $T_X^\bullet$ (see \cite{Kim16}). Since $\mathbb{H}^2(X,T_X^\bullet)\to  \mathbb{H}^2(X, f^* T_Y^\bullet)$ is injective and the image of $(\{-\frac{1}{2}Q_i\},\{ -\Lambda_{ij}'\},\{ \eta_{ijk}\})$ is zero from (\ref{ll80}),(\ref{ll81}), and (\ref{ll82}), there exist $d_{ij}\in C^1(\mathcal{U}, T_X)$ and $\lambda_i\in C^0(\mathcal{U}, \wedge^2 T_X)$ such that $d_{ki}+d_{jk}+d_{ij}=-\eta_{ijk}$, $\lambda_j -\lambda_i+[\Lambda_0,d_{ij}]=\Lambda_{ij}'$, and $[\Lambda_0,\lambda_i]=\frac{1}{2}Q_i$. Then we claim that $\{r_{ij}+td_{ij}\},\{\Lambda_i-t\lambda_i\}$ defines a flat Poisson deformation $(\mathcal{X}',\Lambda')$ of $(X,\Lambda_0)$ inducing $(\bar{\mathcal{X}},\bar{\Lambda})$. Indeed,
\begin{align} \label{mm3}
&(r_{ij}+td_{ij})(r_{jk}+td_{jk})(r_{ki}+td_{ki})=Id+t\eta_{ijk}+td_{ki}+td_{jk}+td_{ij}=Id.\\
&[\Lambda_i-t\lambda_i,\Lambda_i-t\lambda_i]=[\Lambda_i,\Lambda_i]-t 2[\Lambda_0, \lambda_i]=0 \notag\\
& (r_{ij}+td_{ij})((\Lambda_j-t\lambda_j)(a,b))-(\Lambda_i-t\lambda_i)((r_{ij}+td_{ij})(a),(r_{ij}+td_{ij})(b)) \notag \\
&=(\Lambda_i+t\Lambda_{ij}')(r_{ij}(a),r_{ij}(b))-t\lambda_j(a,b)+td_{ij}(\Lambda_0(a,b))-\Lambda_i(r_{ij}(a), r_{ij}(b))-t\Lambda_0(d_{ij}(a),b)-t\Lambda_0(a, d_{ij}(b))+t\lambda_i(a,b) \notag \\
&=t\Lambda_{ij}'(a,b)-t\lambda_j(a,b)-t[\Lambda, d_{ij}](a,b)+t\lambda_i(a,b)=0 \notag
\end{align}

Consider the following diagram (which is not necessarily commutative nor Poisson except for $g_{ij}$)
\begin{center}
$\begin{CD}
(\Gamma(V_{ij},\mathcal{O}_Y)\otimes_k A,  \Pi_j) @>g_{ij}>> ( \Gamma(V_{ij},\mathcal{O}_Y)\otimes_k A ,\Pi_i)\\
@V\Phi_j VV @VV\Phi_iV\\
(\Gamma(U_{ij},\mathcal{O}_X)\otimes_k A, \Lambda_j-t\lambda_j) @>r_{ij}+td_{ij}>> (\Gamma(U_{ij},\mathcal{O}_X)\otimes_k A,\Lambda_i-t\lambda_i)
\end{CD}$
\end{center}

Then $(r_{ij}+td_{ij}) \Phi_j-\Phi_ig_{ij}=tG_{ij}$ for some $G_{ij}\in \Gamma(U_{ij}, f^*T_Y)$ so that $G_{ij}= E_{ij}+td_{ij}f$. $G_{ij}-G_{ik}+G_{jk}=E_{ij}+td_{ij}f-E_{ik}-td_{ik}f+E_{jk}+td_{jk}f=F(\eta_{ijk})-F(\eta_{ijk})=0$. On the other hand, $\Phi_i^*\Pi_i-\Phi_i(\Lambda_i-t\lambda_i)=tL_i$ for some $L_i\in \Gamma(U_i,\wedge^2 f^* T_Y)$. Then $L_i=P_i+F(\lambda_i)$. We have $\pi(L_i)=\pi(\Pi_i)+F([\Lambda_0, \lambda_i])=F(-\frac{1}{2}Q_i)+F(\frac{1}{2}Q_i)=0$. Lastly we have $L_j-L_i+\pi(G_{ij})=P_j+F(\lambda_j)-P_i-F(\lambda_i)+\pi(E_{ij})+t\pi(F(d_{ij}))=F(-\Lambda_{ij}')+F(\lambda_j-\lambda_i+[\Lambda, d_{ij}])=0$. Hence $(\{G_{ij}\},\{L_i\})$ defines a $1$-cocycle of $f^*T_Y^\bullet$. Hence since $\mathbb{H}^1(X, T_X^\bullet)\to \mathbb{H}^1(X, f^*T_Y^\bullet)$ is surjective, there exist $1$-cocycle $(\{\chi_{ij}\},\{H_i\})\in C^1(\mathcal{U}, T_X)\oplus C^0(\mathcal{U},\wedge^2 T_X)$ of $T_X^\bullet$, and $\{\tau_i\}\in C^0(\mathcal{U},f^* T_Y)$ such that $G_{ij}-F(-\chi_{ij})=\tau_i-\tau_j$, and $L_i-F(-H_i)=\pi(\tau_i)$.

Then we claim that $\{r_{ij}+t(d_{ij}+\chi_{ij})\}$ and $\{\Lambda_i-t(\lambda_i+H_i)\}$ define a flat Poisson deformation $(\mathcal{X},\Lambda)$ of $(X,\Lambda_0)$ inducing $(\bar{\mathcal{X}},\bar{\Lambda})$. Indeed, from (\ref{mm3}), we have
\begin{align*}
&(r_{ij}+td_{ij}+t\chi_{ij})(r_{jk}+td_{jk}+t\chi_{jk})(r_{ki}+td_{ki}+\chi_{ki})=Id+t(\chi_{ij}+\chi_{jk}+\chi_{ki})=Id,\\
&[\Lambda_i-t\lambda_i-tH_i,\Lambda_i-t\lambda_i-tH_i]= -2t[\Lambda_0,H_i]=0\\
&(r_{ij}+td_{ij}+t\chi_{ij})((\Lambda_j-t\lambda_j-tH_j)(a,b))-(\Lambda_i-t\lambda_i-tH_i)((r_{ij}+td_{ij}+t\chi_{ij})(a),(r_{ij}+td_{ij}+t\chi_{ij})(b))\\
&=-tH_j(a,b)+t\chi_{ij}(\Lambda_0(a,b))-t\Lambda_0(\chi_{ij}(a),b)-t\Lambda_0(a,\chi_{ij}(b))+tH_i(a,b)=-t(H_j-H_i+[\Lambda, \chi_{ij}])=0
\end{align*}
We have the following commutative diagram
\begin{center}
$\begin{CD}
(\Gamma(V_{ij},\mathcal{O}_Y)\otimes_k A,  \Pi_j) @>g_{ij}>> ( \Gamma(V_{ij},\mathcal{O}_Y)\otimes_k A ,\Pi_i)\\
@V\Phi_j+t\tau_j VV @VV\Phi_i+t\tau_iV\\
(\Gamma(U_{ij},\mathcal{O}_X)\otimes_k A, \Lambda_j-t(\lambda_j+H_j)) @>r_{ij}+td_{ij}+t\chi_{ij}>> (\Gamma(U_{ij},\mathcal{O}_X)\otimes_k A,\Lambda_i-t(\lambda_i+H_i))
\end{CD}$
\end{center}
Indeed, $(r_{ij}+td_{ij}+t\chi_{ij})(\Phi_j(a)+t\tau_j(a))=\Phi_i(g_{ij}(a))+tG_{ij}(a)+t\chi_{ij}(f(a))+t\tau_j(a)=\Phi_i(g_{ij}(a))+t\tau_i(a)=(\Phi_i+t\tau_i)(g_{ij}(a))$. Lastly we claim that $\Phi_j+t\tau_j$ is a Poisson homomorphism. Indeed, 
\begin{align*}
&(\Phi_j+t\tau_j)(\Pi_j(a,b))-(\Lambda_j-t\lambda_j-tH_j)(\Phi_j(a)+t\tau_j(a),\Phi_j(b)+t\tau_j(b))\\
&=(\Phi_j(\Lambda_j-t\lambda_j)+tL_j)(a,b)+t\tau_j(\Pi_j(a,b))-(\Lambda_j-t\lambda_j)(\Phi_j(a),\Phi_j(a))-t\Lambda_0(f(a),\tau_j(b))-t\Lambda_0(\tau_j(a),b)+tH_j(f(a),f(b))\\
&=t(L_j-\pi(\tau_j)+F(H_j))(a,b)=0
\end{align*}
Hence there exists a flat Poisson deformation $(\mathcal{X},\Lambda)$ over $Spec(A)$ of $(X,\Lambda_0)$, and a Poisson morphism $\Phi:(\mathcal{X},\Lambda)\to (\mathcal{Y},\Pi)$ over $Spec(A)$ which induces $f$ so that induction holds for $n$. This completes the proof of Theorem \ref{mm1}.
\end{proof}

\begin{theorem}[Costability]\label{mm15}
Let $(X,\Lambda_0)$ and $(Y,\Pi_0)$ be two nonsingular projective Poisson varieties, $f:(X,\Lambda_0)\to (Y,\Pi_0)$ a Poisson morphism. Assume that
Assume that
\begin{enumerate}
\item $f^*:\mathbb{H}^1(Y,T_Y^\bullet)\to \mathbb{H}^1(X, f^*T_Y^\bullet)$ is surjective.
\item $f^*:\mathbb{H}^2(Y,T_Y^\bullet)\to \mathbb{H}^2(X, f^*T_Y^\bullet)$ is injective.
\end{enumerate}
Then the functor $f_X$ is smooth. In other words, for a flat Poisson deformation $p:(\mathcal{X},\Lambda)\to Spec(A)$ of $(X,\Lambda_0)$ over $Spec(A)$ for a local artinian $k$-algebra $A$ with the residue $k$, there exist a flat Poisson deformation  $q:(\mathcal{Y},\Pi)\to Spec( A)$ of $(Y,\Pi_0)$, and a Poisson map $\Phi:(\mathcal{X},\Lambda)\to (\mathcal{Y},\Pi)$ over $A$ which induces $f$.
\end{theorem}

\begin{proof}
We will prove by induction on the dimension $n$ of $A$. When $\dim_k A=1$, there is noting to prove. Now assume that the theorem holds for $\dim_k A \leq n-1$. Let $A$ be a local artinian $k$-algebra with $\dim_k A=n$. Assume that the maximal ideal $\mathfrak{m}$ of $A$ satisfy $\mathfrak{m}^{p-1}\ne 0$ and $\mathfrak{m}^p=0$. Choose $t\ne 0\in \mathfrak{m}^p$. Then $A/(t)\in \bold{Art}$ with $\dim_k A/(t)\leq n-1$ and $0\to (t)\to A\to A/(t)\to 0$ is a small extension.

Let $p:(\mathcal{X},\Lambda)\to Spec(A)$ be a flat Poisson deformation of $(X,\Lambda_0)$ over $Spec(A)$. Let $\mathcal{U}_i=\{U_i=Spec(B_i)\}$ be an affine open cover of $X$. Let $\bar{p}:(\bar{\mathcal{X}},\bar{\Pi}):=(\mathcal{X},\Lambda)\times_{Spec(A)} Spec(A/(t))\to Spec(A/(t))$ be the induced deformation over $Spec(A/(t))$. Then by the induction hypothesis, there exists a deformation $\bar{q}:(\bar{\mathcal{Y}},\bar{\Pi})\to Spec(A/(t))$ of $(Y,\Pi_0)$ and $\bar{\Psi}:(\bar{\mathcal{X}},\bar{\Lambda})\to(\bar{\mathcal{Y}},\bar{\Pi})$ which induces $f$.

Let $\mathcal{U}=\{U_i=Spec(B_i)\}_{i\in I}$ be an affine open cover of $X$ and $\mathcal{V}=\{V_i=Spec(C_i)\}_{i \in I}$ be an affine open cover of $Y$ such that $f(U_i)\subset V_i$. We have a Poisson isomorphism  $\bar{\theta}_i:(U_i\times Spec(A/(t)),\bar{\Lambda}_i)\to (\bar{\mathcal{X}}|_{U_i},\bar{\Lambda}|_{U_i})$, where $\bar{\Lambda}_i\in \Gamma(U_i,\wedge^2 T_X)\otimes_k A/(t)$. Then for each $i,j\in I$, $\bar{\theta}_{ij}:=\bar{\theta}_j^{-1}\bar{\theta}_i:(U_{ij}\times Spec(A/(t)),\bar{\Lambda}_j)\to (U_{ij}\times Spec(A/(t)), \bar{\Lambda}_i)$ is a Poisson isomorphism inducing the identity on $(U_{ij},\Lambda_0)$, which corresponds to a Poisson $A/(t)$-isomorphism $\bar{r}_{ij}:(\Gamma(U_{ij},\mathcal{O}_X)\otimes_k A/(t), \bar{\Lambda}_j)\to(\Gamma(U_{ij},\mathcal{O}_X)\otimes_k A/(t),\bar{\Lambda}_i)$. On the other hand, similarly we have a Poisson isomorphism $\bar{e}_i:(V_i\times Spec(A/(t)),\bar{\Pi}_i)\to (\bar{\mathcal{Y}}|_{V_i},\bar{\Pi}|_{V_i})$, where $\bar{\Pi}_i \in \Gamma(V_i, \wedge^2 T_Y)\otimes_k A/(t)$. Then for each $i,j\in I$, $\bar{e}_{ij}:=\bar{e}_j^{-1}\bar{e}_i:(V_{ij}\times Spec(A/(t)),\bar{\Pi}_i)\to (V_{ij}\times Spec(A/(t)),\bar{\Pi}_i)$ is a Poisson isomorphism inducing the identity on $(V_{ij},\Pi_0)$, which corresponds to a Poisson $A/(t)$-isomorphism $\bar{g}_{ij}:(\Gamma(V_{ij},\mathcal{O}_Y)\otimes_k A/(t),\bar{\Pi}_j)\to (\Gamma(V_{ij},\mathcal{O}_Y)\otimes_k A/(t),\bar{\Pi}_i)$. We set $\bar{\Psi}_i=e_i^{-1} \Psi\theta_i:(U_i\times Spec(A/(t)),\bar{\Lambda}_i)\to (V_i\times Spec(A/(t)) ,\bar{\Pi}_i)$ which corresponds to $\bar{\Phi}_i:(C_i\otimes_k A/(t),\bar{\Pi}_i)\to (B_i\otimes_k A/(t), \bar{\Lambda}_i)$.

Let $\Phi_i:C_i\otimes A\to B_i\otimes A$ be a lifting of $\bar{\Phi}_i$, and $g_{ij}:\Gamma(V_{ij},\mathcal{O}_Y)\otimes_k A \to \Gamma(V_{ij},\mathcal{O}_Y)\otimes_k A$ be a lifting of $\bar{g}_{ij}$, and $\Pi_i\in \Gamma(V_i,\wedge^2 T_Y)\otimes_k A$ be a lifting of $\bar{\Pi}_i$. On the other hand, let $r_{ij}:(\Gamma(U_{ij},\mathcal{O}_X)\otimes_k A,\Lambda_j)\to (\Gamma(U_{ij},\mathcal{O}_X)\otimes_k A, \Lambda_i)$ be a lifting of $\bar{r}_{ij}$ and $\bar{\Lambda}_i$ defining $(\mathcal{X},\Lambda)$. Then $g_{ij}g_{jk}g_{ki}=Id+t \xi_{ijk}$ for some $\xi_{ijk}\in \Gamma(V_{ijk},T_Y)$, $[\Pi_i,\Pi_i]=t S_i$ for some $S_i\in \Gamma(V_i,\wedge^2 T_Y)$, and there exists $\Pi_{ij}'$ such that $g_{ij}(\Pi_j(a,b))=(\Pi_i+t\Pi_{ij}')(g_{ij}(a),g_{ij}(b))$ for $a,b\in \Gamma(V_{ij},\mathcal{O}_Y)\otimes_k A$ (for the detail, see \cite{Kim16}). Consider the following diagram (which is not necessarily commutative nor Poisson except for $r_{ij}$)
\begin{center}
$\begin{CD}
(\Gamma(V_{ij},\mathcal{O}_Y)\otimes_k A,\Pi_j) @>g_{ij}>> ( \Gamma(V_{ij},\mathcal{O}_Y)\otimes_k A,\Pi_i)\\
@V\Phi_j VV @VV\Phi_i V\\
(\Gamma(U_{ij},\mathcal{O}_X)\otimes_k A \Lambda_j)@>r_{ij}>> (\Gamma(U_{ij},\mathcal{O}_X)\otimes_k A, \Lambda_i)\\
\end{CD}$
\end{center}
Then since $\bar{r}_{ij}\bar{\Phi}_j-\bar{\Phi}_i \bar{g}_{ij}=0$, we have $r_{ij}\Phi_j-\Phi_i g_{ij}= t E_{ij} $ for some $E_{ij}\in \Gamma(U_{ij},f^* T_Y)$. 
\begin{align*}
t(E_{ij}-E_{ik}+E_{jk})&=t(E_{ij}g_{jk}-E_{ik}+r_{ij}E_{jk})=r_{ij}\Phi_jg_{jk}-\Phi_ig_{ij}g_{jk}-r_{ik}\Phi_k+\Phi_i g_{ik}+r_{ij}r_{jk}\Phi_k-r_{ij}\Phi_j g_{jk}\\
&=-\Phi_i g_{ij}g_{jk}+\Phi_i g_{ik}=-\Phi_i(g_{ik}+t\xi_{ijk})+\Phi_i g_{ik}=-tf \xi_{ijk}
\end{align*}
Hence we have 
\begin{align}\label{mm5}
f^*(\{-\xi_{ijk}\})=\delta(\{E_{ij}\}).
\end{align}

On the other hand, we recall Notation \ref{notation12}.  Then since $\bar{\Phi}_i^*\bar{\Pi}_i-\bar{\Phi}_i\Lambda_i=0$, we have $\Phi_{i}^* \Pi_i -\Phi_i \Lambda_i=tP_i$ for some $P_i\in \Gamma(U_i, \wedge^2 f^* T_Y)$. Then for $a,b,c\in C$, since $[\Lambda_i,\Lambda_i]=0$, we have
\begin{align*}
&t\pi(P_i)(a,b,c)\\
&=t\Lambda_0(P_i(a,b),f(c))-t\Lambda_0(P_i(a,c),f(b))+t\Lambda_0(P_i(b,c),f(a))+tP_i(\Pi_0(a,b),c)-tP_i(\Pi_0(a,c),b)+tP_i(\Pi_0(b,c),a)\\
&=\Lambda_i((\Phi_{i}^*\Pi_i-\Phi_i\Lambda_i)(a,b),\Phi_i(c))-\Lambda_i((\Phi_{i}^*\Pi_i-\Phi_i\Lambda_i)(a,c),\Phi_i(b))+\Lambda_i((\Phi_{i}^*\Pi_i-\Phi_i\Lambda_i)(b,c),\Phi_i(a))\\
&+(\Phi_{i}^*\Pi_i-\Phi_i^*\Lambda_i)(\Pi_i(a,b),c)-(\Phi_{i}^*\Pi_i-\Phi_i\Lambda_i)(\Pi_i(a,c),b)+(\Phi_{i}^*\Pi_i-\Phi_i\Lambda_i)(\Pi_i(b,c),a)\\
&=\Lambda_i(\Phi_i(\Pi_i(a,b)),\Phi_i(c))-\Lambda_i((\Lambda_i(\Phi_i(a),\Phi_i(b)),\Phi_i(c))-\Lambda_i(\Phi_i(\Pi_i(a,c)),\Phi_i(b))+\Lambda_i((\Lambda_i(\Phi_i(a),\Phi_i(c)),\Phi_i(b))\\
&+\Lambda_i(\Phi_i(\Pi_i(b,c)),\Phi_i(a))-\Lambda_i((\Lambda_i(\Phi_i(b),\Phi_i(c)),\Phi_i(a))\\
&+\Phi_i(\Pi_i(\Pi_i(a,b)),c)-\Lambda_i(\Phi_i(\Pi_i(a,b)),\Phi_i(c))-\Phi_i(\Pi_i(\Pi_i(a,c)),b)+\Lambda_i(\Phi_i(\Pi_i(a,c)),\Phi_i(b))\\
&+\Phi_i(\Pi_i(\Pi_i(b,c)),a)-\Lambda_i(\Phi_i(\Pi_i(b,c)),\Phi_i(a))\\
&=\Phi_i(\frac{1}{2}[\Pi_i,\Pi_i](a,b,c))=tf^*(\frac{1}{2}S_i)(a,b,c)
\end{align*}
Hence we have 
\begin{align}\label{mm6}
\pi(P_i)=f^*(\frac{1}{2}S_i).
\end{align}
Lastly, we have, for $a,b\in \Gamma(V_{ij},\mathcal{O}_Y)$,
\begin{align*}
&t(P_i-P_j)=t(g_{ij}P_i-r_{ij}^*P_j)(a,b)=(\Phi_{i}^*\Pi_i-\Phi_i\Lambda_i)(g_{ij}(a),g_{ij}(b))-r_{ij}(\Phi_{j}^*\Pi_j-\Phi_j \Lambda_j)(a,b)\\
&=\Phi_i(\Pi_i(g_{ij}(a),g_{ij}(b)))-\Lambda_i(\Phi_i (g_{ij}(a)),\Phi_i ( g_{ij}(b)))-r_{ij}(\Phi_j(\Pi_j(a,b)))+r_{ij}(\Lambda_j(\Phi_j (a),\Phi_j(b)))\\
&=\Phi_i(g_{ij}(\Pi_j(a,b)))-\Phi_i(t \Pi_{ij}'(g_{ij}(a),g_{ij}(b)))-\Lambda_i((r_{ij}\Phi_j-tE_{ij})(a),(r_{ij}\Phi_j-tE_{ij})(b))\\
&-\Phi_i (g_{ij}(\Pi_j(a,b)))-t E_{ij}(\Pi_j(a,b))+\Lambda_i(r_{ij}(\Phi_j(a)),r_{ij}(\Phi_j(b)))\\
&=-t f(\Pi_{ij}'(a,b))+t\Lambda_0(f(a),E_{ij}(b))+t\Lambda_0(E_{ij}(a),f(b))-t E_{ij}(\Pi_0(a,b))\\
&=-tf(\Pi_{ij}'(a,b))+\pi(E_{ij})(a,b)
\end{align*}
Hence we have
\begin{align}\label{mm7}
\delta(\{P_i\})+\pi(\{E_{ij}\})=f^*(\Pi_{ij}')
\end{align}
We note that $(\{\frac{1}{2}S_i\}, \{\Pi_{ij}'\},\{ -\xi_{ijk}\})$ is a $2$-cocycle of $T_Y^\bullet$ (see \cite{Kim16}). Since $f^*:\mathbb{H}^2(X,T_X^\bullet)\to \mathbb{H}^2(Y, f^* T_Y^\bullet)$ is injective, from (\ref{mm5}),(\ref{mm6}), and (\ref{mm7}), there exists $\{d_{ij}\}\in C^1(\mathcal{U}, T_Y)$ and $\{\lambda_i\}\in C^0(\mathcal{U},\wedge^2 T_Y)$ such that $d_{ki}+ d_{jk}+d_{ij}=-\xi_{ijk}$, $\lambda_j -\lambda_i+[\Pi_0, d_{ij}]=\Pi_{ij}'$, and $[\Pi_0, \lambda_i]=\frac{1}{2}S_i$. Then we claim that $\{g_{ij}+td_{ij}\},\{\Pi_i-t\lambda_i\}$ defines a flat Poisson deformation of $(\mathcal{Y}',\Lambda')$ of $(Y,\Pi_0)$ inducing $(\bar{\mathcal{Y}},\bar{\Pi})$. Indeed,
 \begin{align} \label{mm8}
 &(g_{ij}+td_{ij})(g_{jk}+td_{jk})(g_{ki}+td_{ki})=Id+t\xi_{ijk}+td_{ik}+t d_{jk} +td_{ij}=Id\\
 &[\Pi_i- t\lambda_i,\Pi_i-t\lambda_i]=tS_i-t2[\Pi_0, \lambda_i]=0 \notag \\
 & (g_{ij}+td_{ij})((\Pi_j-t\lambda_j)(a,b))-(\Pi_i-t\lambda_i)((g_{ij}+td_{ij})(a),(g_{ij}+td_{ij})(b)) \notag \\
&=(\Pi_i+t\Lambda_{ij}')(g_{ij}(a),g_{ij}(b))-t\lambda_j(a,b)+td_{ij}(\Pi_0(a,b))-\Pi_i(g_{ij}(a), g_{ij}(b))-t\Pi_0(d_{ij}(a),b)-t\Pi_0(a, d_{ij}(b))+t\lambda_i(a,b) \notag \\
&=t\Pi_{ij}'(a,b)-t\lambda_j(a,b)-t[\Pi_0, d_{ij}](a,b)+t\lambda_i(a,b)=0 \notag
 \end{align}

Consider the following diagram (which is not necessarily commutative nor Poisson except for $r_{ij}$)
\begin{center}
$\begin{CD}
(\Gamma(V_{ij},\mathcal{O}_Y)\otimes_k A ,\Pi_j-t\lambda_j)@>g_{ij}+td_{ij}>> (\Gamma(V_{ij},\mathcal{O}_Y)\otimes_k A,\Pi_i-t\lambda_i)\\
@V\Phi_j VV @VV\Phi_i V\\
(\Gamma(U_{ij},\mathcal{O}_X)\otimes_k A,\Lambda_j)@>r_{ij}>>( \Gamma(U_{ij},\mathcal{O}_X)\otimes_k A,\Lambda_i)\\
\end{CD}$
\end{center}
Then $r_{ij}\Phi_j-\Phi_i (g_{ij}+t d_{ij})=tG_{ij}$ for some $G_{ij}\in \Gamma(U_{ij}, f^* T_Y)$ so that $G_{ij}= E_{ij}-f^*d_{ij}$. $G_{ij}-G_{ik}+G_{jk}= E_{ij}-f^*d_{ij}-E_{ik}+f^*d_{ik}+ E_{jk}-f^*d_{jk}=f^*(-\xi_{ijk})+f^*(\xi_{ijk})=0$. On the other hand, $\Phi_i^*(\Pi_i-t\lambda_i)-\Phi_i \Lambda_i=t L_i$ for some $L_i\in \Gamma(U_i \wedge^2 f^* T_Y)$. Then $P_i-f^*\lambda_i=L_i$ so that $\pi(L_i)=\pi(P_i)-f^*([\Pi_0, \lambda_i])=0$. Lastly, we have $\pi(G_{ij})+L_j-L_i=\pi(E_{ij})-f^*[\Pi_0, d_{ij}]+ P_j-f^*\lambda_j-P_i+f^*\lambda_i=0$. Hence $(\{G_{ij}\},\{L_i\})$ defines a $1$-cocycle of $f^* T_Y^\bullet$. Hence since $f^*:\mathbb{H}^1(Y, T_Y^\bullet) \to \mathbb{H}^1(X, f^* T_Y^\bullet)$ is surjective, there exist a $1$-cocycle
$(\{\chi_{ij}\},\{H_i\})\in C^1(\mathcal{V},T_Y)\oplus C^0(\mathcal{V},\wedge^2 T_Y)$ of $T_Y^\bullet$, and $\{\tau_i\}\in C^0(\mathcal{U},f^* T_Y)$ such that $G_{ij}-f^*(\chi_{ij})=\tau_i-\tau_j$, and $L_i-f^*(H_i)=\pi(\tau_i)$.

Then we claim that $\{g_{ij}+t(d_{ij}+\chi_{ij})\}$ and $\{\Pi_i-t(\lambda_i+H_i)\}$ define a flat Poisson deformation $(\mathcal{Y}, \Pi)$ of $(Y,\Pi_0)$ inducing $(\bar{\mathcal{Y}},\bar{\Pi})$. Indeed, from (\ref{mm8}), we have
\begin{align*}
&(g_{ij}+td_{ij}+t\chi_{ij})(g_{jk}+td_{jk}+t\chi_{jk})(g_{ki}+td_{ki}+\chi_{ki})=Id+t(\chi_{ij}+\chi_{jk}+\chi_{ki})=Id,\\
&[\Pi_i-t\lambda_i-tH_i,\Pi_i-t\lambda_i-tH_i]= -2t[\Pi_0,H_i]=0\\
&(g_{ij}+td_{ij}+t\chi_{ij})((\Pi_j-t\lambda_j-tH_j)(a,b))-(\Pi_i-t\lambda_i-tH_i)((g_{ij}+td_{ij}+t\chi_{ij})(a),(g_{ij}+td_{ij}+t\chi_{ij})(b))\\
&=-tH_j(a,b)+t\chi_{ij}(\Pi_0(a,b))-t\Pi_0(\chi_{ij}(a),b)-t\Pi_0(a,\chi_{ij}(b))+tH_i(a,b)=-t(H_j-H_i+[\Pi_0, \chi_{ij}])=0.
\end{align*}

We have the following commutative diagram
\begin{center}
$\begin{CD}
(\Gamma(V_{ij},\mathcal{O}_Y)\otimes_k A,  \Pi_j-t(\lambda_j+H_j)) @>g_{ij}+td_{ij}+t\chi_{ij}>> ( \Gamma(V_{ij},\mathcal{O}_Y)\otimes_k A ,\Pi_i-(t\lambda_i+H_i))\\
@V\Phi_j+t\tau_j VV @VV\Phi_i+t\tau_iV\\
(\Gamma(U_{ij},\mathcal{O}_X)\otimes_k A, \Lambda_j) @>r_{ij}>> (\Gamma(U_{ij},\mathcal{O}_X)\otimes_k A,\Lambda_i)
\end{CD}$
\end{center}
Indeed, $(\Phi_i+t\tau_i)(g_{ij}(a)+td_{ij}(a)+t\chi_{ij}(a))=r_{ij}(\Phi_j(a))-tG_{ij}(a)-tf(d_{ij}(a))+tf(d_{ij}(a))+tf(\chi_{ij}(a))+t\tau_i(a)=r_{ij}(\Phi_j(a)+t\tau_j(a))$. Lastly we claim that $\Phi_j+t\tau_j$ is a Poisson homomorphism. Indeed, 
\begin{align*}
&(\Phi_j+t\tau_j)(\Pi_j(a,b)-t\lambda_j(a,b)-tH_j(a,b))-\Lambda_j(\Phi_j(a)+t\tau_j(a),\Phi_j(b)+t\tau_j(b))\\
&=\Lambda_j(\Phi_j(a),\Phi_j(b))+tL_j(a,b)-tf(\lambda_j(a,b))-tf(H_j(a,b))+t\tau_j(\Pi_j(a,b))-\Lambda_j(\Phi_j(a),\Phi_j(b))-\Lambda_0(\tau_j(a),b)-\Lambda_0(a, \tau_j(b))\\
&=t(L_j-\pi(\tau_j)-f^*(H_j))(a,b)=0
\end{align*}
Hence there exists a flat Poisson deformation $(\mathcal{Y},\Pi)$ over $Spec(A)$ of $(Y,\Pi_0)$, and a Poisson morphism $\Phi:(\mathcal{X},\Lambda)\to (\mathcal{Y},\Pi)$ over $Spec(A)$ which induces $f$ so that induction holds for $n$. This completes the proof of Theorem \ref{mm15}.

\end{proof}

\begin{corollary} 
Let $(Y,\Pi_0)$ be a nonsingular projective Poisson variety, $(X,\Lambda_0)$ be a nonsingular Poisson subvariety of $(Y,\Pi_0)$, and let $p:(\mathcal{X},\Lambda)\to Spec(A)$ be a flat Poisson deformations of $(X,\Lambda_0)$ over $A\in\bold{Art}$. Assume that $\mathbb{H}^2(Y,T_Y(-X)^\bullet)=0$, where $ T_Y(-X)^\bullet=ker(T_Y^\bullet \to T_Y^\bullet|_X)$. Then there exist a flat Poisson deformation $q:(\mathcal{Y},\Pi)\to Spec(A)$ of deformations of $(Y,\Pi_0)$ over $A$, and an Poisson embedding $(\mathcal{X},\Lambda)\to (\mathcal{Y},\Pi)$ over $Spec(A)$ which induces the natural embedding $(X,\Lambda_0)\to (Y,\Pi_0)$ over $0\in N$.
\end{corollary}

\section{Comparison of two complexes of sheaves associated with the normal bundle }\label{appendix4}

Let $X$ be a holomorphic Poisson submanifold  of a holomorphic Poisson manifold $(Y,\Lambda_0)$, and let $i:X\to (Y,\Lambda_0)$ be a Poisson embedding. In \cite{Kim17}, we defined a complex of sheaves associated with the normal bundle
\begin{align*}
\mathcal{N}_{X/Y}^\bullet:\mathcal{N}_{X/Y}\xrightarrow{\nabla} \mathcal{N}_{X/Y}\otimes T_Y|_X\xrightarrow{\nabla} \mathcal{N}_{X/Y}\otimes \wedge^2 T_Y|_X\xrightarrow{\nabla} \cdots
\end{align*}

We have another complex of sheaves associated with the normal bundle $T_Y|_X$ induced by 
\begin{align*}
\mathcal{N}_i:=coker(T_X^\bullet\hookrightarrow T_Y|_X^\bullet)
\end{align*}
where
\begin{center}
$\begin{CD}
T_X^\bullet: @.T_X@>-[-,\Lambda_0|_X]>> \wedge^2 T_X @>-[-,\Lambda_0|_X]>>\wedge^3 T_X @> -[-,\Lambda_0|_X] >> \cdots \\
@.@V VV @VVV @VVV\\
T_Y^\bullet|_X:@. T_Y|_X@>\pi_i=-[-.\Lambda_0]|_X>>\wedge^2 T_Y|_X@>\pi_i=-[-,\Lambda_0]|_X>>\wedge^3T_Y|_X@>\pi_i=-[-,\Lambda_0]|_X>>\cdots
\end{CD}$
\end{center}

We will compare their $0$-th and first hypercohomology groups. We will define a map $\varphi:\mathcal{N}_i^\bullet \to\mathcal{N}_{X/Y}^\bullet$ extending the identity $\mathcal{N}_{X/Y}\to \mathcal{N}_{X/Y}$.
\begin{center}
$\begin{CD}
\mathcal{N}_i^\bullet: @.\mathcal{N}_i^0=\mathcal{N}_{X/Y}@>>> \mathcal{N}_i^1 @>>>\mathcal{N}_i^2@>>> \cdots \\
@.@V\phi_0=id VV @VV\phi_1V @V\phi_2VV\\
\mathcal{N}_{X/Y}^\bullet: @.\mathcal{N}_{X/Y}@>\nabla>>\mathcal{N}_{X/Y}\otimes T_Y|_X@>\nabla>> \mathcal{N}_{X/Y}\otimes \wedge^2 T_Y|_X@>\nabla>>\cdots
\end{CD}$
\end{center}
Let $\mathcal{U}=\{W_i\}$ be a open covering of $Y$ such that $W_i$ is a polycylinder with a local coordinate $(w_i,z_i)=(w_i^1,...,w_i^r,z_1^1,...,z_1^d)$ and $U_i:=X\cap W_i$ is determined by $w_i^1=\cdots=w_i^r=0$. On the intersection $W_i\cap W_k$, the transition functions are given by $w_i^\alpha=f_{ik}^\alpha(w_k,z_k),\alpha=1,...,r, z_i^\lambda=g_{ik}^\lambda(w_k,z_k),\lambda=1,...,d$. Since $f_{ik}^\alpha(0,z_k)=0$, $w_i^\alpha$ is of the form $w_i^\alpha=\sum_{\beta=1}^r w_k^\beta F_{ik\beta}^\alpha(w_k,z_k)$. Since $X$ is a holomorphic Poisson manifold of $Y$, $[\Lambda_0, w_i^\alpha]$ is of the form $[\Lambda_0,w_i^\alpha]=\sum_{\beta=1}^rw_i^\beta T_{i\alpha}^\beta(w_i,z_i)$ for some $T_{i\alpha}^\beta(w_i,z_i)\in \Gamma(W_i, T_Y)$.

\begin{lemma}\label{qq1}
Let $g\in \Gamma(U_i, \wedge^2 T_Y|_X)$. Then $g\in \Gamma(U_i, \wedge^2 T_X)$ if and only if $[g,w_i^\alpha]|_{w_i=0}=0$ for all $\alpha=1,...,r$.
\end{lemma}

\begin{proof}
It is clear that $g\in \Gamma(U_i,\wedge^2 T_X)$ implies $[g,w_i^\alpha]|_{w_i=0}$ for all $\alpha=1,...,r$. Now assume that $[g,w_i^\alpha]|_{w_i=0}$ for all $\alpha=1,...,r$. Write $g$ in the following form
\begin{align*}
g=\sum_{r,s=1}^dG_{rs}\frac{\partial}{\partial z_i^r}\wedge \frac{\partial}{\partial z_i^s}+\sum_{p=1}^d\sum_{q=1}^rP_{pq}\frac{\partial}{\partial z_i^p}\wedge \frac{\partial}{\partial w_i^q}+\sum_{a,b=1}^rQ_{ab}\frac{\partial}{\partial w_i^a}\wedge \frac{\partial}{\partial w_i^b}
\end{align*}
where $G_{rs},P_{pq}, Q_{ab}\in \Gamma(U_i, \mathcal{O}_X)$ with $Q_{ab}=-Q_{ba}$. We claim that $P_{pq}=Q_{ab}=0$ for all $a,b,p,q$. Indeed, for any $\alpha$, 
\begin{align*}
[g,w_i^\alpha]|_{w_i=0}=\sum_{p=1}^d-P_{p\alpha}\frac{\partial}{\partial z_p}+2\sum_{a=1}^r Q_{\alpha b}\frac{\partial}{\partial w_b}=0
\end{align*}
Then $P_{p\alpha}=0$ and $Q_{\alpha b}=0$ for all $p,b,\alpha$. Hence we get the claim.
\end{proof}

First we will define a map $\tilde{\varphi}:T_Y^\bullet |_X  \to \mathcal{N}_{X/Y}^\bullet$. We note that $\Gamma(U_i, \mathcal{N}_{X/Y}\otimes \wedge^{p} T_Y|_X)\cong \oplus^r \Gamma(U_i, \wedge^{p}T_Y|_X)$. We define $\tilde{\varphi}$ locally in the following way:
\begin{align*}
 \Gamma(U_i,\wedge^{p+1} T_Y|_X)&\xrightarrow{\tilde{\varphi}_i} \oplus^r\Gamma(U_i, \wedge^pT_Y|_X)\\
  g &\mapsto ((-1)^p[g,w_1]|_X,...,(-1)^p[g,w_r]|_X)
 \end{align*}
 We show that $\tilde{\varphi}$ is well-defined. It is sufficient to show that the following diagram commutes.
 \[\xymatrix{
& \Gamma(U_k\cap U_i,\wedge^{p+1} T_Y|_X) \ar[dl]_{\tilde{\varphi_k}} \ar[dr]^{\tilde{\varphi_i}} \\
\oplus^r \Gamma(U_k \cap U_i,\wedge^p T_Y|_X) \ar[rr]^{\cong} & & \oplus^r \Gamma(U_i\cap U_k, \wedge^p T_Y|_X)
}\]
Indeed, for $g\in \Gamma(U_i,\wedge^{p+1}T_Y|_X)$, we have 
\begin{align*}
(-1)^p[g,w_i^\alpha]|_{w_i=0}=\sum_{\beta=1}^r(-1)^p[g,w_k^\beta F_{ik\beta}^\alpha(w_k,z_k)]|_{w_k=0}=\sum_{\beta=1}^r(-1)^p[g,w_k^\beta]|_{w_k=0}F_{ik\beta}^\alpha(0,z_k)
\end{align*}

Next we show that the following diagram commutes
\begin{center}
$\begin{CD}
\Gamma(U_i, \wedge^{p+2} T_Y|_X)@>\tilde{\varphi}_i>> \oplus^r \Gamma(U_i, \wedge^{p+1} T_Y|_X)\\
@A-[-,\Lambda_0]|_XAA @AA\nabla A \\
\Gamma(U_i, \wedge^{p+1} T_Y|_X) @>\tilde{\varphi}_i>> \oplus^r \Gamma(U_i, \wedge^{p} T_Y|_X)
\end{CD}$
\end{center}
For $g\in \Gamma(U_i, \wedge^{p+1} T_Y|_X)$, we have
\begin{align*}
\tilde{\varphi}_i(-[g,\Lambda_0]|_X)=\sum_{\alpha=1}^r (-(-1)^{p+1}[[g,\Lambda_0],w_i^\alpha]|_{w_i=0})e_i^\alpha
\end{align*}
We note that
\begin{align*}
[[g,\Lambda_0], w_i^\alpha]|_{w_i=0}&=[g, [\Lambda_0,w_i^\alpha]]|_{w_i=0}-(-1)^p[\Lambda_0,[g,w_i^\alpha]]=\sum_{\beta=1}^r[g, w_i^\beta T_{i\alpha}^\beta(w_i,z_i)]|_{w_i=0}-[[g,w_i^\alpha],\Lambda_0]|_{w_i=0}\\
&=\sum_{\beta=1}^r (-1)^p [g,w_i^\beta]\wedge T_{i\alpha}^\beta (0,z_i)-[[g,w_i^\alpha],\Lambda_0]|_{w_i=0}
\end{align*}
Hence we get
\begin{align*}
\tilde{\varphi}_i(-[g,\Lambda_0]|_{w_i=0})=\sum_{\alpha=1}^r\left((-1)^{p+1}[[g,w_i^\alpha],\Lambda_0]|_{w_i=0}+\sum_{\beta=1}^r [g,w_i^\beta]|_{w_i=0}\wedge T_{i\alpha}^\beta(0,z_i) \right)e_i^\alpha
\end{align*}

On the other hand,
\begin{align*}
\nabla(\tilde{\varphi}_i(g))=\nabla(\sum_{\alpha=1}^r(-1)^p[g,w_i^\alpha]e_i^\alpha)=\sum_{\alpha=1}^r \left( -(-1)^p[[g,w_i^\alpha],\Lambda_0]|_{w_i=0}+(\sum_{\beta=1}^r [g,w_i^\beta]|_{w_i=0}\wedge T_{i\alpha}^\beta(0,z_i)  \right)e_i^\alpha
\end{align*}
 
Lastly, since $[h,w_i]=0$ for any $h\in \Gamma(U_i,\wedge^{p+1}T_X)$. $\tilde{\varphi}$ induces $\varphi:\Gamma(U_i, \mathcal{N}_i^{p})\to \oplus^r \Gamma(U_i,\wedge^p T_Y|_X)$ so that we have $\varphi:\mathcal{N}_i^\bullet \to \mathcal{N}_{X/Y}^\bullet$.

\begin{lemma}\label{aa10}
$\varphi$ induces $\varphi^0=id:\mathbb{H}^0(X, \mathcal{N}_i^\bullet)\cong \mathbb{H}^0(X,\mathcal{N}_{X/Y}^\bullet)$ and an injection $\varphi^1:\mathbb{H}^1(X,\mathcal{N}_i^\bullet)\to \mathbb{H}^1(X, \mathcal{N}_{X/Y}^\bullet)$. In particular, if $\mathbb{H}^1(X, \mathcal{N}_{X/Y}^\bullet)=0$, then we have $\mathbb{H}^1(X, \mathcal{N}_i^\bullet)=0$.
\end{lemma}

\begin{proof}
We show that $\mathbb{H}^0(X,\mathcal{N}^i)\cong \mathbb{H}^0(X,\mathcal{N}_{X/Y}^\bullet)$. Let $\mathcal{U}=\{U_i\}$ be the covering of $X$.  Every element of $\mathbb{H}^0(X, \mathcal{N}_i^\bullet)$ is represented by some $g=\{g_i\}\in \mathcal{C}^0(\mathcal{U}, T_Y|_X)$ with $g_i\in \Gamma(U_i,T_Y|_X)$. Then by Lemma \ref{qq1}, $-[g_i, \Lambda_0]|_{w_i=0}\in \Gamma(U_i,\wedge^2 T_X)$ if and only if $-[[g_i, \Lambda_0],w_i^\alpha]_{w_i=0}=0$ for all $\alpha=1,...,r$ if and only if
\begin{align*}
0=&[[ g_i,\Lambda_0], w_i^\alpha]|_{w_i=0}=[g_i,[\Lambda_0,w_i^\alpha]]|_{w_i=0}-[\Lambda_0, [g_i,w_i^\alpha]]|_{w_i=0}=[g_i, \sum_{\beta=1}^r w_i^\beta T_{i\alpha}^\beta(w_i,z_i)]|_{w_i=0}-[[g_i ,w_i^\alpha], \Lambda_0]_{w_i=0}\\
=&\sum_{\beta=1}^r[g_i ,w_i^\alpha]|_{w_i=0} T_{i\alpha}^\beta(0,z_i)-[[g_i ,w_i^\alpha],\Lambda_0]|_{w_i=0}
\end{align*}

Next let us show that $\varphi^1$ is injective. Every element of $\mathbb{H}^1(X,\mathcal{N}_i^\bullet)$ is represented by some $(h=\{h_{ij}\},g=\{g_i\})\in C^0(\mathcal{U}, T_Y|_X)\oplus C^1(\mathcal{U},\wedge^2 T_Y|_X)$. Assume that $([h_{ij},w_i^1]|_{w_i=0},...,([h_{ij},w_i^r]|_{w_i=0})\in \bigoplus^r\Gamma(U_{ij},\mathcal{O}_X)$, and  $(-[g_i,w_i^1]|_{w_i=0},...,-[g_i,w_i^r]|_{w_i=0})\in \bigoplus^r \Gamma(U_i, T_Y|_X)$ defines $0$ in $\mathbb{H}^1(X,\mathcal{N}_{X/Y}^\bullet)$ so that there exist $(f_i^1,...,f_i^r)\in\oplus^r \Gamma(U_i, \mathcal{O}_X)$ such that $-[g_i,w_i^\alpha]|_{w_i=0}=-[f_i^\alpha,\Lambda_0]|_{w_i=0}+\sum_{\beta=1}^r f_i^\beta T_{i\alpha}^\beta(0,z_i)$. Then we claim that $g_i-(-[\sum_{\gamma=1}^r f_i^\gamma\frac{\partial}{\partial w_i^\alpha},\Lambda_0]|_{w_i=0})\in \Gamma(U_i, \wedge^2 T_X)$, equivalently, by Lemma \ref{qq1},
\begin{align*}
[g_i,w_i^\alpha]|_{w_i=0}+[[\sum_{\gamma=1}^r f_i^\gamma \frac{\partial}{\partial w_i^\gamma},\Lambda_0],w_i^\alpha]|_{w_i=0}=0,\,\,\,\,\,\,\text{for all $\alpha=1,...,r$.}
\end{align*}
We note that $[\sum_{\gamma=1}^r f_i^\gamma\frac{\partial}{\partial w_i^\gamma},\Lambda_0]=\sum_{\gamma=1}^r [f_i^\gamma,\Lambda_0] \wedge \frac{\partial}{\partial w_i^\gamma}+\sum_{\alpha=1}^r f_i^\gamma [\frac{\partial}{\partial w_i^\gamma},\Lambda_0]$. Then we have
\begin{align*}
[[\sum_{\gamma=1}^r f_i^\gamma\frac{\partial}{\partial w_i^\gamma},\Lambda_0],w_i^\alpha]|_{w_i=0}&=\sum_{\gamma=1}^r [[f_i^\gamma,\Lambda_0],w_i^\alpha]|_{w_i=0}\wedge \frac{\partial}{\partial w_i^\gamma}-[f_i^\alpha,\Lambda_0]|_{w_i=0}+\sum_{\alpha=1}^r f_i^\gamma [[\frac{\partial}{\partial w_i^\gamma},\Lambda_0],w_i^\alpha]|_{w_i=0}\\
&= -[f_i^\alpha,\Lambda_0]|_{w_i=0}+\sum_{\alpha=1}^r f_i^\gamma [[\frac{\partial}{\partial w_i^\gamma},\Lambda_0],w_i^\alpha]|_{w_i=0}
=-[f_i^\alpha, \Lambda_0]|_{w_i=0}+\sum_{\alpha=1}^r f_i^\gamma T_{i\alpha}^\gamma(0,z_i)\\
&=-[g_i, w_i^\alpha]|_{w_i=0}.
\end{align*}
since $[[f_i^\gamma, \Lambda_0],w_i^\alpha]|_{w_i=0}=-[[\Lambda_0, w_i^\alpha], f_i^\gamma]|_{w_i=0}=0$, and $[[\frac{\partial}{\partial w_i^\gamma},\Lambda_0],w_i^\alpha]|_{w_i=0}=[\frac{\partial}{\partial w_i^\gamma},[\Lambda_0, w_i^\alpha]]|_{w_i=0}-[\Lambda_0,\frac{\partial w_i^\alpha}{\partial w_i^\gamma}]|_{w_i=0}=\sum_{\beta=1}^r [\frac{\partial}{\partial w_i^\gamma} , w_i^\beta T_{i\alpha}^\beta(w_i,z_i)]|_{w_i=0}=T_{i\alpha}^\gamma(0,z_i)$.
Hence we get the claim so that $(h,g)$ defines $0$ in $\mathbb{H}^1(X,\mathcal{N}_i^\bullet)$.

\end{proof}

\bibliographystyle{amsalpha}
\bibliography{References-Res2}

\providecommand{\bysame}{\leavevmode\hbox to3em{\hrulefill}\thinspace}
\providecommand{\MR}{\relax\ifhmode\unskip\space\fi MR }
\providecommand{\MRhref}[2]{%
  \href{http://www.ams.org/mathscinet-getitem?mr=#1}{#2}
}
\providecommand{\href}[2]{#2}
\begin{thebibliography}{LGPV13}

\bibitem[Hor74]{Hor74}
Eiji Horikawa, \emph{On deformations of holomorphic maps. {II}}, J. Math. Soc.
  Japan \textbf{26} (1974), 647--667. \MR{0352540 (50 \#5027)}

\bibitem[Hor75]{Hor73}
\bysame, \emph{On deformations of holomorphic maps}, Manifolds---{T}okyo 1973
  ({P}roc. {I}nternat. {C}onf., {T}okyo, 1973), Univ. Tokyo Press, Tokyo, 1975,
  pp.~383--388. \MR{0372254 (51 \#8470)}

\bibitem[Hor76]{Hor76}
\bysame, \emph{On deformations of holomorphic maps. {III}}, Math. Ann.
  \textbf{222} (1976), no.~3., 275--282. \MR{0417458 (54 \#5508)}

\bibitem[Kim14a]{Kim16}
Chunghoon Kim, \emph{Deformations of nonsingular {P}oisson varieties and
  {P}oisson invertible sheaves}, preprint (2014).

\bibitem[Kim14b]{Kim15}
\bysame, \emph{Theorem of existence and completeness for holomorphic {P}oisson
  structures}, preprint (2014).

\bibitem[Kim15]{Kim17}
\bysame, \emph{Deformations of compact holomorphic {P}oisson submanifolds},
  preprint (2015).

\bibitem[Kod05]{Kod05}
Kunihiko Kodaira, \emph{Complex manifolds and deformation of complex
  structures}, english ed., Classics in Mathematics, Springer-Verlag, Berlin,
  2005, Translated from the 1981 Japanese original by Kazuo Akao. \MR{2109686
  (2005h:32030)}

\bibitem[LGPV13]{Lau13}
Camille Laurent-Gengoux, Anne Pichereau, and Pol Vanhaecke, \emph{Poisson
  structures}, Grundlehren der Mathematischen Wissenschaften [Fundamental
  Principles of Mathematical Sciences], vol. 347, Springer, Heidelberg, 2013.
  \MR{2906391}

\bibitem[Ser06]{Ser06}
Edoardo Sernesi, \emph{Deformations of algebraic schemes}, Grundlehren der
  Mathematischen Wissenschaften [Fundamental Principles of Mathematical
  Sciences], vol. 334, Springer-Verlag, Berlin, 2006. \MR{2247603
  (2008e:14011)}

\end{thebibliography}

\end{document}